\documentclass[11pt]{amsbook}

\usepackage{amsfonts, amstext, amsmath, amsthm, amscd, amssymb}
\usepackage{epsfig, graphics, psfrag, overpic, color}
\usepackage{enumerate}
 \usepackage[
 pdfborderstyle={},
 pdfborder={0 0 0},
 pagebackref]{hyperref}

 \renewcommand*{\backref}[1]{}
 \renewcommand*{\backrefalt}[4]{
   \ifcase #1 %
    [No citations.]%
   \or
    [#2]%
   \else
    [#2]%
   \fi
 }

\let\oldmarginpar\marginpar
\renewcommand\marginpar[1]{\oldmarginpar[\raggedleft\footnotesize #1]%
{\raggedright\footnotesize #1}}

\renewcommand{\setminus}{{\smallsetminus}}

\newcommand{\RR}{{\mathbb{R}}}
\newcommand{\ZZ}{{\mathbb{Z}}}
\newcommand{\NN}{{\mathbb{N}}}

\newcommand{\QQ}{{\mathbb{Q}}}
\newcommand{\CalD}{{\mathcal{W}}}

\newcommand{\bdy}{{\partial}}
\newcommand{\guts}{{\rm guts}}
\newcommand{\cut}{{\backslash \backslash}}
\newcommand{\clock}{{\phi}}
\newcommand{\vol}{{\rm vol}}

\newcommand{\pitos}{{PITOS}}
\newcommand{\nsep}{{n_{\rm sep}}}
\newcommand{\negeul}{{\chi_{-}}}

\newcommand{\abs}[1]{{\left\vert #1 \right\vert}}

\newcommand{\GA}{{\mathbb{G}_A}}
\newcommand{\GB}{{\mathbb{G}_B}}
\newcommand{\GRA}{{\mathbb{G}'_A}}
\newcommand{\GRB}{{\mathbb{G}'_B}}
\newcommand{\G}{{\mathbb{G}}}

\def\co{\colon\thinspace}

\theoremstyle{plain}
\newtheorem{theorem}{Theorem}
\newtheorem{corollary}[theorem]{Corollary}
\newtheorem{lemma}[theorem]{Lemma}

\newtheorem{prop}[theorem]{Proposition}

\newtheorem*{namedtheorem}{\theoremname}
\newcommand{\theoremname}{testing}
\newenvironment{named}[1]{\renewcommand{\theoremname}{#1}\begin{namedtheorem}}{\end{namedtheorem}}

\theoremstyle{definition}
\newtheorem{define}[theorem]{Definition}
\newtheorem{remark}[theorem]{Remark}
\newtheorem{notation}[theorem]{Notation}
\newtheorem{example}[theorem]{Example}
\newtheorem{question}[theorem]{Question}
\newtheorem{problem}[theorem]{Problem}

\numberwithin{figure}{chapter}
\numberwithin{theorem}{chapter}
\numberwithin{section}{chapter}
\makeindex

\begin{document}

\title[Guts of surfaces and the colored Jones polynomial]{Guts of surfaces and \\ the colored Jones polynomial}
\author[D. Futer]{David Futer}
\author[E. Kalfagianni]{Efstratia Kalfagianni}
\author[J. Purcell]{Jessica S. Purcell}

\maketitle

\centerline{ \bf \Large{Authors}}
\vspace {0.8in}

\noindent David Futer \\
Department of Mathematics \\
Temple University \\
1805 North Broad Street \\
Philadelphia, PA 19122 \\
USA	\\
\email{ Email: {\tt dfuter@temple.edu}} \\

\vspace{0.4in}

\noindent Efstratia Kalfagianni \\
Department of Mathematics\\ 
Michigan State University\\ 
619 Red Cedar Road\\
East Lansing, MI 48824\\
USA \\
 \email{Email: {\tt kalfagia@math.msu.edu}}\\


\vspace{0.4in}

\noindent Jessica S. Purcell \\
Department of Mathematics \\ 
275 TMCB\\
 Brigham Young  University \\ 
Provo, UT 84602 \\
USA	\\
\email{ Email: {\tt  jpurcell@math.byu.edu}}\\

\newpage
  
\centerline{ \bf \Large{Preface}}
	
\vspace{0.6in}
	


 Around 1980, W.~Thurston proved that every knot complement satisfies the
geometrization conjecture: it decomposes into pieces that admit
locally homogeneous geometric structures.  In addition, he proved that the complement of
any non-torus, non-satellite knot admits a complete hyperbolic
metric which, by the Mostow--Prasad rigidity theorem, is is necessarily
unique up to isometry. As a result, geometric
information about a knot complement, such as its volume, gives
topological invariants of the knot.

\medskip

Since the mid-1980's, knot theory has also been invigorated by ideas from
quantum physics, which have led to powerful and subtle knot
invariants, including the Jones polynomial and its relatives, the
\emph{colored} Jones polynomials.  Topological quantum field theory
predicts that these quantum invariants are very closely connected to
geometric structures on knot complements, and particularly to
hyperbolic geometry.  The \emph{volume conjecture} of
R. Kashaev, H.~Murakami, and J.~Murakami, which asserts that the volume of a hyperbolic
knot is determined by certain asymptotics of colored Jones
polynomials, fits into the context of these predictions.  Despite
compelling experimental evidence, these conjectures are currently
verified for only a few examples of hyperbolic knots.

\medskip

This monograph initiates a systematic study of relations between quantum
and geometric knot invariants.  Under mild diagrammatic hypotheses
that arise naturally in the study of knot polynomial invariants ($A$--
or $B$--adequacy), we derive direct and concrete relations between
colored Jones polynomials and the topology of incompressible spanning
surfaces in knot and link complements.  We prove that the growth of
the degree of the colored Jones polynomials is a boundary slope of an
essential surface in the knot complement, and that certain
coefficients of the polynomial measure how far this surface is from
being a fiber in the knot complement. In particular, the surface is a
fiber if and only if a certain coefficient vanishes.

\medskip

Our results also yield concrete relations between hyperbolic geometry
and colored Jones polynomials: for certain families of links,
coefficients of the polynomials determine the hyperbolic volume to
within a factor of $4$.  Our methods here provide a deeper and more
intrinsic explanation for similar connections that have been
previously observed.

\medskip

Our approach is to generalize the checkerboard decompositions of
alternating knots and links. For $A$-- or $B$--adequate diagrams, we
show that the checkerboard knot surfaces are incompressible, and
obtain an ideal polyhedral decomposition of their complement.  We use
normal surface theory to establish a dictionary between the pieces of
the JSJ decomposition of the surface complement and the combinatorial
structure of certain spines of the checkerboard surface (state
graphs).  In particular, we give a combinatorial formula for the
complexity of the hyperbolic part of the JSJ decomposition (the guts)
of the surface complement in terms of the diagram of the knot, and use
this to give lower bounds on volumes of several classes of knots.
Since state graphs have previously appeared in the study of Jones
polynomials, our setting and methods create a bridge between quantum
invariants and geometries of knot complements.


\tableofcontents

\chapter{Introduction}\label{sec:intro}
In the last three decades, there has been significant progress in
3--dimensional topology, due in large part to the application of new
techniques from other areas of mathematics and from physics.  On the
one hand, ideas from geometry have led to geometric decompositions of
3--manifolds and to invariants such as the $A$--polynomial and
hyperbolic volume.  On the other hand, ideas from quantum physics have
led to the development of invariants such as the Jones polynomial and
colored Jones polynomials.  While ideas generated by these invariants
have helped to resolve several problems in knot theory, their
relationships to each other, and to classical knot topology, are still
poorly understood.  Topological quantum field theory predicts
that these invariants are in fact tightly related, as does mounting
computer evidence.  However, there are few proofs and many open
problems in this area.

In this monograph, we initiate a systematic study of relations between
quantum knot invariants and geometries of knot complements.  We
develop the setting and machinery that allows us to establish direct
and concrete relations between colored Jones knot polynomials and
geometric knot invariants.  In several instances, our results provide
deeper and more intrinsic explanations for the connections between
geometry and quantum topology that have been observed in special cases
in the past.  In addition, this work leads to some surprising new
relations between the two areas, and offers a promising environment
for further exploring such connections.

We begin with some history and background on the problems under
consideration, then give an overview of the work contained in this
manuscript, including some of the results mentioned above.

\section{History and motivation}

W.~Thurston's ground-breaking work in the late 1970s established
the ubiquity and importance of hyperbolic geometry in three--dimensional
topology.  In fact, hyperbolic 3--manifolds had been studied since
the beginning of the 20th century as a subfield of complex analysis.
In the 1960s and '70s, Andreev \cite{Andreev2, Andreev1}, Riley
\cite{Riley1, Riley}, and J{\o}rgensen \cite{Jorgensen} found 
several families of hyperbolic 3--manifolds
with increasingly complex topology.  In particular, Riley constructed the first
examples of hyperbolic structures on complements of knots in the
3--sphere.  In a different direction, Jaco and Shalen
\cite{jaco-shalen} and Johannson \cite{johannson} found a canonical way to decompose a $3$--manifold along surfaces
of small genus (this is now called the \emph{JSJ decomposition} or \emph{torus decomposition})\index{JSJ decomposition}.  In
particular, they observed that \emph{simple} $3$--manifolds, i.e.\ ones  that do not contain
homotopically essential spheres, disks, tori or annuli, have fundamental groups that share similar
properties with the groups of hyperbolic 3--manifolds. 
Thurston's major insight was  that the pieces of the JSJ decomposition
should admit locally homogeneous geometric structures, and furthermore that the simple pieces should admit complete hyperbolic structures. This insight was formalized in the celebrated \emph{geometrization conjecture}. Thurston proved the conjecture for $3$--manifolds with non-empty boundary \cite{thurston:bulletin}, among others. In 2003, Perelman proved the general conjecture \cite{perelman02,
  perelman03, morgan-tian}.

A special case of Thurston's theorem  \cite{thurston:bulletin} is that
link complements in the 3--sphere satisfy the geometrization
conjecture.  In particular, the complement of any
non-torus, non-satellite knot must admit a complete hyperbolic metric.  By Mostow--Prasad rigidity
\cite{mostow, prasad}, this hyperbolic structure is unique up to
isometry.  As a result, geometric information about a hyperbolic knot  complement,
such as its volume, gives topological knot invariants. \index{hyperbolic volume} For
arbitrary knots, one can obtain a similar invariant, called the \emph{simplicial volume}\index{simplicial volume}, by considering the sum of the volumes of the
hyperbolic components in the JSJ decomposition.  The simplicial volume is a constant multiple of the Gromov norm of the knot complement
\index{Gromov norm}
\cite{gromov}.

Since the mid-1980s, low--dimensional topology has also been
invigorated by ideas from quantum physics, which have led to powerful
and subtle invariants.  The first major invariant along these lines is the celebrated
Jones polynomial, first formulated by Jones in 1985 using operator algebras \cite{jones:announce}. Soon after, Kauffman described a direct construction of the polynomial using the combinatorics of link projections  \cite{KaufJones}, and several authors 
 generalized it to links and trivalent graphs
\cite{HOMFLY, Jones,  Kauf2variable, ReTugraphs}. Witten showed that the Jones polynomial
of links in the 3--sphere has an interpretation in terms of a $2+1$ dimensional
\emph{topological quantum field theory} (TQFT)\index{topological quantum field theory (TQFT)}.  At the
same time, he introduced new invariants for links in arbitrary
3--manifolds, as well as invariants of 3--manifolds \cite{Wittengravity, WittenJones}.  The resulting theory, although defined only at the
physical level of rigor, predicted that the Jones--type invariants and
their generalizations are intimately connected to geometric structures
of 3--manifolds, and particularly to hyperbolic geometry \cite[page
  77]{Wittengravity}.  As explained by Atiyah \cite{Abook}, the TQFT
proposed by Witten is completely characterized by
certain ``gluing axioms.''  In the late 1980s, Reshetikhin and Turaev gave the first 
mathematically rigorous construction of
a TQFT that fit this axiomatic description  \cite{ReTu}.  
Unlike that
of \cite{WittenJones}, which is intrinsically 3--dimensional, the
constructions of \cite{ReTu}, as well as those of \cite{KaufJones, ReTugraphs}, relied on combinatorial descriptions of
3--manifolds and the representation theory of quantum groups.  
This
approach makes it harder to establish connections with the geometry of
3--manifolds. 

In the 1990s, Kashaev defined an infinite family of complex
valued invariants of links in 3--manifolds, using the combinatorics of
triangulations and the quantum dilogarithm function  \cite{kashaev6j}. 
For links in the 3--sphere, these invariants can also be formulated in
terms of tangles and $R$--matrices \cite{kashaevdilog}.
Kashaev's invariants are parametrized by the positive integers; there is
an invariant for each $n \in \NN$.  He
conjectured that the large--$n$ asymptotics of these invariants determine
the volume of hyperbolic knots \cite{kashaev:vol-conj}.   Building on these works,
 H.~Murakami and J.~Murakami were able to
recover Kashaev's invariants as special values of the \emph{colored}
Jones polynomials: an infinite family of polynomials, closely related
to the Jones polynomial, also parametrized by $n \in \NN$ \cite{ murakami:vol-conj}.  As a result, Kashaev's original conjecture
has been reformulated into the  \emph{volume conjecture}\index{volume conjecture},
which asserts that the volume of a hyperbolic knot is determined by
the large--$n$ asymptotics of the colored Jones polynomials.
Furthermore, Murakami and Murakami
generalized the conjecture to all knots in $S^3$ by replacing the hyperbolic
volume with the simplicial volume  \cite{murakami:vol-conj}.  The volume conjecture fits into
a more general, conjectural framework relating hyperbolic
geometry and quantum topology; for details, see the survey papers 
\cite{gukov:vc, murakamisurvey} and references therein.  Despite compelling
experimental evidence, the aforementioned conjectures are currently
known for only a few examples of hyperbolic knots.

At the same time, a growing body of evidence points to strong
relations between the coefficients of the Jones and colored Jones
polynomials and the volume of hyperbolic links.  One such form of
evidence consists of numerical computations, for example those by
Champanerkar, Kofman, and Patterson \cite{ckp:simplest-knots}.  A
second form of evidence consists of theorems proved for several
classes of links, for example alternating links by Dasbach and Lin
\cite{dasbach-lin:head-tail}. The authors of this monograph have 
extended those results to closed $3$--braids \cite{fkp:farey},
highly twisted links \cite{fkp:filling}, and certain sums of
alternating tangles \cite{fkp:conway}.  The approach in all of these 
results is somewhat indirect,
in that they relate hyperbolic volume to the Jones polynomial by
estimating both quantities in terms of the twist number of a link
diagram\index{twist number!and hyperbolic volume}.  To mention two examples, for alternating links the result
follows from Lackenby's volume estimate in terms of the twist number
in any alternating projections \cite{lackenby:volume-alt} and the
relation of the twist number to the colored Jones polynomial observed
by Dasbach and Lin \cite{ dasbach-lin:head-tail}.  For highly twisted
links, our argument works as follows. First, we proved an effective version of 
Gromov and Thurston's
$2\pi$--theorem and applied it to estimate the hyperbolic link volume in terms of the
twist number of any highly twisted projection.  Second, we relied on
the combinatorial properties of \emph{Turaev surfaces}\index{Turaev surface}, 
as studied in \cite{dasbach-futer...}, to relate the
twist numbers to the coefficients of Jones polynomials.  However, for
general links, twist numbers have a highly imperfect relationship to
hyperbolic volume  \cite{fkp:coils}.  This
limits the applicability of these methods to special families of knots and links.

In this monograph, we modify our approach to these problems, focusing on
the topology of incompressible surfaces in knot complements and their
relations to the colored Jones knot polynomials.  Our motivation for
the project has been two-fold.  On the one hand, certain spanning
surfaces of knots have been shown to carry information on colored
Jones polynomials \cite{dasbach-futer...}.  On the other hand,
essential surfaces also shed light on volumes of manifolds
\cite{ast} and additional geometry and topology
(e.g. \cite{adams:quasi-fuchsian, menasco:incompress, miyamoto}).
With these ideas in mind, we develop a machine that allows us to
establish relationships between colored Jones polynomials and
topological/geometric invariants.

For example, under mild diagrammatic hypotheses that arise naturally in
the study of Jones--type polynomials, we show that the growth of the
degree of the colored Jones polynomials is a \emph{boundary
  slope}\index{boundary slope} of an essential surface in the knot
complement, as predicted by Garoufalidis
\cite{garoufalidis:jones-slopes}.  Furthermore, certain coefficients
of the polynomials measure how far this surface is from being a fiber
in the knot complement.  Our work leads to direct and detailed
relations between hyperbolic geometry and Jones--type polynomials: for
certain families of links, coefficients of the Jones and colored Jones
polynomials determine the hyperbolic volume to within a factor of $4$.
Compared to previous arguments, which were all somewhat indirect, the
way in which our machine produces volume inequalities gives a clearer
and deeper conceptual explanation for why the hyperbolic volume should
be related to particular coefficients of  the Jones polynomial.

A survey of this monograph, in which the main theorems are illustrated by a running example,  is given in \cite{fkp:survey}. 

\section{State graphs, and state surfaces far from fibers}
We begin with some terminology and conventions.  Throughout this
manuscript, $D = D(K)$ will denote a link diagram, in the equatorial
$2$--sphere of $S^3$. It is worth pointing out two conventions. First,
we always assume (without explicit mention) that link diagrams are
connected. Second, we abuse notation by referring to the projection
$2$--sphere using the common term \emph{projection plane}. In
particular, $D(K)$ cuts the projection ``plane'' into compact regions.

Let $D(K)$ be a (connected) diagram of a link $K$, as above, and let
$x$ be a crossing of $D$.  Associated to $D$ and $x$ are two link
diagrams, each with one fewer crossing than $D$, called the
\emph{$A$--resolution}\index{$A$--resolution} and
\emph{$B$--resolution}\index{$B$--resolution} of the crossing.
\begin{figure}[h]
	\centerline{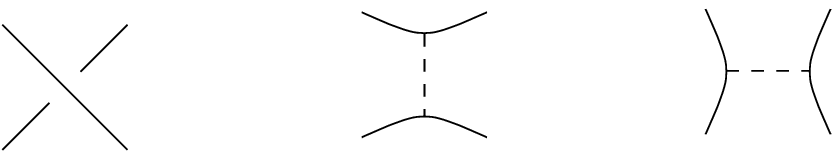}
\caption{$A$-- and $B$--resolutions at a crossing of
  $D$.\index{$A$--resolution}\index{$B$--resolution}}
\label{fig:splicing}
\end{figure}

\begin{define}\label{def:reduced}
A \emph{state}\index{state} $\sigma$ is a choice of $A$-- or
$B$--resolution at each crossing of $D$.  Resolving every crossing, as
in Figure \ref{fig:splicing}, gives rise to a crossing--free diagram
$s_\sigma(D)$, which is a collection of disjoint circles in the
projection plane.  Thus one obtains a \emph{state graph}\index{state graph} $\mathbb{G}_\sigma$, whose vertices correspond to circles of
$s_\sigma$ and whose edges correspond to former crossings.  For a
given state $\sigma$, the \emph{reduced state graph}\index{state graph!reduced}\index{reduced state graph} $\mathbb{G}'_\sigma$ is
the graph obtained from $\mathbb{G}_\sigma$ by removing all multiple
edges between pairs of vertices.
\end{define}

The notion of states on link diagrams was first considered by Kauffman
\cite{KaufJones} during his construction of the bracket polynomial
that provided a new construction and interpretation of the Jones
polynomial.

Our primary focus is on the all--$A$ and all--$B$ states. The
crossing--free diagram $s_A(D)$\index{state!all--$A$ state}\index{all--$A$ state} is obtained by applying the
$A$--resolution to each crossing of $D$.  Its state graph is denoted
$\GA$ or $\GA(D)$\index{$\GA$, $\GB$: state graph}, and its reduced
state graph $\GRA$ or $\GRA(D)$\index{$\GRA$, $\GRB$: reduced state graph}.  Similarly, for the all--$B$ state
$s_B(D)$\index{state!all--$B$ state}\index{all--$B$ state}, the state
graph is denoted $\GB$, and the reduced state graph $\GRB$.

To a state $\sigma$, we associate a \emph{state surface}\index{state surface}
$S_\sigma$ as follows.  The state circles of $\sigma$ bound disjoint
disks in the $3$--ball below the projection plane; these disks can be 
connected to one another by half--twisted bands at the crossings.  The
surface $S_\sigma$ will have $\bdy S_\sigma = K$.  A special case of
this construction is the Seifert surface constructed from the diagram
$D(K)$, where the state $\sigma$ is determined by an orientation on $K$. 

When $\sigma$ is the all--$A$ or all--$B$ state, the surfaces
$S_{\sigma}$ hold significance for both geometric topology and quantum
topology.  The graph $\GA$ canonically embeds as a spine of the
surface $S_A$.  On the quantum side, the combinatorics of this
embedding can be used to recover the colored Jones polynomials
$J^n_K(t)$ \cite{dasbach-futer..., dasbach-lin:head-tail}.  On the
geometric side, as we will see below, the combinatorics of $\GA$
dictates a geometric decomposition of the 3--manifold $M_A$ obtained
by cutting the link complement along the surface $S_A$.  Because every
statement has a $B$--state counterpart (by taking a mirror of the
diagram), we will mainly discuss the all--$A$ state for ease of
exposition.

\begin{define}\label{def:cut}
Let $M=S^3\setminus K$ denote the 3--manifold with torus boundary
component(s) obtained by removing a tubular neighborhood of $K$ from
$S^3$. Let $S_A$ be the all--$A$ state surface, as above, and let
$M\cut S_A$\index{$M\cut S_A$} denote the path--metric closure of $M
\setminus S_A$.  Note that $(S^3\setminus K)\cut S_A$ is homeomorphic
to the 3--manifold $S^3\cut S_A$ obtained by removing a regular
neighborhood of $S_A$ from $S^3$.  We will usually write $S^3\cut S_A$
for short, and denote this manifold with boundary by $M_A$\index{$M_A = S^3 \cut S_A$}.

We will refer to $P=\bdy M_A \cap \bdy M$ as the \emph{parabolic
  locus}\index{parabolic locus} of $M_A$.  This parabolic locus
consists of annuli.  The remaining, non-parabolic boundary $\bdy M_A
\setminus \bdy M$ is the unit normal bundle of $S_A$.
\end{define}

\begin{define}\label{def:essential}
Let $M$ be an orientable $3$--manifold and $S \subset M$ a properly embedded surface. 
We say that $S$ is \emph{essential}\index{essential surface} in $M$ if the boundary of a regular neighborhood of $S$, denoted $\widetilde{S}$, 
 is incompressible and boundary--incompressible. If $S$ is orientable, then $\widetilde{S}$ consists of two copies of $S$, and the definition is equivalent to the standard notion of ``incompressible and boundary--incompressible.''
 If $S$ is non-orientable, this is equivalent to $\pi_1$--injectivity of $S$, the stronger of two possible senses of incompressibility.

In the setting of Definition \ref{def:cut}, the surface $S_A$ is often non-orientable. In this case, $S^3 \cut \widetilde{S_A}$ is the disjoint union of $M_A = S^3 \cut S_A$ and a twisted $I$--bundle over $S_A$. Since we are interested in the topology of $M_A$, 
 it is appropriate to look at the incompressibility of $\widetilde{S_A}$. 
\end{define}

Guided by the combinatorial structure of the state graph $\GA$, we
construct a decomposition of $M_A$ into topological balls.  The
connectivity properties of $\GA$ govern the behavior of this
decomposition; in particular, if $\GA$ has no loop edges, we obtain a
decomposition of $M_A$ into checkerboard ideal polyhedra with
4--valent vertices (Theorem \ref{thm:simply-connected}).  This
decomposition generalizes Menasco's decomposition of alternating link
complements, which has been used frequently in the literature
\cite{menasco:polyhedra}.  As a first application of our machinery, we
use normal surface theory with respect to our polyhedral decomposition
to give a new proof of the following theorem of Ozawa \cite{ozawa}.

\begin{named}{Theorem \ref{thm:incompress} ({\rm Ozawa \cite{ozawa}})}
Let $D(K)$ be a diagram of a link $K$.  Then the all--$A$ state 
surface $S_A$ is essential
 in $S^3
\setminus K$ if and only if $\GA$ contains no 1--edge loops.
Similarly, the surface $S_B$ is essential
in $S^3 \setminus K$ if and only if $\GB$
contains no 1--edge loops.
\end{named}

Our polyhedral decomposition is designed to provide much more detailed
information about the topology and geometry of $M_A = S^3 \cut S_A$.
In particular, we can characterize exactly when the surface $S_A$ is a
fiber of the link complement.

\begin{named}{Theorem \ref{thm:fiber-tree}}
Let $D(K)$ be any link diagram, and let $S_A$ be the spanning surface
determined by the all--$A$ state of this diagram.  Then the following
are equivalent:
\begin{enumerate}
\item[\eqref{item:tree}] The reduced graph $\GRA$ is a tree.
\item[\eqref{item:fiber}] $S^3 \setminus K$ fibers over $S^1$, with fiber $S_A$.
\item[\eqref{item:semi-fiber}] $M_A = S^3 \cut S_A$ is an $I$--bundle over $S_A$.
\end{enumerate}
\end{named}

It is remarkable to note that the state graph connectivity conditions
that ensure incompressibility of the state surfaces first arose in the
study of Jones--type knot polynomials. The following definition,
formulated by Lickorish and Thistlethwaite \cite{lick-thistle,
  thi:adequate}, captures exactly the class of link diagrams whose
polynomial invariants are especially well--behaved.
  
\begin{define}\label{def:adequate}
A link diagram $D(K)$ is called $A$--adequate\index{$A$--adequate}
(resp. $B$--adequate\index{$B$--adequate}) if $\GA$ (resp. $\GB$) has
no 1--edge loops.  If both conditions hold for a diagram $D(K)$, then
$D(K)$ and $K$ are called \emph{adequate}\index{adequate diagram}.  
If $D(K)$ is either $A$-- or $B$--adequate, then $D(K)$ and $K$ are
called \emph{semi-adequate}\index{semi-adequate diagram}.  As we will
discuss in the next section, the hypothesis of semi-adequacy is rather
mild. 
\end{define}

Building on Theorem \ref{thm:fiber-tree}, we start with an
$A$--adequate diagram $D$ and strive to understand the geometric and
topological complexity of $S^3 \cut S_A$.  In Chapter
\ref{sec:decomp}, we will see that the 3--manifold $M_A = S^3 \cut
S_A$ is in fact a handlebody, and thus atoroidal.  The annulus version
of the JSJ decomposition theory \cite{jaco-shalen, johannson}\index{JSJ decomposition} provides
a way to cut $M_A$ along annuli (disjoint from the parabolic locus)
into three types of pieces: $I$--bundles over sub-surfaces of $S_A$,
Seifert fibered spaces, and the \emph{guts}\index{guts}, which is the
portion that admits a hyperbolic metric with totally geodesic
boundary.  The Seifert fibered components are solid tori. Thus
$\chi(\guts(M_A))= 0$ precisely when $\guts(M_A)=\emptyset$ and $M_A$
is a union of $I$--bundles and solid tori.  In this case, $M_A$ is
called a \emph{book of $I$--bundles}\index{book of $I$--bundles} and
$S_A$ is called a \emph{fibroid}\index{fibroid} \cite{culler-shalen}.
The guts are the complex, interesting pieces of the geometric
decomposition of $M_A$.  Because hyperbolic surfaces, and guts, have
negative Euler characteristic, it is convenient to work with the
following definition.

\begin{define}\label{def:neg-euler}
Let $Y$ be a compact cell complex, whose connected components are
$Y_1, \ldots, Y_n$. Then the Euler characteristic of $Y$ can be split
into positive and negative parts:\index{negative Euler characteristic}

$$
\chi_+(Y) = \sum_{i=1}^{n} \max \{  \chi(Y_i), \, 0 \}, \qquad 
\negeul(Y) = \sum_{i=1}^{n} \max \{ - \chi(Y_i), \, 0 \}.
\index{$\xi$@$\negeul(Y)$, $\chi_+(Y)$} 
$$ It follows immediately that $\chi(Y) = \chi_+ (Y) - \negeul(Y).$
This notation is borrowed from the Thurston norm \cite{thurston:norm}.
By convention, when $Y = \emptyset$, the above sums have no terms,
hence $\chi_+(\emptyset) = \negeul(\emptyset) = 0$.
\end{define}

The negative Euler characteristic $\negeul(\guts(M_A))$ serves as a
useful measurement of how far $S_A$ is from being a fiber or a fibroid
in $S^3 \setminus K$. 
In fact, $\negeul(\guts)$ is a key measurement
of complexity in Agol's virtual fibering criterion \cite{agol:fibering-criterion}, which 
is needed in the proof  of the virtual fibering conjecture for hyperbolic 
$3$--manifolds \cite{agol:virtual-haken}.  The Euler
characteristic of guts also has a direct connection to hyperbolic
geometry.  Agol, Storm, and Thurston have shown that for any essential
surface $S$ in a hyperbolic $3$--manifold $M$, a constant times
$\negeul(\guts(M))$ gives a lower bound for $\vol(M)$ \cite{ast}. This
is applied below, in Section \ref{sec:intro-volume}.  On the other
hand, the Euler characteristic $\chi(\GRA)$ of the reduced graph
$\GRA$ first arose in the study of Jones--type polynomials
\cite{dasbach-lin:head-tail, stoimenow:coeffs}, and in fact expresses
one of their coefficients. This is explored in Section
\ref{sec:intro-jones}.

One of our main results is a diagrammatic formula for the guts of
state surfaces for all $A$--adequate diagrams. In relating guts to
reduced state graphs, it provides a bridge between hyperbolic geometry
and quantum topology.

\begin{named}{Theorem \ref{thm:guts-general}}
Let $D(K)$ be an $A$--adequate diagram, and let $S_A$ be the
essential spanning surface determined by this diagram. Then
$$\negeul( \guts(S^3 \cut S_A))= \negeul (\GRA) - || E_c||,$$
where $|| E_c|| \geq 0$ is a diagrammatic quantity defined in
Definition \ref{def:ec}. 
\end{named}

In many cases, the correction term $ || E_c||$ vanishes. For example,
this happens for alternating links \cite{lackenby:volume-alt}, as well
as for most Montesinos links.  See Theorem \ref{thm:monteguts}, stated
on page \pageref{monte-statement-intro} and Corollary
\ref{cor:onlybigons} on page \pageref{cor:onlybigons}.  In each of
these cases, Theorem \ref{thm:guts-general} says that a geometric
quantity, $\negeul(\guts(M_A))$, is equal to $\negeul(\GRA)$, which,
as shown in \cite{dasbach-lin:head-tail}, expresses a coefficient of
the Jones polynomial.

\section{Which links are semi-adequate?}\label{subsec:largeclass}

We will be considering semi-adequate links throughout this
manuscript. (After taking a mirror if necessary, such a link is
$A$--adequate.)  Before we continue with the description of our
results, it is worth making some remarks about the class of
semi-adequate links. It turns out that the class is very broad, and
that the condition that a knot be semi-adequate seems to be rather
mild.  For example, with the exception of two 11--crossing knots that
we will discuss below, and a handful of 12--crossings knots, all knots
with at most 12 crossings are semi-adequate.  Furthermore, every
minimal crossing diagram for each of these semi-adequate knots is
semi-adequate \cite{stoimenow:coeffs, thi:adequate}.  Thus, apart from
a few exceptions, our results in this monograph apply directly to the
diagrams in the knot tables up to 12 crossings.  The situation is
similar with the larger tabulated knots: Stoimenow has computed that
among the 253,293 prime knots with 15 crossings tabulated in
\cite{knotscape}, at least 249,649 are semi-adequate \cite{stoimenow}.

Several well--studied families of links are semi-adequate. These
include alternating links, positive or negative closed braids, all
closed $3$--braids, all Montesinos links, and planar cables of all of
the above.  We refer the reader to \cite{lick-thistle, stoimenow,
  thi:adequate} for more discussion and examples.

Nevertheless, there exist knots and links that are not semi-adequate.
Before discussing examples, we recall that the Jones polynomial can be
used to detect semi-adequency.  Indeed, the last coefficient of an
$A$-adequate link must be $\pm 1$. Similarly, the first coefficient of
an $B$-adequate link must be $\pm 1$ \cite{thi:adequate}.  With the
notation of Knotinfo \cite{knotinfo}, the knot $K=11n_{95}$ has Jones
polynomial equal to $J_K(t)= 2t^2-3t^3+5t^4-6t^5+6t^6-5t^7+4t^8-2t^9$.
Hence, $K$ is not semi-adequate; this is the first such knot in the
knot tables.  An infinite family of non semi-adequate knots, detected
by the extreme coefficients of their Jones polynomial, can be obtained
by \cite[Theorem 5]{manchon}. However, as we discuss below, the
extreme coefficients of the Jones polynomial are not a complete
obstruction to semi-adequecy.

Thistlethwaite \cite{thi:adequate} showed that certain coefficients of
the 2--variable Kauffman polynomial \cite{Kauf2variable} provide the 
obstruction to semi-adequacy. Building on Thistlethwaite's
results, Stoimenow obtained a set of semi-adequacy
criteria and applied them to several knots whose adequacy could not be
determined by the Jones polynomial.  For example, he showed that the
knot $K'=11n_{118}$ is not semi-adequate. Note that
in this case, the last coefficient of the Jones polynomial, 
$J_{K'}(t)=2t^2-2t^3+3t^4-4t^5+4t^6-3t^7+2t^8-t^9$, is $-1$.

Ozawa has considered link diagrams and Kauffman states $\sigma$ that are
adequate\index{$\sigma$--adequate} (meaning ${\G}_\sigma$
has no $1$--edge loops) and
homogeneous\index{$\sigma$--homogeneous}\index{homogeneous state}\index{state!homogeneous} (meaning ${\G}_\sigma$ contains a
set of cut vertices that decompose it into a collection of all--$A$
and all--$B$ state graphs) \cite{ozawa}.  See Definition
\ref{def:sigma-homo} for more details.  
Semi-adequate diagrams clearly have this
property, but the class of \cite{ozawa} is broader.  As an example,
consider the 12--crossing knot $K''=12n_{0706}$.  This is not 
semi-adequate since both the extreme coefficients of the Jones polynomial are equal
to 2.  Indeed
$J_{K''}(t)=2t^{-4}-4t^{-3}+6t^{-2}-8t^{-1}+9-8t+6t^2-4t^3+2t^4$.
However $K''$ can be written as a $5$--string braid that is
homogeneous in the sense of Cromwell \cite{cromwell}.  Thus the
Seifert state of this closed braid diagram is homogeneous and
adequate.

Ozawa proved that the state surface $S_\sigma$ corresponding to a
$\sigma$--adequate, $\sigma$--homogeneous diagram is always essential in
$S^3 \setminus K$.  In \cite{futer}, Futer gave a direct proof of a slightly weaker version of 
Theorem
\ref{thm:fiber-tree}, and also  generalized it  to $\sigma$--adequate, $\sigma$--homogeneous  link diagrams.
It turns out that many properties of the polyhedral decompositions
that we develop below, as well as a number of results proved using
the polyhedral decomposition, also extend to all adequate,
homogeneous states.  
See Sections
\ref{subsec:generalization}, \ref{subsec:idealsigma},
\ref{sec:gensigmahomo},  \ref{subsec:spanningsigma} where, in particular, we obtain analogues
of Theorems \ref{thm:incompress}, \ref{thm:fiber-tree} and \ref{thm:guts-general} in this generalized setting.  Our study of the
geometry of such links is continued in \cite{fkp:qsf}.

\section{Essential surfaces and  colored Jones polynomials}\label{sec:intro-jones}

The Jones and colored Jones polynomials have many known connections to
the state graphs of diagrams. To specify notation, let
$$ J^n_K(t)= \alpha_n t^{m_n}+ \beta_n t^{m_n-1}+ \ldots + \beta'_n
t^{r_n+1}+ \alpha'_n t^{r_n},
$$ denote the $n$-th \emph{colored Jones polynomial}\index{colored Jones polynomial} of a link $K$.  Recall that $J^2_K(t)$ is the
usual Jones polynomial. 
Consider the sequences
$$js_K:= \left\{{{ 4m_n}\over {n^2} } \: : \: n > 0\right\} \quad
\mbox{and} \quad js^*_K:= \left\{ {{ 4 r_n}\over {n^2}} \: : \: n > 0
\right\}.$$

Garoufalidis'  slope conjecture\index{slope conjecture}  predicts that for
each knot $K$, every cluster point (i.e., every limit of a
subsequence) of $js_K$ or $js^*_K$ is a \emph{boundary slope}\index{boundary slope} of $K$ \cite{garoufalidis:jones-slopes}, i.e.\ a fraction $p/q$ such that the homology class $p \mu + q \lambda$ occurs as the boundary of an essential surface in $S^3 \setminus K$.

For a given diagram $D(K)$, there is a
lower bound for $r_n$ in terms of data about the state graph $\GA(D)$,
and this bound is sharp when $D(K)$ is $A$--adequate.  Similarly,
there is an upper bound on $m_n$ in terms of $\GB$ that is realized
when $D(K)$ is $B$--adequate \cite{lickorish:book}.  In
\cite{fkp:PAMS}, building on these properties and using Theorem
\ref{thm:incompress}, we relate the extreme degree of $J^n_K(t)$
to the boundary slope of $S_A$, as predicted by
the slope conjecture.

\begin{theorem}[\cite{fkp:PAMS}]\label{thm:slopes}
Let $D(K)$ be an $A$--adequate diagram of a knot $K$ and let $b(S_A)
\in {\ZZ}$ denote the boundary slope of the essential surface $S_A$.
Then
$$\lim_{n\to \infty} {\frac {4 r_n}{n^2}}= b(S_A),$$
where $r_n$ is the lowest degree of $J^n_K(t)$.
\end{theorem}

Similarly, if $D(K)$ is a $B$--adequate diagram of a knot $K$, let
$b(S_B) \in {\ZZ}$ denote the boundary slope of the essential surface
$S_B$.  Then
$$\lim_{n\to \infty} {\frac {4 m_n}{n^2}}= b(S_B),$$
where $m_n$ is the highest degree of $J^n_K(t)$.

Work of Garoufalidis and Le \cite{garquasi, garoufalidisLe} implies
that each coefficient of $J^n_K(t)$ satisfies linear recursive
relations in $n$.  For adequate links, these relations manifest
themselves in a very strong form: Dasbach and Lin showed that if $K$
is $A$--adequate, then the absolute values $ \abs{\beta'_n}$ and
$\abs{\alpha'_n}$ are independent of $n>1$
\cite{dasbach-lin:head-tail}.  In fact, $\abs{\alpha'_n}=1$ and
$\abs{\beta'_n} = 1 - \chi(\GRA)$, where $\GRA$ is the reduced graph.
Similarly, if $D$ is $B$--adequate, then $\abs{\alpha_n}=1$ and
$\abs{\beta_n} = 1 - \chi(\GRB)$.  Thus we can define the \emph{stable
  values}\index{stable value}\index{$\beta_K$, $\beta_K'$}
$$\beta_K'\: := \: \abs{\beta'_n} \: = \: 1 - \chi(\GRA), \qquad
\mbox{and} \qquad \beta_K \: := \: \abs{\beta_n} \: = \: 1 -
\chi(\GRB).
$$ The main results of this monograph explore the idea that the stable
coefficient $\beta_K'$ does an excellent job of measuring the
geometric and topological complexity of the manifold $M_A = S^3 \cut
S_A$. (Similarly, $\beta_K$ measures the complexity of $M_B = S^3 \cut
M_B$.)  For instance, it follows from Theorem \ref{thm:fiber-tree}
that $\beta'_K$ is exactly the obstruction to $S_A$ being a fiber.

\begin{named}{Corollary \ref{cor:beta-fiber}} 
For an $A$--adequate link $K$, the following are equivalent:
\begin{enumerate}
\item $\beta'_K=0$. 
\item For \emph{every} $A$--adequate diagram of $D(K)$,  $S^3
  \setminus K$ fibers over $S^1$ with fiber the corresponding state
  surface $S_A = S_A(D)$.  
\item For \emph{some} $A$--adequate diagram $D(K)$,  $M_A = S^3 \cut
  S_A$ is an $I$--bundle over $S_A(D)$. 
\end{enumerate}
\end{named}

Similarly, $\abs{\beta_K'} = 1$ precisely when $S_A$ is a fibroid
of a particular type.

\begin{named}{Theorem \ref{thm:fibroid-detect}}
For an $A$--adequate link $K$, the following are equivalent:
\begin{enumerate}
\item $\beta'_K=1$.
\item For \emph{every} $A$--adequate diagram of $K$, the corresponding
  3--manifold $M_A$ is a book of $I$--bundles, with $\chi(M_A)=
  \chi(\GA) - \chi(\GRA)$, and is not a trivial $I$--bundle over the
  state surface $S_A$.
\item For \emph{some} $A$--adequate diagram of $K$, the corresponding
  3--manifold $M_A$ is a book of $I$--bundles, with $\chi(M_A)=
  \chi(\GA) - \chi(\GRA)$.
\end{enumerate}
\end{named}

In general, the geometric decomposition of $M_A$ contains some
non-trivial hyperbolic pieces, namely guts.  In this case,
$\abs{\beta_K'}$ measures the complexity of the guts together with
certain complicated parts of the maximal $I$--bundle of $M_A$.  To
state our result we need the following definition.

\begin{define}\label{primetwist}
A link diagram $D$ is called \emph{prime}\index{prime!diagram} if any
simple closed curve that meets the diagram transversely in two points
bounds a region of the projection plane without any crossings.
  
Two crossings in $D$ are defined to be \emph{twist equivalent} if
there is a simple closed curve in the projection plane that meets $D$
at exactly those two crossings. The diagram is called \emph{twist
  reduced}\index{twist reduced} if every equivalence class of
crossings is a \emph{twist region}\index{twist region} (a chain of
crossings between two strands of $K$).  The number of equivalence
classes is denoted $t(D)$, the \emph{twist number}\index{twist number} of $D$.
\label{def:prime-diagram}
\end{define}

\begin{named}{Theorem \ref{thm:no2loops}}\index{two-edge loop}
Suppose $K$ is an $A$--adequate link whose stable colored Jones
coefficient is $\beta_K' \neq 0$.  Then, for every $A$--adequate
diagram $D(K)$,
$$\negeul( \guts(M_A)) + || E_c || \: = \: \abs{\beta'_K} - 1, $$
where as above $|| E_c|| \geq 0$ is the diagrammatic quantity of
Definition \ref{def:ec}.  Furthermore, if $D$ is prime and every
2--edge loop in $\GA$ has edges belonging to the same twist region,
then $||E_c|| = 0$ and
$$\negeul( \guts(M_A)) \: = \: \abs{\beta'_K}- 1. $$
\end{named}

To briefly discuss the meaning of the correction term $||E_c||$,
recall that the non-hyperbolic components of the JSJ decomposition of
$M_A$ are $I$--bundles and solid tori.  In Chapter \ref{sec:ibundle},
we show that the $I$--bundle components with negative Euler
characteristic are spanned by \emph{essential product disks
  (EPDs)}\index{essential product disk (EPD)}\index{EPD}: properly
embedded essential disks in $M_A$ whose boundary meets the parabolic
locus twice.  These disks come in two types: those corresponding to
(strings of) complementary regions of $\GA$ with just two sides, and
certain ``complicated'' ones, which we call
\emph{complex}\index{complex EPD}\index{EPD!complex}.  (See Definition
\ref{def:simple} on page \pageref{def:simple}).  The minimal number of
complex EPDs in a spanning set is denoted $||E_c||$; this is exactly
the correction term of Theorems \ref{thm:guts-general} and
\ref{thm:no2loops}.

It is an open question whether \emph{every} $A$--adequate link admits
a diagram for which $||E_c ||=0$: see Question \ref{quest:allequal} on
page \pageref{quest:allequal}. For instance, Lackenby showed that this
is the case for prime alternating links \cite{lackenby:volume-alt}. By
Theorem \ref{thm:no2loops}, $||E_c|| = 0$ when every $2$--edge loop of
$\GA$ has edges belonging to the same twist region.  This is also the
case for most Montesinos links (the reader is referred to Chapter
\ref{sec:montesinos} for the terminology).

\begin{named}{Corollary \ref{cor:monte-exact}}
Suppose $K$ is a Montesinos link with a reduced admissible diagram
$D(K)$ that contains at least three tangles of positive slope.  Then
$$\negeul( \guts(M_A)) \:= \: \abs{\beta'_K}-1.$$
Similarly, if $D(K)$ contains at least three tangles of negative
slope, then
$$\negeul( \guts(M_B)) \: = \: \abs{\beta_K}-1.$$
\end{named}

When $||E_c ||=0$, Theorem \ref{thm:no2loops} offers striking evidence
that coefficients of the colored Jones polynomials measure something
quite geometric: when $\abs{\beta'_K}$ is large, the link complement
$S^3 \setminus K$ contains essential spanning surfaces that are
correspondingly far from being a fiber.  Whereas the Alexander
polynomial and its generalization in Heegaard Floer homology are known
to have many connections to the geometric topology of spanning
surfaces of a knot \cite{ni:fiber, ozsvath-szabo:genus,
  ozsvath-szabo:thurston-norm}, the geometric meaning of Jones-type
polynomials has traditionally been a mystery.  Theorems
\ref{cor:beta-fiber}, \ref{thm:fibroid-detect}, and \ref{thm:no2loops}
establish some of the first detailed connections between surface
topology and the Jones polynomial.

\section{Volume bounds from topology and combinatorics}\label{sec:intro-volume}

Recall that by the work of Agol, Storm, and Thurston \cite{ast}, any
computation of, or lower bound on, $\negeul(\guts)$ of an essential
surface $S \subset S^3 \setminus K$ leads to a proportional lower
bound on $\vol(S^3 \setminus K)$.  For instance, Lackenby's
diagrammatic lower bound on the volumes of alternating knots and links
came as a result of computing the guts of checkerboard surfaces
\cite{lackenby:volume-alt}.  However, computing $\negeul(\guts)$ has
typically been quite hard: apart from alternating knots and links,
there are very few infinite families of manifolds for which there are
known computations of the guts of essential surface \cite{agol:guts,
  kuessner:guts}.

The results of this manuscript greatly expand the list of manifolds for
which such computations exist.  In Chapter \ref{sec:applications}, we
combine \cite{ast} with our results in Theorems \ref{thm:guts-general}
and \ref{thm:no2loops}, as well as some of their specializations, to
give lower bounds on hyperbolic volume for all $A$--adequate knots and
links. See Theorem \ref{thm:volume} on page \pageref{thm:volume} for
the most general result along these lines.

We also focus on two well--studied families of links: namely, positive
braids and Montesinos links.  For these families, we are able to
compute or estimate the quantity $\negeul(\guts(M_A))$ in terms of
much simpler diagrammatic data.  As a consequence, we obtain tight,
two--sided estimates on the volumes of knots and links in terms of the
twist number $t(D)$ (see Definition \ref{primetwist}).

\begin{named}{Theorem \ref{thm:positive-volume}}
Let $D(K)$ be a diagram of a hyperbolic link $K$, obtained as the
closure of a positive braid with at least three crossings in each
twist region. Then
$$\frac{2 v_8}{3} \, t(D) \: \leq \:\vol(S^3 \setminus K) \: < \:
10v_3(t(D)-1),$$  
where $v_3 = 1.0149...$ is the volume of a regular ideal tetrahedron  
and $v_8 = 3.6638...$ is the volume of a regular ideal octahedron. 
\end{named}
Observe that the multiplicative constants in the upper and lower
bounds differ by a rather small factor of about $4.155$.  For
Montesinos links, we obtain similarly tight two--sided volume bounds.

\begin{named}{Theorem \ref{thm:monte-volume}}
Let $K \subset S^3$ be a Montesinos link with a reduced Montesinos
diagram $D(K)$.  Suppose that $D(K)$ contains at least three positive
tangles and at least three negative tangles.  Then $K$ is a hyperbolic
link, satisfying
\begin{equation*}
  \frac{v_8}{4} \, \left( t(D) - \#K \right) \:
  \leq \: \vol(S^3 \setminus K) \:
  < \: 2 v_8 \, t(D),
\end{equation*}
where $v_8 = 3.6638...$ is the volume of a regular ideal octahedron
and $\# K$ is the number of link components of $K$.  The upper bound
on volume is sharp.
\end{named}

We also relate the volumes of these links to quantum invariants.
Recall that the volume conjecture of Kashaev and Murakami--Murakami  \cite{kashaev:vol-conj, murakami:vol-conj}\index{hyperbolic volume!volume conjecture}\index{volume conjecture}\index{colored Jones polynomial!volume conjecture}
states that all hyperbolic knots satisfy
$$2\pi\lim_{n\to \infty} { \frac{\log \abs {J^n_K(e^{2\pi i / n} ) }}
  {n}}=\vol(S^3 \setminus K).
$$ If this volume conjecture is true, it would imply for large $n$ a
relation between the volume of a knot $K$ and coefficients of
$J^n_K(t)$.  For example, for $ n \gg 0$ one would have
$\vol(S^3\setminus K) < C || J^n_K||$, where $|| J^n_K||$ denotes the
$L^1$--norm of the coefficients of $J_K^n(t)$, and $C$ is an
appropriate constant.  In recent years, a series of articles by
Dasbach and Lin, as well as the authors, has established such
relations for several classes of knots \cite{dasbach-lin:volumish, fkp:filling, fkp:conway, fkp:farey}.  In fact, in all known cases,
the upper bounds on volume are paired with similar lower bounds.
However, in all of the past results, showing that coefficients of
$J^n_K(t)$ bound volume below required two steps: first, showing that
Jones coefficients give a lower bound on twist number $t(D)$, and then
showing that twist number gives a lower bound on volume.  Each of these
two steps is known to fail outside special families of knots
\cite{fkp:farey, fkp:coils}, and their combination produces an
indirect argument in which the constants are far from sharp.

By contrast, our results in this manuscript bound volume below in
terms of a topological quantity, $\negeul(\guts)$, that is directly
related to colored Jones coefficients. As a consequence, we obtain
much sharper lower bounds on volume, along with an intrinsic and
satisfactory conceptual explanation for why these lower bounds exist.
See Section \ref{sec:jones-volume} in Chapter \ref{sec:applications}
for more discussion.

Our techniques also imply similar results for additional classes of
knots.  For instance, Theorems \ref{thm:positive-volume} and
\ref{thm:monte-volume} have the following corollaries.

\begin{named}{Corollary \ref{cor:positive-vol-jones}}
Suppose that a hyperbolic link $K$ is the closure of a positive braid
with at least three crossings in each twist region. Then
$$v_8 \, ( \abs{\beta'_K}-1 )\: \leq \:\vol(S^3 \setminus K) \: < \: 15 v_3 \, (\abs{\beta'_K}-1) - 10 v_3,$$ 
where $v_3 = 1.0149...$ is the volume of a regular ideal tetrahedron  
and $v_8 = 3.6638...$ is the volume of a regular ideal octahedron.
\end{named}

\begin{named}{Corollary \ref{cor:jones-volumemonte}}
Let $K \subset S^3$ be a Montesinos link with a reduced Montesinos
diagram $D(K)$.  Suppose that $D(K)$ contains at least three positive
tangles and at least three negative tangles.  Then $K$ is a hyperbolic
link, satisfying
$$ v_8 \left( \max \{ \abs{\beta_K}, \abs{\beta'_K} \}  -1 \right) \:
\leq \: \vol(S^3 \setminus K) \: < \: 4 v_8  \left(  \abs{\beta_K} +
\abs{\beta'_K} -2 \right) + 2 v_8 \, (\# K),$$ 
where $\#K$ is the number of link components of $K$.
 \end{named}

\section{Organization}

We now give a brief guide to the organization of this monograph.
  
In Chapter \ref{sec:decomp}, we begin with a connected link diagram
$D(K)$, and explain how to construct the state graph $\GA$ and the
state surface $S_A$.  Guided by the structure of $\GA$, we will cut
the 3--manifold $M_A=S^3 \cut S_A$ along a collection of disks into
several topological balls.  We obtain a collection of \emph{lower}
balls that are in one--to--one correspondence with the alternating
tangles in $D(K)$ and a single \emph{upper} 3--ball\index{upper 3--ball}.  The boundary of each ball admits a checkerboard coloring
into white and shaded regions that we call faces. 
In the last section of the chapter we discuss the generalization of the decomposition to 
$\sigma$--homogeneous and $\sigma$--adequate diagrams.

In Chapter \ref{sec:polyhedra}, we show that if $D(K)$ is
$A$--adequate, each of these balls is a checkerboard colored ideal
polyhedron with 4--valent vertices.  This amounts to showing that the
shaded faces on each of the 3--balls are simply--connected (Theorem
\ref{thm:simply-connected}).  Furthermore, we show that the ideal
polyhedra do not contain normal bigons (Proposition
\ref{prop:no-normal-bigons}), which quickly implies Theorem
\ref{thm:incompress}.  In the last section of the chapter, we generalize these results to
homogeneous and adequate states.

In Chapter \ref{sec:ibundle}, we prove a structural result about the
geometric decomposition of $M_A$. As already mentioned, the JSJ
decomposition yields three kinds of pieces: $I$--bundles, solid tori,
and the guts, which admit a hyperbolic metric with totally geodesic
boundary. Let $B$ be an $I$--bundle in the characteristic submanifold
of $M_A$.  We say that a finite collection of disjoint essential
product disks (EPDs) $\{D_1, \dots, D_n \}$ \emph{spans}\index{spans}
$B$ if $B \setminus (D_1 \cup \dots \cup D_n)$ is a finite collection
of prisms (which are $I$--bundles over a polygon) and solid tori
(which are $I$--bundles over an annulus or M\"obius band).  We prove
the following.

\begin{named}{Theorem \ref{thm:epd-span}}
Let $B$ be a component of the characteristic submanifold of $M_A$
which is not a solid torus. Then $B$ is spanned by a collection of
essential product disks (EPDs) $D_1, \dots, D_n$, with the property
that each $D_i$ is embedded in a single polyhedron in the polyhedral
decomposition of $M_A$.
\end{named}

Like all results from the early chapters, Theorem \ref{thm:epd-span} generalizes  
to $\sigma$--adequate and  $\sigma$--homogeneous diagrams. See Sect. \ref{sec:gensigmahomo} for details.

In Chapter \ref{sec:spanning}, we calculate the number of EPDs
required to span the $I$--bundle of $M_A$.  We do this by explicitly
constructing a suitable spanning set of disks (Lemmas
\ref{lemma:spanning-lower} and \ref{lemma:spanning-upper}).  The EPDs
in the spanning set that lie in the \emph{lower} polyhedra of the
decompositions are well understood; they are in one--to--one
correspondence with 2--edge loops in the state graph $\GA$.  The EPDs
in the spanning set that lie in the upper polyhedron are
\emph{complex}; they are not parabolically compressible to EPDs in the
lower polyhedra.  The construction of this spanning set leads to a
proof of Theorem \ref{thm:guts-general}. The spanning set of Chapter
\ref{sec:spanning} also makes it straightforward to detect when the
manifold $M_A$ is an $I$--bundle, leading to a proof of Theorem
\ref{thm:fiber-tree}.
 
The main tool used in Chapters \ref{sec:polyhedra}, \ref{sec:ibundle},
and \ref{sec:spanning} is normal surface theory. In fact, our results
about normal surfaces in the polyhedral decomposition of $M_A$ can
likely be used to attack other topological problems about
$A$--adequate links: see Section \ref{sec:control} in Chapter
\ref{sec:questions} for variations on this theme.

The results of Chapter \ref{sec:spanning} reduce the problem of
computing the Euler characteristic of the guts of $M_A$ to counting
how many complex EPDs are required to span the $I$--bundle of the
\emph{upper} polyhedron.  In Chapter \ref{sec:epds}, we restrict
attention to prime diagrams and address the problem of how to
recognize such EPDs from the structure of the all--$A$ state graph
$\GA$.  Our main result there is Theorem \ref{thm:2edgeloop}, which
describes the basic building blocks for such EPDs.  Roughly speaking,
each of these building blocks maps onto to a 2--edge loop of $\GA$.

In Chapter \ref{sec:nononprime}, we restrict attention to
$A$--adequate diagrams $D(K)$ for which the polyhedral decomposition
includes no non-prime arcs or switches (see Definition
\ref{def:non-prime} on page \pageref{def:non-prime}).  In this case,
one can simplify the statement of Theorem \ref{thm:guts-general} and
give an easier combinatorial estimate for the guts of $M_A$.  To state
our result, let $b_A$ denote the number of bigons in twist regions
of the diagram such that a loop tracing the boundary of this bigon
belongs to the $B$--resolution of $D$. (The $A$--resolution of these
twist regions is \emph{short} in Figure \ref{fig:twist-resolutions} on
page \pageref{fig:twist-resolutions}.) Then, define $m_A = \chi(\GA) -
\chi(\GRA) - b_A$.  We prove the following estimate.

\begin{named}{Theorem \ref{thm:guts-nononprime}}
Let $D(K)$ be a prime, $A$--adequate diagram, and let $S_A$ be the
essential spanning surface determined by this diagram. Suppose that
the polyhedral decomposition of $M_A = S^3\cut S_A$ includes no
non-prime arcs.  Then
$$ \negeul (\GRA) - 8 m_A \: \leq \: \negeul ( \guts(M_A)) \: \leq \: \negeul(\GRA),$$
where the lower bound is an equality if and only if $m_A = 0$.
\end{named}

In Chapter \ref{sec:montesinos}, we study the polyhedral decompositions
of Montesinos links. The main result is the following.

\begin{named}{Theorem \ref{thm:monteguts}}\label{monte-statement-intro}
Suppose $K$ is a Montesinos link with a reduced admissible diagram
$D(K)$ that contains at least three tangles of positive slope.  Then
$$\negeul( \guts(M_A))=\negeul(\GRA).$$
Similarly, if $D(K)$ contains at least three tangles of negative slope, then
$$\negeul( \guts(M_B))=\negeul(\GRB).$$
\end{named}

The arguments in Chapters \ref{sec:epds}, \ref{sec:nononprime}, and
\ref{sec:montesinos} require a detailed and fairly technical analysis
of the combinatorial structure of the polyhedral decomposition; we
call this analysis \emph{tentacle chasing}\index{tentacle chasing}. In
addition, Chapters \ref{sec:nononprime} and \ref{sec:montesinos}
depend heavily on Theorem \ref{thm:2edgeloop} in Chapter
\ref{sec:epds}.

In Chapter \ref{sec:applications}, we give the applications to volume
estimates and relations with the colored Jones polynomials that were
discussed earlier in this introduction. The results in this chapter do
not use Chapter \ref{sec:nononprime} at all, and do not directly
reference Chapter \ref{sec:epds} or the arguments of Chapter
\ref{sec:montesinos}.  Thus, having the statement of Theorem
\ref{thm:monteguts} at hand, a reader who is eager to see the
aforementioned applications may proceed to Chapter
\ref{sec:applications} immediately after Chapter \ref{sec:spanning}.

In Chapter \ref{sec:questions}, we state several open questions and
problems that have emerged from this work, and discuss potential
applications of the methods that we have developed.

\section{Acknowledgments}

This manuscript is the fruit of a project conducted over several years, in several different
locations, with the assistance of many people and organizations.  We
thank our home institutions, Temple University, Michigan State
University, and Brigham Young University, for providing support and
venues for various collaborative meetings on each of the three
campuses and
the National Science Foundation (NSF) for supporting
this project through grant funding. 
During this project, Futer was supported by
 by NSF grant DMS--1007221, Kalfagianni was supported by grants DMS--0805942 and DMS--1105843,
 and Purcell was supported by NSF grant DMS--1007437 and a Sloan Research Fellowship.

Some of this work was carried out and discussed while the authors were
in attendance at workshops and conferences. These include the
\emph{Moab topology conference}, in Moab, Utah in May 2009;
\emph{Interactions between hyperbolic geometry, quantum topology, and
  number theory}, at Columbia University in June 2009; the \emph{Joint
  AMS--MAA Meetings} in San Francisco in January 2010; \emph{Topology
  and geometry in dimension three}, in honor of William Jaco, at
Oklahoma State University in June 2010; \emph{Knots in Poland III} at
the Banach Center in Warsaw, Poland, in July 2010; and \emph{Faces of
  geometry: 3--manifolds, groups, \& singularities}, in honor of
Walter Neumann, at Columbia University in June 2011.  We thank the
organizers of these conferences for their hard work, and hospitality.

Our tools and techniques owe a sizable intellectual debt to the prior
work of Ian Agol and Marc Lackenby. We are grateful to both of them
for a number of enlightening discussions.  

We thank Jessica Banks for her careful reading and numerous helpful suggestions.
We also thank the
MSU graduate students Cheryl Balm, Adam Giabrone, Christine Lee, and
Indra Schottland for the interesting questions that they raised during
a reading seminar on parts of this monograph.

We thank the referees for their helpful comments. We also thank Chris
Atkinson, John Baldwin, Abhijit Champanerkar, Oliver Dasbach, Charlie
Froh\-man, Cameron Gordon, Eli Grigsby, Joanna Kania-Bartoszynska, Ilya
Kofman, Tao Li, William Menasco, Walter Neumann, Neal Stoltzfus, and Roland
van der Veen for their comments on, and interest in, this project.

Finally, we thank Yael Futer, George Pappas, and Tim Purcell for their
support --- especially during several extended visits that the authors
had to pay at each other's households during the course of this
project.

\chapter{Decomposition into 3--balls}\label{sec:decomp}
In this chapter, we start with a connected link diagram and explain
how to construct state graphs and state surfaces.  We cut the link
complement in $S^3$ along the state surface, and then describe how to
decompose the result into a collection of topological balls whose boundaries have a
checkerboard coloring.  There are two steps to this decomposition; the
first is explained in Section \ref{subsec:decomp-balls}, and the
second in Section \ref{subsec:primeness}.
Finally, in Section \ref{subsec:generalization}, we briefly describe how to generalize
the decomposition to a broader class of links considered by Ozawa in \cite{ozawa}.

The combinatorics of the decomposition will be used heavily in later
chapters to prove our results.  Consequently, in this chapter we will
define terminology that will allow us to refer to these combinatorial
properties efficiently.  Thus the terminology and results of this chapter
are important for all the following chapters.

\section{State circles and state surfaces} 

Let $D$ be a connected link diagram, and $x$ a crossing of $D$.
Recall that associated to $D$ and $x$ are two link diagrams, each with
one fewer crossing than $D$, called the \emph{$A$--resolution} and
\emph{$B$--resolution} of the crossing.  See Figure
\ref{fig:splicing} on page \pageref{fig:splicing}.

A \emph{state}\index{state} of $D$ is a choice of $A$-- or $B$--resolution for each
crossing.  Applying a state to the diagram, we obtain a crossing free
diagram consisting of a disjoint collection of simple closed curves on
the projection plane.  We call these curves \emph{state
  circles}\index{state circle}.  The
\emph{all--$A$ state}\index{state!all--$A$ state}\index{all--$A$ state} of the diagram $D$ chooses the $A$--resolution
at each crossing.  We denote the union of corresponding state circles by
$s_A(D)$\index{$s_A = s_A(D)$}, or simply $s_A$.  Similarly, one can define an all--$B$
state\index{state!all--$B$ state}\index{all--$B$ state} and state
circles $s_B = s_B(D)$\index{$s_B = s_B(D)$}.

Start with the all--$A$ state of a diagram.  From this, we may form a
connected graph in the plane.

\begin{define}\label{def:HA}
Let $s_A$ be the union of state circles in the all--$A$ state of a diagram $D$. To this union of circles, we attach one edge for each crossing, which records the location of the crossing. (These edges are dashed in Figure \ref{fig:splicing} on page
\pageref{fig:splicing}.)  The resulting graph is trivalent, with edges
coming from crossings of the original diagram and from state circles.
To distinguish between these two, we will refer to edges coming from
state circles just as state circles, and edges from crossings as
\emph{segments}\index{segment}.  This graph will be important in the
arguments below.  We will call it \emph{the graph of the
  $A$--resolution}\index{graph of the $A$--resolution, $H_A$}, and denote it by $H_A$\index{$H_A$, graph of the $A$--resolution}.
\end{define}

In the introduction, we introduced the \emph{$A$--state graph} $\GA$.
This will factor into our calculations in later chapters.  For
now, note that $\GA$ is obtained from $H_A$ by collapsing state
circles to single vertices.

We may similarly define the graph of the $B$--resolution, $H_B$, and the
$B$--state graph $\GB$.  Indeed, every construction that follows will work with
only minor modifications (involving handedness) if we replace $A$--resolutions
with $B$--resolutions.  For ease of exposition, we will mostly
consider $A$--resolutions.

We now construct a surface related to a state $\sigma$.  First, draw
the circles of the $\sigma$--resolution, $s_\sigma$.  These state circles
bound disjoint disks in the $3$--ball below the projection plane.  
Form the state surface\index{state surface} $S_\sigma$\index{$S_\sigma$, state surface of $\sigma$} by
 taking this disjoint collection of disks bounded by state circles,
and attaching a twisted band for each crossing.
The result is a surface whose boundary is the link.  A well--known example of a state surface is the Seifert surface constructed from a diagram, where the state $\sigma$
is chosen following an orientation on $K$.

See Figure \ref{fig:statesurface} for a state surface $S_A$\index{$S_A$, all--$A$ state surface}
corresponding to the all--$A$ state.

\begin{figure}
\includegraphics{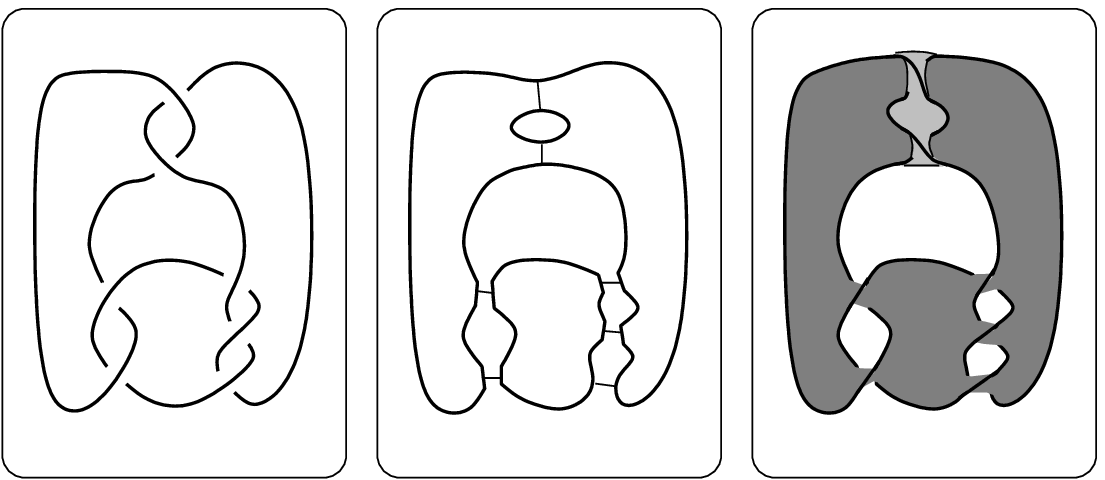}
\caption{Left to right:  A diagram.  The graph
  $H_A$\index{$H_A$, graph of the $A$--resolution!example}.  The state surface $S_A$\index{$S_A$, all--$A$ state surface!example}.}
\label{fig:statesurface}
\end{figure}

\begin{lemma}\label{lemma:ga-spine}
The graph $\mathbb{G}_\sigma$ is a spine for the surface $S_\sigma$.
\end{lemma}

\begin{proof}
By construction, $\mathbb{G}_\sigma$ has one vertex for every circle of
$s_\sigma$ (hence every disk in $S_\sigma$), and one edge for every
half--twisted band in $S_\sigma$. This gives a natural embedding of
$\mathbb{G}_\sigma$ into the surface, where every vertex is embedded
into the corresponding disk, and every edge runs through the
corresponding half-twisted band. This gives a spine for $S_\sigma$.
\end{proof}

\begin{lemma}
\label{lemma:sa-orientable}
The surface $S_\sigma$ is orientable if and only if
$\mathbb{G}_\sigma$ is bipartite.
\end{lemma}

\begin{proof}
It is well--known that a graph is bipartite if and only if all loops
have even length.

For the ``if'' direction, assume that $\mathbb{G}_\sigma$ is
bipartite. Then, we may construct a (transverse) orientation on
$S_\sigma$, as follows. First, pick a normal direction to one disk
(corresponding to one vertex of $\mathbb{G}_\sigma$). Then, extend
over half--twisted bands to orient every adjacent disk, and continue
inductively. This inductive construction of a tranverse orientation on
$S_\sigma$ will never run into a contradiction, precisely because
every loop in $\mathbb{G}_\sigma$ has even length. Thus $S_\sigma$ is
a two--sided surface in $S^3$, hence orientable.

For the ``only if'' direction, suppose $\mathbb{G}_\sigma$ is not
bipartite, hence contains a loop of odd length. By embedding
$\mathbb{G}_\sigma$ as a spine of $S_\sigma$, as in Lemma
\ref{lemma:ga-spine}, we see that this loop of odd length is
orientation--reversing on $S_\sigma$.
\end{proof}

Our primary focus will be on the all $A$--state and the $A$--state
surface $S_A$.  It will be helpful to isotope the state surface $S_A$
into a topologically convenient position.  Recall that by construction, the 
disks used to construct $S_A$ lie in the $3$--ball below the projection plane.  
Each of these disks can be thought of as
consisting of a thin annulus with outside boundary attached to the
state circle, and then a \emph{soup can}\index{soup can} attached to the inside
boundary of the annulus.  That is, a long cylinder runs deep under the
projection plane, with a disk at the bottom.  These soup cans will be
nested, with outer state circles bounding deeper, wider soup cans.
Finally, isotope the state circles and bands of the diagram onto the
projection plane, except at crossings of the diagram in which the
rectangular band runs through a small crossing ball coming out of the
projection plane.  When we have finished, aside from crossing balls,
the link diagram sits on the plane of projection, and the surface
$S_A$ lies below.

Consider the manifold created by cutting $S^3$ cut along (a regular neighborhood of) $S_A$.  We 
will refer to this manifold as $S^3 \cut S_A$\index{$S^3 \cut S_A$, or $S^3$ cut along $S_A$}, or $M_A$\index{$M_A = S^3 \cut S_A$} for short.  With $S_A$
isotoped into the position above, it now is a straightforward matter
to prove various topological conditions on $M_A$.

\begin{lemma}
  \label{lemma:handlebody}
  The manifold $M_A = S^3\cut S_A$ is homeomorphic to a handlebody.
\end{lemma}

\begin{proof}
By definition, the manifold $M_A$ is the complement of a regular
neighborhood of $S_A$ in $S^3$.  A regular neighborhood of $S_A$
consists of the union of a regular neighborhood of the link and a
regular neighborhood of the twisted rectangles, as well as regular
neighborhoods of each of the soup cans.  Note first that the union of
the regular neighborhood of the link and the rectangular bands deformation 
retracts to the projection graph of
the diagram, which is a planar graph, and so its complement is a
handlebody.  Next, when we attach to this a regular neighborhood of a
soup can, we are cutting the complement along a 2--handle.  Since the
result of cutting a handlebody along a finite number of non-separating 2--handles is
still a handlebody, the lemma follows.
\end{proof}

\section{Decomposition into topological balls}\label{subsec:decomp-balls}

We will cut $M_A$ along a collection of disks, to obtain a
decomposition of the manifold into a collection of topological balls.
In fact, we will eventually show this decomposition is actually a
decomposition of $M_A$ into ideal polyhedra, in the sense of the
following definition.

\begin{define}\label{def:polyhedra}
An \emph{ideal polyhedron}\index{ideal polyhedron} is a 3--ball with a
graph on its boundary, such that complementary regions of the graph on
the boundary are simply connected, and the vertices have been removed
(i.e. lie at infinity).
\end{define}

\begin{remark}
Menasco's work \cite{menasco:polyhedra} gives a decomposition of any
link complement into ideal polyhedra.  When the link is alternating,
the resulting polyhedra have several nice properties.  In particular,
they are checkerboard colored, with 4--valent vertices.  However, when
the link is not alternating, these properties no longer hold.  For alternating diagrams, our polyhedra will be exactly the same as Menasco's. More generally, we will
see that our polyhedral decomposition of $M_A$ also has a checkerboard coloring
and 4--valent vertices.
\end{remark}

There are two stages of the cutting.  For the first stage, we will take one disk for each
complementary region in the complement of the projection graph in $S^2$, with the boundary of the disk lying on $S_A$ and the link.  (Here we are using the assumption that our
diagram is connected when we assert that the complementary regions of its
projection graph are disks.)  Note that each region of the projection
graph corresponds to exactly one region of $H_A$, the graph of the
$A$--resolution.  Thus we may also refer to these disks as
corresponding to regions of the complement of $H_A$ in the projection
plane.

To form the disk that we cut along, we isotope the disk of the given
region by pushing it under the projection plane slightly, keeping its
boundary on the state surface $S_A$, so that it meets the link a
minimal number of times.  Since $S_A$ itself lies on or below the projection
plane, except in the crossing balls, we know we can push the disk
below the projection plane everywhere except possibly along the half-twisted rectangles at the crossings.  At each crossing met by the particular
region of $H_A$, the boundary of the region either runs past the
twisted rectangular band without entering it, in which case it can be
isotoped out of the crossing ball, or the boundary of the region runs
along the attached band.  In the latter case, the boundary comes into
the crossing from below the projection plane.  The crossing twists it
such that if it continued to follow the band through the state
surface, it would come out lying above the projection plane.  To avoid
this, isotope such that the boundary of the disk runs over the link
inside the crossing ball, and so exits the crossing ball with the disk
still under the projection plane.

After this isotopy, the result is one of the disks we cut along.  We
call such a disk a \emph{white disk}, 
 or \emph{white face}\index{white face}, indicative
of a checkerboard coloring we will give our polyhedral decomposition.
Notice the above construction immediately gives the following lemma.

\begin{lemma}
White disks meet the link only in crossing balls, and then only at
under-crossings. Additionally, white disks lie slightly under the
projection plane everywhere. \qed
\label{lemma:white-disk}
\end{lemma}

See Figure \ref{fig:disk} for an example.

\begin{figure}
	\includegraphics{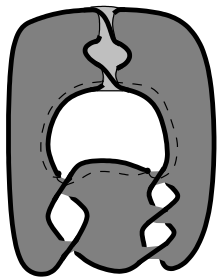}
	\caption{A white disk\index{white face!example} lying just below the
    projection plane, with boundary (dashed line) on underside of
    shaded surface.  Note this disk meets the link in exactly two
    points.}
\label{fig:disk}
\end{figure}

Now, some of the white disks will not meet the link at all.  These
disks are those corresponding to regions of the projection graph whose
boundaries never run through a crossing band.  Therefore, the
boundaries of such disks are isotopic to a state circle of $s_A$.
Hence they are isotopic to soup cans attached to form the state
surface.  We call these particular soup cans \emph{innermost
  disks}\index{innermost disk},
since they will not have any additional soup cans nested inside them.
Remove all white disks isotopic to innermost disks from consideration,
since they are isotopic into the boundary of $M_A$.

We are left with a collection of disks $\CalD$\index{$\CalD$, collection of white disks (faces)}, each
lying in $S^3$ with boundary on the state surface $S_A$ and on the
link $K$.  Cut along these disks.

\begin{lemma}
\label{lemma:3-ball}
Each component of $M_A \cut \CalD$ is homeomorphic to a 3--ball.
\end{lemma}

\begin{proof}
Notice there will be a single component above the projection plane.
Since we have cut along each region of the projection graph, either by
cutting along a soup can or along one of the disks in $\CalD$, this
component must be homeomorphic to a ball.

Next, consider components which lie below the projection plane.  These
lie between soup can disks.  Since any disk cuts the 3--ball below the
projection plane into 3--balls, these components must also each be
homeomorphic to 3--balls. 
\end{proof}

\begin{define}
The single 3--ball of the decomposition which lies above the plane of
projection we call the \emph{upper 3--ball}\index{upper polyhedron}\index{upper 3--ball}.  All
3--balls below the 
plane of projection will be called \emph{lower 3--balls}\index{lower polyhedra}.
  \label{def:upper-lower}

\end{define}

%

We now build up a combinatorial description of the upper and lower
3--balls.  The boundary of any 3--ball will be marked by
2--dimensional regions separated by edges and ideal vertices.  The
regions (faces) come from white disks and portions of the surface
$S_A$, which we shade.\index{shaded face}

\begin{notation}
In the sequel, we will use a variety of colors to label and distinguish the different shaded regions
on the boundary of the $3$--balls. All of these colored regions come from the surface $S_A$, and all of them are considered shaded. 
See, for
example, Figure \ref{fig:lower-schematic} on page
\pageref{fig:lower-schematic}.
\end{notation}

Continuing the combinatorial description of the upper and lower
3--balls, edges on a 3--ball are components of intersection of white
disks in $\CalD$ with  (the boundary of a regular neighborhood of) $S_A$.  Each edge runs between strands of the
link.  As usual, each ideal vertex lies on the torus boundary of the tubular neighborhood of a link component (see
e.g. Menasco \cite{menasco:polyhedra}).

Note each edge bounds a white disk in $\CalD$ on one side, and a
portion of the shaded surface $S_A$ on the other side.  Thus, by
construction, we have a checkerboard coloring of the
2--dimensional regions of our decomposition.  Since the white regions
are known to be disks, showing that our 3--balls are actually
polyhedra amounts to showing that the shaded regions are also simply
connected.

In the process of showing these regions are simply connected, we will
build up a combinatorial description of how the white and shaded faces are super-imposed on 
the projection plane, and how these faces interact with the
planar graph $H_A$.  This combinatorial description will be useful in
proving the main results.

\begin{notation}\index{shaded face}
From here on, we will refer to both white and shaded regions of our
decomposition as \emph{faces}.  We do not assume that the shaded faces
are simply connected until we prove they are, in Theorem
\ref{thm:simply-connected}.
\end{notation}

Consider first the lower 3--balls.

\begin{lemma}
Let $R$ be a non-trivial component of the complement of $s_A$ in the projection plane. Then there is exactly one lower 3--ball corresponding to $R$.
The white faces of this 3--ball correspond to the regions in the complement of $H_A$ that are contained in $R$.
\label{lemma:lower-circles}
\end{lemma}

Here, by a non-trivial component of the complement of $s_A$, we mean a
component which is not itself an innermost disk. 

\begin{proof}
The soup cans attached to the circles $s_A$ when forming $S_A$ cut the
3--ball under the projection plane into the lower 3--balls of the
decomposition.  We will have exactly one such component for each
non-trivial region of $s_A$.  Faces are as claimed, by construction.
\end{proof}

\begin{lemma}
Each lower 3--ball is an ideal polyhedron, identical to the
checkerboard polyhedron obtained by restricting to the alternating
diagram given by the subgraph of $H_A$ contained in a non-trivial
region of $s_A$.
\label{lemma:lower-alt}\index{lower polyhedra}
\end{lemma}

\begin{proof}
Ideal edges of a lower 3--ball stretch from the link, across the state
surface $S_A$, to the link again, and bound a disk of $\CalD$ on one
side.  The disks in $\CalD$, along which we cut, block portions of the
link from view from inside the lower 3--ball.  In particular, because
each disk of $\CalD$ lies below the projection plane except at
crossings, and the link lies on the projection plane except at
crossings, the only parts of the link visible from inside a lower
3--ball correspond to crossings of the diagram.  That is, only small
segments of under-crossings of the link are visible from inside a lower
3--ball.  Since edges meet at the same ideal vertex if and only if
they meet the same strand of the link visible from below, edges of the
lower 3--ball will meet other edges at under-crossings of the link.
Notice that the only relevant under-crossings will be those which
correspond to segments of $H_A$ which lie inside the region $R$ in the
complement of $s_A$, as in Lemma \ref{lemma:lower-circles}.  All other
crossings are contained outside our 3--ball.

For each such under-crossing, note two disks of $\CalD$ meet at the
under-crossing.  These disks correspond to the two regions of $H_A$
adjacent to the segment of $H_A$ at that crossing, and they both meet
in two edges.  Thus each vertex is 4--valent.  Finally, the graph
formed by edges and ideal vertices must be connected, since the region
of the complement of $s_A$ is connected.  Hence we have a 4--valent,
connected graph on the plane, corresponding to a non-trivial subgraph
of $H_A$ contained in a single component of the complement of $s_A$.

Any connected, 4--valent graph on the plane corresponds to an
alternating link, and gives the checkerboard decomposition of such a
link.  Thus it is an honest polyhedron.  The vertices of the 4--valent
graph correspond to crossings of the alternating diagram.  Notice the
vertices of the 4--valent graph also came from crossings of our link
diagram.  Thus we have a correspondance between a lower polyhedron and
an alternating link with exactly the same crossings as in our subgraph
of $H_A$.
\end{proof}

By Lemma \ref{lemma:lower-alt}, we think of the lower polyhedra as
corresponding to the largest alternating pieces of our knot or link
diagram.

Schematically, to sketch a lower polyhedron, start by drawing a
portion of $H_A$ which lies inside a non-trivial region of the
complement of $s_A$.  Mark an ideal vertex at the center of
each segment of $H_A$.  Connect these dots by edges bounding white
disks, as in Figure \ref{fig:lower-schematic}.

\begin{figure}
	\includegraphics{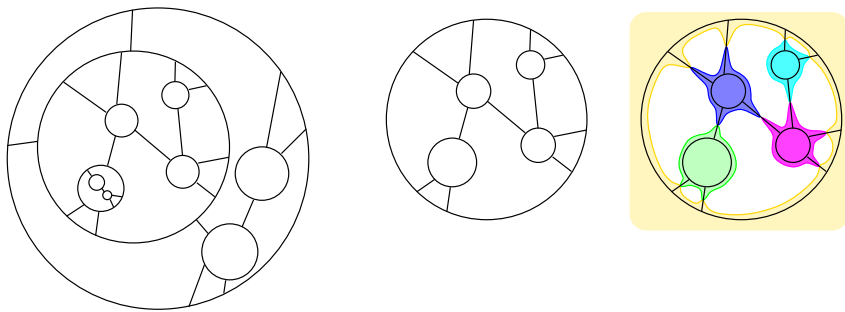}
\caption{Left to right:  An example graph $H_A$.  A subgraph
	corresponding to a region of the complement of $s_A$.  White and
	shaded faces of the corresponding lower polyhedron\index{lower polyhedra!example}.}
\label{fig:lower-schematic}
\end{figure}

\vskip 0.04in

Now we consider the upper 3--ball, or the single 3--ball lying above
the plane of projection.  Again, ideal edges on this 3--ball will meet
at ideal vertices corresponding to strands of the link visible from
inside the 3--ball.  However, the identification no longer occurs only at single crossings.  Still, we obtain the following.

\begin{lemma}
The upper 3--ball admits a checkerboard coloring, and all ideal
vertices are 4--valent.
\label{lemma:top}
\end{lemma}

\begin{proof}
By shading the surface $S_A$ gray and disks of $\CalD$ white, we
obtain the checkerboard coloring.

Ideal vertices are strands of the link visible from inside the
3--ball.  Each strand of the link is cut off as it enters an
under-crossing.  Thus visible portions of the link from above will lie
between two under-crossings.  At each under-crossing, two edges bounding
a single disk meet the link on each side, as illustrated in Figure
\ref{fig:under-crossing}.  Thus these two edges will share an ideal
vertex with the two edges bounding a disk at the next under-crossing.
Since no other edges meet the link strand between under-crossings, the
vertex must be 4--valent.
\end{proof}

\begin{figure}
  \begin{center}
    \input{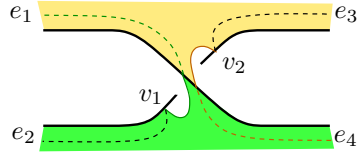}
  \end{center}
  \caption{Shown are portions of four ideal edges, terminating at
    under-crossings on a single crossing.  Ideal edges $e_1$ and $e_2$
    bound the same white disk and terminate at the ideal vertex $v_1$.
    Ideal edges $e_3$ and $e_4$ bound the same white disk and
    terminate at the ideal vertex $v_2$.}
  \label{fig:under-crossing}
\end{figure}

We give a description of the upper 3--ball by drawing the faces, ideal
edges, and ideal vertices of its boundary 
superimposed upon $H_A$, the graph of the $A$--resolution.  We will see
that the combinatorics of $H_A$ determine the combinatorics of this
upper 3--ball.  We will continue referring to it as the ``upper 3--ball'' until we prove in Theorem \ref{thm:simply-connected} that it is indeed a polyhedron.

\begin{notation}\label{notation:top-right}
The combinatorial picture of the upper 3--ball superimposed on the graph of
$H_A$ will be described by putting together local moves that occur at
each segment.  In order to describe a particular move at a particular
segment $s$, we will assume that the diagram has been rotated so that
$s$ is vertical.  Now there are two state circles meeting $s$, one at
the top of $s$ and one at the bottom of $s$.  Note the choice of top
and bottom is not well--defined, but depends on how we rotated $s$ to
be vertical.  However, the pair of edges of $H_A$, coming from state
circles, which meet the segment at the top--right and at the
bottom--left, is a well--defined pair.

Moreover, note that each crossing of the link diagram meets two shaded
faces, one on the top and one on the bottom of the crossing, when the
crossing is rotated to be as in Figure \ref{fig:splicing}.  See also
Figure \ref{fig:under-crossing}.  Given a choice of shaded face, and a
choice of crossing which that shaded face meets, this shaded face will
determine a well--defined state circle meeting the segment $s$ of
$H_A$ corresponding to that crossing.  We may rotate $H_A$ so that the
segment $s$ is vertical, and the state circle corresponding to our
chosen shaded face is on the top, as in Figure \ref{fig:tentacle}.
Once we have performed this rotation, the state circle on top of $s$,
as well as that on the bottom, and the four edges top--left,
top--right, bottom--left, and bottom--right, are now completely
well--defined.  
\end{notation}

The following description of the upper 3--ball, superimposed on the
graph $H_A$, is illustrated in Figures \ref{fig:tentacle},
\ref{fig:undercr-tentacles}, and \ref{fig:tentacle-mult}.

\begin{figure}
\includegraphics{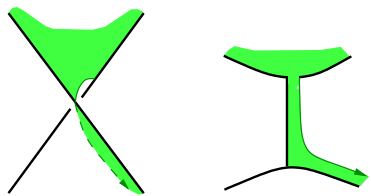}
\caption{Left: A tentacle\index{tentacle}\index{shaded face!tentacle} continues a shaded face in
the upper $3$--ball. Right: visualization of the tentacle on the graph
  $H_A$.}
\label{fig:tentacle}
\end{figure}

First, we discuss ideal vertices\index{upper 3--ball!ideal vertex} of
the upper 3--ball and how these 
appear on the graph $H_A$.  Each ideal vertex corresponds to a strand of
the diagram between two under-crossings, running over a (possibly
empty) sequence of over-crossings.  In terms of the graph $H_A$, for
any given segment $s$, rotate $H_A$ so that $s$ is vertical, as in
Notation \ref{notation:top-right}.  Now, draw two dots on the graph
$H_A$ on the edges at the top--right and bottom--left of $s$, near $s$.
Erase a tiny bit of state circle between the segment $s$ and each dot.
After this erasure is performed at all segments of $H_A$, the connected components that remain 
(which are piecewise linear arcs between a pair of dots), correspond to
ideal vertices. 
 This erasing has
been done on the top--right of Figure \ref{fig:tentacle}, and for all
edges shown in the right side of Figure \ref{fig:tentacle-mult}.  It
has also been done on the top--right and bottom--left of Figure
\ref{fig:undercr-tentacles}.  

\begin{figure}
  \input{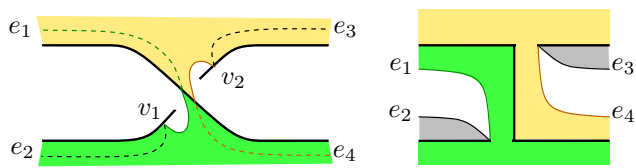}
  \caption{The shaded faces around the crossing of Figure
    \ref{fig:under-crossing}. (Note: For grayscale
    versions of this monograph, green faces will appear as dark gray,
    orange faces as lighter gray.)}
  \label{fig:undercr-tentacles}
\end{figure}

Every ideal edge of the upper $3$--ball\index{upper 3--ball!ideal edge} 
bounds a white disk on one side, and a shaded face on the other, and
starts and ends at under-crossings.  This ideal edge runs through a
crossing, for which we may assume the orientation is as in Figure
\ref{fig:splicing}, with the shaded face at the top of the crossing. Then, following
Notation \ref{notation:top-right},  this ideal edge
will run from the top--right of a crossing, through the crossing, and continues parallel to the link strand on the bottom--right. 
%
Superimposed on $H_A$, the ideal edge will start at the top--right of a segment
$s$, run adjacent to $s$, and then continue horizontally parallel to the state circle at
the bottom of $s$.  
This is shown for a single edge in Figure
\ref{fig:tentacle}.  Thus the white face adjacent to this ideal edge
of the decomposition corresponds to the region of the complement of
$H_A$ to the right of this particular segment of $H_A$.

Note the shaded face\index{upper 3--ball!shaded face} adjacent to
this ideal edge is the same shaded face on top of the segment of
$H_A$, or from the top of the crossing.  It runs between the link and
the ideal edge, adjacent to the white face, until the ideal edge
terminates at another under-crossing.  We draw it on the graph $H_A$ as
shown in the right panel of Figure \ref{fig:tentacle}.  Figure \ref{fig:tentacle-mult}
shows multiple ideal edges.  Figure \ref{fig:undercr-tentacles} shows
a single segment with the two ideal edges that are adjacent to it.
Notice in that figure that when the green\footnote{Note: For
grayscale versions of this monograph, green will refer to the darker
gray shaded face, orange to the lighter gray.} 
shaded face is rotated to be 
on top, the green ideal edge runs from the top--right to the
bottom--right of the segment.  However, rotated as shown, the green
runs from the bottom--left to the top--left.  These are symmetric.

\begin{define}
A \emph{tentacle}\index{tentacle}\index{shaded face!tentacle} is defined to be the strip of shaded
face running from the top right of a segment of $H_A$, adjacent to the
link, along the bottom right state circle of the edge of $H_A$, as
illustrated in Figure \ref{fig:tentacle}.  Notice that a tentacle is
bounded on one side by a portion of the graph $H_A$, and on the other
side by exactly one ideal edge.
\label{def:tentacle}
\end{define}

Note that a tentacle will continue past segments of $H_A$ on the
opposite side of the state circle without terminating, spawning new
tentacles, but will terminate at a segment on the same side of the
state surface.  This is shown in Figure \ref{fig:tentacle-mult}.

\begin{figure}
  \begin{center}
    \input{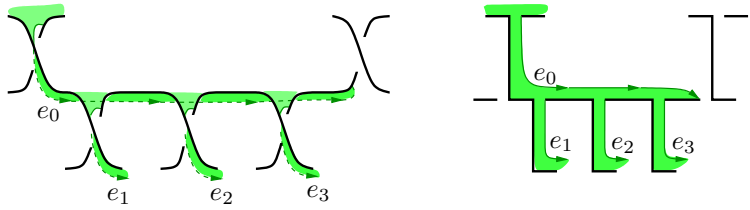}
  \end{center}
  \caption{Left: part of a shaded face in an upper $3$--ball. Right: the
    corresponding picture, superimposed on $H_A$.  The tentacle next to the ideal
    edge $e_0$ terminates at a segment on the same side of the state
    circle on $H_A$.  It runs past segments on the opposite side of
    the state circle, spawning new tentacles associated to ideal edges
    $e_1$, $e_2$, $e_3$.}
  \label{fig:tentacle-mult}
\end{figure}


\begin{define}
A given tentacle is adjacent to a segment of $H_A$.  Rotate the
segment to be vertical, so that the tentacle lies to the right of this segment.  The
\emph{head}\index{head!of tentacle}\index{tentacle!head} of a tentacle
is the portion attached to the top right of the segment of $H_A$.  The
\emph{tail}\index{tail (of tentacle)}\index{tentacle!tail} is the part adjacent to
the lower right of the state circle. We think of the tentacle as directed from head to tail.

Alternately, if we think of ideal edges of the decomposition as
beginning at the top--right of a crossing, rotated as in Notation
\ref{notation:top-right}, and ending at the bottom--left, this orients each ideal edge.  Since each tentacle is bounded by an ideal
edge on one side, this in turn orients the tentacle. See Figure \ref{fig:tentacle-mult}.
\label{def:tentacle-direct}
\end{define}

Notice that for any segment, we will see the head of some tentacle at
the top right of the segment, and the head of another tentacle at the
bottom left of the segment.  In Figure \ref{fig:undercr-tentacles}, we
see the heads of two tentacles.  The orange
one is on the top--right,
the green one on the bottom--left.  In addition, the tails of two
other tentacles are shown in gray in that figure.

\section{Primeness}\label{subsec:primeness}

Near the beginning of the last section, we stated that there were
two stages to our polyhedral decomposition.  The first stage is that
explained above, given by cutting along white disks corresponding to
complementary regions of $H_A$ in the projection plane.  This may not
cut $M_A$ into sufficiently simple pieces.  In many cases, we may
have to do some additional cutting to obtain polyhedra with the
correct properties.  This additional cutting is described in this
section.

Recall from Definition \ref{primetwist} that a link diagram is
\emph{prime}\index{prime!diagram} if any simple closed curve that
meets the diagram transversely in two points bounds a region of the
projection plane with no crossings.  By Lemma \ref{lemma:lower-alt}, the
polyhedra in our decomposition that lie below the projection plane
correspond to alternating link diagrams.  At this stage, these
diagrams may not all be prime.  We need to modify the polyhedral
decomposition so that polyhedra below the projection plane do, in
fact, correspond to prime alternating diagrams.

\begin{define}
The graph $H_A$ is \emph{non-prime}\index{non-prime!graph} if
there exists a state circle $C$ of $s_A$ and an arc $\alpha$ with both
endpoints on $C$ such that the interior of $\alpha$ is disjoint from
$H_A$ and $\alpha$ cuts off two non-trivial subgraphs when restricted
to the subgraph of $H_A$ corresponding to the region of $s_A$
containing $\alpha$.  The arc $\alpha$ is defined to be a
\emph{non-prime arc}\index{non-prime!arc}.

More generally, define a non-prime arc inductively as follows.
Suppose $\alpha_1, \dots, \alpha_n$ are non-prime arcs with endpoints
on the same state circle $C$.  Suppose there is a region $R$ of the
complement of $H_A \cup (\cup_{i=1}^n \alpha_i)$ and an arc $\beta$
embedded in that region with both endpoints on $C$ such that $\beta$
splits $R$ into two new regions, both containing state circles of
$s_A$.  Then $H_A \cup (\cup_{i=1}^n \alpha_i)$ is
\emph{non-prime}\index{non-prime!graph},
and $\beta$ is defined to be a \emph{non-prime arc}\index{non-prime!arc}.
\label{def:non-prime}
\end{define}

Figure \ref{fig:nonprime-ex} gives an example of a collection of
non-prime arcs.

\begin{figure}
  \begin{center}
    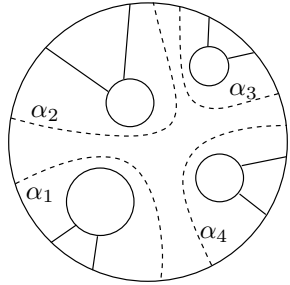
  \end{center}
  \caption{Arcs $\alpha_1$, $\alpha_2$, and $\alpha_3$ are all
    non-prime arcs.  However, $\alpha_4$ is not non-prime in $H_A \cup
    (\alpha_1 \cup \alpha_2 \cup \alpha_3)$, since there is a region
    in the complement of $H_A \cup (\alpha_1 \cup \dots \cup
    \alpha_4)$ which contains no state circles\index{non-prime!arc}.}
  \label{fig:nonprime-ex}
\end{figure}

We call this non-prime because the corresponding alternating diagram
of the polyhedron below the projection plane will no longer be a prime
alternating diagram in the presence of such an arc.

The arc $\alpha$ meets two ideal edges of a polyhedron below the
projection plane, and two ideal edges of the $3$--ball above the
projection plane.  Modify the polyhedral decomposition as follows: we
take our finger and push the arc $\alpha$ down against the soup can
corresponding to the state circle $C$\index{prime decomposition}.
That is, we surger along the disk bounded by $\alpha$ and an arc
parallel to $\alpha$ running along the soup can.  Topologically, this
divides the corresponding lower polyhedron into two, replacing two
shaded faces by one, and one white face by two.  No new ideal vertices
are added, but the two ideal edges met by the non-prime arc are
modified so that each runs from its original head to a neighborhood of
the non-prime arc, then parallel to the non-prime arc, then along the
tail of the other original ideal edge to where that ideal edge
terminates.

Combinatorially, this does the following.  In the lower polyhedron,
under the projection plane, cut the polyhedron into two polyhedra by
joining the ideal edges attached by the non-prime arc.  See Figure
\ref{fig:nonprime-lower}.  The lower polyhedra now correspond to
alternating links whose state circles contain $\alpha$.

\begin{figure}
\includegraphics{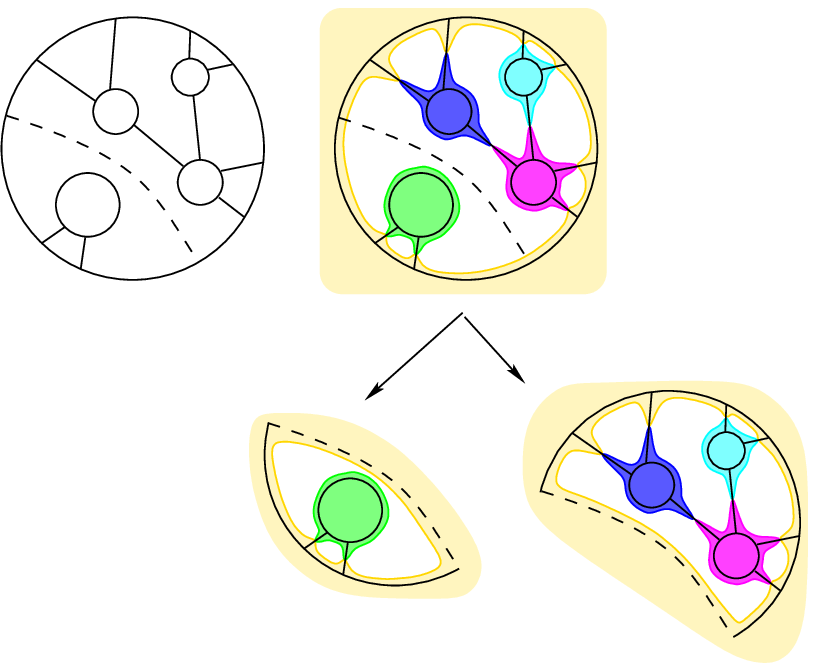}
\caption{Splitting a lower polyhedron\index{prime decomposition!example} into two along a non-prime arc.
  These give two polyhedral regions, defined in Definition
  \ref{def:polyhedral-region} on page \pageref{def:polyhedral-region}.}
\label{fig:nonprime-lower}
\end{figure}

On the boundary of the upper $3$--ball, connect tentacles at both endpoints of the
non-prime arc $\alpha$ by attaching a small regular neighborhood of
$\alpha$, for example as in Figure \ref{fig:nonprime-top}.  We call
this neighborhood of $\alpha$ connecting tentacles a \emph{non-prime
  switch}\index{non-prime!switch}\index{shaded face!non-prime switch}.  A priori, a non-prime switch might join two 
shaded faces into one, or else connect a shaded face to itself. (In fact, we will show in the next chapter that it does the former.)  A non-prime switch also reroutes ideal edges adjacent to the
connected tentacles to run adjacent to the non-prime arc.  Notice that
these two edges still have well defined orientations, although unlike
the case of tentacles, this does not give a direction to the non-prime
switch.

\begin{figure}
\includegraphics{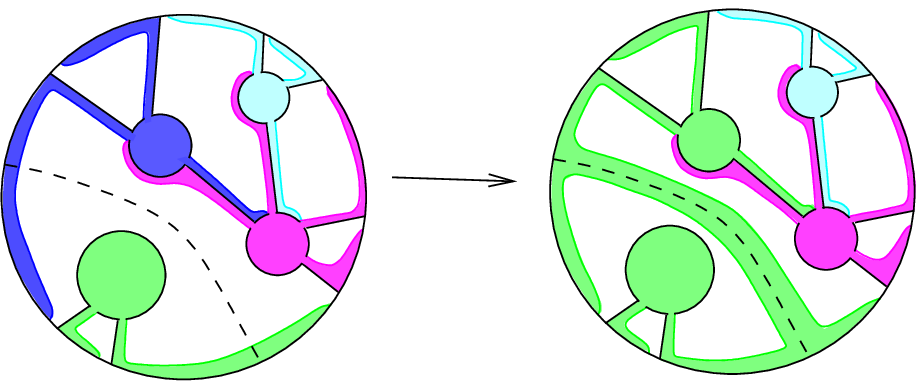}
\caption{Splitting the upper $3$--ball along a non-prime
  arc\index{prime decomposition!example}.}
\label{fig:nonprime-top}
\end{figure}

\begin{define}
Let $D$ be a diagram of a link, with $H_A$ the corresponding graph of
its $A$--resolution.  Let $\alpha_1, \dots, \alpha_n$ be a collection
of non-prime arcs for $H_A$ that is maximal, in the sense that each
$\alpha_i$ is a non-prime arc in $H_A \cup (\cup_{j=1}^{i-1}\alpha_j)$
and there are no non-prime arcs in $H_A \cup (\cup_{j=1}^n\alpha_j)$.
Cut $M_A$ into upper and lower 3--balls along disks $\CalD$, as in
Lemma \ref{lemma:3-ball}.  Then modify the decomposition by cutting
lower polyhedra along each non-prime arc $\alpha_i$, $i=1, \dots, n$,
as described above.  This decomposes $M_A$ into 3--balls, which we
continue to call upper and lower 3--balls.  We refer to the
decomposition as a \emph{prime decomposition} of $M_A$.\index{prime decomposition}
  \label{def:max-nonprime}
\end{define}

Notice that the choice of a maximal collection of non-prime arcs may
not be unique.  In fact, by appealing to certain results about
orbifolds, one can show that the pieces of the prime decomposition are
unique.  While we do not need this for our applications, the argument
is outlined in Remark
\ref{rem:prime-uniqueness} on page \pageref{rem:prime-uniqueness}.

\begin{define}
\label{def:prime}
We say a polyhedron is \emph{prime}\index{prime!polyhedron} if every pair of faces
meet along at most one edge.

Equivalently, we will see that any prime polyhedron admits no normal
bigons, as in Definition \ref{def:normal-bigon}.
\end{define}

In our situation, we also have the following equivalent notion of
\emph{prime}.  Recall that, by Lemma \ref{lemma:lower-alt}, each lower polyhedron corresponds to an alternating link diagram (which can be recovered from a sub-graph of $H_A$). The $4$--valent graph of the polyhedron is identical to the $4$--valent graph of the alternating diagram. Then the polyhedron will be prime if and only if the corresponding alternating diagram is prime, in the sense of Definition \ref{primetwist}. This is one motivation for the notion of prime polyhedra.


The effect of the prime decomposition of $M_A$ is summarized in the
following lemma.

\begin{lemma}
A prime decomposition of $M_A$, along a maximal collection of
non-prime arcs $\alpha_1, \dots, \alpha_n$, has the following
properties:
\begin{enumerate}
\item\label{item:upper-lower} It decomposes $M_A$ into one upper and at least one lower
  3--ball.
\item\label{item:4-valent} Each 3--ball is checkerboard colored with 4--valent vertices.
\item\label{item:poly-regions} Lower 3--balls are in one to one correspondence with non-trivial
  complementary regions of $s_A \cup (\cup_{i=1}^n\alpha_i)$.
\item\label{item:poly-alt} All lower 3--balls are ideal polyhedra identical to the
  checkerboard polyhedra of an alternating link\index{lower polyhedra}, where the alternating
  link is obtained by taking the restriction of $H_A \cup
  (\cup_{i=1}^n\alpha_i)$ to the corresponding region of $s_A \cup
  (\cup_{i=1}^n\alpha_i)$, and replacing segments of $H_A$ with
  crossings (using the $A$--resolution).  
\item\label{item:lower-prime} The alternating diagram corresponding to each lower polyhedron is prime. Consequently, each lower polyhedron is itself prime.
\item\label{item:white-regions} White faces of the 3--balls correspond to regions of the
  complement of $H_A \cup (\cup_{i=1}^n\alpha_i)$.
\end{enumerate}
\label{lemma:nonprime-3balls}
\end{lemma}

\begin{proof}
Cutting along non-prime arcs may slice lower 3--balls into multiple
pieces, but it will not subdivide the upper 3--ball.  Hence we still
have one upper and at least one lower 3--ball after a prime
decomposition, giving item \eqref{item:upper-lower}.

Note that ideal edges are modified by the prime decomposition, but each ideal edge
still bounds a white face on one side and a shaded face on the
other.  Hence the 3--balls are still checkerboard colored.  Moreover,
the prime decomposition does not affect any ideal vertices, and so
these remain 4--valent, as in Lemmas \ref{lemma:lower-alt} and
\ref{lemma:top}.  This gives item \eqref{item:4-valent}.

For item \eqref{item:poly-regions}, recall that before cutting along non-prime arcs, lower
3--balls corresponded to non-trivial regions of the complement of the
state circles $s_A$.  Now we cut these along non-prime arcs, splitting
them into regions corresponding to components of the complement of $s_A \cup
(\cup_{i=1}^n\alpha_i)$, as in Figure \ref{fig:nonprime-lower}.

Item \eqref{item:poly-alt} follows from Lemma \ref{lemma:lower-alt} and from the fact
that we cut along a maximal collection of non-prime arcs.  Lower
3--balls were known to be ideal polyhedra corresponding to alternating
links.  When we cut along non-prime arcs, we modify the diagrams of
these links by splitting into two along the non-prime arc.  Because
the collection of non-prime arcs is maximal, in the final result all
such diagrams will be prime, proving \eqref{item:lower-prime}.

Finally, before cutting along non-prime arcs, white faces corresponded
to non-trivial regions in the complement of $H_A$.  A non-prime arc
will run through such a region, with its endpoints on the same state
circle in the boundary of such a region.  Hence after cutting along a
non-prime arc, we have separated such a region into two.  Item \eqref{item:white-regions}
follows.
\end{proof}

At this stage, we have quite a bit of information about the lower $3$--balls: we know that each lower ball is an ideal polyhedron, and that it is prime. The same statements are true for the upper $3$--ball as well, although they are harder to prove. Proving these results for the upper $3$--ball is one of the main goals of the next chapter.

\section{Generalizations to other states}  \label{subsec:generalization} 
So far, we have described how to decompose the surface complement $M_A = S^3 \cut S_A$ into $3$--balls. By reflecting a $B$--adequate diagram to make it $A$--adequate, one could apply the same decomposition to $S^3 \cut S_B$.  In this section, we briefly describe how to generalize the
decomposition into $3$--balls to the much broader class
of $\sigma$--homogeneous states considered by Ozawa in \cite{ozawa}.

Given a state $\sigma$ of a link diagram $D(K)$, recall that $s_\sigma$
is a collection of disjointly embedded circles on the projection
plane.  We obtain a trivalent graph $H_\sigma$ by attaching
edges, one for each crossing of the original diagram $D(K)$. The edges
of $H_\sigma$ that come from crossings of the diagram are referred to
as \emph{segments}, and the other edges are portions of state circles.

\begin{figure}
  \includegraphics{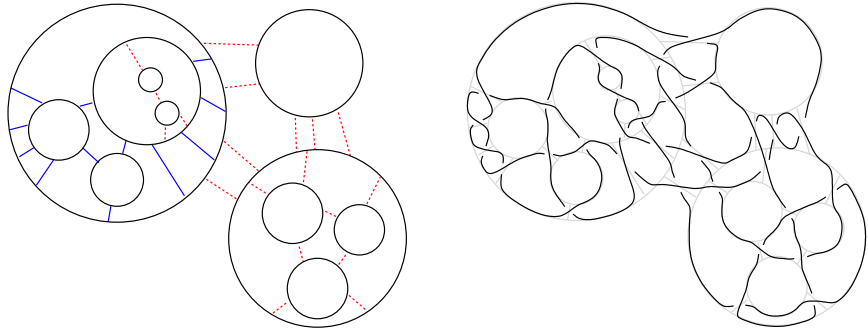}
  \caption{An example of a $\sigma$--homogeneous diagram (on the
    right) and its graph $H_\sigma$ (left).  Blue segments represent
    the $B$--resolution, red segments the $A$--resolution.  Note
    $\sigma$--homogeneity means each component in the complement of
    the state circles (black circles on left) has all segments of the
    same color.}
  \label{fig:sigma-homo}
\end{figure}

\begin{define}\label{def:sigma-homo}
 Given  a  state $\sigma$ of a link diagram $D(K)$, the
circles of $s_\sigma$ divide the projection plane
into components.  Within each such component, we have a collection of
segments coming from crossings of the diagram.  Label each segment $A$
or $B$, depending on whether the corresponding crossing is given an $A$ or
$B$--resolution in the state $\sigma$.  If all edges within each component have the same $A$ or $B$ label, we say that $\sigma$ is a \emph{homogeneous} state, and the diagram $D$ is
\emph{$\sigma$--homogeneous}\index{$\sigma$--homogeneous}.
See Fig. \ref{fig:sigma-homo} for an example.

Let $G_\sigma$ be the graph obtained by collapsing the circles of
$H_\sigma$ into vertices. If $G_\sigma$ contains no loop edges, we say that $\sigma$ is an \emph{adequate} state, and the diagram that gave rise to this state is  \emph{$\sigma$--adequate}\index{$\sigma$--adequate}. 
\end{define}

Let $\sigma$ be a homogeneous state of a link diagram $D(K)$, and let
$S_{\sigma}$ denote the corresponding state surface.  We let $M = S^3 \setminus K$
denote the link complement, and let
$M_\sigma := M\cut S_\sigma$ denote the path--metric closure of $M
\setminus S_\sigma$.  Note that $M_\sigma = (S^3\setminus K)\cut
S_\sigma$ is homeomorphic to $S^3\cut S_\sigma$, obtained by removing
a regular neighborhood of $S_\sigma$ from $S^3$.
As above, we will refer to $P=\bdy M_\sigma \cap \bdy M$ as the \emph{parabolic
  locus} of $M_\sigma$; it consists of annuli. 

We cut $M_\sigma = S^3 \cut S_\sigma$ along disks, one for each region
of the complement of $H_\sigma$, excepting innermost circles of
$s_\sigma$.  We refer to these disks as \emph{white disks}, and we
denote the collection of such disks by $\CalD$.  This cuts $M_\sigma$
into 3--balls:  one \emph{upper} 3--ball lying above the plane of
projection, and multiple \emph{lower} 3--balls, one for each component
of $s_\sigma$.  On the surface of each 3--ball is a graph, with edges
coming from intersections of white disks and the surface $S_\sigma$,
dividing the surface of the 3--ball into regions:  \emph{white faces},
coming from the white disks, and \emph{shaded regions}, coming from
portions of the state surface $S_\sigma$.



Because the diagram $D(K)$ is $\sigma$--homogeneous, each lower 3--ball comes
from a sub-diagram of $D$ that consists of only $A$-- or only
$B$--resolutions.  Because this sub-diagram is contained in a single
non-trivial component of the complement of the state circles
$s_\sigma$, it is alternating. The proofs of Lemmas \ref{lemma:lower-circles} and \ref{lemma:lower-alt}
go through in the $\sigma$--homogeneous
setting, and we  immediately obtain analogous results for the lower
3--balls in this case.

\begin{lemma}\label{lemma:lower3ball}
Let $\sigma$ be a homogeneous state of a diagram $D(K)$. 
Let  $R$ be a non-trivial component of the
complement of $s_\sigma$ in the projection plane.  Then:
\begin{enumerate}
\item There is exactly one lower 3--ball corresponding to $R$.  Its
  white faces correspond to the regions in the complement of
  $H_\sigma$ that are contained in $R$.
\item Each lower 3--ball is an ideal polyhedron, identical to the
  checkerboard polyhedron obtained by restricting to the alternating
  diagram given by the subgraph of $H_\sigma$ contained in a
  non-trivial region of $s_\sigma$.  \qed
\end{enumerate}
\end{lemma} 

As above, we call $R$ a \emph{polyhedral region}.

The ideal edges of the upper 3--ball are given
by the intersection of white disks with the surface $S_\sigma$.  Since
each white disk is contained in a single polyhedral region $R$ in the
complement of the state circles $s_\sigma$, each crossing that the
white disk borders has been assigned the same resolution, $A$ or $B$,
by $\sigma$.  Thus the local description of these ideal edges is
identical to that in the all--$A$ or all--$B$ case.  In particular,
obtain the analogue of  Lemma \ref{lemma:top}: the upper 3--ball is
checkerboard colored, and all ideal vertices are 4-valent.  In the $\sigma$--homogeneous setting, 
the definitions of \emph{tentacles} and their \emph{head} and
\emph{tail} directions in the $\sigma$--homogeneous case, are completely analogous to
Definitions and \ref{def:tentacle} and  \ref{def:tentacle-direct}.

\begin{define}\label{def:tentacleh}
  For any segment of $H_\sigma$, rotate $H_\sigma$ so that the segment
  is vertical.  A \emph{tentacle} is defined to be the strip of shaded
  face running from the top of the segment, adjacent to the link,
  along the bottom state circle adjacent to this segment.  When the
  segment comes from a crossing with the $A$--resolution, the tentacle
  runs from the top to the right.  When the segment comes from a
  crossing with the $B$--resolution, the tentacle runs from the top
  and to the left.  A tentacle is bounded on one side by a portion of
  the graph $H_\sigma$, and on the other side by exactly one ideal
  edge.

  The \emph{head} of the tentacle is the portion attached to the top
  of the segment.  The tail is adjacent to the state circle.  We think
  of a tentacle as directed from head to tail.  
\end{define}

In the $\sigma$--homogeneous setting, some tentacles run in the
right--down direction (corresponding to the $A$--resolution) and some
in the left--down direction (corresponding to the $B$--resolution).
However, within any component of the complement of $s_\sigma$, all tentacles run in the
same direction.  Thus the only way to switch from right--down to
left--down, or vice versa, is to cross over a circle of $s_\sigma$.

In the upper 3--ball, we attach tentacles to tentacles across state
circles as shown in Figure \ref{fig:tentacle-mult}.  However, if the
state circle separates $A$--resolutions from $B$--resolutions, we attach
left--down tentacles to right--down tentacles, or vice-versa.  See Figure
\ref{fig:tentacles-homo}.

\begin{figure}
\input{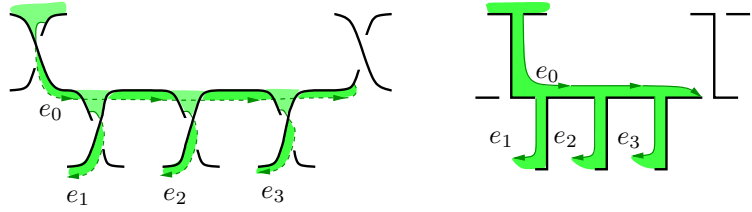}
\caption{Analogue of Figure \ref{fig:tentacle-mult}.  When resolutions
switch from all--$A$ to all--$B$ across a separating state circle of
$s_\sigma$, we attach left--down tentacles rather than right--down.}
\label{fig:tentacles-homo}
\end{figure}

Finally, to complete the decomposition, just as in the all--$A$ or
all--$B$ case we need to ensure primeness of the polyhedra.  To do so,
we add a maximal collection of non-prime arcs, defined exactly as in
Section \ref{subsec:primeness}, and then surger our polyhedra along
disks bounded by these arcs.  Because non-prime arcs connect a state
circle to itself, and therefore separate only all--$A$ or all--$B$
resolutions (by $\sigma$--homogeneity), all the discussion in Section
\ref{subsec:primeness} goes through without modification in the
$\sigma$--homogeneous case (except to replace all $A$'s with all $B$'s
if necessary, which does not affect the argument).

\chapter{Ideal Polyhedra}\label{sec:polyhedra}
Recall that $M_A = S^3 \cut S_A$ is $S^3$ cut along the surface $S_A$.
In the last chapter, starting with a link diagram $D(K)$, we obtained
a prime decomposition of $M_A$ into 3--balls. One of our goals in this
chapter is to show that, if $D(K)$ is $A$--adequate (see Definition
\ref{def:reduced} on page \pageref{def:reduced}), each of these balls
is a checkerboard colored ideal polyhedron with 4--valent vertices.
This amounts to showing that the shaded faces on each of the 3--balls
are simply--connected, and is carried out in Theorem
\ref{thm:simply-connected}.

Once we have established the fact that our decomposition is into ideal
polyhedra, as well as a collection of other lemmas concerning the
combinatorial properties of these polyhedra, two important results
follow quickly.  The first is Proposition \ref{prop:no-normal-bigons},
which states that all of the ideal polyhedra in our decomposition are prime.  
The second is a new proof of Theorem
\ref{thm:incompress}, originally due to Ozawa \cite{ozawa}, that the
surface $S_A$ is essential
in the link complement if and only if the diagram of our link is
$A$--adequate.

All the results of this chapter generalize to $\sigma$--adequate, $\sigma$--homogeneous diagrams. We discuss this generalization in Section \ref{subsec:idealsigma}.

The results of this chapter will be assumed in the sequel.  To prove
many of these results, we will use the combinatorial structure of the
polyhedral decomposition of the previous chapter, in a method of proof
we call \emph{tentacle chasing}\index{tentacle chasing}.  This method
of proof, as well as many lemmas established here using this method,
will be used again quite heavily in parts of Chapters
\ref{sec:ibundle}, \ref{sec:epds}, \ref{sec:nononprime}, and
\ref{sec:montesinos}.  Therefore, the reader interested in those
chapters should read the tentacle chasing arguments carefully, to be
prepared to use such proof techniques later.  In particular, tentacle chasing 
methods form a crucial component in the proofs of our main results,
which reside in  Chapters
\ref{sec:spanning} and \ref{sec:applications} respectively.

However, a reader who is eager to get to the main theorems and their 
applications, and who seeks only a top-level outline of the proofs, may opt to 
survey the results of this chapter while taking the proofs on faith. The top-level proofs of the 
main results in Chapter \ref{sec:spanning} and the applications in Chapter \ref{sec:applications} 
will not make any direct reference to tentacle chasing.

\section{Building blocks of shaded faces}

To prove the main results of this chapter, first we need to revisit
our construction of shaded faces for the upper $3$--ball.  Shaded
faces in the upper $3$--ball are built of one of three pieces:
innermost disks, tentacles, and non-prime switches.  See Figure
\ref{fig:shaded-pieces}.  Recall that a tentacle is directed, starting
at the portion adjacent to the segment of $H_A$ (the head) and ending
where the tentacle terminates adjacent to the state circle (the tail).
This direction leads naturally to the definition of a \emph{directed
spine}\index{directed spine (for shaded face)}\index{shaded face!directed spine} for any shaded face on
the upper $3$--ball, as follows.  For each tentacle, take a directed
edge running through the core of the tentacle, with initial vertex on
the state circle to which the segment of the tentacle is attached, and
final vertex where the tentacle terminates, adjacent to the state
circle.  For each innermost disk, take a vertex.  Notice that
innermost disks are sources of directed edges of the spine, with one
edge running out for each segment adjacent to the disk, but no
directed edges running in.  A non-prime arc is also represented as a 
vertex of the spine, with two incoming edges and two outgoing edges.
This motivates the term \emph{non-prime switch}. 
See Figure \ref{fig:directed-spine}.

\begin{figure}
  \includegraphics{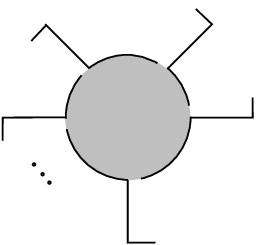} \hspace{.1in}
  \includegraphics{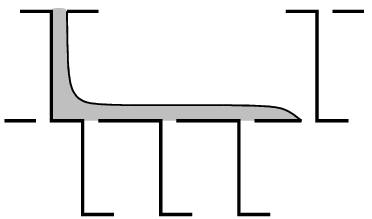} \hspace{.1in}
  \includegraphics{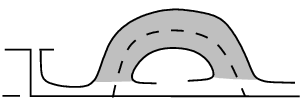}
  \caption{Building blocks of a shaded face\index{shaded face}: an innermost disk\index{innermost disk}\index{shaded face!innermost disk}, a
    tentacle\index{tentacle}\index{shaded face!tentacle}, and a non-prime switch\index{non-prime!switch}\index{shaded face!non-prime switch}.}
  \label{fig:shaded-pieces}
\end{figure}

In the language of directed spines, the statement that shaded faces are simply connected (Theorem \ref{thm:simply-connected}) can be rephrased to say 
that the directed spine of each shaded face is, in fact, a
directed tree.

\begin{figure}[h]
  \includegraphics{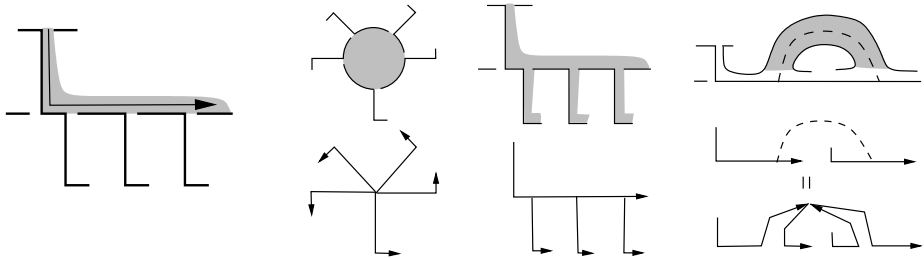}
  \caption{Far left: A directed spine\index{directed spine (for shaded face)} of a
    tentacle.  Left to right: Shown is how directed tentacles connect
    to an innermost disk, to another tentacle, across a non-prime
    switch.}
  \label{fig:directed-spine}
\end{figure}

\begin{define}
When an oriented arc running through a tentacle in a shaded face is
running in the same direction as that of the orientation above, or in
the same direction as the edge of the directed spine, we say the path
is running \emph{downstream}\index{downstream}.  When the oriented
path is running opposite the direction on the tentacle, we say the
path is running \emph{upstream}\index{upstream}.
\label{def:downstream-upstream}
\end{define}

Figure \ref{fig:directed-spine}, far left, shows an arc running
through a single tentacle in the downstream direction.  All the arrows
in the remainder of that figure point in the downstream direction.

\begin{define}\label{def:simple-face}
Suppose a directed arc $\gamma$, running through a shaded face of the
upper $3$--ball, has been homotoped to run monotonically through each
innermost disk, tentacle, and non-prime switch it meets.  Suppose
further that $\gamma$ meets any innermost disk, tentacle, and
non-prime switch at most once.  Then we say that $\gamma$ is \emph{simple
  with respect to the shaded face}\index{simple with respect to shaded face}.

Note that paths through the spine of a shaded face are simple if and
only if they are embedded on the spine.

We say that $\gamma$ is \emph{trivial} if it does not cross any state circles.

\end{define}

\section{Stairs and arcs in shaded faces}

The directions given to portions of shaded faces above lead to natural
directions on subgraphs of $H_A$.  One subgraph of $H_A$ that we will
see repeatedly 
is called a
right--down staircase.  

\begin{define}
A \emph{right--down staircase}\index{right--down staircase} is a
connected subgraph of $H_A$ determined by an alternating sequence of
state circles and segments of $H_A$, oriented so that every turn from
a state circle to a segment is to the right, and every turn from a
segment to a state circle is to the left.  (So the portions of state
circles and edges form a staircase moving down and to the right.)

In fact, right--down staircases could be named left--up, except that
the down and right follows the convention of Notation \ref{notation:top-right}.
\label{def:right-down-stair}
\end{define}

In this section, we present a series of highly useful lemmas that will
allow us to find particular right--down staircases in the graph $H_A$
associated with shaded faces.  These lemmas lead to the proof of
Theorem \ref{thm:simply-connected}, and will be referred to frequently
in Chapters \ref{sec:ibundle}, \ref{sec:epds}, \ref{sec:nononprime},
and \ref{sec:montesinos}.

\begin{lemma}[Escher stairs\index{Escher stairs lemma}]
  In the graph $H_A$ for an $A$--adequate diagram, the following are true:
  \begin{enumerate}
    \item\label{item:escher-loop} no right--down staircase forms a loop, and
    \item\label{item:escher-same} no right--down staircase has its top and bottom on the
      same state circle.
  \end{enumerate}
  \label{lemma:escher}
\end{lemma}

Cases \eqref{item:escher-loop} and \eqref{item:escher-same} of Lemma \ref{lemma:escher} are illustrated in
Figure \ref{fig:escher}.

\begin{figure}
  \begin{center}
    \begin{tabular}{ccccc}
      \eqref{item:escher-loop} & \includegraphics{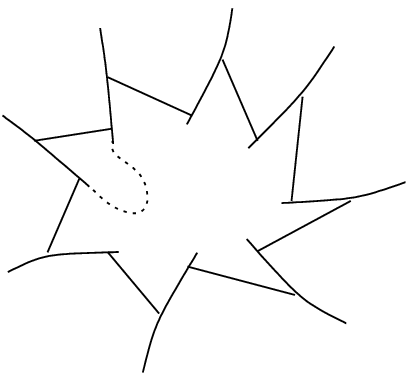} & \hspace{.2in} &
      \eqref{item:escher-same} &  \input{figures/loop-stairs.pstex_t}
    \end{tabular}
  \end{center}
  \caption{Left:  a right--down staircase forming a loop.  Right:  A
    single right--down staircase with its top and bottom connected to
    the same state circle.}
  \label{fig:escher}
\end{figure}

\begin{proof}
Suppose there exists a right--down staircase forming a loop.  Notice
that the staircase forms a simple closed curve in the projection
plane.  Each state circle of the staircase intersects that loop.
Because state circles are also simple closed curves, they must
intersect the loop an even number of times.  Because state circles
cannot intersect segments, each state circle within the loop must be
connected to another state circle within the loop.  There must be an
outermost such connection.  These two state circles will form adjacent
stairs, and connect within the loop.  But then the segment between
them gives a segment with both endpoints on the same state circle,
contradicting $A$--adequacy of the diagram, Definition
\ref{def:adequate} (page \pageref{def:adequate}).

Similarly, suppose a right--down staircase has its top and bottom on
the same state circle.  Then the staircase and this state circle forms
a loop, as above, and state circles that enter the loop must connect
to each other.  Again there must be some outermost connected pair.
This pair will be two adjacent stairs.  Again the segment between them
will then give a segment with both endpoints on the same state circle,
contradicting $A$--adequacy.
\end{proof}

Lemma \ref{lemma:escher} is the first place where we have used $A$--adequacy\index{$A$--adequate}\index{adequate diagram}. In fact, as the following example demonstrates, this hypothesis 
(or a suitable replacement, such as $\sigma$--adequacy) is crucial for both the lemma and for future results.

\begin{example}\label{ex:two-crossings}
Consider the unique connected, two-crossing diagram of a two-component unlink. This diagram is not $A$--adequate. Its graph $H_A$ features both a loop staircase (with two steps), and a one-step staircase with its top and bottom on the same state circle, violating both conclusions of Lemma \ref{lemma:escher}. 

The loop staircase also gives rise to a non-trivial loop in the directed spine of the (unique) shaded face. Thus the upper $3$--ball of this diagram is not a polyhedron. Therefore, all the proof techniques requiring a polyhedral decomposition will fail for this inadequate diagram.
\end{example}

\begin{define}\label{def:nonprime-halfdisk}
Every non-prime arc $\alpha_i$ has its endpoints on some state circle
$C$, and cuts a disk in the complement of $C$ into two regions, called
\emph{non-prime half-disks}\index{non-prime!half-disk}\index{half-disk}. 
\end{define}

The following lemma will help us deal with combinatorial behavior when
we encounter non-prime arcs.

\begin{lemma}[Shortcut lemma\index{Shortcut lemma}]
Let $\alpha$ be a non-prime arc with endpoints on a state circle $C$.
Suppose a directed arc $\gamma$ lies entirely on a single shaded face,
and is simple with respect to that shaded face, in the sense of
Definition \ref{def:simple-face}.  Suppose $\gamma$ runs across
$\alpha$ into the interior of the non-prime half-disk bounded by
$\alpha$ and $C$, and then runs upstream.  Finally, suppose that
$\gamma$ exits the interior of that half-disk across the
state circle $C$.  Then $\gamma$ must exit by following a tentacle
downstream (that is, it cannot exit running upstream).
\label{lemma:np-shortcut}
\end{lemma}

\begin{proof}
Consider an innermost counterexample.  That is, if there exists a
counterexample, then there exists one for which $\gamma$ does not
cross any other non-prime arc and then run upstream when exiting the
non-prime half-disk bounded by $C$ and $\alpha$.  Consider the subarc
of $\gamma$ which runs from the point where it crosses $\alpha$ to the
point where it crosses $C$.  We will abuse notation slightly and call
this arc $\gamma$.

After crossing $\alpha$, the arc $\gamma$ is running upstream in a
tentacle adjacent to $C$.  Note that since we are assuming this is a
counterexample, it will not cross $C$ immediately, for to do so it
would follow a tentacle running downstream.  Additionally, it cannot
cross some other non-prime arc $\alpha_1$ with endpoints on $C$, for
because we are assuming this counterexample is innermost, it would
then exit the region bounded by $\alpha_1$ and $C$ running downstream,
contradicting our assumption that it crosses $C$ running upstream.
Finally, it may reach a non-prime arc $\alpha_1$ and run around it
without crossing, but then we are still running upstream on a tentacle
adjacent to $C$, so we may ignore this case.  

Hence the only possibility is that $\gamma$ crosses $\alpha$ and then
runs up the head of a tentacle with tail on $C$.  The head of this
tentacle is adjacent to a single step of a right--down stair.
Consider what $\gamma$ may do at the top of this stair.
\begin{enumerate}
\item It may continue upstream, following another tentacle.
\item It may change direction, following a tentacle downstream, or
  crossing a non-prime arc $\alpha_1$ with endpoints on $C_1$ and then
  (eventually) running downstream across $C_1$.
\item It may run over a non-prime switch without crossing the
  non-prime arc.
\end{enumerate}
By assumption (counterexample is innermost), it cannot run over a
non-prime arc $\alpha_1$ with endpoints on $C_1$ and (eventually)
cross $C_1$ running upstream.  Notice that if $\gamma$ enters an
innermost disk, it must leave the disk running downstream, case (2),
since an innermost disk is a source for edges of the directed spine.
Also, in case (3), $\gamma$ remains adjacent to the same state circle
before and after, and so we ignore this case.

In case (1), we follow $\gamma$ upstream to a new stair, and the same
options are again available for $\gamma$, so we may repeat the
argument.

We claim that $\gamma$ is eventually in case (2).  For, suppose not.
Then since $\gamma$ crosses $C$, and the graph $H_A$ is finite, by
following tentacles upstream we form a finite right--down staircase
whose bottom is on $C$, and whose top is on $C$ as well.  This contradicts
Lemma \ref{lemma:escher} (Escher stairs).

So eventually $\gamma$ must change direction, following a tentacle
downstream.  After following the tentacle downstream, $\gamma$ will be
adjacent to another state circle.  At this point, it may do one of two
things:
\begin{enumerate}
\item It may continue downstream through another tentacle, or by
  running through a non-prime arc first and then continuing
  downstream.
\item It may run over a non-prime switch without crossing the non-prime arc.
\end{enumerate}
Notice that these are the only options because first, no arc running
downstream can enter an innermost disk (because such a disk is a
source).  Second, by assumption (innermost) $\gamma$ cannot cross a
non-prime arc and then cross the corresponding state circle running
upstream.  Third, tentacles only connect to tentacles in a downstream
direction (Figure \ref{fig:directed-spine} center).  Again we ignore
case (2), as $\gamma$ will be adjacent to the same state circle before
and after running over the non-prime switch.

But since these are the only possibilities, $\gamma$ must continue
running downstream, and cannot change direction again to run upstream.
Thus $\gamma$ must exit $C$ by running over a tentacle in the
downstream direction.
\end{proof}

\begin{define}\label{def:tentacle-chasing}
The proof of the previous lemma involved following arcs through
oriented tentacles, keeping track of local possibilities.  We call
this proof technique \emph{tentacle chasing}\index{tentacle chasing}.
We will use it repeatedly in the sequel.
\end{define}

\begin{lemma}[Staircase extension\index{Staircase extension lemma}]
Let $\gamma$ be a directed arc lying entirely in a single shaded face,
such that $\gamma$ is simple with respect to the shaded face
(Definition \ref{def:simple-face}).  Suppose also that $\gamma$ begins
by crossing a state circle running downstream.  Suppose that every
time $\gamma$ crosses a non-prime arc $\alpha$ with endpoints on $C$
and enters the non-prime half-disk bounded by $\alpha$ and $C$, that
it exits that half-disk.  Then $\gamma$ defines a
right--down staircase such that every segment of the staircase is
adjacent to $\gamma$, with $\gamma$ running downstream. Moreover, the
endpoints of $\gamma$ lie on tentacles that are adjacent to the first
and last stairs of the staircase.
\label{lemma:staircase}
\end{lemma}

\begin{proof}
The arc $\gamma$ runs through a tentacle downstream.  The tentacle is
attached to a state circle at its head, is adjacent to a segment of
$H_A$, and then adjacent to a second state circle at its tail.  Form
the first steps of the right--down staircase by including the state
circle at the head, the segment, and the state circle at the tail.

Now we consider where $\gamma$ may run from here.  Note it cannot run
into an innermost disk, since each of these is a source (and so is
entered only running upstream).  Thus it must do one of the following:
\begin{enumerate}
\item It runs through another tentacle downstream.
\item It runs through a non-prime switch, without changing direction.
\item It runs through a non-prime switch, changing direction.  
\end{enumerate}

In case (1), we extend the right--down staircase by attaching the
segment and state circle of the additional tentacle.  If $\gamma$
continues, we repeat the argument with $\gamma$ adjacent to this new
state circle.

We ignore case (2), because $\gamma$ will remain adjacent to the same
state circle in this case, still running in the downstream direction.

In case (3), $\gamma$ is adjacent to a state circle $C$, then enters a
non-prime half-disk bounded by a non-prime arc and $C$.  By
hypothesis, $\gamma$ also exits that non-prime half-disk.  Since it
cannot exit along the non-prime switch, by hypothesis that $\gamma$
runs monotonically through non-prime switches and meets each at most
once, $\gamma$ must exit by crossing $C$.  Then Lemma
\ref{lemma:np-shortcut} implies that $\gamma$ exits by following a
tentacle downstream.  This tentacle will be adjacent to some segment
attached to $C$ and a new state circle attached to the other endpoint
of this segment.  Extend the right--down staircase by attaching this
segment and state circle to $C$.  See Figure \ref{fig:extend}.  If
$\gamma$ continues, we may repeat the argument.

\begin{figure}
  \begin{center}
    \begin{tabular}{ccc}
      \includegraphics{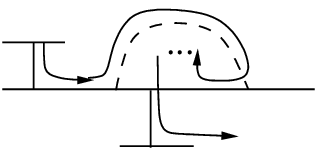} & $\rightarrow$ &
      \includegraphics{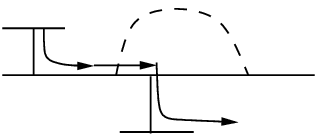}
    \end{tabular}
  \end{center}
  \caption{Extend a right--down staircase over a non-prime switch.}
  \label{fig:extend}
\end{figure}

After a finite number of repetitions, $\gamma$ must terminate, and we
have our extended right--down staircase as claimed in the lemma.
\end{proof}

The following is an immediate, highly useful consequence.  

\begin{lemma}[Downstream continues down, or Downstream lemma\index{Downstream lemma}]
Let $\gamma$ be as in Lemma \ref{lemma:staircase}.  Then $\gamma$
crosses the last state circle of the staircase by running downstream.
\qed
\label{lemma:downstream}
\end{lemma}

We can now prove a result, which is called the Utility lemma because
we will use it repeatedly in the upcoming arguments.

\begin{lemma}[Utility lemma\index{Utility lemma}]
\label{lemma:utility}
Let $\gamma$ be a simple, directed arc in a shaded face, which starts
and ends on the same state circle $C$. Then $\gamma$ starts by running
upstream from $C$, and then terminates at $C$ while running
downstream.

Furthermore, $\gamma$ cannot intersect $C$ more than two times.
\end{lemma}

\begin{proof}
First, suppose that $\gamma$ runs downstream from its first
intersection with $C$.  This will lead to a contradiction.

We begin by applying Lemma \ref{lemma:staircase} (Staircase extension)
to find a right--down staircase starting on $C$, such that $\gamma$
runs downstream, adjacent to each segment of the staircase. This
staircase will continue either until the terminal end of $\gamma$, or
until $\gamma$ crosses a non-prime arc $\alpha$ and enters (but does
not exit) a half-disk $R$ bounded by $\alpha$ and some
state circle $C'$.  But any such non-prime half-disk $R$ will not
contain the initial endpoint of $\gamma$ (else $\gamma$ would have
crossed $C'$ running downstream earlier, and we would have created a
right--down staircase from $C'$ to $C'$, contradicting Lemma
\ref{lemma:escher}), hence $R$ will not contain $C$ unless $C'=C$.
Because the final endpoint of $\gamma$ is on $C$, either no such
region $R$ exists, or $\alpha$ has both endpoints on $C$.  In either
case, we will have constructed a right--down staircase that starts and
ends on $C$, contradicting Lemma \ref{lemma:escher} (Escher stairs).
So $\gamma$ cannot run downstream from $C$.

Next, suppose that the terminal end of $\gamma$ meets $C$ running
upstream. Then we simply reverse the orientation on $\gamma$, and
repeat the above argument to obtain a contradiction.  Therefore,
$\gamma$ first runs upstream from $C$, then terminates on $C$ while
running downstream.

Finally, suppose that $\gamma$ meets $C$ more than twice. Let $x_1,
\ldots, x_n$ be its points of intersection with $C$. Applying the
above argument to the sub-arc of $\gamma$ from $x_1$ to $x_2$, we
conclude that $\gamma$ must arrive at $x_2$ while running
downstream. But then the sub-arc of $\gamma$ from $x_2$ to $x_3$
departs $C$ running downstream, which is a contradiction.
\end{proof}

Given the above tools, we are now ready to show that our decomposition
is into ideal polyhedra. The following is one of the main results of
this chapter.

\begin{theorem}
Let $D(K)$ be an $A$--adequate link diagram. Then, in the prime
decomposition of $M_A$, shaded faces on the 3--balls are all simply
connected.  This gives a decomposition of $M_A$ into checkerboard
colored ideal polyhedra with 4--valent vertices.
\label{thm:simply-connected}\index{upper 3--ball!actually a polyhedron}\index{directed spine (for shaded face)!is a tree}
\end{theorem}

\begin{proof}
By Lemma \ref{lemma:nonprime-3balls}, part \eqref{item:poly-alt}, the
lower $3$--balls are ideal polyhedra, with simply connected faces.
Hence, we need only consider the shaded faces on the upper $3$--ball.

We have constructed a spine for each shaded face on the upper
$3$--ball.  The shaded face will be simply connected if and only if
the spine is a tree.  Hence, we show the spine is a tree.

If the spine is not a tree, then there is a non-trivial embedded loop
$\gamma$ in the spine for the shaded face.  Since $\gamma$ is embedded
in the spine, any sub-arc is simple in the sense of Definition
\ref{def:simple-face}.

Now, suppose $\gamma$ crosses a state circle $C$.  Since $\gamma$ is a
simple closed curve, as is the state circle, $\gamma$ must actually
cross $C$ at least twice.  Then we can express $\gamma$ as the union
of two directed arcs $\gamma_1, \gamma_2$, with endpoints at $C$, such
that $\gamma_1, \gamma_2$ meet only at their endpoints. Suppose that
both arcs are directed along a consistent orientation of $\gamma$.
Then Lemma \ref{lemma:utility} (Utility lemma) says that $\gamma_1$
terminates at $C$ running downstream. This means that $\gamma_2$
starts at $C$ by running downstream, which contradicts the Utility
lemma.

So $\gamma$ never crosses a state circle.  Since $\gamma$ is
non-trivial, contained in a single shaded face, it must run over a
sequence of non-prime switches, all with endpoints on the same state
circle $C$.  When $\gamma$ runs from one non-prime switch into
another, it cannot meet any segments of $H_A$ coming out of $C$, else
the tentacle that $\gamma$ runs through would terminate ($\gamma$
would have to exit the shaded face).  But then $\gamma$ bounds a
region in the projection plane which contains no state circles, since
our diagram is assumed to be connected.  This contradicts the
definition of a collection of non-prime arcs, Definition
\ref{def:non-prime} on page \pageref{def:non-prime}: the last such arc
added to our collection divides a region of the complement of $H_A$
and the other non-prime arcs into two pieces, one of which does not
contain any state circles.  See Figure \ref{fig:nonprime-loop}.

\begin{figure}
  \includegraphics{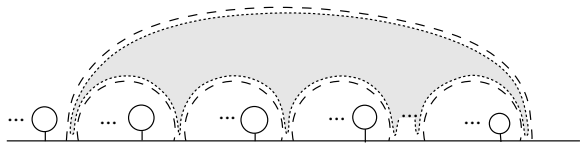}
  \caption{If $\gamma$ runs over a sequence of non-prime arcs, then
    $\gamma$ bounds a region (shown shaded above) containing no state
    circles, giving a contradiction.  Compare with Figure
    \ref{fig:nonprime-ex}.}
  \label{fig:nonprime-loop}
\end{figure}

So shaded faces are simply connected.  Since white faces are disks by definition, a prime decomposition
of $M_A = S^3\cut S_A$ is a decomposition into ideal polyhedra.  The
fact that it is 4--valent and checkerboard colored follows from Lemma
\ref{lemma:nonprime-3balls}.
\end{proof}

Recall that lower 3--balls are ideal polyhedra corresponding to
non-trivial complementary regions of $s_A \cup (\bigcup \alpha_i)$, where
the $\alpha_i$ form a maximal collection of non-prime arcs.

\begin{define}
A \emph{polyhedral region}\index{polyhedral region} is a complementary
region of $s_A \cup(\bigcup \alpha_i)$ on the projection plane. With
the convention that the ``projection plane'' is a $2$--sphere, it
follows that each polyhedral region is compact.
\label{def:polyhedral-region}
\end{define}


\begin{lemma}[Parallel stairs\index{Parallel stairs lemma}]\label{lemma:enter-through-circle}\label{lemma:different-streams}\label{lemma:parallel-stairs}
Let $\sigma_1$ and $\sigma_2$ be simple, disjoint, directed arcs
through the spines of shaded faces $F_1$ and $F_2$. (These shaded
faces are allowed to coincide, so long as the $\sigma_i$ are
disjoint.) Suppose that both $\sigma_1$ and $\sigma_2$ begin at the
same state circle $C$, running downstream, and terminate in the same
polyhedral region $R$. Then the following hold.
\begin{enumerate}
\item\label{i:stairs-exist} There are disjoint right--down staircases
for the $\sigma_i$, such that $\sigma_1$ runs downstream along each
segment of the first staircase and $\sigma_2$ runs downstream along
each segment of the second staircase.  
\item\label{i:alltheway} The terminal endpoint of each $\sigma_i$ is
adjacent to the last step (state circle) of its staircase. 
\item\label{i:same-steps} The $j$-th step of the first staircase is on
the same state circle as the $j$-th step of the second staircase,
except possibly the very last step.
\item\label{i:no-white-face} The arcs $\sigma_1$ and $\sigma_2$ cannot
  terminate on the same white face.
\end{enumerate}
\end{lemma}

\begin{proof}
Conclusions \eqref{i:stairs-exist} and \eqref{i:alltheway} will follow
from Lemma \ref{lemma:staircase} (Staircase extension), as soon as we
verify that this lemma applies to the entire length of $\sigma_1$ and
$\sigma_2$. That is, we need to check that each time $\sigma_i$ enters
a non-prime half-disk through a non-prime arc, it leaves that
half-disk.

Suppose, for a contradiction, that $\sigma_1$ enters some non-prime
half-disk through a non-prime arc, and does not leave it. All such
half-disks are ordered by inclusion.  Let $R_1$ be the
\emph{largest} such non-prime half-disk. Let $\alpha_1$ be the
non-prime arc through which $\sigma_1$ enters $R_1$, and let $C_1$ be
the state circle to which it is attached.  Since $\sigma_2$ also
terminates inside $R \subset R_1$, and is disjoint from $\sigma_1$, it
must cross into $R_1$ by crossing $C_1$.

Let $\gamma$ denote the portion of $\sigma_1$ from $C$ to
$\alpha_1$.  By Lemma \ref{lemma:staircase} (Staircase extension),
there is a right-down staircase corresponding to $\gamma$. Thus $C$ is
connected to $C_1$ by a sequence of segments, and adjacent to the last
such segment is a tentacle that meets the non-prime switch
corresponding to $\alpha_1$. Since the arc $\alpha_1$ is next to the
last stair, it is on the same side of $C_1$ as the stair.  It follows
that $C$ and $\alpha_1$ are on the same side of $C_1$.  Thus
$\sigma_2$ must actually cross $C_1$ twice, and by Lemma
\ref{lemma:utility} (Utility lemma), it does so first running
upstream, then running downstream.

But $\sigma_2$ left $C$ running downstream.  By Lemma
\ref{lemma:staircase} (Staircase extension),
the only way $\sigma_2$ can later cross $C_1$ running upstream is if
$\sigma_2$ crossed over a non-prime arc $\alpha_2$ with endpoints on
$C_2$, where $\alpha_2$ separates $C_1$ from $C$.  Let $R_2$ be the
non-prime half-disk bounded by $\alpha_2$ and $C_2$ and containing $R_1$.  Since
$R_1 \subset R_2$, $\sigma_1$ must also enter $R_2$, and it must do so
by crossing $C_2$.  Since $\sigma_1$ enters $R_1$ through non-prime
arc $\alpha_1$ (and not through a state circle), we conclude that $R_1
\neq R_2$.

By applying to $\sigma_1$ the argument we used for $\sigma_2$ above,
we conclude that $\sigma_1$ must cross $C_2$ twice, first running
upstream and then downstream.  Again, $\sigma_1$ cannot run upstream
after leaving $C$ in the downstream direction, unless $R_2$ is
contained in a non-prime half-disk that $\sigma_1$ enters through a
non-prime arc.  But by construction, $R_1 $ is the largest such
half-disk, contradicting the strict inclusion $R_1
\subsetneq R_2$.  This proves
\eqref{i:stairs-exist}--\eqref{i:alltheway}.

\smallskip

To prove \eqref{i:same-steps}, let $C= C_0, C_1, \ldots, C_m$ be the
steps of the staircase of $\sigma_1$. Note that $\sigma_1$ runs
downstream across each $C_i$ (for $i = 0, \ldots, m-1$). Thus, by
Lemma \ref{lemma:utility} (Utility lemma), once $\sigma_1$ crosses a
circle $C_i$, it may not cross it again. In other words, $C_0, C_1,
\ldots, C_m$ are nested, and $\sigma_1$ runs deeper into this chain of
nested circles.

Similarly, let $C = D_0, D_1, \ldots, D_n$ be the steps of the
staircase of $\sigma_2$. Again, $\sigma_2$ runs downstream along $D_0,
\ldots, D_{n-1}$, and cannot cross these circles a second time. Thus
$D_0, \ldots, D_n$ are also nested.

By hypothesis, the terminal ends of $\sigma_1$ and $\sigma_2$ are in
the same polyhedral region $R$. By the above work, each $\sigma_i$
enters this region $R$ by crossing a state circle running
downstream. (Otherwise, $\sigma_i$ would enter a non-prime half-disk
across a non-prime arc without exiting, and we have ruled out this
possibility.)  Thus $\sigma_1$ enters $R$ by crossing $C_{m-1}$, while
$\sigma_2$ enters $R$ by crossing $D_{n-1}$. Since the $C_i$ are
nested, as are the $D_j$, the only way this can happen is if $m=n$,
and the stairs $C_j = D_j$ coincide for $j= 0, \ldots, n-1$.

\smallskip

For \eqref{i:no-white-face}, suppose that $\sigma_1$ and $\sigma_2$
terminate at the same white face $W$.  Then we can draw an arc $\beta$
entirely contained in $W$ which meets the ends of both $\sigma_1$ and
$\sigma_2$.  Recall that a white face corresponds to a region of the
complement of $H_A \cup (\bigcup \alpha_i)$.  Thus the arc $\beta$
corresponds to an arc, which we still denote $\beta$, in the
complement of $H_A \cup (\bigcup \alpha_i)$ which meets the final
segment of each right--down staircase on the right side of that
segment, when the staircases are in right--down position.  The two
staircases, the state circle at the top, and the arc $\beta$ form a
loop in the sphere on which the graph $H_A$ lies.  See Figure
\ref{fig:prime-stairs}.

\begin{figure}
  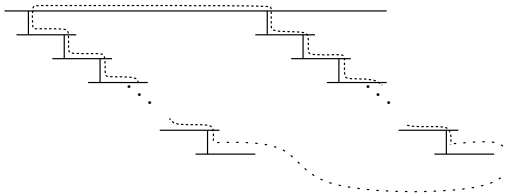
  \caption{There exists a closed curve in $H_A$ of the form of the
    dotted line above, where the arc with wider dots lies entirely in
    a region of the complement of $H_A \cup (\cup \alpha_i)$. }
  \label{fig:prime-stairs}
\end{figure}

By conclusion \eqref{i:same-steps}, all steps of the staircases,
except for the last, are on the same state circles.  Note that the
bottom stair $C_n$ on the left is not inside the shown bounded region
enclosed by the dotted curve $\beta$, but both ends of the bottom
stair $D_n$ on the right are inside the region enclosed by
$\beta$. Since $C_j = D_j$ for $j = 0, \ldots, n-1$, i.e.\ all stairs
but the last connect from left to right, the two ends of the bottom
right stair $D_n$ must connect to each other only (and to none of the
other state circles within the dotted curve), to form a state circle
that does not intersect the dotted line at all, but lies entirely
within it.

But then the arc $\beta$ can be pushed to have both endpoints lying on
the state circle $C_n-1$ just above the bottom segment.  It then gives a
non-prime arc.  By maximality of our polyhedral decomposition,
Definition \ref{def:max-nonprime}, there must be a collection of
non-prime arcs $\alpha_{j_1}, \dots, \alpha_{j_k}$ from our maximal
decomposition so that the collection $\beta \cup (\cup \alpha_{j_i})$
bounds no state circles in its interior.  But then one of these
$\alpha_{j_i}$ must separate the bottom stair on the left from the
bottom stair on the right.  This non-prime arc would separate the
bottom stairs into two distinct regions of the complement of $H_A \cup
(\cup \alpha_i)$, contradicting our assumption that $\beta$ lies in a
single such region.
\end{proof}


\section{Bigons and compression disks}

In an ideal polyhedral decomposition, any properly embedded essential
surface (with or without boundary) can be placed into normal form.
See, for example, Lackenby \cite{lackenby:volume-alt} or Futer and
Gu\'eritaud \cite{fg:arborescent}.
  
\begin{define}
A surface in \emph{normal form}\index{normal!form for a surface}
satisfies five conditions:
\begin{enumerate}
\item[(i)] its intersection with ideal polyhedra is a collection of
  disks;
\item[(ii)] each disk intersects a boundary edge of a polyhedron at
  most once;
\item[(iii)] the boundary of such a disk cannot enter and leave an
  ideal vertex through the same face of the polyhedron;
\item[(iv)] the surface intersects any face of the polyhedra in arcs,
  rather than simple closed curves;
\item[(v)] no such arc can have endpoints in the same ideal vertex of
  a polyhedron, nor in a vertex and an adjacent edge.
\end{enumerate}
\label{def:normal}
\end{define}


\begin{define}
\label{def:normal-disk}
A disk of intersection between a polyhedron and a normal surface is
called a \emph{normal disk}\index{normal!disk}. For example, a
\emph{normal bigon}\index{normal!bigon} is a normal disk with two
sides, which meets two distinct edges of its ambient polyhedron. Note
that in a checkerboard colored polyhedron, one face met by a normal
bigon must be white, and the other shaded.
\label{def:normal-bigon}
\end{define}

Recall that, in Definition \ref{def:prime}, we said that a polyhedron
is prime if each pair of faces meet along at most one edge. This is
equivalent to the absence of normal bigons.\index{prime!polyhedron}


Recall as well that our choice of a maximal collection of non-prime arcs may
not have been unique, as pointed out just after Defintion
\ref{def:max-nonprime}.  However, using the idea of normal bigons, one
can show that the prime polyhedral decomposition, obtained in Theorem
\ref{thm:simply-connected}, is unique.  Because the result is not
needed for our applications, we only outline the argument in the
remark below.  We point the reader to Atkinson \cite{atkinson:decomp} for more
details.

\begin{remark} \label{rem:prime-uniqueness}\index{prime decomposition}
One can see that the pieces of the prime decomposition are unique, as
follows.  We know, from Lemma \ref{lemma:lower-alt}, that the lower
$3$--balls are ideal polyhedra with $4$--valent ideal vertices. For
each lower polyhedron $P$, we may place a dihedral angle of $\pi/2$ on
each edge, and construct an orbifold $\mathcal{O}_P$ by doubling $P$
along its boundary.  $\mathcal{O}_P$ is topologically the $3$--sphere,
with singular locus the planar $1$--skeleton of $P$. Because we have
doubled a dihedral angle of $\pi/2$, every edge in the singular locus
has cone angle $\pi$.

There is a version of the prime decomposition for orbifolds, which
involves cutting $\mathcal{O}_P$ along \emph{orbifold spheres}, namely
2--dimensional orbifolds with positive Euler characteristic. Let $S$
be one such orbifold sphere. In our setting, because the singular
locus is a $4$--valent graph, $S$ must have an even number of cone
points.  Since the $1$--skeleton of $P$ is connected, the orbifold
sphere $S$ must intersect the singular locus, hence must have at least
two cone points, with angle $\pi$. Therefore, since each singular edge
has angle $\pi$, and $S$ has positive Euler characteristic, it must
have exactly two cone points.

Recall (e.g. from \cite{atkinson:decomp, petronio:decomp}) that the
prime decomposition of the orbifold $\mathcal{O}_P$ is equivariant
with respect to the reflection along $\bdy P$. Thus any orbifold
sphere $S$ is constructed by doubling a normal bigon in $P$.  Since
the prime decomposition of $\mathcal{O}_P$ is unique, and corresponds
to cutting $P$ along normal bigons, it follows that the decomposition
of $P$ along normal bigons is also unique.
\end{remark}


The following proposition shows that our earlier definition of
\emph{prime decomposition} along non-prime arcs actually results in prime polyhedra.   This, in turn, will be important in proving that the state
surface $S_A$ is essential in the link complement (Theorem
\ref{thm:incompress}).

\begin{prop}[No normal bigons\index{No normal bigons proposition}]
Let $D(K)$ be an $A$--adequate link diagram, and let $S_A$ be the all--$A$ state surface of $D$.
A prime decomposition of $S^3\cut S_A$ into 3--balls, as in Definition
\ref{def:max-nonprime}, gives polyhedra which contain no normal
bigons. In other words, every polyhedron is prime.
\label{prop:no-normal-bigons}
\end{prop}

\begin{proof}
Recall that by Lemma \ref{lemma:nonprime-3balls}, part
\eqref{item:lower-prime}, the lower polyhedra are prime. Since a
normal bigon is the obstruction to primeness, the lower polyhedra do
not contain any normal bigons.
 
Suppose, by way of contradiction, that there exists a normal bigon in
the upper polyhedron.  Then its boundary consists of two arcs, one,
$\gamma_s$ embedded in a shaded face, and one, $\gamma_w$ embedded on
a single white disk $W$.  Consider the arc $\gamma_s$ in the shaded face.
We may homotope this arc to lie on the spine of the shaded face.
Since the spine is a tree, by Theorem \ref{thm:simply-connected},
there is a unique embedded path between any pair of points on the
tree.  Hence $\gamma_s$ is simple with respect to the shaded face.

First, note that $\gamma_s$ must cross some state circle, for if not,
$\gamma_s$ remains on tentacles and non-prime switches adjacent to the
same state circle $C_0$, and so $\gamma_w$ contradicts part (ii) of the
definition of normal, Definition \ref{def:normal}. 

So $\gamma_s$ crosses a state circle $C$.  The endpoints of $\gamma_s$
are both on $W$, which means $\gamma_s$ crosses $C$ twice. If we cut
out the middle part of $\gamma_s$ (from $C$ back to $C$), we obtain
two disjoint sub-arcs from $C$ to $W$. If we orient these sub-arcs
away from $C$ toward $W$, Lemma \ref{lemma:utility} (Utility lemma)
implies they run downstream from $C$. Now, part
\eqref{i:no-white-face} of Lemma \ref{lemma:parallel-stairs} (Parallel
stairs) says that the ends of $\gamma_s$ cannot both be on $W$, which
is a contradiction.
\end{proof}

Recall that the state surface $S_A$ may not be orientable.  In this
case, Definition \ref{def:essential} on page \pageref{def:essential}
says that $S_A$ is \emph{essential} if the boundary $\widetilde{S_A}$
of its regular neighborhood is incompressible and
boundary--incompressible. Since $S^3 \cut \widetilde{S_A}$ is the
disjoint union of $M_A = S^3 \cut S_A$ and an $I$--bundle over $S_A$,
the computation of the guts is not affected by replacing $S_A$ with
$\widetilde{S_A}$.

\begin{theorem}[Ozawa]
Let $D$ be a (connected) diagram of a link $K$.  The surface $S_A$ is
essential in $S^3 \setminus K$ if and only if $D$ is $A$--adequate.
\label{thm:incompress}
\index{$A$--adequate}\index{$S_A$, all--$A$ state surface!is incompressible}
\end{theorem}

\begin{proof}
If $D$ is not $A$--adequate, then there is an edge of $H_A$ meeting
the same state circle at each of its endpoints.  To form $S_A$, we
attach a twisted rectangle with opposite sides on a disk bounded by
that same state circle.  Note in this case, $S_A$ will be
non-orientable.  The boundary of a disk $E$ runs along $S_A$, over the
twisted rectangle, meets the knot at the crossing of the rectangle,
then continues along $S_A$ through the disk bounded by that state
circle.  This disk $E$ will give a boundary compression disk for
$\widetilde{S_A}$, as follows.  A regular neighborhood of $S_A$ will
meet $E$ in a regular neighborhood of $\partial E \cap S_A$.  Hence
$E \setminus N(\partial E \cap S_A)$ is a compression disk for
$\widetilde{S_A}$.

Now, suppose $D$ is $A$--adequate, and let $\widetilde{S_A}$ be the
boundary of a regular neighborhood of $S_A$.  This orientable surface
is the non-parabolic part of the boundary of $M_A$. If
$\widetilde{S_A}$ is compressible, a compressing disk $E$
has boundary on $\widetilde{S_A}$.  Since $S^3 \cut \widetilde{S_A}$ is
the disjoint union of an $I$--bundle over $S_A$ and $M_A$, the disk
$E$ must be contained either in the $I$--bundle or in $M_A$.  It
cannot be in the $I$--bundle, or in a neighborhood of
$\widetilde{S_A}$ it would lift to a horizontal or vertical disk,
contradicting the fact that it is a compression disk.  Hence $E$ lies
in $M_A$.

Put the compressing disk $E$ into normal form with respect to the
polyhedral decomposition of $M_A$.  The intersection of $E$ with white
faces contains no simple closed curves, so all intersections of $E$
and the white faces are arcs.  Consider an outermost disk.  This has
boundary a single arc on a white face, and a single arc on a shaded
face.  Hence it cuts off a normal bigon, which is a contradiction of
Proposition \ref{prop:no-normal-bigons} (No normal bigons).  So the
surface $\widetilde{S_A}$ is incompressible.

If $\widetilde{S_A}$ is boundary compressible, then a boundary
compression disk $E$ again lies in $M_A$ rather than the $I$--bundle.
Its boundary consists of two arcs, one on $\widetilde{S_A}$, which we
denote $\beta$, and one which lies on the boundary of $S^3\setminus K$
(the parabolic locus), which we denote $\alpha$.  Put $E$ in normal
form.  First, we claim the arc $\alpha$ on $\bdy (S^3\setminus K)$
lies in a single polyhedron on a single ideal vertex.  If not, it must
meet one of the white faces of the polyhedron.  Take an outermost arc
of intersection of the white faces with $E$ which cuts off a disk $E'$
whose boundary contains a portion of the arc $\alpha$.  Either $E'$
has an edge on a white face and an edge on $\alpha$, in which case the
surface $E$ contradicts the first part of condition (v) of the
definition of normal, or else $E'$ has an edge on a white face, an
edge on $\alpha$, and an edge on $S_A$.  In this case, $E$ contradicts
the second part of condition (v).  Hence $\alpha$ lies entirely within
one polyhedron.

Consider arcs of intersection of $E$ with white faces.  An outermost
such arc must contain an ideal vertex, or we get a normal bigon as
above, which is a contradiction.  But if $E'$ is outermost and $E'$
contains an ideal vertex, then $E\setminus E'$ is a disk which does
not contain an ideal vertex.  Again we get a contradiction looking at
the outermost arc of intersection of $E\setminus E'$ with white faces.
\end{proof}


\begin{lemma}\label{lemma:white-incompress}
Every white face of the polyhedral decomposition is boundary
incompressible in $M_A$.
\end{lemma}

\begin{proof}
If $E$ is a boundary compression disk for a white face, it can be
placed in normal form. Then, as above, $E$ must contain an outermost
normal bigon, which contradicts Proposition
\ref{prop:no-normal-bigons} (No normal bigons). 
\end{proof}

Recall that a link diagram $D$ is \emph{prime}\index{prime!diagram} if
any simple closed curve which meets the diagram transversely exactly
twice does not bound crossings on each side.  Theorem
\ref{thm:incompress} has the following corollary that shows that for
prime, non-split links, working with prime diagrams is not a
restriction. Starting in Chapter \ref{sec:epds}, we will restrict to
prime adequate diagrams.
 
\begin{corollary}\label{cor:primelinkadequate}
Suppose that $K$ is an $A$--adequate, non-split, prime link.  Then
every $A$--adequate diagram $D(K)$ without nugatory crossings is
prime.
\end{corollary}

\begin{proof}
Suppose $D(K)$ is an $A$--adequate diagram of $K$ and let $\gamma$
denote a simple closed curve on the projection plane that intersects
$D(K)$ at exactly two points.  Now $\gamma$ splits $D(K)$ into a
connect sum of diagrams $D_1\#D_2$.  Since $K$ is prime, one of them,
say $D_1$, must be an $A$--adequate diagram of $K$, and $D_2$ must be
an $A$--adequate diagram of the unknot.  The state surface $S_A$
splits along an arc of $\gamma$ into surfaces $S_1$ and $S_2$, where
$S_i$ is the all--$A$ state surface of $D_i$, $i=1, 2$.  By Theorem
\ref{thm:incompress}, $S_2$ is incompressible, and thus it must be a
disk.  The graph $\GA(D_2)$ is a spine for $S_2$.  Since $S_2$ is a
disk, $\GA(D_2)$ is a tree.  But then each edge of $\GA(D_2)$ is
separating, hence each crossing is nugatory.  Since we assumed that
$D$ contains no nugatory crossings, $D_2$ must be embedded on the
projection plane.  Thus $D(K)$ is prime, as desired.
\end{proof}

The converse to Corollary \ref{cor:primelinkadequate} is open.  See
Problem \ref{problem:composite} in Chapter \ref{sec:questions}.


\section{Ideal polyhedra for $\sigma$--homogeneous diagrams}\label{subsec:idealsigma}

In this section, we show that the decomposition for
$\sigma$--homogeneous diagrams discussed in Section
\ref{subsec:generalization} becomes an ideal polyhedral decomposition
under the additional hypothesis of $\sigma$--adequacy. The
arguments are almost identical to the already-discussed case of $A$--adequate links. 
Thus our exposition here will
be brief, indicating only the cases where the argument calls for slight
modifications.

In the $\sigma$--homogeneous setting,
shaded faces decompose into portions associated with a directed spine.
An edge of the directed spine lies in each tentacle, and runs
adjacent to a segment and then along a state circle.  The only
difference now is that when we are in a polyhedral region for which
each resolution is the $B$--resolution, these directed edges run
left--down rather than right--down.  Innermost disks are still
sources, and non-prime arcs give rise to switches (non-prime
switches).  The resulting pieces are illustrated in Figure
\ref{fig:directed-spine-homo}, which should be compared to Figure
\ref{fig:directed-spine}.

\begin{figure}
  \includegraphics{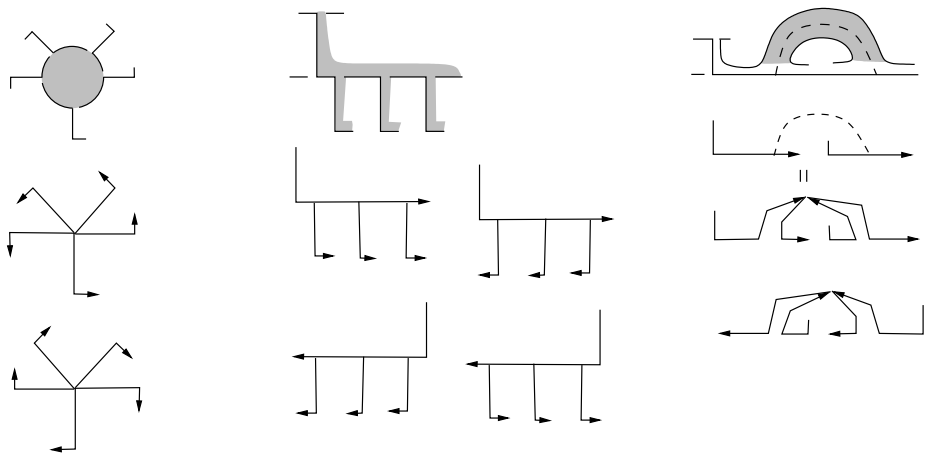}
  \caption{Building blocks of the directed spine of a shaded face, in a $\sigma$--homogeneous diagram.}
  \label{fig:directed-spine-homo}
\end{figure}

As before, when an oriented arc in a shaded face runs in the direction
of the directed spine, we say it is running \emph{downstream}.
Otherwise, it is running \emph{upstream}.  When such an arc has been
homotoped to run monotonically through each tentacle, innermost disk,
and non-prime switch, and to meet each at most once, we say it is
\emph{simple with respect to the shaded face}.

These definitions agree with Definitions \ref{def:downstream-upstream}
and \ref{def:simple-face}, modified to accommodate left--down edges.
Similarly, we have the following definition.

\begin{define}\label{def:staircaseh}
A \emph{staircase} is an alternating sequence of state circles and
segments.  The direction of the staircase is determined by the directions
of tentacles running along those staircases, which are determined by
the resolution.  Those of the $A$--resolution run ``right--down''.
Those of the $B$--resolution run ``left--down''.  All stairs in the
same component of $s_\sigma$ run in the same direction, by
$\sigma$--homogeneity.
\end{define}

It turns out that the existence of a directed staircase is all that is
needed for our main results.  ``Right--down-ness'' and
``left--down-ness'' are only peripheral, and the theory developed so far in
this chapter so far goes through without a problem.  Hence we may prove
the following analogue of Theorem \ref{thm:simply-connected}.

\begin{theorem}\label{thm:sigma-homo-poly}
Let $\sigma$ be an adequate, homogeneous state of a diagram $D$.  
Then the decomposition described above gives
a polyhedral decomposition of the surface complement $M_\sigma$ into
4--valent ideal polyhedra.
\end{theorem}

\begin{proof}
By $\sigma$--homogeneity, each lower polyhedron is identical to a
polyhedron in Menasco's decomposition of an alternating link, which
corresponds to the subgraph of $H_\sigma$ coming from a polyhedral
region.  As for the upper polyhedron, ideal vertices are 4--valent,
and white faces are simply connected.  We need to show that shaded
faces are simply connected in the $\sigma$--homogeneous case.  Each
shaded face deformation retracts to a directed spine, and we need to
show this spine is a tree.  The result follows from a sequence of
lemmas established in the previous sections concerning how these
directed graphs may be super-imposed on $H_\sigma$.  The proofs of
these lemmas work equally well when staircases run ``right--down'' and
``left--down,'' as they will when $A$ and $B$ resolutions are mixed.
What is key in all the proofs of these lemmas is that edges of the
graph corresponding to the shaded faces have a direction, and the
direction only changes in non--prime switches.  In addition, the
proofs repeatedly use the hypothesis that the state $\sigma$ defining
the graph $H_\sigma$ is adequate (recall Example
\ref{ex:two-crossings}). Hence the following technical lemmas
generalize without any modification of the proofs, except to remove
the words ``right--down'' and replace ``$A$--adequate'' with
``$\sigma$--adequate.''

\smallskip

\underline{Lemma \ref{lemma:escher} (Escher stairs):} No staircase
forms a loop, and no staircase has its top and bottom on the same
state circle.

\smallskip

\underline{Lemma \ref{lemma:np-shortcut} (Shortcut lemma):} If a
directed arc $\gamma$ in a shaded face runs across a non-prime arc
$\alpha$ with endpoints on a state circle $C$, and then upstream, the
arc $\gamma$ must exit the non-prime half-disk bounded by $\alpha$
and $C$ by running downstream across $C$.

\smallskip

\underline{Lemma \ref{lemma:staircase} (Staircase extension):} If
$\gamma$ runs downstream across a state circle, and every time
$\gamma$ crosses a non-prime arc with endpoints on a state circle $C$,
the arc $\gamma$ exits the non-prime half-disk bounded by $\alpha$
and $C$, then $\gamma$ defines a staircase such that $\gamma$ is
adjacent to each segment of the staircase, running downstream.

\smallskip

\underline{Lemma \ref{lemma:downstream} (Downstream):} For $\gamma$
as above, it must cross the last state circle of the staircase running
downstream.

\smallskip

\underline{Lemma \ref{lemma:utility} (Utility lemma):} Let $\gamma$ be
a simple, directed arc in a shaded face, which starts and ends on the
same state circle $C$. Then $\gamma$ starts by running upstream from
$C$, and then terminates at $C$ while running downstream.
Furthermore, $\gamma$ cannot intersect $C$ more than two times.
\smallskip

Now the proof of Theorem \ref{thm:simply-connected}  goes through
verbatim, only replacing $H_A$ with $H_\sigma$.  Hence the upper
polyhedron is also a 4--valent ideal polyhedron.
\end{proof}

Once we have a polyhedral decomposition of $M_\sigma$ for a
$\sigma$--adequate, $\sigma$--homogeneous diagram, we may use this to
generalize Proposition \ref{prop:no-normal-bigons} and Theorem
\ref{thm:incompress} in the setting of $\sigma$--adequate and
$\sigma$--homogeneous diagrams.

In order to do so, we need Lemma \ref{lemma:parallel-stairs} (Parallel
stairs). More specifically, we need part \eqref{i:no-white-face} of
that lemma, but we state the entire lemma for completeness.

\smallskip

\underline{Lemma \ref{lemma:parallel-stairs} (Parallel stairs):} Let
$\sigma_1$ and $\sigma_2$ be simple, disjoint, directed arcs through
the spines of shaded faces $F_1$ and $F_2$. (These shaded faces are
allowed to coincide, so long as the $\sigma_i$ are disjoint.) Suppose
that both $\sigma_1$ and $\sigma_2$ begin at the same state circle
$C$, running downstream, and terminate in the same polyhedral region
$R$. Then
\begin{enumerate}
\item There are disjoint staircases for the
  $\sigma_i$, such that $\sigma_1$ runs downstream along each segment
  of the first staircase and $\sigma_2$ runs downstream along each
  segment of the second staircase.
\item The terminal endpoint of each $\sigma_i$ is
  adjacent to the last step (state circle) of its staircase.
\item The $j$-th step of the first staircase is on
  the same state circle as the $j$-th step of the second staircase,
  except possibly the very last step. 
\item The arcs $\sigma_1$ and $\sigma_2$ cannot
  terminate on the same white face.
\end{enumerate}
 
\smallskip

As in the case of the $A$--adequate links, the proof constructs
staircases for $\sigma_1$ and $\sigma_2$, using Lemma
\ref{lemma:staircase} (Staircase extension).  Furthermore, the proof
of the (generalized) lemma requires $\sigma$--homogeneity, in that if
both arcs running downstream along the staircases end in tentacles
meeting the same white face, then at the bottom the arcs are both
either running in the right--down or the left--down direction, and we
obtain a diagram as in Figure \ref{fig:prime-stairs} or its
reflection.  That is, we obtain a sequence of stairs on the right and
the left, with bottom segments of the stairs connected by an arc
$\beta$ in the complement of $H_\sigma \cup (\bigcup \alpha_i)$ which
runs from the right side of one last segment to the right side of the
other, or from the left side of one last segment to the left side of
the other.  In either case, the argument of the proof of that lemma
will still imply that stairs connect left to right, excepting the two
bottom stairs, and that the arc $\beta$ can have its endpoints pushed
to the state circle just above both bottom stairs to give a non-prime
arc, contradicting maximality of our choice of a system of non-prime
arcs.  Then the proof of Proposition \ref{prop:no-normal-bigons} goes
through verbatim to give the following.

\begin{prop}[No Normal Bigons]\label{prop:no-normal-bigons-hom}
  Let $D(K)$ be a link diagram with an adequate, homogeneous state $\sigma$, and let $S_\sigma$ be the state
  surface of $\sigma$.  Then thee decomposition of $S^3 \cut S_\sigma$ as above
  gives polyhedra without normal bigons.  In other words,
  every polyhedron is prime. \qed
\end{prop}

Finally, given these pieces, we obtain Theorem \ref{thm:incompress} 
in this setting, without modification to the proof.  The theorem is
originally due to Ozawa \cite{ozawa}.

\begin{theorem}[Ozawa]\label{thm:ozawa}
  Let $D$ be a (connected) diagram of a link $K$, such that $D$ is
  $\sigma$--homogeneous for some state $\sigma$.  The surface
  $S_\sigma$ is essential
   in $S^3
  \setminus K$ if and only if $D$ is $\sigma$--adequate. \qed
\end{theorem}

\chapter{$I$--bundles and essential product disks}\label{sec:ibundle}
Recall that we are trying to relate geometric and topological aspects
of the knot complement $S^3\setminus K$ to quantum invariants and
diagrammatic properties.  So far, we have identified an
essential state surface $S_A$, and we
have found a polyhedral decomposition of $M_A = S^3\cut S_A$.  On the
one hand, the surface $S_A$ is known to have relations to 
 the Jones and colored Jones polynomials \cite{dasbach-futer...,
  dasbach-lin:head-tail, fkp:PAMS}.  On the other hand, the Euler characteristic
of the \emph{guts} of $M_A$, whose definition is recalled immediately
below, is known to have relations to the volume \cite{ast}.  As
mentioned in the introduction, we will see in Chapter
\ref{sec:applications} that the Euler characteristic of the guts of
$M_A$ forms a bridge between geometric and quantum invariants.
In this chapter, we take a first step toward computing this
Euler characteristic, using the polyhedral decomposition from Chapter
\ref{sec:polyhedra}.

By the annulus version of the JSJ decomposition\index{JSJ decomposition!characteristic submanifold}, there is a canonical
way to decompose $M_A= S^3\cut S_A$ along a collection of essential
annuli that are disjoint from the parabolic locus.  (In our case,
recall from Definition
\ref{def:cut} that the parabolic locus of $M_A$ will be the remnant of the
boundary of a regular neighborhood of $K$ in $M_A$.)  The JSJ decomposition yields two kinds of pieces: the
\emph{characteristic submanifold}\index{characteristic submanifold},
consisting of $I$--bundles\index{$I$--bundle} and Seifert fibered
pieces, and the \emph{guts}\index{guts}, which admit a hyperbolic
metric with totally geodesic boundary.  We consider the components of
the characteristic submanifold of $M_A$ which affect Euler
characteristic.  In this chapter, we show that such components
decompose into well--behaved pieces.  In particular, we show that they
are spanned by essential product disks (Definition \ref{def:epd})
which are each embedded in a single polyhedron of the polyhedral
decomposition of $M_A$ from Chapter \ref{sec:polyhedra}.  This is the
content of Theorem \ref{thm:epd-span}, which is the main result of the
chapter.

\section{Maximal $I$--bundles}\label{sec:Ibdls}
Let $B$ be a component of the characteristic submanifold of $M_A$; so
$B$ is either a Seifert fibered component or an $I$--bundle.  Our
first observation implies that only $I$--bundles can have non-trivial
Euler characteristic.

\begin{lemma}  \label{lemma:Ibdl}
Let $B$ be a component of the the characteristic submanifold of
$M_A$. Then $\chi(B) \leq 0$, and $B$ can come in one of two flavors:
\begin{enumerate}
\item If $\chi(B)<0$, then $B$ is an $I$--bundle. We call such
  components
\emph{non-trivial}\index{$I$--bundle!non-trivial component}\index{non-trivial component of $I$--bundle}.

\item If $\chi(B)=0$, then $B$ is a solid torus.
  We call these solid tori \emph{trivial}.
\end{enumerate}
\end{lemma}

The reason for this terminology is that removing a solid torus $B$
does not affect the Euler characteristic of what remains. Thus, for
computing the Euler characteristic of the guts, one only needs to
worry non-trivial $I$--bundles.

\begin{proof}
Recall that by Lemma  \ref{lemma:handlebody}, $M_A$ is topologically a handlebody, 
with non-positive Euler characteristic. To find the characteristic submanifold, we cut $M_A$ along essential annuli. Thus every component of the complement of these annuli will again have non-positive Euler characteristic.

The component $B$ is either an $I$--bundle or a Seifert fibered
piece. If $B$ is an $I$--bundle, then $\chi(B)$ will vanish if and
only if $B$ is a solid torus (viewed as an $I$--bundle over an annulus
or M\"obius band).

Next, suppose $B$ is a Seifert fibered component.  Since $M_A$
is a handlebody, its fundamental group is free, and
 cannot contain a $\ZZ \times \ZZ$ subgroup.  On
the other hand, $B$ is a Seifert fibered $3$--manifold with boundary.
Its base orbifold $O$ must be an orbifold with boundary.  But $B$
cannot contain any essential tori, hence the orbifold $O$ cannot
contain any essential loops.  This is possible only if $O$ a disk with
at most one cone point, and $B$ is a solid torus, with zero Euler
characteristic.
\end{proof}

The main result of this chapter is that all the non-trivial components
of $I$--bundle can be found by studying \emph{essential product
  disks}.

\begin{define}  \label{def:epd}
An \emph{essential product disk}\index{essential product disk (EPD)}\index{EPD} (EPD for short) is a properly
embedded essential disk in $M_A$, whose boundary meets the parabolic
locus of $M_A$ twice. See Figure \ref{fig:EPD}.
\end{define}

\begin{figure}
\psfrag{S}{$S_A$}
\psfrag{P}{$P$}
 \includegraphics{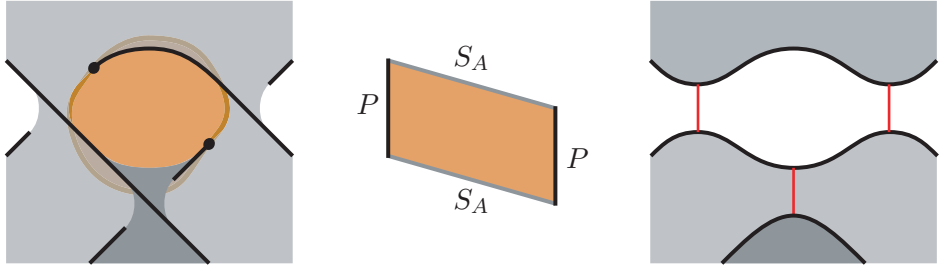}
\caption{Left: an EPD in $M_A$. 
Center: the product structure of the EPD. Right: the corresponding subgraph of $H_A$  contains  a 2-edge loop. In this example, the two segments come from different twist regions in the link diagram.}
\label{fig:EPD}
\end{figure}

Essential product disks play an important role in computing guts in
Lackenby's volume estimates for alternating
links \cite{lackenby:volume-alt}.  They will play an important role in
our setting as well.

Recall that $M_A$ is a handlebody, so certainly it admits a number of
compression disks. However, a compression disk for $M_A$ that is
disjoint from the parabolic locus would be a compression disk for
$S_A$; by Theorem \ref{thm:incompress}, such disks cannot exist.
Similarly, a compression disk for $M_A$ that meets the parabolic locus
only once would be a boundary compression disk for $S_A$; by Theorem
\ref{thm:incompress}, such disks also cannot exist.  Thus essential
product disks can be seen as the simplest compression disks for $M_A$.

Notice that a regular neighborhood of an essential product disk is an
$I$--bundle, and is thus contained in the characteristic submanifold  of
$M_A$.

\begin{define}\label{def:span}
Let $B$ be an $I$--bundle in the characteristic submanifold of $M_A$.
We say that a finite collection of disjoint essential product disks
$\{D_1, \dots, D_n \}$ \emph{spans}\index{spans} $B$ if $B \setminus
(D_1 \cup \dots \cup D_n)$ is a finite collection of prisms (which are
$I$--bundles over a polygon) and solid tori (which are $I$--bundles
over an annulus or M\"obius band).
\end{define}

Our main result in this chapter is the following theorem, which
reduces the problem of understanding the $I$--bundle in the
characteristic submanifold of $M_A$ to the problem of understanding
and counting EPDs in individual polyhedra. For instance, the EPD of Figure \ref{fig:EPD} is embedded in a lower polyhedron. In the following chapters
we will study such EPDs.

\begin{theorem}\label{thm:epd-span}\index{$I$--bundle!is spanned by EPDs}
Let $B$ be a non-trivial component of the characteristic submanifold
of $M_A$. Then $B$ is spanned by a collection of essential product
disks $D_1, \dots, D_n$, with the property that each $D_i$ is embedded
in a single polyhedron in the polyhedral decomposition of $M_A$.
\end{theorem}

The proof of the theorem will occupy the remainder of this
chapter. Before we give an outline of the proof, we need the following
definition. 

\begin{define}
  A surface $S$ in $M_A$ is \emph{parabolically
    compressible}\index{parabolically compressible} if there
  is an embedded disk $E$ in $M_A$ such that:
  \begin{enumerate}
    \item[(i)] $E \cap S$ is a single arc in $\bdy E$;
    \item[(ii)] the rest of $\bdy E$ is an arc in $\bdy M_A$ that has
      endpoints disjoint from the parabolic locus $P$ of $M_A$ and
      that intersects $P$ in at most one transverse arc;
    \item[(iii)] 
 $E$ is not parallel into $S$ under an isotopy that keeps $E \cap S$ fixed and keeps $E \cap P$ on the parabolic locus.   
  \end{enumerate}
  We say $E$ is a
  \emph{parabolic compression disk}\index{parabolic compression disk}.
  (See \cite[Figure 5]{lackenby:volume-alt}.) 
  \label{def:para-compress}
\end{define}

Definition \ref{def:para-compress} differs slightly from the corresponding definition in Lackenby's work \cite[Page 209]{lackenby:volume-alt}. Conditions (i) and (ii) are exactly the same, while our condition (iii) is slightly less restrictive.

If $D$ is an essential product disk and $E$ is a parabolic compression
for $D$, then compressing $D$ to the parabolic locus along $E$ will produce
a pair of new essential product disks, $D'$ and $D''$. See Figure
\ref{fig:para-compression}. Observe that if $D, D_1, \ldots, D_n$ span
an $I$--bundle $B$, then $D', D'', D_1, \ldots, D_n$ will span $B$ as
well. Thus we may perform parabolic compressions at will, without
losing the property that the disks in question span $B$.

\begin{figure}
  \includegraphics{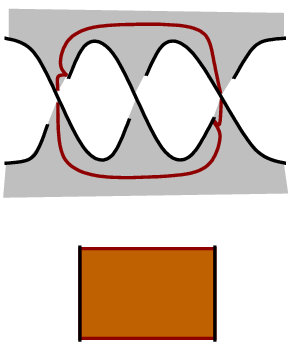}
  \hspace{.2in}
  \includegraphics{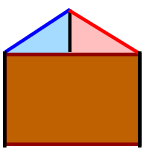}
  \hspace{.2in}
  \includegraphics{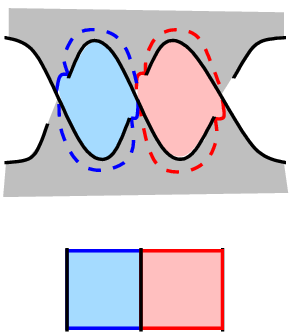}
  \caption{The EPD shown in the left panel parabolically compresses
  to the EPDs shown in the right panel.}
  \label{fig:para-compression}
\end{figure}

\begin{proof}[Top-level proof of Theorem  \ref{thm:epd-span}]

The argument has three main steps:

\begin{enumerate}
\item[{\underline {Step 1}:}]
Given a non-trivial component $B$ of the characteristic submanifold
of $M_A$ we show $B$ meets the parabolic locus
(Proposition \ref{prop:Ibdl-parabolic}).

\item[{\underline {Step 2}:}]
We show that if a component $B$ as above meets the parabolic locus, it is
spanned by essential product disks
(Proposition \ref{prop:some-epds-span}).

\item[{\underline {Step 3}:}]
We show that every essential product disk in $M_A$ parabolically
compresses to a collection of essential product disks, each of which
is embedded in a single polyhedron (Proposition \ref{prop:epds-poly}).
\end{enumerate}

Step 2 will be completed by straightforward topological argument.  On
the other hand, Steps 1 and 3 require a number of technical tools that
we will develop in the next sections.  Thus we postpone the proofs
of all three propositions until the end of the chapter.  Modulo these
propositions, the proof of Theorem \ref{thm:epd-span} is
complete.
\end{proof}

Here is the outline of the rest of the chapter.  Section
\ref{sec:normal-squares} uses normal surface theory to examine pieces
of the boundaries of $I$--bundles.  Section \ref{sec:para-compress}
uses tentacle chasing arguments to force parabolic compressions.  In
Section \ref{subsec:pfthmepd-span}, we put it all together to
finish the proof of Steps 1--3. In Section \ref{sec:gensigmahomo}, we discuss the (straightforward) 
extension to $\sigma$--adequate, $\sigma$--homogeneous diagrams.

\section{Normal squares and gluings}\label{sec:normal-squares}

In this section, we will consider properties of normal squares\index{normal!square}
(i.e. normal disks with four sides; see Definitions \ref{def:normal} and \ref{def:normal-disk}).
In Lemma \ref{lemma:squares}, we will see that normal squares arise
naturally as intersections of annuli, boundary components of our
characteristic submanifold, and our ideal polyhedra.  With this in
mind, we need to examine how normal squares glue across white faces of
the ideal polyhedra.  The results of this section, which are of a
somewhat technical nature, will be used to examine this gluing.

\begin{lemma}
Let $B_0$ be a component of the maximal $I$--bundle for $M_A$ with
negative Euler characteristic.  Then $B_0$ contains a product bundle
$Y=Q \times I$, where $Q$ is a pair of pants.  Moreover, when put into
normal form in a prime decomposition of $M_A$, the three annuli of
$\bdy Y$ are composed of disjointly embedded normal squares.
\label{lemma:squares}
\end{lemma}

The proof of Lemma \ref{lemma:squares} should be compared to that of
\cite[Lemma 7.1]{agol:guts}.

\begin{proof}
Since $B_0$ is a $3$--dimensional submanifold of $S^3$, it must be
orientable.  Thus $B_0$ is either the product $I$--bundle over an
orientable surface $F$, or the twisted $I$--bundle over a
non-orientable surface $F$.  In either case, since $\chi(B_0) =
\chi(F) < 0$, $F$ contains a pair of pants $Q$. (In a non-orientable
surface, cutting a once--punctured M\"obius band along an
orientation--reversing closed curve produces a pair of pants.) The
$I$--bundle over the pair of pants $Q$ must be trivial, so $B_0$
contains a product bundle $Y = Q \times I$.

Consider the three annuli of $\bdy Y$.  We view the union of these
three annuli as a single embedded surface.  Move this surface into
normal form in the polyhedral decomposition of $M_A$, keeping the
surface embedded.  The annuli of $\bdy Y$ stay disjoint.  The
intersection of the annuli with the faces of the polyhedra cuts the
annuli into polygons, each of which must have an even number of edges
due to the checkerboard coloring of the polyhedra.

Consider an arc of intersection between the white faces and an annulus
$A \subset \bdy Y$. If this arc $\alpha$ starts and ends on the same
boundary circle of $A$, then $A$ cuts off a bigon disk. An outermost
such arc would cut off a normal bigon in a single polyhedron -- but by
Proposition \ref{prop:no-normal-bigons}, there are no normal
bigons. Thus the arc $\alpha$ must run from one boundary circle of $A$
to the other boundary circle. Because every arc of intersection
between $\bdy Y$ and the white faces is of this form, every normal
disk must be a square.
\end{proof}

While studying the checkerboard surfaces of alternating links,
Lackenby has obtained useful results by super-imposing normal squares in the
upper polyhedron onto normal squares in the lower polyhedron
\cite{lackenby:volume-alt}.  For alternating knots and links, the
$1$--skeleton of each polyhedron is the $4$--valent graph of the link
projection; thus there is a natural ``identity map'' from one
polyhedron to the other. Lackenby's method will also be useful for our
results, although we need to take some care defining maps between the
upper and lower polyhedra.

For each white face $W$, the disk $W$ appears as a face of the upper
polyhedron and exactly one lower polyhedron.  These two faces are
glued via the gluing map\index{gluing map}, which is just the reverse
of the cutting moves we did in Chapter \ref{sec:decomp} to form the
polyhedra.

\begin{define} \label{def:clockwise}
Let $W$ be a white face of the upper polyhedron $P$, and suppose that
$W$ has $n$ sides. For the purpose of defining continuous functions,
picture $W$ as a regular $n$--gon in $\RR^2$.  Let $W'$ be the face of
a lower polyhedron that is glued to $W$ in the polyhedral
decomposition.  Then we define a \emph{clockwise map}\index{clockwise map}
$\clock\co W \to W'$ to be the composition of the gluing map with a
$2\pi/n$ clockwise rotation. In other words, both the gluing map and
the clockwise map send $W$ to $W'$, but the two maps differ by one
side of the polygon.
\end{define}

Combinatorially, in the upper polyhedron, white faces are sketched
with edges on tentacles and non-prime switches, and with vertices
adjacent to a state circle at the top--right (bottom--left) at a
crossing, or segment of $H_A$, as in the right of Figure
\ref{fig:nonprime-top} on page \pageref{fig:nonprime-top}.  However in
the lower polyhedron, white faces are drawn with vertices in the
center of segments of $H_A$, as in Figure \ref{fig:nonprime-lower} on
page \pageref{fig:nonprime-lower}.  The gluing map gives the white
faces on the upper polyhedron a slight rotation counterclockwise,
moving a vertex adjacent to a segment of $H_A$ to lie at the center of
that same segment, and then maps the region on the upper polyhedron to
the corresponding region on the lower polyhedron by the identity.  See
Figure \ref{fig:clockwise}.  On the other hand, instead of rotating
counterclockwise in the upper polyhedron to put vertices at the
centers of segments of $H_A$, $\clock$ rotates clockwise to the
nearest adjacent edge in the clockwise direction.

\begin{figure}
  \input{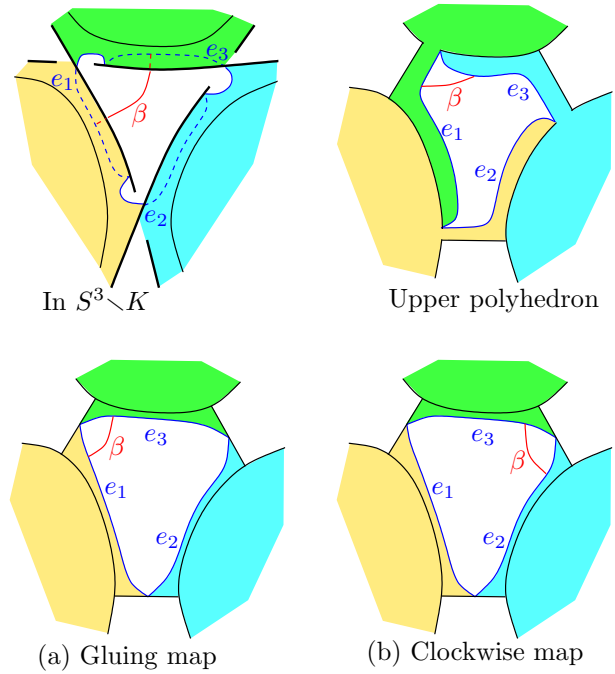}
  \caption{An example showing the image of an arc $\beta$ under (a) the gluing
    map, and (b) the clockwise map\index{clockwise map!example}.}
  \label{fig:clockwise}
\end{figure}

It is instructive to compare our setting with Menasco's polyhedral
decomposition of alternating links \cite{menasco:polyhedra}. In an
alternating diagram $D(K)$, the state surface $S_A$ is a (shaded)
checkerboard surface for $K$, and the union of all the white faces of
the polyhedra is the other (white) checkerboard surface $S_B$. If the
$1$--skeleta of both the top polyhedron and the bottom polyhedron are
identified with the $4$--valent graph of $D(K)$, then the gluing map
rotates all white faces counterclockwise and all shaded faces
clockwise. In other words, on all the white faces, the identity map
differs from the gluing map by a clockwise rotation. Furthermore, the
identity map is of course defined on the entire polyhedron, not just on
the white faces.

In our case, the clockwise map $\clock$ is an analogue of the identity
map, and also differs from the gluing map by a $2\pi/n$ clockwise
rotation. In keeping with the analogy, the domain of definition of
$\clock$ can be extended beyond the white faces (although not all the
way to the entire top polyhedron).

\begin{lemma}
\label{lemma:clockwise}
Let $U$ be a polyhedral region of the projection plane, that is, a
region of the complement of $s_A \cup (\cup_i \alpha_i)$. Let $W_1,
\ldots, W_n$ be the white faces in $U$, and let $P'$ be the lower
polyhedron associated to $U$. Then the clockwise map $\clock: W_1 \cup
\ldots \cup W_n \to P'$ has the following properties:
\begin{enumerate}
\item\label{cws:item1} If $x$ and $y$ are points on the boundary of
  white faces in $U$ that belong to the same shaded face of the upper
  polyhedron $P$, then $\clock(x)$ and $\clock(y)$ belong to the same
  shaded face of $P'$.

\item\label{cws:item2} Let $S \subset P$ be a normal square in the upper
  polyhedron, such that the white faces $V$, $W$ intersected by $S$
  belong to $U$. Let $\beta_v = S \cap V$ and $\beta_w = S \cap W$.
  Then the arcs $\clock(\beta_v)$ and $\clock(\beta_w)$ can be joined
  along shaded faces to give a normal square $S' \subset P'$, defined
  uniquely up to normal isotopy. We write $S' = \clock(S)$.

\item\label{cws:item3} If $S_1$ and $S_2$ are disjoint normal squares
  in $P$, all of whose white faces belong to $U$, then $S'_1$ and
  $S'_2$ are disjoint normal squares in $P'$.
\end{enumerate}
\end{lemma}

\begin{proof}
For conclusion \eqref{cws:item1}, let $F$ be a shaded face of the
upper polyhedron $P$, and let $x$ and $y$ be points on $(\bdy F) \cap
U$.  Then $x$ and $y$ can be connected by an arc $\gamma$ running
through $F$, and we can make $\gamma$ simple with respect to $F$
(Definition \ref{def:simple-face}).  If the arc $\gamma$ is parallel
to an ideal edge $e$ of the upper polyhedron, then $x,y \in e$, hence
$\clock(x), \clock(y) \in \clock(e)$, and the conclusion holds.
Otherwise, the arc $\gamma$ must cross some state circle $C$, hence is
non-trivial.  Because both of its endpoints are in the same polyhedral
region, in fact $\gamma$ must cross $C$ twice, first running upstream
and then downstream by Lemma \ref{lemma:utility}.  Thus we can split
$\gamma$ into two disjoint arcs beginning at $C$, running downstream,
and terminating in the same polyhedral region.  By Lemma
\ref{lemma:parallel-stairs} (Parallel stairs), $\gamma$ must run up
and down a pair of right--down staircases, and by part
\eqref{i:same-steps} of that lemma, the first state circle $C_1$ that
$\gamma$ crosses when running from $x$ to $y$ must be the same as the
last state circle $C_1$.  Now $C_1 \subset \bdy U$ corresponds to a
shaded face $F'$ of the lower polyhedron $P'$ (see Figure
\ref{fig:nonprime-lower} on page \pageref{fig:nonprime-lower}). Thus
both $\clock(x)$ and $\clock(y)$ must lie on the boundary of $F'$.

For \eqref{cws:item2}, let $S \subset P$ be a normal square, such that
the white faces $V$, $W$ intersected by $S$ belong to $U$. Let
$u,x,y,z$ be the four points of intersection between $S$ and the edges
of $P$, such that $u,x \in \bdy V$ and $y,z \in \bdy W$. Then $x,y$
lie on the boundary of the same shaded face $F$. By conclusion
\eqref{cws:item1}, $\clock(x), \clock(y)$ belong to the same shaded
face $F' \subset P'$. Since $F'$ is simply connected by Theorem
\ref{thm:simply-connected}, $\clock(x)$ and $\clock(y)$ can be
connected by a unique isotopy class of arc in $F'$. Similarly,
$\clock(z)$ and $\clock(u)$ can be connected by a unique isotopy class
of arc in a shaded face of $P'$. These normal arcs in shaded faces
combine with the arcs $\clock(\beta_v) \subset \clock(V)$ and
$\clock(\beta_w) \subset \clock(W)$ to form a normal square $S'
\subset P'$, which is unique up to normal isotopy.
 
For conclusion \eqref{cws:item3}, let $S_1$ and $S_2$ be disjoint
normal squares in $P$, all of whose white faces belong to $U$. The
arcs of $S_1$ and $S_2$ that lie in white faces of $U$ are mapped
homeomorphically (hence disjointly) to white faces in $P'$. Thus it
remains to check that the arcs of $S_1'$ and $S_2'$ are also disjoint
in the shaded faces. Suppose that both $S_1$ and $S_2$ pass
through a shaded face $F$, disjointly. Then we can label points
$w,x,y,z$, in clockwise order around $\bdy F$, such that $S_1$
intersects $\bdy F$ at points $w,x$ and $S_2$ intersects $F$ at points
$y,z$. Then the four points $\clock(w), \clock(x), \clock(y),
\clock(z)$ are also arranged in clockwise order around a shaded face
$F'$ of $P'$. Thus $S_1'$ and $S_2'$ are disjoint in $F'$.
\end{proof}

One part of proving the main result in this chapter is to show that
certain normal squares in the upper polyhedron are parabolically
compressible. For that, we will map them to squares in the lower
polyhedra, using the clockwise map and Lemma \ref{lemma:clockwise},
and consider their intersections with certain normal squares in the
lower polyhedra.  We use the following lemma, which is due to Lackenby
\cite[Lemma 7]{lackenby:volume-alt}. We include the proof for
completeness.

\begin{lemma}
  \label{lemma:marcs}
  Let $P$ be a prime, checkerboard--colored polyhedron with 4--valent
  vertices. Let $S$ and $T$ be normal squares in $P$, which have been 
  moved by normal isotopy into a position that minimizes their
intersection number. Then $S$ and
  $T$ are either disjoint, or they have an
  essential intersection in two faces of the same color.
\end{lemma}

Recall, from Definition \ref{def:prime} on page \pageref{def:prime}, that a polyhedron $P$ is
prime if it contains no normal bigons. By Proposition
\ref{prop:no-normal-bigons} on page \pageref{prop:no-normal-bigons}, our polyhedra are all prime.

\begin{proof}
The four sides of $S$ run through four distinct faces of the
polyhedron, as do the four sides of $T$.  A side of $S$ intersects a
side of $T$ at most once.  If all four sides of $S$ intersect sides of
$T$, then $S$ and $T$ are isotopic and can be isotoped off each other.
So $S$ and $T$ intersect at most three times.  However, $\bdy S$ and
$\bdy T$ form closed curves on the boundary of the polyhedron, hence
$S$ intersects $T$ an even number of times, so $0$ or $2$ times.  If
twice, suppose $S$ and $T$ intersect in faces of opposite color.  Then
each arc of $S\setminus T$ and $T\setminus S$ intersects the edges of
the polyhedron an odd number of times.  Hence one of the four
complementary regions of $S\cup T$ has two points of intersection with
the edges of the polyhedron in its boundary.  Because the polyhedron
is prime, this gives a bigon which cannot be normal, hence $S$ and $T$
can be isotoped off each other.
\end{proof}

This lemma has the following useful consequence, illustrated in Figure
\ref{fig:squares-int}.

\begin{lemma}
\label{lemma:squares-int} 
Let $P$ be a prime, checkerboard--colored polyhedron with 4--valent
vertices. Let $S$ and $T$ be normal squares in $P$, moved by normal
isotopy to minimize their intersection number.  Suppose that $S$ and
$T$ pass through the same white face $W$, and that the edges $S \cap
W$ and $T \cap W$ differ by a single rotation of $W$ (clockwise or
counterclockwise).  Then exactly one of the following two conclusions
holds:
\begin{enumerate}
\item\label{squares-int:1} Each of $S$ and $T$ cuts off a single ideal
  vertex in $W$, and $S$ and $T$ are disjoint.
\item\label{squares-int:2} Neither $S$ nor $T$ cuts off a single
  vertex in $W$. The two normal squares intersect in $W$ and in
  another white face $W'$.
\end{enumerate}
\end{lemma}

\begin{figure}
  \begin{center}
    \begin{tabular}{ccccc}
      \eqref{squares-int:1} & \includegraphics{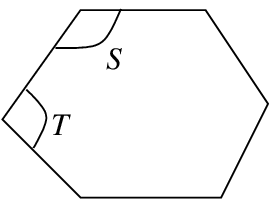}
      & \hspace{.2in} & \eqref{squares-int:2} &
      \includegraphics{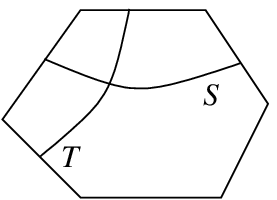}
    \end{tabular}
  \end{center}
  \caption{Cases \eqref{squares-int:1} and \eqref{squares-int:2} of
    Lemma \ref{lemma:squares-int} are illustrated for an example.
    Note if \eqref{squares-int:2} happens, $S$ and $T$ intersect in
    two white faces.}
  \label{fig:squares-int}
\end{figure}

\begin{proof}
First, suppose that $S$ cuts off a single ideal vertex in $W$.  Then
so does $T$.  Hence $S$ and $T$ do not intersect in $W$.  By Lemma
\ref{lemma:marcs}, we can conclude that if $S$ and $T$ intersect at
all, they intersect in two shaded faces, $F$ and $G$.  Since $W$ meets
each shaded face at only one edge, $S \cap W$ and $T \cap W$ must run
in parallel through $W$.  Thus their intersections in $F$ and $G$ can
be isotoped away, and $S$ and $T$ are disjoint.  This proves
\eqref{squares-int:1}.

Next, suppose that $S$ does not cut off a single ideal vertex in $W$.
Then neither does $T$.  Thus, since $S \cap W$ and $T \cap W$ differ
by a clockwise or counterclockwise rotation of $W$, they must have an
essential intersection.  By Lemma \ref{lemma:marcs}, they must also
intersect in another white face $W'$. This proves
\eqref{squares-int:2}.
\end{proof}

When case \eqref{squares-int:1} of Lemma \ref{lemma:squares-int}
holds, note that there is a parabolic compression of $S$, through $W$,
to the single ideal vertex of $W$ that it cuts off.  Similarly for
$T$.

\begin{define}
  Let $P$ be a truncated, checkerboard--colored ideal polyhedron. Then
  a \emph{normal trapezoid}\index{normal!trapezoid} in $P$ is a normal
  disk that passes through two shaded faces, one white face, and one
  truncated ideal vertex.
\label{def:trapezoid}
\end{define}

Trapezoids give the following analogue of Lemma \ref{lemma:squares-int}.

\begin{lemma}\label{lemma:trapezoid}
  Let $P$ be a prime, checkerboard--colored polyhedron with 4--valent
  vertices. Let $S$ be a normal square in $P$, and $T$ a normal
  trapezoid. Suppose that $S$ and $T$ pass through the same white face
  $W$, and that their arcs of intersections with $W$ differ by a
  single rotation (clockwise or counterclockwise). Then $S$ and $T$
  are disjoint, and each of $S$ and $T$ is parabolically compressible
  to an ideal vertex of $W$.
\end{lemma}

\begin{proof}
Let $T'$ be a normal square obtained by pulling $T$ off an ideal
vertex, into a white face $W'$. Then $S$ and $T'$ are normal squares
that satisfy the hypotheses of Lemma \ref{lemma:squares-int}.  Because
$T'$ cuts off a single ideal vertex of $W'$, $S$ and $T'$ cannot
intersect in that white face.  Thus conclusion \eqref{squares-int:1}
of Lemma \ref{lemma:squares-int} holds: $S$ and $T'$ are disjoint, and
each one cuts off a single ideal vertex in $W$. Thus $S$ and $T$ are
also disjoint, and each one is parabolically compressible to an ideal
vertex of $W$.
\end{proof}

By applying Lemmas \ref{lemma:squares-int} and \ref{lemma:trapezoid}, we will show that many normal
squares are also parabolically compressible. See e.g.\  the proof of Lemma \ref{lemma:trap-compress} for a preview of the argument.

\section{Parabolically compressing normal
  squares}\label{sec:para-compress}

Results of Section \ref{sec:normal-squares} are enough to handle normal
squares with sides in the same polyhedral region.   Note this is all
that occurs for alternating knots, as in \cite{lackenby:volume-alt}.
In this section, we will use tentacle chasing arguments\index{tentacle chasing} to extend our
tools, so that we can deal with normal squares with their sides in
different polyhedral regions.  This is the content of the next
proposition, which is the main result in this section.

\begin{prop}
  \label{prop:square-compress}
  Let $S$ be a normal square in the upper polyhedron, with boundary
  consisting of arcs $\beta_v$, $\beta_w$ on white faces $V$ and $W$,
  and arcs $\sigma_1$, $\sigma_2$ on shaded faces.  Suppose that $V$
  and $W$ are in different polyhedral regions.  Finally, suppose that
  $S$ is glued to a normal square $T$ at $W$.  Then $S$ cuts off a
  single ideal vertex in $W$, hence is parabolically
  compressible at $W$.
\end{prop}

Proposition \ref{prop:square-compress} is a crucial ingredient for the
proof of the main result of the chapter, Theorem \ref{thm:epd-span},
which is given in the next section.  Before we give the proof of the
proposition, we need to establish some technical lemmas.  We advise the
reader that only the statement of Proposition
\ref{prop:square-compress}, and not those of the intermediate
technical lemmas, are required for the proof of Theorem
\ref{thm:epd-span}.  Thus readers eager to get to the proof of the
main result of the chapter may, at this point, move to the next
section, on page \pageref{subsec:pfthmepd-span}, without loss of
continuity.  However, several of the technical lemmas in the remainder
of this section are repeatedly used in Chapter \ref{sec:epds}.

\begin{lemma}[Opposite sides\index{Opposite sides lemma}]
Let $S$ be a normal square with boundary arcs $\beta_v$ and $\beta_w$
on white faces $V$ and $W$, and arcs $\sigma_1$ and $\sigma_2$ on
distinct shaded faces.  Suppose $\sigma_1$ and $\sigma_2$ intersect
the same state circle $C$.  Then the intersections are in tentacles
attached to edges on opposite sides of $C$, and $C$ must separate $V$
and $W$.  
\label{lemma:opp-sides}
\end{lemma}

Recall again that arcs in a shaded face can only intersect state circles
at the heads of tentacles, adjacent to segments of $H_A$. (See Definitions
\ref{def:tentacle} and
\ref{def:tentacle-direct}, as well as  Figure \ref{fig:tentacle-mult}, on page
\pageref{fig:tentacle-mult}.)  Lemma \ref{lemma:opp-sides} (Opposite
sides) asserts that under the given hypotheses, $\sigma_1$ and
$\sigma_2$ run adjacent to heads of tentacles attached to $C$, but
adjacent to segments on opposite sides of $C$.

\begin{proof}
We will first show that $C$ must separate $W$ and $V$, and then that
when we direct $\sigma_1$ and $\sigma_2$ to run across $C$ away from
$V$ and toward $W$, one of $\sigma_1$, $\sigma_2$ runs upstream and
one runs downstream.  This will imply the result.

Suppose, by way of contradiction, that $C$ does not separate $V$ and
$W$, but that both lie on the same side of $C$.  Then both $\sigma_1$
and $\sigma_2$ must intersect $C$ twice.  Direct $\sigma_1$ and
$\sigma_2$ away from $V$.  We may assume each is simple with respect
to its shaded face.  Now Lemma \ref{lemma:utility} (Utility lemma)
implies that both arcs cross $C$ first running upstream, then running
downstream.  Consider the portion of the arcs running downstream.
These are both running downstream from $C$, connected at their ends by
$\beta_w$.  This contradicts part \eqref{i:no-white-face} of Lemma
\ref{lemma:parallel-stairs} (Parallel stairs).

Now, suppose $C$ separates $V$ and $W$, but $\sigma_1$ and $\sigma_2$
run in the same direction across $C$.  Then, switching $V$ and $W$ if
necessary, we may assume that both $\sigma_1$ and $\sigma_2$ run away
from $V$ across $C$ in the downstream direction.  Again we have arcs
$\sigma_1$ and $\sigma_2$ running downstream from $C$, connected at
their ends by $\beta_w$.  Again this contradicts Lemma
\ref{lemma:parallel-stairs} (Parallel stairs).  
\end{proof}

\begin{lemma}[Entering polyhedral region\index{Entering polyhedral region lemma}]
Let $S$ be a normal square with boundary consisting of arcs $\beta_v$
and $\beta_w$ on white faces $V$ and $W$, and arcs $\sigma_1$ and
$\sigma_2$ on shaded faces.  Suppose also that $V$ and $W$ are in
distinct polyhedral regions $R_V$ and $R_W$.  Then (up to relabeling),
when $\sigma_1$ and $\sigma_2$ are directed away from $V$ towards $W$,
we have the following: 
\begin{enumerate}
\item The arc $\sigma_1$ first enters $R_W$ through a state circle $C$
  running in the downstream direction, and immediately connects to
  $\beta_w$ (i.e., without intersecting any additional state circles
  or non-prime arcs).
\item The arc $\sigma_2$ first enters $R_W$ either through $C$ running
  upstream, or through a non-prime arc with both endpoints on $C$.  In
  any case, if $\sigma_2$ crosses $C$, then it must do so only once,
  running upstream.
\end{enumerate}
\label{lemma:enter-RW}
\end{lemma}

\begin{proof}
Since $V$ and $W$ are in distinct polyhedral regions, they are either
on opposite sides of some state circle, or if they are not on opposite
sides of any state circle, they are separated by a non-prime arc
$\alpha$ with both endpoints on a state circle $C$.  In the latter
case, $C$ does not separate $V$ and $W$, nor does any state circle
contained inside the non-prime half-disk bounded by $\alpha$ and $C$
separate $V$ and $W$. We distinguish two cases.

\smallskip

{\underline{Case 1:}} Suppose that $V$ and $W$ are separated by a
non-prime arc $\alpha$ with both endpoints on some state circle $C$,
but that $C$ does not separate $V$ and $W$.  Suppose also that within
the non-prime half-disk bounded by $\alpha$ and $C$ that contains
$W$, no other state circle separates $V$ and $W$.  Without loss of
generality, we may suppose that $\alpha$ is innermost with this
property with respect to $W$, that is, that $\alpha$ is the non-prime
arc with this property closest to $W$.

Notice that one of $\sigma_1$, $\sigma_2$ must cross $C$, since the
arcs are on distinct shaded faces.  After relabeling, we may assume
$\sigma_1$ crosses $C$.  Since $V$ and $W$ are on the same side of
$C$, $\sigma_1$ must actually cross $C$ twice.  But then Lemma
\ref{lemma:utility} (Utility lemma) implies that it crosses $C$ first
running upstream, then running downstream.  So $\sigma_1$ crosses
running downstream when it enters $R_W$.

Since $\sigma_1$ is running downstream, it will be adjacent to some
state circle $C_4$ attached to $C$ by a segment of the graph $H_A$.
By assumption, $C_4$ does not separate $V$ and $W$.

Since $\sigma_1$ is going downstream, Lemma \ref{lemma:downstream}
(Downstream continues down) implies that it can only continue
downstream next, or cross a non-prime arc with $W$ inside, or exit the
shaded face immediately to $\beta_w$.  The first of these three
possibilities cannot happen: $\sigma_1$ cannot continue downstream,
else it crosses into $C_4$, and must cross back out, which is
impossible by Lemma \ref{lemma:utility} (Utility lemma).  The second
possibility also cannot hold by assumption:  $\alpha$ was assumed to
be innermost with respect to $W$.  Thus the only possibility is that
$\sigma_1$ exits the shaded face immediately to $\beta_w$.

If $\sigma_2$ also crosses $C$, then it must do so twice, since its
endpoints are not separated by $C$.  Lemma \ref{lemma:utility}
(Utility lemma) implies that $\sigma_2$ first crosses $C$ running
upstream, then crosses running downstream.  But now $\sigma_1$ and
$\sigma_2$ both cross $C$ running downstream, then have endpoints
attached at $\beta_w$.  This contradicts Lemma
\ref{lemma:parallel-stairs} (Parallel stairs).  So $\sigma_2$ does
not cross $C$, which implies it enters $W$ across a non-prime arc, as
desired.

\smallskip

{\underline{Case 2:}} Suppose that $V$ and $W$ are on opposite sides
of some state circle.  Then there is some such state circle which is
closest to $W$.  Call this state circle $C_W$.  The arcs $\sigma_1$
and $\sigma_2$ must both intersect $C_W$.  By Lemma
\ref{lemma:opp-sides} (Opposite sides), the intersections are in
opposite directions, when $\sigma_1$ and $\sigma_2$ are both directed
toward $W$.  Relabel, if necessary, so that $\sigma_1$ is the arc
running downstream toward $W$ across $C_W$.
  
Since $\sigma_1$ is running downstream, it will be adjacent to some
state circle $C_4$ attached to $C_W$ by a segment of the graph $H_A$.
Again since $\sigma_1$ is running downstream, Lemma
\ref{lemma:downstream} (Downstream continues down) implies that it can
only go downstream next, or cross a non-prime arc with $W$ inside, or
exit the shaded face to $\beta_w$.  The first possibility cannot
happen: $\sigma_1$ cannot continue downstream, else it crosses into
$C_4$, and must cross back out, which is impossible by Lemma
\ref{lemma:utility} (Utility lemma).  Suppose the second possibility
holds, that is that $\sigma_1$ crosses a non-prime arc $\alpha$ with
$W$ inside.  This non-prime arc $\alpha$ has both endpoints on $C_4$,
and $C_4$ does not separate $V$ and $W$ by choice of $C_W$.  Moreover,
within the non-prime half-disk bounded by $\alpha$ and $C_4$ which
contains $W$, no state circle can separate $V$ and $W$, again by
choice of $C_W$.  Thus if this second possibility holds, we are in
Case 1, with $C_4$ playing the role of $C$, and the lemma is true
(after relabeling $\sigma_1$ and $\sigma_2$ again).

The only remaining possibility is that $\sigma_1$ exits the shaded
face immediately to $\beta_w$ after crossing $C_W$.  This proves
statement (1) of the Lemma, with $C= C_W$.

Finally, the fact that $\sigma_2$ must cross $C$ running upstream, if
it crosses $C$ at all, follows immediately from the fact that
$\sigma_1$ crosses $C$ running downstream, and Lemma
\ref{lemma:opp-sides} (Opposite sides).
\end{proof}

\begin{lemma}
Let $S$ be a normal square with boundary arcs $\beta_v$ and $\beta_w$
on white faces $V$ and $W$, and arcs $\sigma_1$ and $\sigma_2$ on
shaded faces.  Suppose $V$ and $W$ are in distinct polyhedral regions
$R_V$ and $R_W$.  Let $C_W$ be the state circle in the conclusion of
Lemma \ref{lemma:enter-RW} (Entering polyhedral region), and relabel
if necessary so that $\sigma_1$ and $\sigma_2$ are as in the
conclusion of that Lemma, when directed away from $V$ towards $W$.
Suppose, moreover, that arc $\sigma_2$ does not immediately connect to
$\beta_w$ after it first enters $R_W$. Then, $\sigma_2$ runs across a
state circle $C_2$ running upstream, then eventually crosses $C_2$
again into $R_W$ running downstream, at which point it immediately
connects to $\beta_w$ (i.e. without crossing any other state circles).
\label{lemma:sigma_2-connect}
\end{lemma}

Note that Lemmas \ref{lemma:enter-RW} and
\ref{lemma:sigma_2-connect} imply that if $\sigma_2$ does not
immediately connect to $\beta_w$, the region $R_W$ is of the form
shown in Figure \ref{fig:w-picture}.

Essentially, what these two lemmas say is that $\sigma_1$ only enters
the region $R_W$ once, to connect to $\beta_W$.  The arc $\sigma_2$,
on the other hand, may enter $R_W$, then leave and travel elsewhere,
but when it returns it will not leave again, but connect immediately
to $\beta_W$.

\begin{figure}
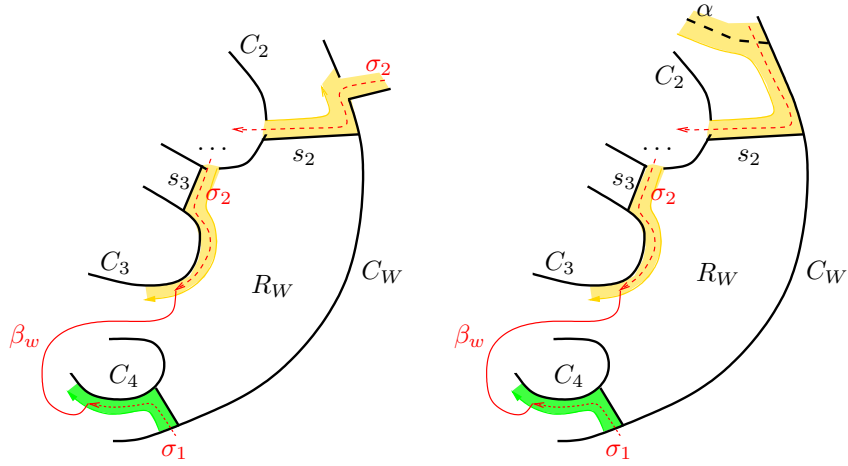

  \begin{center}
    \input{figures/w-picture.pstex_t}
    \hspace{.2in}
    \input{figures/non-prime-wpicture.pstex_t}
  \end{center}
  \caption{The region $R_W$ in the case that $\sigma_2$ does not
    immediately connect to $\beta_w$.  Left: Initially $\sigma_2$
    enters $R_W$ by running upstream.  Right: Initially $\sigma_2$
    enters $R_W$ along a non-prime arc.  In both cases, $\sigma_2$
    runs across a state circle $C_2$ and eventually re-enters $R_W$.}
  \label{fig:w-picture}
\end{figure}

\begin{proof}
If after entering $R_W$, $\sigma_2$ does not immediately meet
$\beta_w$, then it must cross a state circle or non-prime arc first.
It does not cross a non-prime arc, for if so, it would enter a
non-prime half-disk bounded by the non-prime arc and some state
circle $C$, so must exit this non-prime half-disk along $C$, and by
Lemma \ref{lemma:np-shortcut} (Shortcut lemma), must do so running
downstream.  Then $\sigma_2$ must cross $C$ again, to re-enter $R_W$,
but then Lemma \ref{lemma:utility} (Utility lemma) implies it must
first cross running upstream.  This is impossible.  So $\sigma_2$ does
not cross a non-prime arc on the boundary of $R_W$ between entering
$R_W$ and connecting to $\beta_w$.

Similarly, $\sigma_2$ cannot follow a tentacle downstream after
crossing into $R_W$, or as above it would not be able to re-enter
$R_W$.  The only other possibility is that $\sigma_2$ follows a
tentacle upstream, crossing into a state circle $C_2$.  Then
$\sigma_2$ must exit out of $C_2$.  Lemma \ref{lemma:utility} (Utility
lemma) implies that $\sigma_2$ exits $C_2$ in the downstream
direction.

Now $\sigma_2$ is running downstream, so will be on the tail of a
tentacle adjacent to a state circle $C_3$, attached to $C_2$ by a
segment of $H_A$.  Since $\sigma_2$ is running downstream, Lemma
\ref{lemma:downstream} (Downstream) implies it either continues
running downstream, crossing into $C_3$, or crosses a non-prime arc
$\alpha$ with $W$ on the opposite side, or exits the shaded face to
$\beta_w$.  The first possibility cannot hold: $\sigma_2$ cannot cross
$C_3$ running downstream, since it must cross out again to meet
$\beta_w$, and this condradicts Lemma \ref{lemma:utility} (Utility
lemma).

The seccond possibility will also lead to a contradiction.  If
$\sigma_2$ crosses a non-prime arc $\alpha$ with $W$ on the opposite
side, then $\alpha$ would have both endpoints on the state circle
$C_3$.  The non-prime half-disk containing $W$ bounded by $\alpha$
and $C_3$ is therefore separated from the region containing $C_W$.
But this is impossible: $C_W$ meets the boundary of $R_W$.  So the
second possibility cannot hold either.

The only remaining possibility is that $\sigma_2$ immediately connects
to $\beta_w$, as desired.
\end{proof}

We are now ready to give the proof of Proposition
\ref{prop:square-compress}: a normal square $S$ whose white faces are
in different polyhedral regions, which is glued to a normal square in
a lower polyhedron, must parabolically compress.

\begin{proof}[Proof of Proposition \ref{prop:square-compress}]
As usual, let $\sigma_1$ and $\sigma_2$ be the arcs of the square $S$
on the shaded faces.  By Lemma \ref{lemma:enter-RW} (Entering
polyhedral region), we may assume that $\sigma_1$ enters $R_W$, the
polyhedral region containing $W$, by crossing a state circle $C_W$ in
the downstream direction, and then immediately connects to $\beta_w$.
We may also assume that $\sigma_2$ enters either in the upstream
direction or across a non-prime arc.

\smallskip

\underline{Case 1:} Suppose first that $\sigma_2$ also meets $\beta_w$
immediately, without meeting any other boundary components of $R_W$.
Then the region $R_W$ and the arc $\beta_w$ will have the form of one
of the two graphs shown in Figure \ref{fig:head-nonprime},
corresponding to the two possibilities for $\sigma_2$.

\begin{figure}[h]
	\input{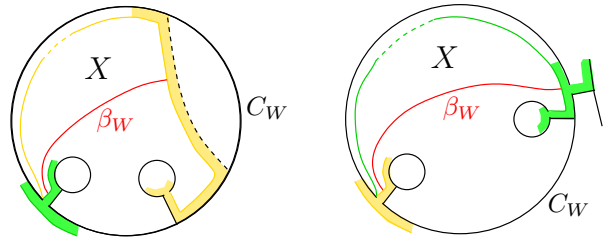}
  \caption{Regions $R_W$ in the case $\sigma_2$ connects immediately
    to $\beta_W$.}
  \label{fig:head-nonprime}
\end{figure}

If the region marked $X$ in the diagrams of Figure
\ref{fig:head-nonprime} contains state circles, then if we push the
endpoints of the arc $\beta_w$ to the state circle on the outside in
each diagram, we form a non-prime arc $\alpha$, bounding state circles
on either side.  This contradicts the maximality of our prime
decomposition, Definition \ref{def:max-nonprime}.  
Thus the diagrams in Figure \ref{fig:head-nonprime} must contain no
state circles in the regions marked $X$.  Then in both cases, the
tentacle running through $X$ around the interior of the outermost
state circle will terminate at the top of the tentacle where $\beta_w$
has its other endpoint, as illustrated.  Thus $\beta_w$ cuts off a
single ideal vertex of the white face $W$.  But then the portion of
the white disk bounded by these two shaded faces and the arc $\beta_w$
forms a parabolic compression disk for $S$, as desired.

\smallskip

\underline{Case 2:} Now suppose that $\sigma_2$ does not immediately
connect to $W$ after crossing $C_W$.  Then Lemma
\ref{lemma:sigma_2-connect} implies that the region $R_W$ is as shown
in Figure \ref{fig:w-picture}.
 Recall that by assumption, $S$ is glued to a square $T$ in the lower
polyhedron at $W$.  Apply the clockwise map to $\beta_w$, sending it
to an arc which differs from the arc of $T$ lying in $W$ by a
clockwise rotation.  The image of $\beta_w$ is shown in Figure
\ref{fig:beta-image}.

\begin{figure}
  \input{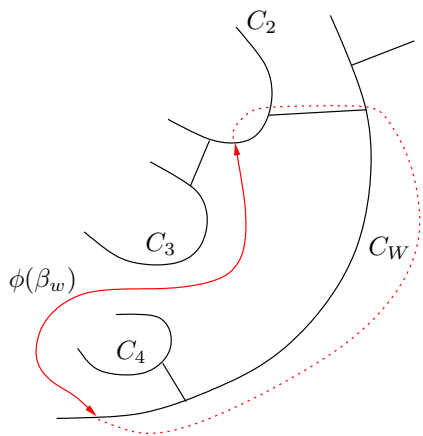}
  \caption{Image of $\beta_w$ under clockwise map $\phi\co W \to W$.
    Note there is an arc (dotted red) running through two state
    circles and a single segment of $H_A$ connecting its endpoints.
    All red arcs shown form a normal trapezoid in the lower
    polyhedron}
    \label{fig:beta-image}
\end{figure}

Notice in the lower polyhedron that there exists an arc through two
innermost disks (boundary components of $R_W$) adjacent to a single
segment of $H_A$ which connects the endpoints of the image of
$\beta_w$ under the clockwise map; see Figure \ref{fig:beta-image}.
In fact, this gives a normal trapezoid $S'$ contained in the lower
polyhedron with one of its sides on $W$, two sides on the two shaded
faces corresponding to the two innermost disks, and a side running
over the ideal vertex of the polyhedron corresponding to the center of
this edge of $H_A$.

Recall that we have the normal square $T$ in the lower polyhedron
with one side on the white face $W$, differing from the side of $S'$
on $W$ by a single clockwise rotation.  Lemma \ref{lemma:trapezoid}
implies that  $S'$ and $T$ are  parabolically compressible to an
ideal vertex of $W$. Thus, $S$ is parabolically compressible at
$W$.
\end{proof}


\section{$I$--bundles are spanned by essential product
  disks}\label{subsec:pfthmepd-span}

We can now complete the proof of Theorem \ref{thm:epd-span}.  Recall
from the beginning of the chapter that the proof of Theorem
\ref{thm:epd-span} required three steps, whose proofs we have
postponed until now.  The first step, Proposition
\ref{prop:Ibdl-parabolic}, relies on the following general lemma, which will also be needed in Chapter \ref{sec:spanning}.

\begin{lemma}[Product rectangle in white face]\label{lemma:product-in-face}\index{Product rectangle in white face lemma}
Let $B$ be an $I$--bundle in $M_A$, whose vertical boundary is
incompressible. Suppose that $B$ has been moved by isotopy to minimize
its intersections with the white faces. Then, for any white face $W$,
the intersection $B \cap W$ is a union of product rectangles whose
product structure comes from the $I$--bundle structure of $B$. In
other words, each component of $B \cap W$ has the form
$$D = \alpha \times I,$$
where $\alpha \times \{0 \}$ and $\alpha \times \{1 \}$ are sub-arcs
of ideal edges of $W$. 
\end{lemma}

\begin{proof}
Suppose, first, that $B = Q \times I$ is a product $I$--bundle over an
orientable base $Q$. At the end of the proof, we will consider the
case of a non-orientable base.

Let $D$ be a component of $B \cap W$. Notice that $\bdy D$ cannot
contain any simple closed curves in the interior of $W$, because an
innermost such curve would bound a compression disk for $\bdy B$, and
can be removed by isotopy. Similarly, $\bdy D$ cannot contain an arc
from an ideal edge of $W$ to the same ideal edge, since an outermost
such arc can be removed by isotopy.

Truncate the ideal vertices of $W$, so that every ideal vertex becomes
an arc (parallel to the parabolic locus). Abusing notation slightly,
the portion of $D$ in this truncated face is still denoted $D$. Then,
by the above paragraph, $D$ must be a $2n$--gon, with sides $\alpha_1,
\beta_1, \ldots, \alpha_n, \beta_n$. Here, every $\alpha_i$ is a
sub-arc of an ideal edge of $W$ (and comes from the horizontal
boundary of $B$), while every $\beta_i$ is a normal arc that connects
distinct ideal edges of $W$ (and comes from the vertical boundary of
$B$).

We claim that every $\beta_i$ spans the product bundle $B = Q \times
I$ top to bottom. For, suppose for concreteness that both endpoints of
$\beta_1$ are on $Q \times \{ 1 \}$. Then $\beta_1 \subset W$ is
parallel to $Q \times \{ 1 \}$ through $Q \times I$, which gives a
boundary compression disk for the white face $W$. This contradicts
Lemma \ref{lemma:white-incompress} on page
\pageref{lemma:white-incompress}, proving the claim. Note that this
implies $n$ is even.

Next, we claim that $n = 2$. For, suppose for a contradiction that $n
> 2$. Then the sides $\alpha_1$ and $\alpha_3$ belong to the same (top
or bottom) boundary of $B$, say $Q \times \{ 1 \}$. There is an arc
$\gamma$ through the polygon $D \subset W$ that connects $\alpha_1$ to
$\alpha_3$. This arc is parallel to $Q \times \{ 1 \}$ through $Q
\times I$, which again gives a boundary compression disk,
contradicting Lemma \ref{lemma:white-incompress}.
We conclude that $D$ is a rectangle, with $\alpha_1, \alpha_2$
horizontal and $\beta_1, \beta_2$ vertical. Thus, after an appropriate
isotopy, $D$ is a vertical rectangle in the product structure on $B$.

\smallskip

Now, suppose that $B = Q \widetilde{\times} I$, where $Q$ is
non-orientable. Let $\gamma_1, \ldots, \gamma_n$ be a maximal
collection of disjoint, embedded, orientation--reversing loops in
$Q$. Then the $I$--bundle over each $\gamma_i$ is a M\"obius band
$A_i$. Furthermore, $Q \setminus (\cup \gamma_i)$ is an orientable
surface $Q_0$, such that $B_0 = B \setminus (\cup A_i)$ is a product
$I$--bundle over $Q_0$.
Let $W$ be a white face, and let $D$ be a component of $B \cap W$. By
the orientable case already considered, every component of $B_0 \cap
W$ is a product rectangle $\alpha \times I$. Also, $A_i \cap W$ is a
union of arcs, hence the regular neighborhood of each vertical
M\"obius band $A_i$ intersects $W$ in rectangular product strips. Each
of these strips respects the $I$--bundle structure of $B$. Thus $D$ is
constructed by joining together several product rectangles of $B_0
\cap W$, along product rectangles in the neighborhood of $A_i \cap
W$. Therefore, all of $D$ is a product rectangle, as desired.
\end{proof}


\begin{prop}  \label{prop:Ibdl-parabolic}{\rm [Step 1]}
  Let $B$ be a non-trivial $I$--bundle of the characteristic
  submanifold of $M_A$.  Then $B$ meets the parabolic locus of $M_A$.
\end{prop}

\begin{proof}
Recall that since $B$ be is non-trivial we have $\chi(B)<0$.  By Lemma
\ref{lemma:squares}, $B$ contains a product bundle $Y=Q \times I$,
where $Q$ is a pair of pants.  Moreover, when put into normal form in
a prime decomposition of $M_A$, the three annuli of $\bdy Y$ are
composed of disjointly embedded normal squares.  Label the squares
$S_1, \dots, S_n$.  Observe each has the form of a product $S_i =
\gamma_i \times I$, where $\gamma_i$ is a sub-arc of $\partial Q$.

If some $S_i$ is parabolically compressible, observe that the
parabolic compression disk $E$ that connects $S_i$ to the parabolic
locus is itself a product $I$--bundle, which can be homotoped to have
product structure matching that of $S_i$.  Hence $E \subset B$, and
$B$ borders the parabolic locus, as desired.

If $Y$ passes through more than one lower polyhedron, some square
$S_i$ must pass through white faces in different polyhedral
regions. Thus, by Proposition \ref{prop:square-compress}, $S_i$ is
parabolically compressible. Hence, as above, $B$ borders the parabolic
locus.

For the rest of the proof, assume that every $S_i$ is parabolically
incompressible, and so $Y$ is entirely contained in the upper
polyhedron and exactly one lower polyhedron $P$.  This assumption will
lead to a contradiction.

Consider the intersections between $Y = Q \times I$ and the white faces.
By Lemma \ref{lemma:product-in-face}, each component of intersection 
is a product rectangle $\alpha \times I$, where $\alpha$
is an arc through the interior of $Q$.
Thus $Y$ intersects the individual polyhedra in a finite number of
prisms, each of which is the product of a polygon with an interval.
Vertical faces of the prism alternate between product rectangles on
white faces and normal squares $S_i$.  Let $Y_0 = Q_0 \times I$ be the
prism whose base polygon has the greatest number of sides.  Since $Q$
is a pair of pants, and has negative Euler characteristic, $Q_0$ must
have at least six sides, half of which lie on normal squares $S_i$.
See Figure \ref{fig:pants-prism}.

\begin{figure}
\begin{overpic}{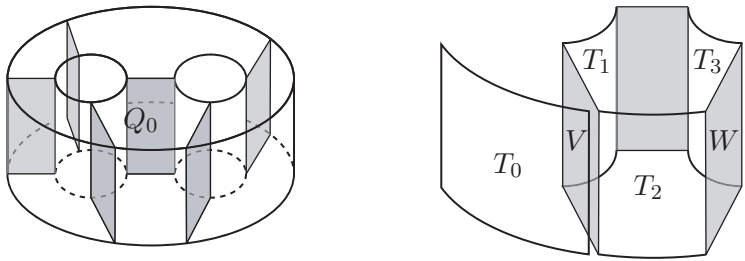}
\put(66, 12){$T_0$}
\put(78, 26){$T_1$}
\put(85, 9){$T_2$}
\put(93, 26){$T_3$}
\put(75.7, 15){$V$}
\put(95, 15){$W$}
\put(16, 18){$Q_0$}
\end{overpic}
\caption{Left: the product bundle $Q \times I$ for a pair of pants
  $Q$. Right: the prism $Q_0 \times I$.}
    \label{fig:pants-prism}
\end{figure}

Let $T_1, \dots, T_k$ denote the normal squares that bound $Y_0$,
listed in order.  By the above, $k\geq 3$.  Let $V$ be the white face
containing the rectangle of $Y_0$ between $T_1$ and $T_2$, and let $W$
be the white face containing the rectangle between $T_2$ and $T_3$.
Finally, let $T_0$ denote the normal square of $\bdy Y$ glued to $T_2$
at the white face $V$.  Thus each $T_j$ is one of the $S_i$,
relabeled.  Note that if $Y_0$ is contained in the upper polyhedron,
then so are $T_1, \dots, T_k$, but $T_0$ is contained in the lower
polyhedron.  Similarly, if $Y_0$ is in the lower polyhedron, then so
are $T_1, \dots, T_k$, but $T_0$ is in the upper polyhedron.

Using Lemma \ref{lemma:clockwise}, map all the $T_j$ to the lower
polyhedron, by the clockwise map.  If $T_0$ is in the upper polyhedron,
let $T_0'$ be its image under the clockwise map.  Otherwise, if $T_0$
is in the lower polyhedron, let $T_0' = T_0$.  Similarly, if $T_i$,
$1\leq i\leq k$ is in the upper polyhedron, let $T_i'$ be its image
under the clockwise map.  Otherwise, let $T_i' = T_i$.  Notice that
because $T_1, \dots, T_k$ are pairwise disjoint, Lemma
\ref{lemma:clockwise}(3) implies that $T_1', \dots, T_k'$ must also be
disjointly embedded in the lower polyhedron $P$.

Now, since $T_0$ is glued to $T_2$ across the white face $V$, $T_0'$
must differ from $T_2'$ by a single rotation of $V$.  Since we are
assuming that $T_0'$ and $T_2'$ cannot be parabolically compressible
at $V$, Lemma \ref{lemma:squares-int} implies $T_0'$ and $T_2'$ must
intersect, both in $V$ and in the other white face met by $T_2'$,
which recall is the face $W$.

\begin{figure}
  \input{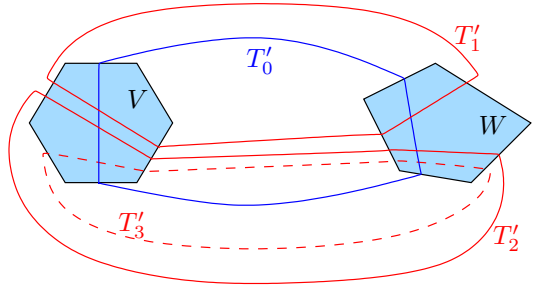}
  \caption{Normal squares in the lower polyhedron that occur in the
    proof of Proposition \ref{prop:Ibdl-parabolic}. }
  \label{fig:stripes}
\end{figure}

Now, $T_1'$ runs parallel to $T_2'$ through $V$.  Thus $T_0'$ must
also intersect $T_1'$, and again Lemma \ref{lemma:squares-int} implies
that $T_0'$ and $T_1'$ intersect in $W$.  Similarly, $T_3'$ runs
parallel to $T_2'$ through $W$, so Lemma \ref{lemma:squares-int}
implies that $T_0'$ intersects $T_3'$ in $W$ and in $V$.  But now,
$T_1'$, $T_2'$, and $T_3'$ are disjoint normal squares through $V$ and
$W$, so must be parallel, and in particular, $T_2'$ must separate
$T_1'$ and $T_3'$.  See Figure \ref{fig:stripes}.  On the other hand,
$T_1'$, $T_2'$ and $T_3'$ are all lateral faces of the same
contractible prismatic block $Y_0$.  This is a contradiction.
\end{proof}

The following proposition supplies Step 2 of the proof. We note that
the proof of the proposition is a straightforward topological argument
that doesn't use any of the machinery we have built.

\begin{prop}  \label{prop:some-epds-span}{\rm [Step 2]}
  Suppose $B$ is a non-trivial connected component of the maximal
  $I$--bundle of $M_A$, which meets the parabolic locus.  Then $B$ is
  spanned by essential product disks.
\end{prop}

\begin{proof}
Since $B$ is a $3$--dimensional submanifold of $S^3$, it is
orientable.  Hence $B$ is either the product $I$--bundle over an
orientable surface $F$, or the twisted $I$--bundle over a
non-orientable surface $F$.  In either case, $\chi(F)<0$.  Since $B$
meets the parabolic locus of $M_A$, so does $F$.
  
We may fill $F$ with disjoint edges with endpoints on the parabolic
locus of $F$, subdividing $F$ into disjoint triangles.  Now consider
the set of points lying over any such edge of the triangulation.  This
will be an essential product disk, meeting the parabolic locus of
$M_A$ at those points that lie over the endpoints of the edge, and
meeting $S_A$ elsewhere.  Remove all such essential product disks from
$B$.

In $F$, removing such arcs gives a finite collection of open disks.
The $I$--bundle over such a disk is a prism over a triangle, so we
have satisfied the definition of spanning, Definition \ref{def:span}.
\end{proof}

To complete the third and  final step of the proof of Theorem
\ref{thm:epd-span}, we need the following lemma.

\begin{lemma}  \label{lemma:trap-compress}
  Suppose $S$ is a normal disk in the upper polyhedron, glued to a
  normal disk $T$ in the lower polyhedron along a white face $W$.  If
  one of $S$ or $T$ is a normal trapezoid and the other is either a
  normal trapezoid or a normal square, then both $S$ and $T$
  parabolically compress at $W$.
\end{lemma}

\begin{proof}
For each of $S$ and $T$, define normal squares $S'$ and $T'$ in the
following way. If $S$ is a normal square, then let $S'=S$. If $S$ is a
trapezoid, then let $S'$ be the normal square obtained by pulling $S$
off the parabolic locus, and into a white face $V$. Similarly, if $T$
is a normal square, we let $T'=T$; otherwise, if $T$ is a trapezoid,
we define $T'$ to be the normal square obtained by pulling $T$ off the
parabolic locus.

Notice that the resulting squares $S'$ and $T'$ are glued to each
other at $W$.  By Proposition \ref{prop:square-compress}, if $V$ and
$W$ do not belong to the same region of $s_A \cup (\cup_i \alpha_i)$,
$S'$ is parabolically compressible at $W$.  If $S'$ parabolically
compresses at $W$, then so does $T'$ (because it is glued to $S'$ at
$W$).

Thus, we may assume that $V$ and $W$ belong to the same region of the
complement of $s_A \cup (\cup_i \alpha_i)$.  Then the entire boundary
of $S'$ can be mapped to the boundary of a square $S''$ in the lower
polyhedron containing $T'$, via the clockwise map of Definition
\ref{def:clockwise} and Lemma \ref{lemma:clockwise}.

By hypothesis, either $S$ or $T$ (or both) is a trapezoid. Thus either
$S''$ or $T'$ cuts off a single ideal vertex in a white face other
than $W$, hence $S''$ and $T$ do not intersect in any white face other
than $W$.  This means we must have conclusion \eqref{squares-int:1} of
Lemma \ref{lemma:squares-int}: $S''$ and $T'$ do not intersect at all,
and each of them cuts off an ideal vertex of $W$.  Therefore, both $S$
and $T$ are parabolically compressible at $W$.
\end{proof}

\begin{prop}  \label{prop:epds-poly} {\rm [Step 3]}
  Let $D$ be an essential product disk in $M_A$.  Then $D$
  parabolically compresses to a collection of essential product disks,
  each of which is embedded in a single polyhedron.
\end{prop}

\begin{proof}
Let $D$ be an essential product disk in $M_A$.  If $D$ is disjoint
from all white faces, then we are done: $D$ is contained in a single
polyhedron.

If $D$ meets white faces, then they split it into disks $S_1, \dots,
S_n$.  Because our polyhedra cannot contain any normal bigons (Lemma
\ref{prop:no-normal-bigons}), every arc of intersection between $D$
and a white face must run from one side of $D$ to the opposite side.
Thus $S_1$ and $S_n$ are normal trapezoids, and $S_2, \dots, S_{n-1}$
are normal squares.
Consider $S_1$ and $S_2$.  One of these is in the upper polyhedron,
and one in the lower.  They meet at a white face $W$.  Lemma
\ref{lemma:trap-compress} implies that both $S_1$ and $S_2$ must
parabolically compress at $W$.  So $D$ compresses to essential product
disks $D_1$ and $D_2$, where $D_1$ is the compressed image of $S_1$
and $D_2$ is the compressed image of $S_2\cup \dots \cup S_n$.  Repeat
this argument for the essential product disk $D_2$.  Continuing in
this manner, we see that $D$ parabolically compresses to essential
product disks, each in a single polyhedron.
\end{proof}


\section{The $\sigma$--adequate, $\sigma$--homogeneous
  setting}\label{sec:gensigmahomo}

All the results of this chapter also hold in the setting of the
ideal polyhedral decomposition for $\sigma$--adequate,
$\sigma$--homogeneous diagrams.  Here, we briefly discuss this
generalized setting.

Lemma \ref{lemma:squares} holds, and the proof requires no changes.
Hence in the characteristic $I$--bundle, we may find a product bundle
$Y=Q\times I$, where $Q$ is a pair of pants and the annuli of
$\partial Y$ are composed of embedded normal squares.

In Definition \ref{def:clockwise}, the \emph{clockwise map} was defined
to map faces of the upper polyhedron to faces of the lower.  Each of
our white faces in the $\sigma$--homogeneous case is contained in an
all--$B$ or all--$A$ polyhedral region.  In the latter case, we use
the same clockwise map as before.  As the map is identical, all the
results in that section hold.  In the all--$B$ case, rather than
mapping by one rotation in the clockwise direction, we need to map by one
rotation in the counter-clockwise direction.  However, the properties
in Lemma \ref{lemma:clockwise} will still hold in the all--$B$ case,
and the proof goes through without change.

We then have Lemma \ref{lemma:marcs}, which discusses the intersections of normal
squares in a checkerboard polyhedron. This lemma is due to Lackenby \cite[Lemma
  7]{lackenby:volume-alt}, and holds in complete generality. This
immediately implies Lemma \ref{lemma:squares-int}: two normal squares
with arcs in the same white face which differ by a single rotation,
will either each cut off a single ideal vertex in that face and not
intersect at all, or intersect nontrivially in both of their
corresponding white faces.  We also obtain Lemma
\ref{lemma:trapezoid}.

The results of Section \ref{sec:para-compress} will hold as
well.  A check through their proofs indicates that they require
the named lemmas from Chapter \ref{sec:polyhedra}, which we have shown to hold
in the $\sigma$--adequate, $\sigma$--homogeneous case.  In particular,
Proposition \ref{prop:square-compress} holds: A normal square whose
white faces are in different polyhedral regions, glued to a normal
square in a lower polyhedron, must parabolically compress.  More
particularly, it will cut off a single ideal vertex in the white face.
The proof of the proposition uses Lemmas \ref{lemma:enter-RW} and
\ref{lemma:sigma_2-connect}, as well as Lemma \ref{lemma:trapezoid},
which continue to hold.  There are two cases of the proof.  The second
uses the clockwise map.  In the case that the polyhedral region is
all--$B$, we must use the ``counter-clockwise map'' instead.
This requires reflecting the figures that illustrate the proof, but 
 the combinatorics of the situation will remain unchanged.

Finally, we step through the results of Section
\ref{subsec:pfthmepd-span}.  Lemma \ref{lemma:product-in-face}
(Product rectangle in white face) 
requires only Lemma \ref{lemma:white-incompress}, which follows
immediately from Proposition \ref{prop:no-normal-bigons-hom} (No
normal bigons) in the $\sigma$--adequate, $\sigma$--homogeneous case.
Hence it continues to hold in this setting.  Similarly, Proposition
\ref{prop:Ibdl-parabolic} holds. The proof applies verbatim, with the
sole modification that the ``clockwise map'' must be replaced by the 
``counter-clockwise map'' in an all--$B$ polyhedral region.

Proposition \ref{prop:some-epds-span} holds without change.
Lemma \ref{lemma:trap-compress} holds after replacing ``clockwise''
with ``clockwise or counter-clockwise'' in the proof. Finally, Proposition
\ref{prop:epds-poly} holds without change.

Thus all the results of this chapter hold for $\sigma$--adequate,
$\sigma$--homogeneous diagrams.  In particular, the following generalization of
Theorem \ref{thm:epd-span} 
reduces the problem of understanding the $I$--bundle 
of the characteristic submanifold of  $M_{\sigma}=S^3\cut S_{\sigma}$
to the problem of understanding
and counting EPDs in individual polyhedra.  

\begin{theorem}\label{thm:sigma-epd-span}
Let $D$ be a (connected) diagram of a link $K$, 
and let $\sigma$ be an adequate, homogeneous state of $D$.
Let $B$ be a non-trivial component of the characteristic submanifold
of $M_{\sigma} =S^3\cut S_{\sigma}$. Then $B$ is spanned by a collection of essential product
disks $D_1, \dots, D_n$, with the property that each $D_i$ is embedded
in a single polyhedron in the polyhedral decomposition of $M_{\sigma}$. \qed
\end{theorem}

\chapter{Guts and fibers}\label{sec:spanning}
This chapter contains one of the main results of the manuscript, namely a
calculation of the Euler characteristic of the guts of $M_A$  in Theorem \ref{thm:guts-general}.  The calculation will
be in terms of the number of essential product disks (EPDs) for $M_A$
which are \emph{complex}, as in Definition \ref{def:simple}, below.
In subsequent chapters, we will find bounds on the number of such EPDs
in terms of a diagram, for general and particular types of diagrams
(Chapters \ref{sec:epds}, \ref{sec:nononprime}, and
\ref{sec:montesinos}), and use this information to bound volumes, and
relate other topological information to coefficients of the colored
Jones polynomial (Chapter \ref{sec:applications}).

Recall that we have shown in Theorem \ref{thm:epd-span} that the
$I$--bundle of $M_A$ is spanned by EPDs, each of which is embedded in a single
polyhedron of the polyhedral decomposition.  (See Definitions
\ref{def:epd} and \ref{def:span} on page
\pageref{def:span} to recall the terminology.)  Thus to calculate the
Euler characteristic of the guts, we calculate the minimal number of
such a collection of spanning EPDs.  We will do this by explicitly
constructing a spanning set of EPDs with desirable properties (Lemmas
\ref{lemma:spanning-lower} and \ref{lemma:spanning-upper}).  In
Proposition \ref{prop:spanningset}, we will compute exactly how
redundant the spanning set is.  This leads to the Euler characteristic
computation in Theorem \ref{thm:guts-general}.  Along the way, we also
give a characterization of when the link complement fibers over $S^1$
with fiber the state surface $S_A$, in terms of the reduced state
graph $\GRA$, in Theorem \ref{thm:fiber-tree}.

\section{Simple and non-simple disks}

By Theorem \ref{thm:epd-span}, every non-trivial component in the
characteristic submanifold of $M_A$ is spanned by essential product
disks in individual polyhedra. Our goal is to find and count these
disks, starting with the lower polyhedra.

\begin{lemma} \label{lemma:EPDlower}
Let $D$ be an $A$--adequate diagram of a link in $S^3$.
Consider a prime polyhedral decomposition of $M_A = S^3\cut S_A$.  The
essential product disks embedded in the lower polyhedra are in
one--to--one correspondence with the 2--edge loops in the graph $\GA$.
\end{lemma}

\begin{proof}
By definition, an EPD in a lower polyhedron must run over a pair of
shaded faces $F$ and $F'$.  By Lemma \ref{lemma:nonprime-3balls} on
page \pageref{lemma:nonprime-3balls}, these shaded faces correspond to
state circles $C$ and $C'$.  Furthermore, every ideal vertex shared by
$F$ and $F'$ corresponds to a segment of $H_A$ between $C$ and $C'$,
or equivalently, to an edge of $\GA$ between $C$ and $C'$.  Since an
EPD must run over two ideal vertices between $F$ and $F'$, it
naturally defines a $2$--edge loop in $\GA$, whose vertices are the
state circles $C$ and $C'$.  In the other direction, the two edges of
a $2$--edge loop in $\GA$ define a pair of ideal vertices shared by
$F$ and $F'$, hence an EPD. Thus we have a bijection.
\end{proof}

Typically, we do not need \emph{all} the disks in the lower polyhedra
to span the $I$--bundle. We will focus on choosing disks that are as
simple as possible.

\begin{define}
\label{def:simple}
Let $P$ be a checkerboard--colored ideal polyhedron. An essential
product disk $D \subset P$ is called
\begin{enumerate}
\item \emph{simple}\index{simple EPD}\index{EPD!simple} if $D$ is the
  boundary of a regular neighborhood of a white bigon face of $P$,
\item \emph{semi-simple}\index{semi-simple EPD}\index{EPD!semi-simple}
  if $D$ parabolically compresses to a union of simple disks (but is
  not itself simple),
\item \emph{complex}\index{complex EPD}\index{EPD!complex} if $D$ is
  neither simple nor semi-simple.
\end{enumerate}
For example, in Figure \ref{fig:para-compression} on page
\pageref{fig:para-compression}, the disk on the left is semi-simple,
and the disks on the right are simple.
\end{define}

In certain special situations (for example, alternating diagrams
studied by Lackenby \cite{lackenby:volume-alt} and Montesinos diagrams
studied in Chapter \ref{sec:montesinos}), simple disks suffice to span
the $I$--bundle of $M_A$. In general, however, we may need to use
complex disks.

\begin{example}\label{example:bad}
Consider the $A$--adequate link diagram shown in Figure \ref{fig:bad},
left. The graph $H_A$ for the diagram is shown in the center of the
figure. Note that there are exactly $3$ polyhedral regions, hence
exactly $3$ lower polyhedra. In each polyhedral region, there is
exactly one 2--edge loop of $\GA$. Thus, by Lemma
\ref{lemma:EPDlower}, there are exactly $3$ EPDs in the lower
polyhedra. Two of these (in the innermost and outermost polyhedral
regions) are simple by Definition \ref{def:simple}, and may be
isotoped through bigon faces into the upper polyhedron.  However, the
green and orange\footnote{Note: For grayscale versions of this
monograph, green will refer to the darker gray shaded face, orange to
the lighter one.} shaded faces of the upper polyhedron shown in the
right panel of Figure \ref{fig:bad} meet in a total of $5$ ideal
vertices. Thus a minimum of $4$ EPDs are required to span the part of
the $I$--bundle contained in the upper polyhedron. This requires using
complex EPDs, for example the ones shown in red in Figure
\ref{fig:bad}.

\begin{figure}
\includegraphics[width=1.5in]{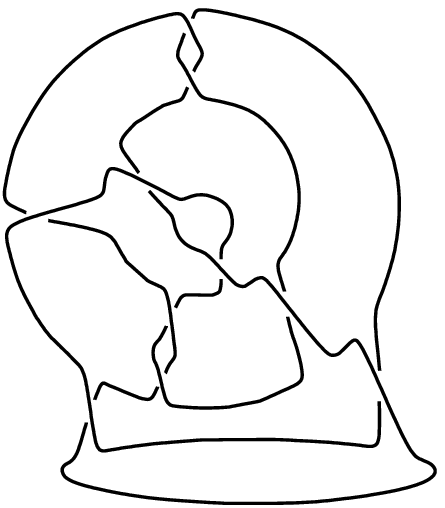}
\hspace{0.1in}
\includegraphics[width=1.5in]{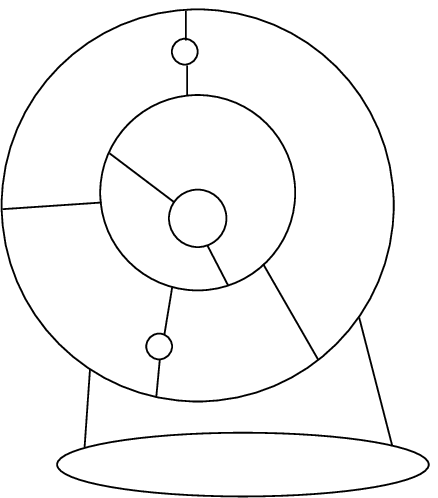}
\hspace{0.1in}
\includegraphics[width=1.5in]{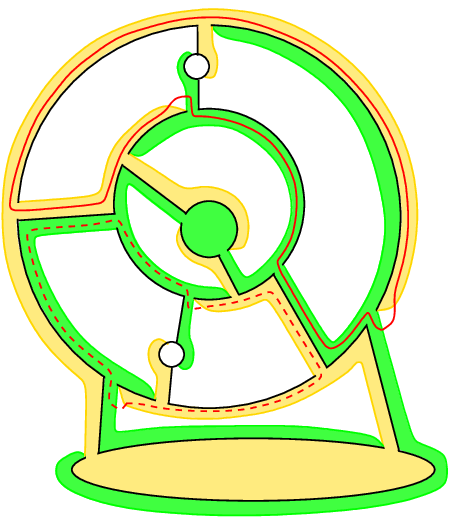}
\caption{Left: an $A$--adequate link diagram. Center: its graph
$H_A$. Right: shaded faces in the upper polyhedron, with two complex
EPDs shown in red.} 
\label{fig:bad}
\end{figure}

\begin{figure}
\includegraphics{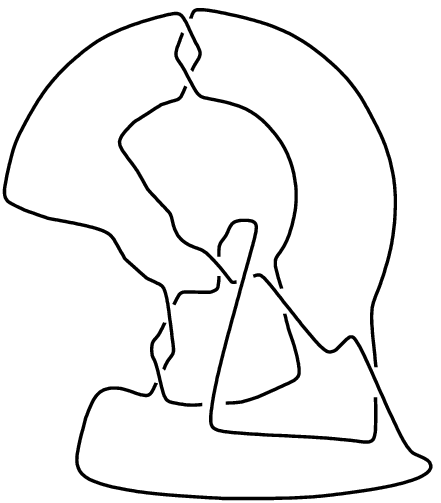}
\hspace{0.1in}
\includegraphics{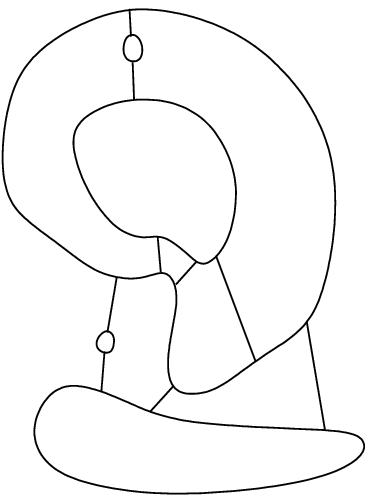}
\caption{Left: an alternate $A$--adequate diagram of the link in
Figure \ref{fig:bad}. Right: the graph $H_A$ for this diagram.} 
\label{fig:bad-fix}
\end{figure}

One feature of this example is that modifying the link diagram fixes
the problem. The modified link diagram in Figure \ref{fig:bad-fix} is
still $A$--adequate. This time, all the EPDs in the lower polyhedra
are simple or semi-simple. Furthermore, simple EPDs (isotoped across
bigon faces from the lower polyhedra into the upper) account for all
the ideal vertices in the upper polyhedron where an EPD may cross from
one shaded face into another. Thus, in the modified diagram, simple
EPDs suffice to span the $I$--bundle. This phenomenon of modifying a
diagram to remove complex EPDs is discussed again in
Chapter \ref{sec:questions}.
\end{example}

\begin{define}\label{def:pitos}
Let $D$ be an essential product disk in a polyhedron $P$. Since $P$ is
a ball, $D$ separates $P$ into two sides. We say that $D$ is
\emph{parabolically incompressible to one side}\index{parabolically incompressible to one side (\pitos)} (or \emph{\pitos} for
short) if all parabolic compression disks for $D$ lie on the same side
of $D$.

Note that simple disks, which have a bigon face to one side, are
automatically \pitos.
\end{define}

A convenient way to characterize \pitos \ disks is via the following
lemma.

\begin{lemma}\label{lemma:pitos}
Let $P$ be a checkerboard--colored ideal polyhedron, and let $F$ and
$G$ be shaded faces of $P$. Let $v_1, \ldots, v_n$ be the ideal
vertices at which $F$ meets $G$. Then
\begin{enumerate}
\item\label{item:vertex-ordering} If we label $v_1, \ldots, v_n$ such
  that the vertices are ordered consecutively around the boundary of
  $F$, for example according to a clockwise orientation on $\bdy F =
  S^1$, then $v_1, \ldots, v_n$ will also be ordered consecutively on
  $\bdy G$, but with the reverse orientation (counterclockwise).
\item\label{item:pitos-consec} With the consecutive ordering of
  \eqref{item:vertex-ordering}, an essential product disk, running
  through faces $F$ and $G$ and ideal vertices $v_i$ and $v_j$, is
  \pitos \ if and only if $j = i \pm 1 \pmod n$.
\item\label{item:pitos-simple} If $n \geq 3$ and every \pitos \ disk
  through faces $F$ and $G$ is simple, then $F$ and $G$ are the only
  shaded faces of $P$.
\end{enumerate}
\end{lemma}

\begin{proof}
We may identify $\bdy P$ with $S^2 \cong \RR^2 \cup \{ \infty \}$, in
such a way that $\infty$ falls in the interior of a white face. Then
the orientation on $\RR^2$ induces a (clockwise) orientation on the
boundary of every shaded face of $P$.

Let $n$ be the number of ideal vertices at which $F$ meets $G$. If $n
< 2$, then there are no EPDs through the pair of faces $F, G$, and the
claims of the lemma are trivial. Thus we may assume $n \geq 2$.

Order these vertices $v_1, \ldots, v_n$, clockwise around the boundary
of $F$.  For conclusion \eqref{item:vertex-ordering}, we claim that
the ideal vertices $v_1, \ldots, v_n$ are ordered counterclockwise
around the boundary of $G$.

Let $v_i$ and $v_{i+1}$ be vertices that are consecutive on $\bdy
F$. Then there is an essential product disk that runs through $F$ and
$G$, and meets exactly these ideal vertices. Let $\gamma$ be the
boundary of this disk. By the Jordan curve theorem, $\gamma$ cuts
$\RR^2$ into an inside and an outside region. Let $\alpha_F$ be the
oriented (clockwise) arc of $\bdy F$ from $v_i$ to $v_{i+1}$. Without
loss of generality, $\alpha_F$ lies inside $\gamma$.

Now, consider the portion of $G$ that lies inside $\gamma$. If this
portion of $G$ has any ideal vertices meeting $F$, they would have to
meet $F$ inside $\gamma$. But, by construction, the portion of $\bdy
F$ inside $\gamma$ is a single arc $\alpha_F$, without any additional
ideal vertices. Therefore, inside $\gamma$, $G$ has no vertices
meeting $F$. Hence, $v_i$ and $v_{i+1}$ must be consecutive from the
point of view of $G$. The orientation on the plane means that the
clockwise arc $\alpha_G \subset \bdy G$ that lies inside $\gamma$ must
run from $v_{i+1}$ to $v_i$. Thus the vertices $v_1, \ldots, v_n$ are
in counterclockwise order around $\bdy G$,
proving \eqref{item:vertex-ordering}.

\smallskip

For \eqref{item:pitos-consec}, observe that if an essential $D$ runs
through consecutive vertices $v_i$ and $v_{i+1}$, then all other
vertices shared by $F$ and $G$ are on the same side. Thus all
parabolic compressions of $D$ are on the same side, and $D$
is \pitos. Conversely, if $v_i$ and $v_j$ are not consecutive, then
there are parabolic compressions on both sides, and $D$ is not \pitos.

\smallskip

It remains to show \eqref{item:pitos-simple}.  By
Definition \ref{def:simple}, any simple disk $D$ through $F$ and $G$
is parallel to a white bigon face of $P$.  When $n \geq 3$, one
component of $P \setminus \bdy D$ contains an extra ideal vertex,
hence cannot be a bigon.  The bigon face must be on the other side,
which we call the \emph{inside} of $D$.  Thus, when $n \geq 3$ and
all \pitos \ disks through $F$ and $G$ are simple, the insides of
these disks are disjoint.

Under these hypotheses, we have mapped out the entire polyhedron $P$.
Inside each of the $n$ essential product disk is a white bigon face,
with no extra vertices.  Each essential product disk meets $F$ in an
arc and $G$ in an arc.  Thus outside all these disks, there is an
$n$--gon in $F$ containing no additional ideal vertices, meeting an
$n$--gon in $G$ containing no additional ideal vertices, where the
$n$--gons meet at their vertices.  Since there are no additional
vertices, there can be no additional faces, white or shaded.  Thus $F$
and $G$ are the only shaded faces in the polyhedron $P$.
\end{proof}

\section{Choosing a spanning set}
Next, we will construct a spanning set for the part of the $I$--bundle
that is contained in each individual polyhedron.  By
Definition \ref{def:span}, a collection of EPDs span the $I$--bundle
of $M_A$ if the complement of these disks in the $I$--bundle is a
union of prisms and solid tori.  Because our goal is to count the
Euler characteristic of the $I$--bundle, prisms and solid tori are
counted differently.  As we construct the spanning set, we will keep
track of the number of prisms created.

Recall that a lower polyhedron $P$ corresponds to a polyhedral region
of the diagram.  Let $e_A(P)$ be the number of segments of $H_A$
(equivalently, the number of edges of $\GA$) in this polyhedral
region, and $e'_A(P)$ be the number of reduced edges (after duplicates
are removed).  We may now choose a spanning set of EPDs for the
polyhedron $P$.

\begin{lemma}\label{lemma:spanning-lower}
Let $P$ be a lower polyhedron in the polyhedral decomposition of
$M_A$.  Then all the essential product disks in $P$ are spanned by a
particular spanning set $E_l(P)$\index{$E_l(P)$}, with the following properties:
\begin{enumerate}
\item\label{item:simple-belongs} Every simple disk in $P$ belongs to
  the spanning set $E_l(P)$.
\item\label{item:no-semi-simple} No disks in $E_l(P)$ are semi-simple.
  (Recall semi-simple disks are not simple by definition.)
\item\label{item:cardinality} The cardinality of $E_l(P)$ is
  $||E_l(P)|| = e_A(P) - e'_A(P) + \varepsilon_P$, where
  $\varepsilon_P$ is either $0$ or $1$.

\item\label{item:equivalences} The following are equivalent:
  \begin{enumerate}
  \item $P$ has exactly two shaded faces.
  \item All white faces of $P$ are bigons.
  \item $P \cut E_l(P)$ contains a prism over an ideal $n$--gon. 
  \item $e'_A(P) = 1$, and this single edge separates the graph $\GRA$. 
  \item $\varepsilon_P = 1$.
  \end{enumerate}
  
\end{enumerate}
\end{lemma}

\begin{proof}
We construct the spanning set as follows. For every pair of shaded
faces $F,G$ of the polyhedron $P$, let $E_{F,G}$ be the set of
all \pitos \ essential product disks that run through $F$ and $G$. If
the number of ideal vertices shared by $F$ and $G$ is $n$, then
$E_{F,G}$ will be non-empty precisely when $n \geq 2$.  By
Lemma \ref{lemma:pitos}, these disks are in 1--1 correspondence with
consecutive pairs of vertices $v_i, v_{i+1}$ shared by $F$ and $G$.

Now, we consider two cases.

\smallskip

\underline{Case 1:} $F$ and $G$ are the only shaded faces in $P$. In this
case, we let $E_l(P) = E_{F,G}$. 

Let us check the conclusions of the lemma.  Note that every simple
disk is \pitos, hence must belong to $E_l(P) = E_{F,G}$.  Conversely,
every disk in $E_{F,G}$ is \pitos, hence contains no vertices between
$F$ and $G$ on one side, hence contains no vertices at all on that
side, and thus can only be simple.  

Recall that the shaded faces $F$ and $G$ correspond to state circles
$C_F$ and $C_G$. Since these are the only shaded faces of $P$, then
$C_F$ and $C_G$ are the only state circles in the polyhedral region of
$P$.  All the edges of $\GA$ in the polyhedral region must connect
$C_F$ and $C_G$; there are $n = e_A(P)$ such edges total.  In the
reduced graph $\GRA$, these $n$ edges are identified to one, hence
$e'_A(P) = 1$.  Thus
$$||E_l(P)|| = n = e_A(P) - e'_A(P) + \varepsilon_P, \quad \mbox{where} \quad \varepsilon_P = 1.$$

In Case 1, all the conditions of \eqref{item:equivalences} will be
true.  The polyhedron $P$ has exactly two shaded faces $F$ and $G$,
and all the white faces are bigons parallel to the simple disks.
Cutting $P$ along the disks of $E_l(P)$ produces a prism over an ideal
$n$--gon.  We have already seen that $\varepsilon_P = e'_A(P) = 1$.
Finally, because every state circle in $S^2$ is separating, any path
in $H_A$ between state circles $C_F$ and $C_G$ must pass through the
polyhedral region of $P$, hence must use one of the $n$ edges that are
identified to one in $\GRA$.

\smallskip
\underline{Case 2:} $F$ and $G$ are not the only shaded faces in
$P$.

When $F$ and $G$ share $n$ vertices with $n\geq 2$, we will see that
we may remove one of the $n$ disks in $E_{F,G}$ to obtain a set
$E'_{F,G}$ of $(n-1)$ disks, which still span all the EPDs through
faces $F$ and $G$.  We make the choices as follows.  If $n=2$, then
the two disks $E_{F,G}$ both run through vertices $v_1$ and $v_2$, and
are parallel. So we may omit one.  If $n \geq 3$, then
Lemma \ref{lemma:pitos} implies that one of the disks in $E_{F,G}$ is
non-simple.  Thus we omit a non-simple disk.  Note that by
construction, one copy of each simple disk through faces $F$ and $G$
remains in $E'_{F,G}$.  Note further that there is a prism between all
disks of $E'_{F,G}$, so the removed disk is spanned by the remaining
ones.  Since all EPDs through faces $F$ and $G$ are spanned
by \pitos \ ones, the remaining set of $(n-1)$ disks spans all EPDs
through $F$ and $G$, as claimed.

When $F$ and $G$ share less than $2$ ideal vertices, $E_{F,G}$ is
empty, and for notational convenience we set $E'_{F,G}$ to be empty.  

Let us check that in the non-trivial cases, all disks in $E'_{F,G}$
satisfy conclusion \eqref{item:no-semi-simple}: that is, none of them
is semi-simple.  If $n=2$, there is nothing to check, because no
parabolic compressions are possible.  Thus, suppose that $n \geq 3$,
and we obtained $E'_{F,G}$ by omitting a non-simple disk in $E_{F,G}$.

Suppose for a contradiction that $D \in E'_{F,G}$ is not simple, but
parabolically compresses to simple disks.  Because the ideal vertices
$v_i, v_{i+1}$ met by $D$ are consecutive, the parabolic compression
must be to the \emph{outside} of $D$ --- that is, away from the arc of
$\bdy F$ that runs from $v_i$ to $v_{i+1}$. Any simple disks to which
$D$ compresses must be \pitos, hence belong to $E_{F,G}$. But one of
the disks to the outside of $D$ is not simple, contradicting the
hypothesis that $D$ compresses to simple disks to its outside.

Now, let $E_l(P)$ be the union of all the sets $E'_{F,G}$, as $(F,G)$
ranges over all unordered pairs of shaded faces in $P$.  We have
already checked that this spanning set satisfies
conclusions \eqref{item:simple-belongs}
and \eqref{item:no-semi-simple}.

Observe that the shaded faces $F$ and $G$ correspond to state circles
$C_F$ and $C_G$. If there are two or more edges of $\GA$ connecting
$C_F$ to $C_G$ (i.e., if $n \geq 2$), then some of these edges will be
removed as we pass to the reduced graph $\GRA$.  The number of edges
removed is exactly $n-1$, which is equal to the cardinality of
$E'_{F,G}$.  Thus
$$||E_l(P)|| = \sum_{(F,G)} || E'_{F,G}|| = 
e_A(P) - e'_A(P)
+ \varepsilon_P, \quad \mbox{where} \quad \varepsilon_P = 0.$$ 
The sum is over unordered pairs of shaded faces $(F,G)$.

In Case 2, all the conditions of \eqref{item:equivalences} will be
false.  By hypothesis, the polyhedron $P$ has more than two shaded
faces, hence some white face is not a bigon.  Since the polyhedral
region of $P$ has more than two state circles and is connected, this
region must contain more than one edge of $\GRA$.  The one non-trivial
statement in \eqref{item:equivalences} is that $P \cut E_l(P)$ cannot
contain a prism.

Suppose, for a contradiction, that $P \cut E_l(P)$ contains a prism
over an $n$--gon.  Then the top and bottom faces of this prism are on
shaded faces $F$ and $G$ of $P$, and the lateral faces are EPDs in
$E_l(P)$.  By construction, these lateral faces must belong to
$E'_{F,G}$.  But then one of these lateral EPDs must parabolically
compress to the remaining $(n-1)$ EPDs --- which is impossible, since
we removed the one redundancy in $E_{F,G}$ when constructing the set
$E'_{F,G}$.
\end{proof}

Lemma \ref{lemma:spanning-lower} has the following immediate
consequence.

\begin{lemma}\label{lemma:lower-combined}
Let $E_l$\index{$E_l$} be the union of all the spanning sets $E_l(P)$,
as $P$ ranges over all the lower polyhedra. Then every essential
product disk in one of the lower polyhedra is spanned by the disks in
$E_l$. The set $E_l$ contains all simple disks in the lower
polyhedra. Furthermore,
$$|| E_l || = e_A - e'_A + \nsep,$$
where $\nsep$\index{$n_{\rm sep}$} is the number of separating edges in
$\GRA$, and $\nsep$ is also equal to the number of prisms in the lower
polyhedra in the complement of $E_l$.
\end{lemma} 

\begin{proof}
The properties that $E_l$ contains all simple disks and spans all the
EPDs in the lower polyhedra follow immediately from the same
properties of the constituent sets $E_l(P)$. To compute the
cardinality of $E_l$, it suffices to observe that the total number of
edges removed as we pass from $\GA$ to $\GRA$ is
$$e_A - e'_A = \sum_P e_A(P) - e'_A(P),$$
and the the total
contribution of the terms $\varepsilon_P$ in
Lemma \ref{lemma:spanning-lower} is exactly $\nsep$. By
Lemma \ref{lemma:spanning-lower}, these $\nsep$ edges are in one-to-one
correspondence with prisms in the lower polyhedra in the complement of
$E_l$.
\end{proof}

Finally, we choose a spanning set of EPDs for the upper polyhedron.

\begin{lemma}\label{lemma:spanning-upper}
Let $P$ denote the upper polyhedron in the decomposition of $M_A$.  Then
there exists a set $E_s \cup E_c$ of essential products disks embedded
in $P$, such that the following hold:
\begin{enumerate}
\item\label{item:all-simple} $E_s$\index{$E_s$} is the set of all
  simple disks in $P$.
\item\label{item:min-complex} $E_c$\index{$E_c$} consists of complex
  disks. Furthermore, $E_c$ is minimal, in the sense that no disk in
  $E_c$ parabolically compresses to a subcollection of $E_s \cup E_c$.
\item\label{item:sc-spans} The set $E_s \cup E_c$ spans the essential
  product disks in $P$.
\item\label{item:top-prism} The following are equivalent:
  \begin{enumerate}
  \item \label{ii:tree} $\GRA$ is a tree.
  \item \label{ii:bigons} Every white face is a bigon.
  \item  \label{ii:prism} $P \cut (E_s \cup E_c)$ contains exactly one prism.
  \item \label{ii:all-prisms} Every (upper or lower) polyhedron is a
    prism, with horizontal faces shaded and lateral faces white.
  \end{enumerate}
%

\end{enumerate}
\end{lemma}

\begin{proof}
The construction is identical to the construction in Lemma
\ref{lemma:spanning-lower}. For every pair of shaded faces $F$ and
$G$, we let $E_{F,G}$ be the set of all \pitos\ disks that run through
$F$ and $G$. If $F$ and $G$ are the only shaded faces of the upper
polyhedron $P$, we let $E_s = E_{F,G}$. In this case, all white faces
of $P$ are bigons, and all disks in $E_{F,G}$ are simple. Hence, $E_c
= \emptyset$.

Alternately, if $F$ and $G$ are not the only shaded faces of $P$, we
proceed as in Case 2 of Lemma \ref{lemma:spanning-lower}. We prune the
set $E_{F,G}$ by one disk, while keeping all simple disks, to obtain
$E'_{F,G}$. As in the proof of Lemma \ref{lemma:spanning-lower}, no
disk in $E'_{F,G}$ is semi-simple. Then, we let $E_s \cup E_c$ be the
union of all sets $E'_{F,G}$ as $F,G$ range over the shaded faces of the
polyhedron $P$. This combined set is composed of simple disks in $E_s$
and complex disks in $E_c$.

In either case, we have constructed a set $E_s \cup E_c$ that
satisfies conclusions \eqref{item:all-simple} and
\eqref{item:sc-spans}.

To prove \eqref{item:min-complex}, observe that by construction, each
disk in $E_c$ is complex and \pitos.  Suppose, for a contradiction,
that some disk $D \in E_c$ parabolically compresses to other disks in
$E_s \cup E_c$. Then, $D$ would need to compress to the remaining
$(n-1)$ \pitos \ disks that share the same shaded faces $F$ and $G$
(where $n$ is the number of vertices at which $F$ and $G$ meet). But
by construction, the only scenario in which all $n$ disks of $E_{F,G}$
remain in $E_s \cup E_c$ is when all of these disks are simple, hence
$E_c = \emptyset$, which is a contradiction.

\smallskip 

It remains to prove the equivalent conditions of \eqref{item:top-prism}.

$\eqref{ii:tree} \Leftrightarrow \eqref{ii:bigons}$: The connected
graph $\GRA$ is a tree if and only if every edge separates. Hence,
this equivalence is immediate from Lemma
\ref{lemma:spanning-lower}\eqref{item:equivalences}.

\smallskip

$\eqref{ii:bigons} \Rightarrow \eqref{ii:all-prisms}$: Let $P_0$ be
any polyhedron in the decomposition, and suppose that every white face
of $P_0$ is a bigon. Then the white faces of $P_0$ must line up
cyclically end to end, and there are exactly two shaded faces. Since a
bigon shaded face is the product of an ideal edge with $I$, this
product structure extends over the entire polyhedron. Thus $P_0$ is a
prism whose horizontal faces are shaded and whose lateral faces are
white bigons.

\smallskip

$\eqref{ii:all-prisms} \Rightarrow \eqref{ii:prism}$: Suppose the top
polyhedron $P$ is a prism, whose lateral faces are white
bigons. Parallel to every bigon face of $P$ is a simple essential
product disk, and by property \eqref{item:all-simple}, each of these
simple EPDs is in the spanning set $E_s$. Thus $P \cut E_s$ consists
of a product region parallel to each white face, as well as a prism
component separated from all white faces.

 
\smallskip

$\eqref{ii:prism} \Rightarrow \eqref{ii:bigons}$: Suppose that $P \cut
(E_s \cup E_c)$ contains a prism over an $n$--gon. Then the top and
bottom faces of this prism are on shaded faces $F$ and $G$ of $P$, and
the lateral faces are EPDs in $E_{F,G}$. Notice that one of these
lateral EPDs must parabolically compress to the remaining $(n-1)$
EPDs. This would be impossible if we removed one of the $n$ disks in
$E_{F,G}$ while passing to the reduced set $E'_{F,G}$. Thus every disk
of $E_{F,G}$ must belong to $E_s \cup E_c$, which means that $F$ and
$G$ are the only shaded faces in polyhedron $P$. Therefore, every
white face of $P$ is a bigon.  But since every white face of a lower
polyhedron is glued to $P$, all white faces must be bigons.
\end{proof}
 
\begin{define}\label{def:ec}\index{$E_c$}
The spanning set $E_c$ is defined in the statement of Lemma \ref{lemma:spanning-upper}. Equivalently, we may define its cardinality $||E_c||$ as follows: $||E_c||$ is the smallest number of complex disks required to span the $I$--bundle of the upper polyhedron. 

The equivalence of this alternative definition is proved in Lemma \ref{lemma:spanning-upper}\eqref{item:min-complex}.
\end{define}

Since the polyhedral decomposition is uniquely specified by the diagram $D(K)$ (see Chapter \ref{sec:decomp} and Remark \ref{rem:prime-uniqueness}), $||E_c||$ is a diagrammatic quantity, albeit one that is not easy to eyeball. In Chapter \ref{sec:nononprime}, we will bound the quantity $||E_c||$ in terms of simpler diagrammatic quantities, and in Chapter \ref{sec:montesinos}, we will prove that for most Montesinos links, $||E_c|| = 0$.

We also record the following property of the spanning set $E_s \cup
E_c$, which will be needed in Chapter \ref{sec:nononprime}.

\begin{lemma}\label{lemma:ec-tentacle}
Let $F$ and $G$ be shaded faces of the upper polyhedron $P$, and let
$E_s \cup E_c$ be the spanning set of Lemma \ref{lemma:spanning-upper}.
Then, for every tentacle of $F$, at most two disks of $E_s \cup E_c$
run through $F$ and $G$ and intersect this tentacle.
\end{lemma}

\begin{proof}
Let $\alpha$ be an arc that
cuts across a tentacle of $F$. Thus, from the point of view of $P$,
$\alpha$ is an arc from an ideal vertex $w$ to a point $x$ in the
interior of a side of $P$. 

Let $v_1, \ldots, v_n$ be the ideal vertices shared by $F$ and $G$,
labeled in order, as in
Lemma \ref{lemma:pitos}\eqref{item:vertex-ordering}.  The point
$x \in \bdy F \cap \alpha$ falls between a consecutive pair of
vertices $v_i$ that connect $F$ to $G$. Then no generality is lost in
assuming that $x$ lies on the oriented arc from $v_n$ to $v_1$.

Recall that all disks in $E_s \cup E_c$ are \pitos, and that by Lemma
\ref{lemma:pitos}, \pitos\ disks through $F$ and $G$ must meet
consecutive ideal vertices. The proof will be complete once we show
that $\alpha$ can only meet at most two such disks (up to isotopy).

If $w$ is not one of the vertices at which $F$ meets $G$, then it lies
between vertices $v_i$ and $v_{i+1}$. In this case, $\alpha$
partitions $\{v_1, \ldots, v_n\}$ into two subsets: namely, $\{v_1,
\ldots, v_i\}$ and $\{v_{i+1}, \ldots, v_n\}$. Any disk through $F$
and $G$ whose vertices belong to the same subset will be disjoint from
$\alpha$ (up to isotopy). Thus $\alpha$ can only intersect the two
disks that run from $v_i$ to $v_{i+1}$ and from $v_n$ to $v_1$.

If $w$ is one of the vertices $v_i$, then the argument is the same. In
this case, $\alpha$ can only intersect the disk that runs from $v_n$
to $v_1$.
\end{proof}

\section{Detecting fibers}

In this section, we prove that the Euler characteristic $\chi(\GRA)$
detects whether $S_A$ is a fiber. See also Corollary
\ref{cor:beta-fiber} on page \pageref{cor:beta-fiber}.

\begin{theorem}\label{thm:fiber-tree}\index{fiber!detected by reduced graph $\GRA$}\index{$\GRA$, $\GRB$: reduced state graph!detects fiber surfaces}
Let $D(K)$ be any link diagram, and let $S_A$ be the spanning surface
determined by the all--$A$ state of this diagram. Then the following
are equivalent:
\begin{enumerate}
\item\label{item:tree} The reduced graph $\GRA$ is a tree.
\item\label{item:fiber} $S^3 \setminus K$ fibers over $S^1$, with fiber $S_A$.
\item\label{item:semi-fiber} $M_A = S^3 \cut S_A$ is an $I$--bundle over $S_A$.
\end{enumerate}
\end{theorem}

We would like to emphasize that the theorem applies to \emph{all}
diagrams. It turns out that each of \eqref{item:tree},
\eqref{item:fiber}, and \eqref{item:semi-fiber} implies that $D$ is
connected and $A$--adequate. The point of including condition
\eqref{item:semi-fiber} is that $S_A$ is never a \emph{semi-fiber}:
that is, $S_A$ cannot be a non-orientable surface that lifts to a
fiber in a double cover of $S^3 \setminus K$.

\begin{proof}
For $\eqref{item:tree} \Rightarrow \eqref{item:fiber}$, suppose that
$\GRA$ is a tree. Then $D$ must be connected because $\GRA$ is
connected. Also, since $\GRA$ contains no loops, $\GA$ must contain no
$1$--edge loops, hence $D$ is $A$--adequate. In particular, we have a
polyhedral decomposition of $M_A = S^3 \cut S_A$, and all the results
of the previous chapters apply to this polyhedral decomposition.

Since $\GRA$ is a tree, Lemma
\ref{lemma:spanning-upper}\eqref{item:top-prism} implies that every
polyhedron of the polyhedral decomposition is a prism, and every white
face is a bigon. Observe that a prism is an $I$--bundle over its base
polygon, and a white bigon face is also an $I$--bundle with the same
product structure. Thus the $I$--bundle structures of the individual
polyhedra can be glued along the bigon faces to obtain an $I$--bundle
structure on all of $M_A$.

Finally, since $\GRA$ is a tree, it is bipartite, hence $\GA$ is also
bipartite. Thus, by Lemma \ref{lemma:sa-orientable} on page
\pageref{lemma:sa-orientable}, $S_A$ is orientable. Since $S_A$ is an
orientable surface whose complement is an $I$--bundle, it must be a
fiber in a fibration over $S^1$.

\smallskip

The implication $\eqref{item:fiber} \Rightarrow
\eqref{item:semi-fiber}$ is trivial.

\smallskip

For $\eqref{item:semi-fiber} \Rightarrow \eqref{item:tree}$, suppose
that $M_A$ is an $I$--bundle over $S_A$. Thus, in particular, $S_A$ is
connected, hence $D$ is connected. Also, $S_A$ must be essential
in $S^3 \setminus K$. Thus, by Theorem
\ref{thm:incompress}, $D$ is $A$--adequate, and all of our polyhedral
techniques apply.

All white faces of the polyhedral decomposition are contained in
$M_A$. Thus, by Lemma \ref{lemma:product-in-face} (Product rectangle in white face), each white face is a product $\alpha
\times I$, where $\alpha \times \{0, 1\}$ are ideal edges. In other
words, every white face is a bigon. Thus, by Lemma
\ref{lemma:spanning-upper}, $\GRA$ is a tree.
\end{proof}

\begin{remark}
We have seen in Lemma \ref{lemma:nonprime-3balls} that each polyhedral
region corresponds to an alternating link diagram, whose all--$A$
surface is a checkerboard surface. Ozawa has observed that the state
surface $S_A$ is a Murasugi sum of these individual checkerboard
surfaces \cite{ozawa}; this was the basis of his proof that $S_A$ 
is essential (Theorem \ref{thm:incompress}). Now, a
theorem of Gabai \cite{gabai:natural, gabai:fibered} states that the
Murasugi sum of several surfaces is a fiber if and only if the
individual summands are fibers. Thus an alternate proof of Theorem
\ref{thm:fiber-tree} would argue by induction: here, the base case is
that of prime, alternating diagrams, and Gabai's result gives the
inductive step.

Restricted to prime, alternating diagrams, Theorem
\ref{thm:fiber-tree} says that the checkerboard surface $S_A$ is a
fiber if and only if $D(K)$ is a negative $2$--braid. (We are using
the convention that positive braid generators are as depicted in
Figure \ref{fig:3braid-generators} on page
\pageref{fig:3braid-generators}.)  This special case follows quickly
from a theorem of Adams \cite[Theorem 1.9]{adams:quasi-fuchsian}, and
can also be proved by applying Lemma \ref{lemma:product-in-face}  
 to Menasco's polyhedral
decomposition of alternating link complements
\cite{menasco:polyhedra}.

In fact, this line of argument extends to give a version of Theorem \ref{thm:fiber-tree} for state surfaces of $\sigma$--homogeneous states. See the recent paper by Futer \cite{futer} for a proof from this point of view.
\end{remark}

\section{Computing the guts}\label{sec:compute-guts}

To compute the guts of $M_A = S^3 \cut
S_A$, it suffices to take the spanning sets of the previous section,
count the EPDs in the spanning sets, and also count how many prisms
will occur in the complement of these disks. The counts work as
follows. 

\begin{prop}\label{prop:spanningset}
Every non-trivial component of the the characteristic submanifold of
$M_A$ is spanned by a collection $E_l \cup E_c$ of essential product
disks, such that
\begin{enumerate}
\item\label{item:el-count} The disks of $E_l$\index{$E_l$} are
  embedded in lower polyhedra, and $||E_l || = e_A - e'_A + \nsep$,
  where $\nsep$\index{$n_{\rm sep}$} is the number of separating edges in
  $\GRA$.

\item\label{item:ec-count} The disks of $E_c$\index{$E_c$} are
  embedded in the upper polyhedron. All these disks are
  complex. Furthermore, no disk in $E_c$ parabolically compresses to
  bigon faces and other disks in $E_c$.

\item\label{item:prism-count} After the characteristic submanifold is
  cut along $E_l \cup E_c$, the total number of prism pieces will be
  $\nsep + \chi_+(\GRA)$, where $\nsep$ is the number of separating
  edges in $\GRA$ and $\chi_+(\GRA) = \max \{ 0, \chi(\GRA) \}$ equals
  $1$ if $\GRA$ is a tree and $0$ otherwise.

\end{enumerate}
\end{prop}

\begin{proof}
By Theorem \ref{thm:epd-span}, every non-trivial component of the
characteristic submanifold is spanned by EPDs in individual
polyhedra. We have constructed spanning sets for the individual
polyhedra in Lemmas \ref{lemma:spanning-lower} and
\ref{lemma:spanning-upper}; these sets are denoted $E_l$ in the lower
polyhedra and $E_s \cup E_c$ in the upper polyhedron. Note that by
construction, every white bigon face in the polyhedral decomposition
has a disk in $E_l$ parallel to it, as well as a disk in $E_s$
parallel to it. We do not need both of these parallel disks to span
the characteristic submanifold. Thus we may discard $E_s$, and
conclude that $E_l \cup E_c$ spans the characteristic submanifold.

Conclusion \eqref{item:el-count}, which counts the cardinality of
$E_l$, is a restatement of Lemma \ref{lemma:lower-combined}.
Conclusion \eqref{item:ec-count} is a restatement of Lemma
\ref{lemma:spanning-upper}\eqref{item:min-complex}.


It remains to count the prism components cut off by $E_l \cup E_c$.
Recall that every white bigon face in the polyhedral decomposition has
a disk in $E_l$ parallel to it, as well as a disk in $E_s$ parallel to
it. Thus every prism cut off by $E_l \cup E_c$ is isotopic (through
white bigon faces) to a prism cut off by $E_l \cup (E_c \cup E_s)$. By
Lemma \ref{lemma:lower-combined}, the number of these prisms in the
lower polyhedra is equal to $\nsep$. By Lemma
\ref{lemma:spanning-upper}\eqref{item:top-prism}, the number of these
prisms in the upper polyhedron is $0$ or $1$, and is equal to
$\chi_+(\GRA)$. Thus the proof will be complete once we show that
every prism cut off by $E_l \cup E_c$ is isotopic into a single
polyhedron.

\begin{figure}
\psfrag{W}{$W$}
\psfrag{b}{$\bdy M_A$}
\psfrag{e}{EPD}
\includegraphics{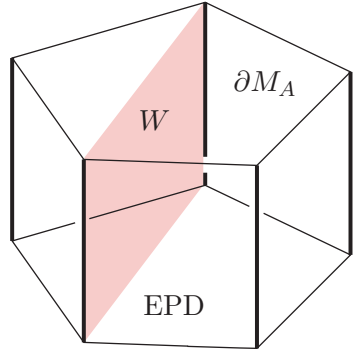}
\caption{A prism $R$, whose lateral faces are EPDs in the spanning
set.  The parabolic locus is in bold.  Any white face $W$ that
intersects $R$ must respect the product structure of $R$, hence is a
bigon face.}
\label{fig:ideal-prism}
\end{figure}

Let $R$ be a prism over an $n$--gon, cut off by $E_l \cup
E_c$. Suppose that there is a white face $W$ of the polyhedral
decomposition that intersects $R$ (otherwise we are done).  By Lemma
\ref{lemma:product-in-face}  (Product rectangle in white face), 
each component of $R \cap W$ is a product
rectangle $\alpha \times I$, whose top and bottom sides $\alpha \times
\{0,1 \}$ are sub-arcs of edges of $W$. But by construction, each
lateral face of $R$ is an EPD belonging to $E_l \cup E_c$, hence lies
in a single polyhedron and is disjoint from $W$. Thus $\alpha \times
\{0,1\}$ must be disjoint from the lateral EPDs, and must run from the
parabolic locus to the parabolic locus.  In other words, $\alpha
\times I$ fills up the entirety of the white face $W$, hence $W$ is a
bigon.  See Figure \ref{fig:ideal-prism}.

Recall that by Lemma \ref{lemma:lower-combined}, every simple disk in
the lower polyhedra belongs to $E_l$. Thus $W$ is parallel to a disk
of $E_l$, hence to a lateral face of the prism $R$. By isotoping $R$
through this white face $W$, we move a lateral face of $R$ from a disk
of $E_l$ to a parallel disk of $E_s$, while removing a component of
intersection with the white faces. Continuing inductively in this
fashion, we conclude that if $R$ was not already in a single
polyhedron, it can be isotoped into the top polyhedron. Thus $R$ was
already accounted for in the count of $\nsep + \chi_+(\GRA)$ prisms,
and the proof is complete.
\end{proof}

We can now prove the main theorem.
   
\begin{theorem} \label{thm:guts-general}\index{guts!in terms of $\chi(\GRA)$}\index{$\GRA$, $\GRB$: reduced state graph!relation to guts}\index{$E_c$!role in computing guts}
Let $D(K)$ be an $A$--adequate diagram, and let $S_A$ be the
essential spanning surface determined by this diagram. Then
$$\negeul( \guts(S^3 \cut S_A))= \negeul (\GRA) - || E_c||,$$
where $\negeul(\cdot)$ is the negative Euler characteristic as in
Definition \ref{def:neg-euler}, and where $|| E_c||$ is the smallest number of complex disks required to span the $I$--bundle of the upper polyhedron, as in Definition \ref{def:ec}.


In particular, if every essential product disk in the upper polyhedron
is simple or semi-simple, then
  $$\negeul( \guts(S^3 \cut S_A))= \negeul(\GRA).$$
\end{theorem}

\begin{proof}
Recall that the graph $\GA$ embeds as a spine for the surface
$S_A$. Thus, by Alexander duality, $M_A = S^3 \cut S_A$ has Euler
characteristic
\begin{equation}\label{eq:alexander}
\chi( M_A) \: = \:  \chi(S_A) \: = \: \chi(\GA).
\end{equation}

Recall that $M_A =
\guts(M_A) \cup CS(M_A)$, where $CS(M_A)$ is the characteristic submanifold of $M_A$, and the intersection $\guts(M_A) \cap
CS(M_A)$ consists of annuli. Thus their Euler characteristics sum to
the Euler characteristic of $M_A$.  Furthermore, by Lemma
\ref{lemma:Ibdl}, all the trivial components of the $I$--bundle are
solid tori glued along annuli, which do not contribute to the Euler
characteristic count. Therefore, if $B$ denotes the maximal
$I$--bundle in the characteristic submanifold,
\begin{equation}\label{eq:euler-submanifold}
\chi(\GA) \: = \:  \chi(M_A) \: = \: \chi(\guts(M_A)) + \chi(B).
\end{equation}  

By Proposition \ref{prop:spanningset}, the maximal $I$--bundle $B$ is
spanned by a collection $E_l \cup E_c$ of essential product
disks. Notice that cutting $B$ along a disk increases its Euler
characteristic by $1$. By Definition \ref{def:span}, we know that $B
\cut (E_l \cup E_c)$ consists of solid tori (Euler characteristic $0$)
and prisms (Euler characteristic $1$). Thus, by Proposition
\ref{prop:spanningset},
\begin{equation}\label{eq:euler-ibundle}
\begin{array}{r c l c l}
\chi(B)
& = & - || E_l \cup E_c || &+& (\mbox{number of prisms}) \\
& = & - (e_A - e'_A) - \nsep - || E_c || &+& (\nsep + \chi_+(\GRA)) \\
& = & \chi(\GA) - \chi(\GRA) - || E_c || &+& \chi_+(\GRA) \\
& =  &   \chi(\GA) + \negeul(\GRA) - || E_c ||. & & 
\end{array}
\end{equation}

Since every component of $\guts(M_A)$ is bounded by a hyperbolic
surface, we have $\negeul(\guts(M_A)) = - \chi(\guts(M_A))$. Thus
plugging the conclusion of \eqref{eq:euler-ibundle} into
equation \eqref{eq:euler-submanifold} gives
$$\negeul(\guts(M_A)) \: = \: - \chi(\guts(M_A)) \: = \: \negeul(\GRA) - || E_c ||,$$
which completes the proof.
\end{proof}

\begin{remark}\label{negative}
The reliance on the notation $\negeul(\cdot)$ in Theorem
\ref{thm:guts-general} is only necessary in the special case when
$\GRA$ is a tree and $\guts(M_A)$ is empty. On the other hand, when
$\guts(M_A) \neq \emptyset$, every component of it will have negative
Euler characteristic. Thus, when $\guts(M_A) \neq \emptyset$, the
conclusion of the theorem can be rephrased as
       $$\chi( \guts(M_A))= \chi (\GRA) - || E_c|| < 0.$$
\end{remark}
 \smallskip
 
\section{Modifications of the diagram}\label{sec:mod-diagram}

In this section we use Theorem \ref{thm:guts-general} to study the
effect on the guts of two well know link diagrammatic moves:
adding/removing crossings to a twist region, and taking planar cables.
Lemma \ref{lemma:remove-bigon} characterizes the effect on the guts of
adding/removing crossings to a twist region.  Planar cables, which are
of particular significance to us since they are used in the
calculation of the colored Jones polynomials \cite{lickorish:book},
are discussed in Corollary \ref{cor:cable}.

\begin{define}\label{def:long-short}
Let $R$ be a twist region of the diagram $D$, and suppose that $R$
contains $c_R > 1$ crossings. Consider the all--$A$ and all--$B$
resolutions of $R$.  One of the graphs associated to $D$, say $\GB$,
will inherit $c_R -1$ vertices from the $c_R -1$ bigons contained in
$R$. We say that this is the \emph{long resolution}\index{long resolution} of the twist 
region $R$. The other graph, say $\GA$, contains $c_R$ parallel edges
(only one of which survives in $\GRA$). This is the \emph{short
  resolution} of $R$. \index{short resolution}\index{twist region!long and short resolutions} See Figure \ref{fig:twist-resolutions}.
  
We say that the twist region $R$ is an \emph{$A$--region}\index{$A$--region} if the
all--$A$ resolution is the short resolution of $R$. In other words,
$R$ is an $A$--region if it contributes exactly one edge to $\GRA$.
\end{define}

\begin{figure}[ht]
\begin{center}
\input{figures/contribution.pstex_t}
\end{center}
\caption{Resolutions of a twist region $R$.
  \index{long resolution!example}\index{short resolution!example}
  This twist region is an $A$--region\index{$A$--region!example},
  because the all--$A$ resolution is short.}
\label{fig:twist-resolutions}
\end{figure}

\begin{lemma}\label{lemma:remove-bigon}
Let $D$ be an $A$--adequate link diagram, with spanning surface
$S_A(D)$ and the associated prime polyhedral decomposition of $S^3
\cut S_A(D)$. Let $\widehat{D}$ be the $A$--adequate diagram obtained
by removing one crossing in an $A$--region of $D$. (Note that this
operation very likely changes the link type.)

Then the effect of removing one crossing from an $A$--region is as
follows:
\begin{enumerate}
\item\label{item:GRA-isomorphic} The reduced graphs $\GRA(D)$ and
  $\GRA(\widehat{D})$ are isomorphic.
\item\label{item:EC-isomorphic} In the upper polyhedra of the
  respective diagrams, the spanning sets $E_c(D)$ and
  $E_c(\widehat{D})$ have the same cardinality.
\item\label{item:guts-isomorphic} The complements of the spanning
  surfaces $S_A(D)$ and $S_A(\widehat{D})$ have the same guts:
$$\negeul \,  \guts(S^3 \cut S_A(D)) = \negeul \, \guts(S^3 \cut
  S_A(\widehat{D})).$$
\end{enumerate}
\end{lemma}

\begin{proof}
Following Definition \ref{def:long-short}, let $R$ be a twist region
of the diagram $D$ which has at least two crossings, and in which the
all--$A$ resolution is short. Then, removing one crossing from twist
region $R$ amounts to removing one segment from the graph
$H_A(D)$. All the state circles are unaffected, and the other segments
of $H_A$ are also unaffected. See Figure \ref{fig:remove-crossing}.

Recall that the vertices of $\GA$ are the state circles of $H_A$, and
the edges of $\GA$ are the segments of $H_A$. Thus the graphs $\GA(D)$
and $\GA(\widehat{D})$ have exactly the same vertex set, with $\GA(D)$
having one more edge in the short resolution of the twist region
$R$. Because duplicate edges of $\GA$ are identified together in
$\GRA$, the two reduced graphs $\GRA(D)$ and $\GRA(\widehat{D})$ are
isomorphic, proving \eqref{item:GRA-isomorphic}.

\smallskip

\begin{figure}
\includegraphics{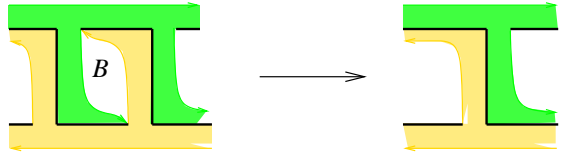}
\caption{The effect of removing a crossing from an $A$--region of $D$
  on the graph $H_A$ and the polyhedral decomposition. In the upper
  polyhedron, a bigon face $B$ between two shaded faces becomes
  collapsed to a single ideal vertex. }
\label{fig:remove-crossing}
\end{figure}

Now, consider the effect of removing a crossing from $R$ on the
polyhedral decomposition. A bigon in the twist region $R$ corresponds
to a white bigon face of the upper polyhedron in the polyhedral
decomposition of $D$. Let $F$ and $G$ be the two shaded faces that are
adjacent to this bigon $B$. As Figure \ref{fig:remove-crossing} shows,
removing one crossing from $R$ amounts to collapsing the bigon face
$B$ to a single ideal vertex.

Next, consider the essential product disks through faces $F$ and $G$
that form part of the spanning set $E_s(D) \cup E_c(D)$. By Lemma
\ref{lemma:spanning-upper}\eqref{item:all-simple}, the simple disk
parallel to bigon $B$ is part of the spanning set
$E_s(D)$. Furthermore, by Lemma \ref{lemma:pitos}, all other
\pitos\ disks through faces $F$ and $G$ remain \pitos\ if we collapse
$B$ to a single ideal vertex.

Recall that in the proof of Lemma \ref{lemma:spanning-upper}, we
considered two cases. If $F$ and $G$ are the only shaded faces in the
upper polyhedron, then $E_c = \emptyset$. This will remain true after
we remove one bigon face. Alternately, if $F$ and $G$ are not the only
shaded faces, then the contribution of these shaded faces to $E_s(D)
\cup E_c(D)$ consists of all \pitos\ disks through $F$ and $G$. The
property that a \pitos\ disk is complex will not change as we collapse
the bigon $B$. Thus $|| E_c(D) || = ||E_c (\widehat{D}) ||$, proving
\eqref{item:EC-isomorphic}.

\smallskip

Finally, \eqref{item:guts-isomorphic} follows immediately from
\eqref{item:GRA-isomorphic}, \eqref{item:EC-isomorphic}, and Theorem
\ref{thm:guts-general}.
\end{proof}
  
\begin{remark}
For alternating diagrams, Lackenby observed that there
is actually a homeomorphism from $\guts(S^3 \cut S_A(D))$ to
$\guts(S^3 \cut S_A(\widehat{D}))$, which carries parabolic locus to
parabolic locus. See \cite[Page 215]{lackenby:volume-alt}. This
statement holds in complete generality, including in our
setting. However, we will only need the equality of Euler
characteristics in Lemma
\ref{lemma:remove-bigon}\eqref{item:guts-isomorphic}.
\end{remark}

By combining Theorem \ref{thm:guts-general} and Lemma
\ref{lemma:remove-bigon} with Theorem \ref{thm:2edgeloop} on page
\pageref{thm:2edgeloop} (which will be proved in the next chapter), we
obtain the following corollary.

\begin{corollary}\label{cor:onlybigons}\index{guts!in terms of $\chi(\GRA)$}\index{$\GRA$, $\GRB$: reduced state graph!relation to guts}
Suppose that $D(K)$ is a prime, $A$--adequate diagram, such that for
each 2-edge loop in $\GA$ the edges belong to the same twist region of
$D(K)$.  Then
$$\negeul( \guts(M_A))=\negeul(\GRA).$$
\end{corollary}

\begin{proof}
Let $D$ be as in the statement of the corollary, and let $\widehat{D}$
be the diagram that results from removing \emph{all} bigons in the
$A$--regions of the diagram $D$. Applying Lemma
\ref{lemma:remove-bigon} inductively, we conclude that this removal of
bigons does not affect either the reduced graph $\GRA$ or spanning set
$E_c$. Also, since every 2-edge loop of $\GA(D)$ belongs to a single
twist region, the removal of bigons also removes all 2-edge
loops. Thus $\GRA(D) = \GRA(\widehat{D}) = \GA(\widehat{D})$.

By Theorem \ref{thm:2edgeloop}, every essential product disk in the
upper polyhedron of $\widehat{D}$ must run over tentacles adjacent to
the segments of a 2-edge loop of $\GA(\widehat{D})$. But by
construction, there are no 2-edge loops in $\GA(\widehat{D})$. Thus,
by Lemma \ref{lemma:remove-bigon}, $E_c(D) = E_c(\widehat{D}) =
\emptyset$.  Therefore, according to the formula of Theorem
\ref{thm:guts-general}, $\negeul( \guts(M_A))=\negeul(\GRA).$
\end{proof}

Given a diagram $D=D(K)$ of a link $K$, and a number $n\in {\NN}$, let
$D^n$\index{$D^n$, the $n$--cabling of a diagram} denote the the
$n$--cabling\index{cabling} of $D$ using the
blackboard framing.  If $D$ is $A$--adequate then $D^n$ is
$A$--adequate for all $n\in {\NN}$.  Furthermore, the Euler
characteristic of the reduced all--$A$ graph corresponding to $D^n$,
is the same as that of the reduced all--$A$ graph corresponding to
$D$. That is,
$$\chi(\GRA (D^n)) =\chi(\GRA(D)),$$
for all $n\geq 1$ \cite[Chapter 5]{lickorish:book}.
We have  the following:

\begin{corollary} \label{cor:cable}
Let $D:=D(K)$ be an $A$--adequate diagram, of a link $K$. Let
$D^n$ denote the $n$--cabling of $D$ using the blackboard framing, and
let $S^n_A$ be the all--$A$ state surface determined by $D^n$. Then
$$\negeul (\guts(S^3 \cut S^n_A))+ || E_c(D^n)||= \negeul (\GRA(D)),$$
for every $n\geq 1$.  Here $\negeul(\cdot)$, $|| \cdot ||$ and
$E_c(D^n)$ are the quantities of the statement of Theorem
\ref{thm:guts-general} corresponding to $D^{n}$.
\end{corollary}

\begin{proof}
By Theorem \ref{thm:guts-general}, 
we have
$$\negeul(\guts(S^3 \cut S^n_A))+ || E_c(D^n)||= \negeul
(\GRA(D^n)).$$ Since $\chi(\GRA (D^n)) = \chi(\GRA(D))$, for all
$n\geq 1$, the result follows.
\end{proof}

It is worth observing that by Corollary \ref{cor:cable} and Theorem \ref{thm:guts-general},
\begin{eqnarray*}
\negeul (\guts(S^3 \cut S^n_A))+ || E_c(D^n)|| & = &
\negeul( \guts(S^3 \cut S_A))+ || E_c(D)|| \\
& = & \negeul (\GRA(D)),
\end{eqnarray*}
for every $n\geq 1$. Thus the left-hand side is independent of $n$. It is worth asking whether the summands $\negeul (\guts(S^3 \cut S^n_A))$ and $|| E_c(D^n)||$ are also independent of $n$; see Question \ref{quest:cableques} in Chapter \ref{sec:questions}.\index{guts!stability under cabling}\index{cabling!effect on guts}

In fact, by Lemma \ref{lemma:stabilized} on page \pageref{lemma:stabilized}, the quantity $\negeul (\GRA(D))$ is
actually an invariant of the link $K$; it is independent of the
$A$--adequate diagram.  However, Example
\ref{example:bad} on page \pageref{example:bad} demonstrates that
$||E_c(D)||$ (and thus $\negeul (\guts(S^3 \cut S_A))$) does, in
general, depend on the diagram used: Figure \ref{fig:bad} shows a diagram $D$ with $|| E_c(D)||\neq 0$, while
Figure \ref{fig:bad-fix} shows a
different diagram $D'$ of the same link with $|| E_c(D')||= 0$.  We
will revisit this discussion in Chapter \ref{sec:questions}.\index{$E_c$!dependence on diagram}

\section{The $\sigma$--adequate, $\sigma$--homogeneous setting}\label{subsec:spanningsigma}

The results of this chapter extend immediately to $\sigma$--adequate,
$\sigma$--homo\-geneous states, using only the fact that the polyhedral
decomposition in this case cuts $M_\sigma$ into checkerboard polyhedra 
(Theorem
\ref{thm:sigma-homo-poly}).  This is because the proofs in this
section use only normal surface theory specific to checkerboard
polyhedra, and nothing dependent on tentacles or choice of
resolution.  

In particular, Lemma \ref{lemma:EPDlower} holds, and its proof needs
no change, using the fact that lower polyhedra still correspond to
checkerboard polyhedra of alternating links.  Definitions
\ref{def:simple} and \ref{def:pitos}, as well as Lemma
\ref{lemma:pitos}, are all stated (and proved) for any checkerboard
colored ideal polyhedron.  Lemmas \ref{lemma:spanning-lower} and
\ref{lemma:lower-combined} concern only lower polyhedra, which we know
correspond to ideal polyhedra of alternating links in the
$\sigma$--homogeneous case.  Hence their proofs will hold in this 
general setting.  Lemmas \ref{lemma:spanning-upper} and
\ref{lemma:ec-tentacle}, concerning upper polyhedra, use only
properties of checkerboard ideal polyhedra, hence these lemmas still hold if we replace $\GRA$
with $G_\sigma'$.  Similarly, Theorem
\ref{thm:fiber-tree} will immediately generalize to the
$\sigma$--adequate, $\sigma$--homogeneous setting. The proof uses Lemma \ref{lemma:spanning-upper} and Proposition \ref{lemma:product-in-face} (Product rectangle in white face), and we have seen that these results hold in the $\sigma$--adequate, $\sigma$--homogeneous case. Thus the proof of Theorem \ref{thm:fiber-tree} applies verbatim to give the following general result.

\begin{theorem}\label{thm:fibers-homo}
Let $D(K)$ be a link diagram, and let $S_\sigma$ be the state surface
of a homogeneous state $\sigma$.  Then the following are equivalent.
\begin{enumerate}
\item\label{item:tree-homo} The reduced graph $G_\sigma'$ is a tree.
\item\label{item:fiber-homo} $S^3 \setminus K$ fibers over $S^1$ with fiber
$S_\sigma$. 
\item\label{item:Ibdl-homo} $M_\sigma = S^3 \cut S_\sigma$ is an $I$--bundle over
$S_\sigma$. \qed
\end{enumerate}
\end{theorem}

In particular, Theorem \ref{thm:fibers-homo} implies the classical result, due to Stallings \cite{stallings:fibered}, that homogeneous closed braids are fibered, with fiber the Seifert surface $S_\sigma$ associated to the Seifert state $\sigma$.

The results of Section \ref{sec:compute-guts} will also extend
immediately to the
$\sigma$--adequate, $\sigma$--homogeneous setting.  In particular, every non-trivial
component of the characteristic submanifold of $M_\sigma$ is spanned
by a collection $E_l \cup E_c$ of EPDs, with the properties of
Proposition \ref{prop:spanningset}, with $\sigma$ replacing $A$ in the
appropriate places.  In the proof of Proposition \ref{prop:spanningset}, one would need to use
Theorem \ref{thm:sigma-epd-span} in place of Theorem \ref{thm:epd-span}.
Theorem \ref{thm:guts-general} also generalizes to this setting, and
we obtain
$$\chi_{-}(\guts(S^3 \cut S_\sigma)) = \chi_{-}(\G'_\sigma)-||E_c||.$$

As for Section \ref{sec:mod-diagram}, in the case of a
$\sigma$--adequate, $\sigma$--homogeneous diagram, analogous to an
$A$--region, we define a twist region $R$ to be a $\sigma$--region if
its $\sigma$--resolution gives the short resolution of $R$.  In other
words, $R$ contributes exactly one edge to $\G'_\sigma$.  With this
modification, Lemma \ref{lemma:remove-bigon} will hold, with the same
proof, replacing $A$--adequate with $\sigma$--adequate,
$\sigma$--homogeneous in the statement, as well as $S_A$ with
$S_\sigma$, $A$--region with $\sigma$--region, and $\GRA$ with
$\G'_\sigma$.

However, this is where we stop.  Corollary \ref{cor:onlybigons}
requires results from Chapter \ref{sec:epds}, which we have not analyzed
in the $\sigma$--adequate, $\sigma$--homogeneous case.  
It is entirely possible that the results of Chapter \ref{sec:epds} will generalize, 
but we leave this analysis for a future time.

%

\chapter{Recognizing essential product disks}\label{sec:epds}
Theorem \ref{thm:guts-general} reduces the problem of computing the
Euler characteristic of the guts of $M_A$ to counting how many complex
EPDs are required to span the $I$--bundle of the upper polyhedron. Our
purpose in this chapter is to recognize such EPDs from the structure
of the all--$A$ state graph $\GA$.  The main result is Theorem
\ref{thm:2edgeloop}, which describes the basic building blocks for
such EPDs.  Each corresponds to a 2--edge loop of the graph $\GA$.

The proofs of this chapter require detailed tentacle chasing
arguments, and some are quite technical.  To assist the reader, we
break the proof of Theorem \ref{thm:2edgeloop} into four steps, and
keep a running outline of what has been accomplished, and what still
needs to be accomplished.  The tentacle chasing does pay off, for by
the end of the chapter we obtain a mapping from any EPD to one of only
seven possible sub-graphs of $H_A$.  By investigating the occurrence
of such subgraphs, we are able to count complex EPDs in large classes
of link complements.  Two such classes are studied in detail in
Chapters \ref{sec:nononprime} (links with diagrams without non-prime
arcs) and \ref{sec:montesinos} (Montesinos links).  Together with the
results of Chapter \ref{sec:spanning}, these give applications to
guts, volumes, and coefficients of the colored Jones polynomials.

\section{$2$--edge loops and  essential product disks}

To find essential product disks in the upper polyhedron, we will
convert any EPD into a normal square and use machinery developed
in Chapter \ref{sec:ibundle}.

\begin{lemma}[EPD to oriented square\index{EPD to oriented square lemma}]
Let $D$ be a prime, $A$--adequate diagram of a link in $S^3$, with
prime polyhedral decomposition of $M_A = S^3\cut S_A$.  Suppose there
is an EPD embedded in $M_A$ in the upper polyhedron.  Then the
boundary of the EPD can be pulled off the ideal vertices to give a
normal square in the polyhedron with the following properties.
\begin{enumerate}
\item Two opposite edges of the square run through shaded faces, which
  we label green and orange.\footnote{Note: For grayscale versions of
    this chapter, the figures will show green faces as darker gray,
    orange faces as lighter gray.}
\item The other two opposite edges run through white faces, each
  cutting off a single vertex of the white face.
\item The single vertex of the white face, cut off by the white edge,
  is a triangle, oriented such that in counter--clockwise order, the
  edges of the triangle are colored orange--green--white.
\end{enumerate}
With this convention, the two white edges of the normal square cannot
lie on the same white face of the polyhedron.
\label{lemma:EPDtoSquare}
\end{lemma}

\begin{proof}
The EPD runs through two shaded faces, green and orange, and two ideal
vertices. Any ideal vertex meets two white faces.  Thus we may push an
arc running over an ideal vertex slightly off the vertex to run
through one of the adjacent white faces instead.  Note that there are
two choices of white face into which we may push the arc, giving
oppositely oriented triangles.  For each vertex, we choose to push in
the direction that gives the triangle oriented as in the statement of
the lemma.

Finally, to see that the two white edges of the normal square do not
lie on the same white face, we argue by contradiction.  Suppose the
two white edges do lie on the same white face.  Then white faces are
simply connected, so we may run an arc from one to the other through
the white face.  Since the shaded faces are simply connected, we may
run an arc through the green face meeting the white edges of the
square at their boundaries.  Then the union of these two arcs gives a
closed curve which separates the two ideal vertices.  This contradicts
Proposition \ref{prop:no-normal-bigons} (No normal bigons).
\end{proof}

As in Chapter \ref{sec:ibundle}, call the arcs of the normal square
which lie on white faces $\beta_W$ and $\beta_V$.

We will use Lemma \ref{lemma:EPDtoSquare} (EPD to oriented square) to
prove Theorem \ref{thm:2edgeloop}, which is the main result of this
chapter.  Before we state the theorem we need two definitions.  For
the first, note that the portion of a tentacle adjacent to a segment
has a natural product structure, homeomorphic to the product of the
segment and an interval.  In particular, the center point $p$ of the
segment defines a line $p \times I$ running across the tentacle.

\begin{define} 
We say that an arc through the tentacle runs
\emph{adjacent}\index{adjacent arc (to segment)} to a segment $s$ if
it runs transversely exactly once through the line $p \times I$, where
$p$ is the center of the segment, and the portion of the tentacle
adjacent to the segment is homeomorphic to $s \times I$.
\label{def:tentacle-adjacent}
\end{define}

\begin{define}\label{def:zig-zag}
Recall that an ideal vertex in the upper polyhedron is described on
the graph $H_A$ by a connected component of the knot, between two
undercrossings.  Such vertices will be right--down staircases,
containing zero or more segments of $H_A$.  A
\emph{zig-zag}\index{zig-zag} is defined to be one of these ideal
vertices.
\end{define}

\begin{theorem} \label{thm:2edgeloop}\index{two-edge loop}\index{EPD!determines two-edge loop}
Let $D(K)$ be a prime, $A$--adequate diagram of a link $K$ in $S^3$,
with prime polyhedral decomposition of $M_A = S^3\cut S_A$.  Suppose
there is an essential product disk embedded in $M_A$ in the upper
polyhedron, with associated normal square of Lemma
\ref{lemma:EPDtoSquare} (EPD to oriented square).  Then there is a
2--edge loop in $\GA$ so that the normal square runs over tentacles
adjacent to segments of the 2--edge loop.

Moreover, the normal square has one of the types $\mathcal{A}$ through
$\mathcal{G}$ shown in Figure \ref{fig:casesABCDEFG}.
\end{theorem}

\begin{figure}
  \begin{tabular}{l}
    $\mathcal{A}$ \\
    \includegraphics{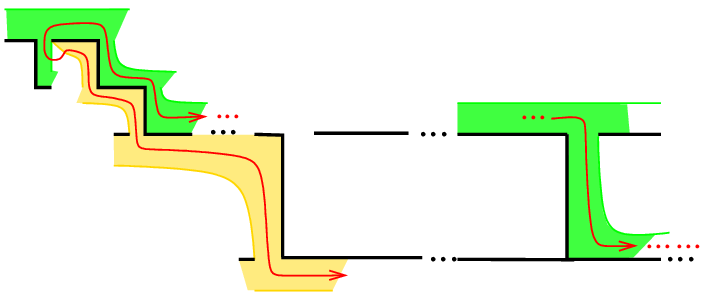} \\
    $\mathcal{B}$ \\
    \includegraphics{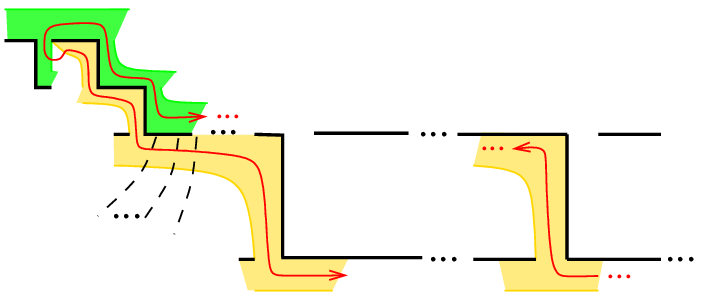} \\
  \end{tabular}

  \vspace{.1in}

  \begin{tabular}{ll}
    $\mathcal{C}$ & $\mathcal{D}$ \\
  \includegraphics{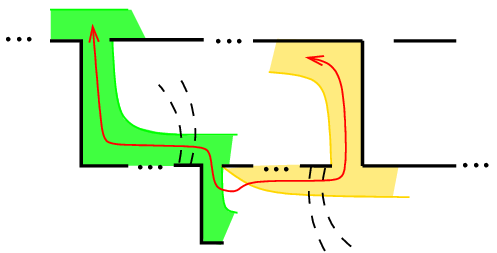} &
  \includegraphics{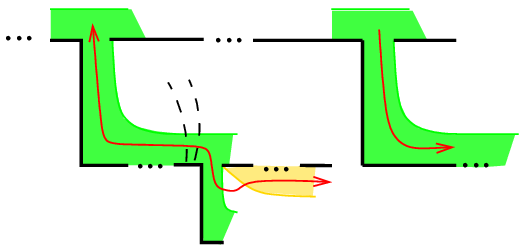} \\
  \end{tabular}

  \vspace{.1in}

  \begin{tabular}{l}
    $\mathcal{E}$ \\
  \includegraphics{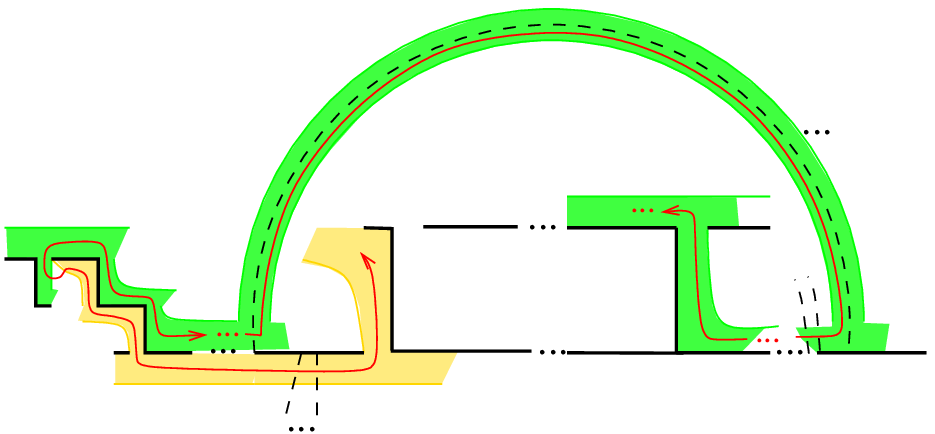}  \\
  \begin{tabular}{ll}
    $\mathcal{F}$ & $\mathcal{G}$ \\
  \includegraphics{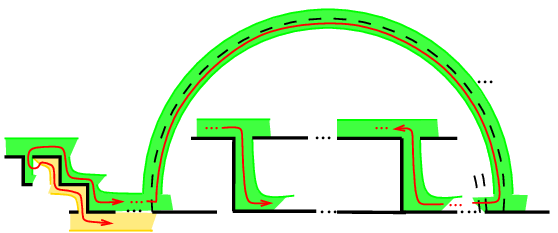} &
  \includegraphics{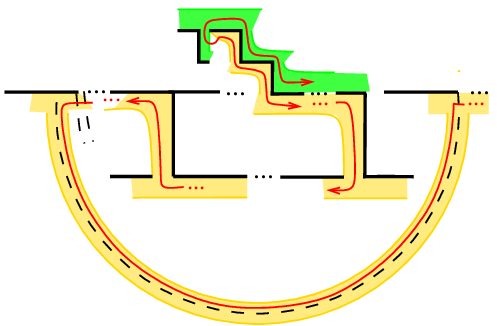} \\
  \end{tabular}
  \end{tabular}
  \caption{Building blocks of EPDs in top polyhedron.  (Note: For grayscale
    versions of this monograph, green faces will appear as dark gray,
    orange faces as lighter gray.)
    \index{EPD!Types $\mathcal{A}$ through $\mathcal{G}$}
    \index{Types $\mathcal{A}$ through $\mathcal{G}$ (of EPDs)}
    }
  \label{fig:casesABCDEFG}
\end{figure}

Before we proceed with the proof of Theorem \ref{thm:2edgeloop}, which
will occupy the remainder of this chapter, we describe the properties
and features of the types $\mathcal{A}$ through $\mathcal{G}$ in some
detail.  We want to emphasize that the colors have been selected so
that colors and orientations at vertices must be exactly as described,
or exactly as shown in Figure \ref{fig:casesABCDEFG}.  This is a
consequence of the choice of orientation in Lemma
\ref{lemma:EPDtoSquare} (EPD to oriented square).

\begin{enumerate}
\item[$\mathcal{A}$]\label{typeA}
  The square runs through distinct shaded faces adjacent to the two
  segments. One vertex of the EPD is a zig-zag (possibly without
  segments) with one end on one of the state circles met by the two
  segments of the loop, and the arcs of the normal square both run
  adjacent to this zig-zag along its length, meeting at a white face
  at its end.

\item[$\mathcal{B}$]\label{typeB} The square runs through the same
  (orange) shaded face adjacent to the two segments.  One vertex of
  the EPD is a zig-zag with one end on one of the state circles met by
  the two segments as before, with the arcs of the normal square
  running adjacent to this zig-zag along its length, meeting at a
  white face at its end.
  
\item[$\mathcal{C}$]\label{typeC} The boundary of the EPD runs through
  distinct shaded faces adjacent to two segments, as in type
  $\mathcal{A}$ above, and as in that case the vertex has an end on
  one of the two state circles met by the two segments.  However, in
  this case the vertex is on the opposite side of the 2--edge loop,
  and so the boundary of the EPD at the vertex does not run adjacent
  to the zig-zag of the vertex, but immediately runs into a white
  face.

\item[$\mathcal{D}$]\label{typeD} The boundary of the EPD runs through
  the same shaded face adjacent to the two segments, as in type
  $\mathcal{B}$ above, and meets a vertex on one of the two state
  circles, but the vertex is on the opposite side as that in type
  $\mathcal{B}$.  With colors and orientations chosen, this forces the
  square to run through two green tentacles, whereas in type
  $\mathcal{B}$ it must run through two orange tentacles.

\item[$\mathcal{E}$]\label{typeE} The boundary of the EPD runs through
  distinct shaded faces adjacent to two segments, which are separated
  from one of the vertices by a non-prime arc.

  After running downstream adjacent to one of the segments (inside
  non-prime arc in figure shown), the boundary of the EPD immediately
  crosses at least one non-prime arc with endpoints on the same state
  circle as the segment.  On the other side of these separating
  non-prime arcs, the boundary of the EPD runs directly to one of the
  vertices.

\item[$\mathcal{F}$]\label{typeF} Identical to type $\mathcal{E}$,
  only the 2--edge loop runs through two green faces rather than
  distinct colors.

\item[$\mathcal{G}$]\label{typeG} Similar to type $\mathcal{F}$, only
  the 2--edge loop runs through two orange faces.  Because the faces
  are orange, the zig-zag vertex (which still may contain no
  segments), is adjacent to the opposite side of a state circle
  meeting both segments of the 2--edge loop.
\end{enumerate}

\begin{remark}
In the statement of Theorem \ref{thm:2edgeloop}, we require the
diagram to be prime, as in Definition \ref{def:prime-diagram}. We have not used the hypothesis of prime diagrams until now, but it will be crucial going forward.  In
fact, Theorem \ref{thm:2edgeloop} does not hold for diagrams that are not prime.
\index{prime!diagram}  For
example, the connected sum of two left--handed trefoils is not prime,
and its all--$A$ state graph $\GA$ has no 2--edge loops. See Figure
\ref{fig:connsum-trefoil}.  But, if $\Sigma$ is the sphere along which
we performed the connect sum, then $\Sigma \cut S_A$ is an essential
product disk in $S^3 \cut S_A$.  One may check that this EPD is
isotopic into the upper polyhedron (indeed, by Lemma
\ref{lemma:EPDlower} on page \pageref{lemma:EPDlower}, there are no
EPDs in the lower polyhedra).
\end{remark}

\begin{figure}
  \includegraphics{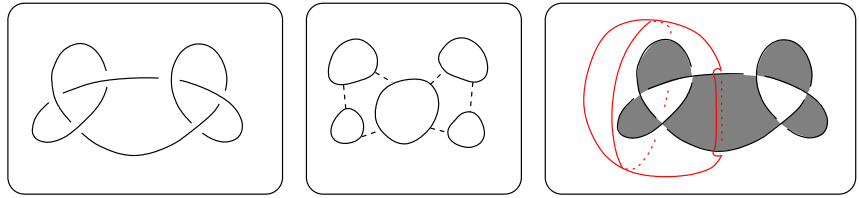}
  \caption{Left: The connect sum of two left--handed trefoils.
    Middle: The all--$A$ state.  Note there are no 2--edge loops.
    Right: The state surface $S_A$ with EPD shown in red.}
  \label{fig:connsum-trefoil}
\end{figure}

Before embarking on the proof of Theorem \ref{thm:2edgeloop}, we
record the following corollary of the theorem, which removes the
dependence on the pulling--off procedure of Lemma
\ref{lemma:EPDtoSquare}.

\begin{corollary}\label{cor:epd-loop}
  Let $D(K)$ be a prime, $A$--adequate diagram of a link $K$ in $S^3$,
  with prime polyhedral decomposition of $M_A = S^3\cut S_A$. Let $E$
  be an essential product disk embedded in the upper polyhedron of
  $M_A$. Then $\bdy E$ runs over tentacles adjacent to segments of a
  2--edge loop, of one of the types $\mathcal{A}$ through
  $\mathcal{G}$ shown in Figure \ref{fig:casesABCDEFG}.
\end{corollary}
 
\begin{proof}
The essential product disk $E$ may be pulled off the ideal vertices of
the upper polyhedron $P$, as in Lemma \ref{lemma:EPDtoSquare} (EPD to
oriented square). By Theorem \ref{thm:2edgeloop}, the resulting normal
square $S$ must run over a $2$--edge loop, as in Figure
\ref{fig:casesABCDEFG}. We may recover $E$ from $S$ by pulling the
segments of $\bdy S$ in the white faces back onto the idea vertices of
the polyhedron $P$.
  
Recall, from Definition \ref{def:zig-zag}, that an ideal vertex of $P$
is seen as a zig-zag (right--down staircase) on the graph $H_A$. Thus,
after we pull $\bdy S$ back onto the ideal vertices, the disk $E$ will
cross from one shaded face into the other at some point of the
zigzag. Performing this operation in panels $\mathcal{A}$ through
$\mathcal{G}$ in Figure \ref{fig:casesABCDEFG}, we see that $\bdy E$
still runs over tentacles adjacent to segments of a 2--edge loop.
\end{proof}

\section{Outline and first step of proof}

As mentioned above, the proof of Theorem \ref{thm:2edgeloop} requires
a significant amount of tentacle chasing\index{tentacle chasing}, which is done in the
remainder of this chapter.  The reader who wishes to avoid tentacle
chasing for now may move on to Chapter \ref{sec:montesinos} and
continue reading from there.  Chapter \ref{sec:nononprime}, which is
independent from Chapters \ref{sec:montesinos},
\ref{sec:applications}, and \ref{sec:questions}, will also involve
tentacle chasing, and requires results from Section
\ref{subsec:near-vertices} below.

In addition to tentacle chasing, the proof of Theorem
\ref{thm:2edgeloop} requires the analysis of several cases.  In each
case, we show either there is a 2--edge loop of the proper form, or
that we can further restrict the diagram.  Thus as the proof
progresses, we are left with more and more restrictions on the
diagram, until we analyze a handful of special cases to finish the
proof of the theorem.  The proof follows four basic steps, which we
summarize as follows.
\begin{enumerate}
\item[{\underline {Step 1}:}] Prove Theorem \ref{thm:2edgeloop} holds
  if $\beta_V$ and $\beta_W$ lie in the same polyhedral region, where
  recall $\beta_V$ and $\beta_W$ denote the sides of the normal square
  which lie on white faces.
\item[{\underline {Step 2}:}] If $\beta_V$ and $\beta_W$ are not in
  the same polyhedral region, then prove the theorem holds, or
  $\beta_V$ and $\beta_W$ run through a portion of the diagram of one
  of two particular forms near $\beta_V$, $\beta_W$.  These are
  illustrated in Figure \ref{fig:VW-forms} on page
  \pageref{fig:VW-forms}.
\item[{\underline {Step 3}:}] Prove the theorem holds, or the normal
  square runs along both sides of a zig-zag --- one of the vertices of
  the EPD --- and then into a non-prime arc separating $\beta_V$ and
  $\beta_W$.
\item[{\underline {Step 4}:}] Analyze behavior inside a first
  separating non-prime arc.
\end{enumerate}
\smallskip

Step 1 of the proof is treated in the next lemma, which shows that
Theorem \ref{thm:2edgeloop} holds if $\beta_V$ and $\beta_W$ lie in
the same polyhedral region.

\begin{lemma}[Step 1]
  Suppose we have a prime, $A$--adequate diagram with prime polyhedral
  decomposition, and an EPD intersecting white faces $V$ and $W$ in
  arcs $\beta_v$ and $\beta_w$, respectively.  If $V$ and $W$ are in
  the same polyhedral region, then there is a 2--edge loop of type
  $\mathcal{B}$.  In particular, Theorem \ref{thm:2edgeloop} holds in
  this case.
\label{lemma:VW-samepoly}
\end{lemma}

\begin{proof} 
Apply the clockwise map.  Lemma \ref{lemma:clockwise} implies that we
may join the images of $\beta_v$ and $\beta_w$ into a square $S'$ in
the lower polyhedron.  Because each of $\beta_v$, $\beta_w$ cuts off a
single ideal vertex in the upper polyhedron, each will cut off a
single ideal vertex in the lower polyhedron, with the same orientation
as in the upper polyhedron, and thus this new square $S'$ either is
inessential, or can be isotoped to an essential product disk.
We will treat the two cases separately.

\smallskip

{\underline{Case 1:}}
$S'$ is isotoped to an essential product disk in a lower polyhedron.
Then Lemma \ref{lemma:EPDlower} implies that the disk runs over two
segments of $H_A$ corresponding to a 2--edge loop in the lower
polyhedron. This loop must come from a 2--edge loop in the upper
polyhedron.  In the lower polyhedron, the image of the arc $\beta_W$
must be the arc on the left of Figure \ref{fig:VW-samepoly-1},
adjacent to the segment on the left.  Similarly, the image of
$\beta_V$ must be adjacent to the segment on the right.  The preimages
of these arcs are shown on the right of Figure
\ref{fig:VW-samepoly-1}.  Note that the dashed lines on the portion of
the graph of $H_A$ on the right are to indicate that the ideal edge
may run over non-prime switches between its head and tail, which will
not affect the argument.

For both vertices, there is an (orange) tentacle shown whose tail
meets the vertex, which runs upstream adjacent to a segment.  We will
show that $\sigma_1$ and $\sigma_2$ run upstream through these
tentacles, and so the 2--edge loop is of type $\mathcal{B}$ of Theorem
\ref{thm:2edgeloop}.

\begin{figure}
  \includegraphics{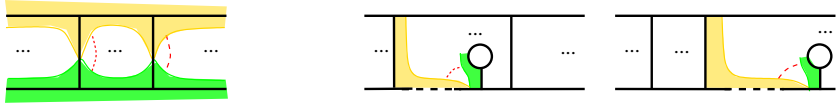}
  \caption{Left: Red (dotted and dashed) arcs show images of $\beta_w$
    under the clockwise map.  Right: preimages of arcs on left.}
  \label{fig:VW-samepoly-1}
\end{figure}

Suppose first that $\sigma_1$ or $\sigma_2$ crosses a state circle
running downstream.  Because its other endpoint is on the opposite
side of that state circle, it must cross the state circle again.  But
this contradicts Lemma \ref{lemma:utility} (Utility lemma), as it
first crossed running downstream.

Next suppose that $\sigma_1$ (or $\sigma_2$) crosses a non-prime arc.
Again since its other endpoint is on the opposite side of the half-disk
bounded by the non-prime arc and the segment of state circle between
its endpoints, $\sigma_1$ (or $\sigma_2$) must cross back out of this
non-prime half-disk.  Since $\sigma_1$ ($\sigma_2$) is assumed to be simple, it may
only exit the region by running downstream across the state circle.
As in the previous paragraph, this leads to a contradiction to the
Utility lemma.

Thus the arcs $\sigma_1$ and $\sigma_2$ must run adjacent to the
segment of the 2--edge loop, as desired. This finishes the proof Case
1.

\smallskip

{\underline{Case 2:}} $S'$ is inessential in the lower polyhedron.  By
choice of orientation on our vertices, the only way $S'$ can be
inessential is if both of its white arcs cut off the same vertex with
opposite orientation.  Thus one of the arcs, say $\beta_V$, cuts off
vertices to both sides, and thus lies in a bigon face.  Hence there is
a 2--edge loop in the upper polyhedron, and tracing back through the
clockwise map as above, we conclude that the boundary of the EPD
encircles the bigon, and the 2--edge loop is of type $\mathcal{B}$
again. 
\end{proof}

\section{Step 2: Analysis near vertices}\label{subsec:near-vertices}
In this section, we will complete Step 2 of the outline given
earlier.  The main result here is Proposition \ref{prop:step2}, which
shows that either Theorem \ref{thm:2edgeloop} holds, or the polyhedral
region near the arcs $\beta_W$ and $\beta_V$ have a very particular
form.  Before we can state this result, we need two auxiliary lemmas
concerning directed arcs in shaded faces.

\begin{lemma}[Adjacent loop\index{Adjacent loop lemma}]
  Let $\sigma$ be a directed simple arc contained in a single shaded
  face, adjacent to a state circle $C$ at a point $p$ on $C$.  Suppose
  $\sigma$ runs upstream across a state circle $C'$ after leaving $p$,
  but then eventually continues on to be adjacent to $C$ again at a
  new point $p'$.  Then $\sigma$ must run adjacent to two distinct
  segments of $H_A$ connecting $C$ to $C'$.
  \label{lemma:adjacent-loop}
\end{lemma}

Lemma \ref{lemma:adjacent-loop} is illustrated in Figure
\ref{fig:adjacent-loop}.

\begin{figure}
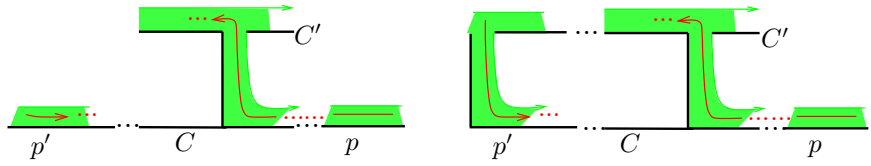

  \begin{center}
    \input{figures/lemma5-6-hypoth.pstex_t}
    \hspace{.2in}
    \input{figures/lemma5-6-conclusion.pstex_t}
  \end{center}
  \caption{Left:  Hypothesis of Lemma \ref{lemma:adjacent-loop}.
    Right:  Conclusion of the lemma:  There is a 2--edge loop and
    $\sigma$ runs adjacent to both segments of the loop.}
  \label{fig:adjacent-loop}
\end{figure}

\begin{proof}
Since $\sigma$ runs from $C$, through $C'$, and eventually back to
$C$, it must cross $C'$ twice.  Lemma \ref{lemma:utility} (Utility
lemma) implies it first crosses $C'$ running upstream, adjacent to
some segment connecting $C$ and $C'$, then runs downstream.  When it
runs downstream, it must run adjacent to a segment connecting $C'$ to
some state circle $C''$.  We show $C''$ must be $C$.  Then, since
$\sigma$ is simple, the two segments connecting $C$ to $C'$ must be
distinct, and we have the result.

Suppose $C''$ is not $C$.  Because $\sigma$ must run adjacent to $C$
further down the directed arc, $\sigma$ must leave $C''$.  Recall that
the only possibilities are that $\sigma$ runs over a non-prime switch
or runs downstream across $C''$.  If downstream across $C''$, then it
must cross $C''$ again.  Lemma \ref{lemma:utility} (Utility lemma)
implies this is impossible, as it is running downstream for the first
crossing.  If $\sigma$ runs over a non-prime switch without crossing
into the half-disk bounded by the non-prime arc, then on the opposite
side it is adjacent to $C''$ again, and we have no change.  If it
crosses into the non-prime half-disk bounded by $C''$ and the non-prime arc and
exits out again, then Lemma \ref{lemma:np-shortcut} (Shortcut) implies
that it exits running downstream across $C''$, which again gives a
contradiction.

So the only remaining possibility is that $\sigma$ crosses into the
non-prime half-disk 
and does not exit out
again.  This means $C$ must be contained in this half-disk.  But $C'$ is
on the opposite side, since $\sigma$ leaves the region containing $C'$
when crossing the non-prime arc.  This is impossible: A segment
connects $C$ to $C'$, hence $C$ and $C'$ must be on the same side of
the non-prime arc.  So $C''$ must equal $C$.
\end{proof}

\begin{lemma}
Suppose there are arcs $\sigma_1$ and $\sigma_2$ in distinct shaded
faces in the upper polyhedron, but that each runs adjacent to points
(in a neighborhood of points) $p_1$ and $p_2$ on the same state circle
$C$.  Then either
\begin{enumerate}
\item at least one of $\sigma_1$ or $\sigma_2$ runs upstream across
  some other state circle and Lemma \ref{lemma:adjacent-loop} applies;
  or
\item both arcs remain adjacent to the same portion of $C$ between
  $p_1$ and $p_2$.
\end{enumerate}
\label{lemma:adjacent-smooth}
\end{lemma}

In other words, Lemma \ref{lemma:adjacent-smooth} says that if neither
$\sigma_1$ nor $\sigma_2$ cross a state circle between $p_1$ and
$p_2$, then they cannot run over non-prime switches, either.  They
can, in fact, intersect a single endpoint of a non-prime arc.  But
they cannot run adjacent to both endpoints of a non-prime arc.

\begin{proof}
If one of $\sigma_1$, $\sigma_2$ crosses another state circle between
points $p_1$ and $p_2$ there is nothing to prove.  Suppose that
neither $\sigma_1$ nor $\sigma_2$ cross another state circle between
points $p_1$ and $p_2$; then each is embedded in the complement of the
graph $H_A$.  Form a simple closed curve meeting $H_A$ exactly twice
by connecting the portions of $\sigma_1$ and $\sigma_2$ between $p_1$
and $p_2$ with small arcs crossing $C$ at $p_1$ and $p_2$.  Replacing
segments of $H_A$ with crossings, this gives a simple closed curve in
the diagram of the link meeting the link transversely exactly twice.
Because the diagram is assumed to be prime, the curve must contain no
crossings on one of its sides.  Since each side contains a portion of
$C$ between $p_1$ and $p_2$ in $H_A$, one of those portions must not
be connected to any segments of $H_A$.  Because non-prime arcs are
required to bound segments on both sides, this means there are no
non-prime arcs attached to this portion of $C$ as well.  Then a single
tentacle runs adjacent to each side of this portion of $C$, and
because shaded faces are simply connected, $\sigma_1$ must run through
one, and $\sigma_2$ must run through the other.
\end{proof}


We are now ready to state and prove the main result of this
section.

\begin{prop} 
With the hypotheses of Theorem \ref{thm:2edgeloop}, either 
\begin{enumerate}
\item the conclusion of the theorem is true and we have a 2--edge loop
  of type $\mathcal{A}$, $\mathcal{B}$, $\mathcal{C}$, or
  $\mathcal{D}$, or
\item the polyhedral regions containing $\beta_V$ and $\beta_W$ are of
  one of two forms, shown in Figure \ref{fig:VW-forms}.
\end{enumerate}
In both cases in the figure, $\sigma_2$ immediately leaves the
polyhedral region, either through a non-prime arc, or by crossing some
state circle.
\label{prop:step2}
\end{prop}

\begin{proof}
By Lemma \ref{lemma:VW-samepoly}, we may assume that $\beta_V$ and
$\beta_W$, in white faces $V$ and $W$ respectively, are in distinct
polyhedral regions.  Then Lemma \ref{lemma:enter-RW} (Entering
polyhedral region) implies that if we direct $\sigma_1$ toward
$\beta_W$, it first enters the region containing $W$ running
downstream across a state circle, which we denote $C_W$, while
$\sigma_2$ enters the region of $W$ either running upstream across
$C_W$, or across a non-prime arc.  In either case, $\sigma_1$ connects
immediately to $\beta_w$, that is, without crossing any additional
state circles or non-prime arcs.

\begin{figure}
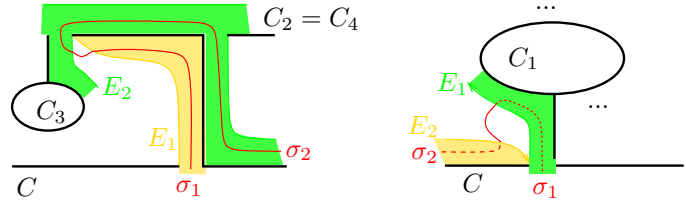

  \begin{center}
    \input{figures/hE2-tE1-s2indirect-2.pstex_t}
    \hspace{.5in}
    \input{figures/hE1-tE2-s2direct.pstex_t}
  \end{center}
  \caption{Conclusion of step 2 in the proof of Theorem
    \ref{thm:2edgeloop}: Either there is a 2--edge loop, or each
    polyhedral region containing $\beta_V$, $\beta_W$ is one of these
    two forms, with the specified colors.}
\label{fig:VW-forms}
\end{figure}

Let $E_1$ be the ideal edge of the polyhedral decomposition on which
$\sigma_1$ meets $\beta_w$.  Note $E_1$ is a directed edge, with its
head on $C_W$ and tail on some other state circle connected to $C_W$ by a
segment.  Let $E_2$ denote the ideal edge on which $\sigma_2$ meets
$\beta_w$.  Since $\beta_w$ cuts off a single ideal vertex, either the
head of $E_2$ meets the tail of $E_1$, or vice versa.  We must
consider both cases.
\smallskip

{\underline{Case 1:}} Suppose $\sigma_2$ connects immediately to
$\beta_W$ upon entering the polyhedral region of $W$, that is, without
crossing any additional state circles.  As noted above, the ideal edge
$E_1$ has its head on $C_W$, runs adjacent to a segment which we
denote $s_1$ connecting $C_W$ to a state circle $C_1$, then has its
tail on $C_1$.  The arc $\sigma_1$ runs across $C_W$ and adjacent to
$s_1$. There are two subcases to consider.

\smallskip

{\underline{Subcase 1a:}} The head of $E_2$ meets the tail of $E_1$.
Then the head of $E_2$ must also lie on the state circle $C_1$.  Since
$\sigma_2$ connects immediately to $\beta_w$ by assumption, and since
$\sigma_2$ is adjacent to $C_W$ when it enters the region of $W$ (by
Lemma \ref{lemma:enter-RW} (Entering polyhedral region)), $E_2$ has
its tail on $C_W$ and thus $E_2$ runs adjacent to a segment $s_2$
connecting $C_1$ and $C_W$.  We may isotope $\beta_w$ to cut off a
very small portion of the white face $W$, forcing $\sigma_2$ to run
adjacent to $s_2$.  See Figure \ref{fig:hE2-tE1-s2direct}, left.

\begin{figure}
  \begin{center}
  \input{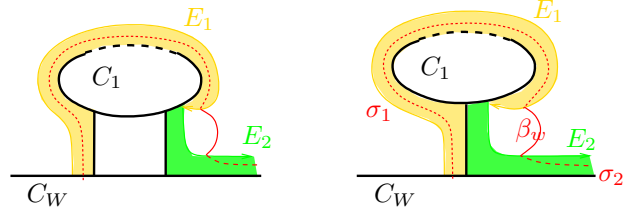}
  \end{center}
  \caption{Case: Head of $E_2$ meets tail of $E_1$, $\sigma_2$ runs
    directly to $\beta_w$.  On left: obvious 2--edge loop.  Right:
    nugatory crossing contradicts prime.  The dashed line on the state
    circle indicates that there may be non-prime switches.}
  \label{fig:hE2-tE1-s2direct}
\end{figure}

Now, provided $s_1 \neq s_2$, we have found two segments connecting
$C_W$ and $C_1$ with the boundary of the EPD running adjacent to both,
through distinct shaded faces on the segments.  Note that in this
case, the boundary of the EPD is of type $\mathcal{A}$ of the
statement of the theorem.  Thus option (1) in the statement of the
proposition holds.

So suppose $s_1=s_2$, so that we don't pick up this 2--edge loop.  We
will now show this leads to a contradiction to the fact that the
diagram is prime.  See Figure \ref{fig:hE2-tE1-s2direct}, right.  Form
a loop in $H_A$ by following $\sigma_1$ from the point where it is
adjacent to the segment $s_1$ to $\beta_w$, then following $\beta_w$
to $\sigma_2$, then following $\sigma_2$ to the point where it is
adjacent to the segment $s_2$.  Since $s_1=s_2$, connect these into a
loop by drawing a line through this segment connecting the endpoints.
Call the loop $\gamma$.  Because $\sigma_1$ connects immediately to
$\beta_w$ without running through additional state circles, $\gamma$
is embedded in the complement of $H_A$, except where it crosses the
segment $s_1=s_2$.  Replace all segments of $H_A$ with crossings of
the diagram, and push $\gamma$ slightly off the crossing of the
segment.  The result is a loop meeting the diagram twice transversely
with crossings on both sides, contradicting the fact that the diagram
is prime.  

\smallskip

{\underline{Subcase 1b:}} The head of $E_1$ meets the tail of $E_2$.
The head of $E_1$ lies on $C_W$.  Hence the tail of $E_2$ must also
lie on $C_W$.  The general form of $\beta_w$ and the ends of the arcs
$\sigma_1$ and $\sigma_2$ where they connect to $\beta_w$ is on the
right of Figure \ref{fig:VW-forms}.  Thus option (2) in the statement
of the proposition holds for $W$.

\smallskip

{\underline{Case 2:}} Suppose that $\sigma_2$ does not immediately
connect to $\beta_w$.  Lemma \ref{lemma:sigma_2-connect} implies that
$\sigma_2$ crosses upstream into some state circle $C_2$, hence
adjacent to some segment $s_2$ connecting $C_W$ and $C_2$, then out of
$C_2$ again running downstream, hence adjacent to some segment $s_3$
connecting $C_2$ and some state circle $C_3$, as in Figure
\ref{fig:w-picture}, on page \pageref{fig:w-picture}.  At this point,
$\sigma_2$ immediately meets $\beta_w$, without crossing any
additional state circles or non-prime arcs.  Hence the edge $E_2$ has
its head on $C_2$ and its tail on $C_3$.

\begin{figure}
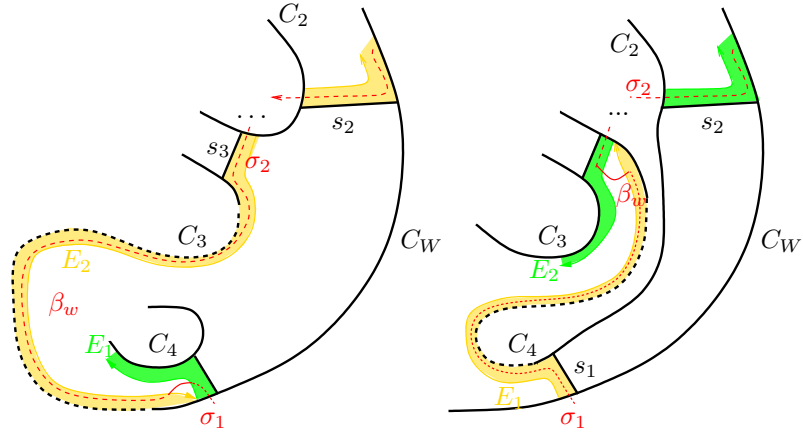

  \begin{center}
    \input{figures/c=c3.pstex_t}
    \hspace{.1in}
    \input{figures/c4=c2.pstex_t}
  \end{center}
  \caption{Possibilities for $R_W$ when met by an essential product
    disk.}
  \label{fig:ci=cj}
\end{figure}

\smallskip

{\underline{Subcase 2a:}} The head of $E_1$ meets the tail of $E_2$.
Recall that the head of $E_1$ is on $C_W$, and the tail of $E_2$ is on
$C_3$.  In order for these edges to meet in this way, we must have
$C_W=C_3$.  But now, we have a segment $s_2$ connecting $C_W$ to
$C_2$, and $\sigma_2$ runs adjacent to this segment as it runs
upstream into $C_2$.  We also have a segment $s_3$ connecting $C_2$ to
$C_3=C_W$, and $\sigma_2$ runs adjacent to this segment as it runs
downstream out of $C_2$ to meet $\beta_w$.  In this case, $s_2$ cannot
equal $s_3$, since $\sigma_2$ is assumed to be simple.  So $s_2$ and
$s_3$ form the two segments giving the desired 2--edge loop.  This
case is shown on the left of Figure \ref{fig:ci=cj}, where again the
dashed line on the state circle $C_3=C_W$ indicates that there may be
non-prime switches.  Note this is of type $\mathcal{B}$ in Theorem
\ref{thm:2edgeloop}. Thus option (1) in the statement of the proposition
holds for $W$.

\smallskip

{\underline{Subcase 2b:}} The head of $E_2$ meets the tail of $E_1$.
The tail of $E_1$ is on some state circle $C_4$ connected to $C_W$ by
a segment $s_1$, which $\sigma_1$ runs adjacent to.  Since the head of
$E_2$ is on $C_2$, $C_2$ must equal $C_4$.  Then we have a segment
$s_1$ connecting $C$ to $C_4=C_2$, with $\sigma_1$ adjacent to $s_1$,
and a segment $s_2$ connecting $C_W$ to $C_2=C_4$, with $\sigma_2$
adjacent to $s_2$.  Provided $s_1 \neq s_2$, this gives the desired
result.  This is shown on the right in Figure \ref{fig:ci=cj}.  Note
in this case our loop is of type $\mathcal{A}$ in the statement of
Theorem \ref{thm:2edgeloop}.

Suppose $s_1 = s_2$.  In this case, we will find a 2--edge loop of
type $\mathcal{D}$ stacked on the opposite side of $C_2=C_4$ from the arc
$\beta_W$, or show that our diagram is as on the right of Figure
\ref{fig:VW-forms}.  

Now, we are assuming $s_1=s_2$.  Consider the circle $C_2=C_4$.  Both
$\sigma_1$ and $\sigma_2$ are adjacent to this state circle at the
point where the segment $s_1=s_2$ meets it.  Additionally, by
shrinking $\beta_w$, we see that both $\sigma_1$ and $\sigma_2$ are
adjacent again to it at the point where $E_1$ meets $E_2$.  So Lemmas
\ref{lemma:adjacent-smooth} and \ref{lemma:adjacent-loop} imply either
that $\sigma_2$ runs adjacent to distinct segments forming a 2--edge
loop in $\GA$ --- note such a loop will be of type $\mathcal{D}$,
since the arc $\beta_W$ is on the opposite side of a state circle
meeting both segments of the loop --- or there are no segments
attached to $C_2=C_4$ between these points of adjacency, and
$\sigma_1$ and $\sigma_2$ run through tentacles adjacent to the state
circle.  In this latter case, we are on the left in Figure
\ref{fig:VW-forms}.

\smallskip

In all cases we have shown that either (1) or (2) of the statement of
the proposition is true for $W$. Since the argument is symmetric with
respect to the two faces $V$ and $W$, the proposition follows.
\end{proof}

\section{Step 3: Building staircases}
By Proposition \ref{prop:step2}, we may assume the polyhedral regions
containing $V$ and $W$ each look like one of the diagrams of Figure
\ref{fig:VW-forms}.  We have two vertices, with corresponding arcs
$\beta_V$ and $\beta_W$, and two corresponding state circles $C_V$ and
$C_W$, respectively, with $C$ playing the role of $C_V$, $C_W$, in
Figure \ref{fig:VW-forms}.  (Note that we may have $C_V=C_W$.)
Consider first $\beta_W$.  In Figure \ref{fig:VW-forms}, the arc
$\sigma_1$ crosses $C_W$ running downstream toward $\beta_W$.  The arc
$\sigma_2$, if it crosses $C_W$ at all, must do so running upstream
toward $\beta_W$.

Now direct $\sigma_2$ away from $\beta_W$.  If it crosses $C_W$, it
does so running downstream.  We will use Lemma \ref{lemma:downstream}
(Downstream) to build a staircase of $\sigma_2$ away from $\beta_W$.

\begin{lemma}[Building the first stair of a zig-zag]
Suppose $C_W \neq C_V$, and that $\sigma_2$, directed away from $C_W$,
crosses $C_W$ running downstream.  Then either
\begin{enumerate}
\item the conclusion of Theorem \ref{thm:2edgeloop} holds and we have
  a 2--edge loop, or
\item $\beta_W$ is as on the left of Figure \ref{fig:VW-forms}, and
  $\sigma_1$ and $\sigma_2$ run parallel to the same two segments on
  either side of $C_W$, both of which are part of the same vertex, cut
  off by $\beta_W$.
\end{enumerate}
\label{lemma:first-stair}
\end{lemma}

\begin{proof}
The arc $\sigma_2$ runs downstream across $C_W$, through a tentacle
which is then adjacent to some state circle $C$.  Say the tentacle has
its head adjacent to a segment $s_2$ connecting $C_W$ and $C$.  Note
that $C$ might equal $C_V$.

\smallskip

\textbf{Claim:} The arc $\sigma_1$ must cross $C$ running upstream,
when running from $\beta_W$ to $\beta_V$.

{\em{Proof of Claim:}} If $C$ separates $\beta_W$ and $\beta_V$, then
$\sigma_2$ and $\sigma_1$ must cross $C$.  Lemma
\ref{lemma:downstream} (Downstream) implies that $\sigma_2$ crosses in
the downstream direction.  Lemma \ref{lemma:opp-sides} (Opposite
sides) implies that $\sigma_1$ must cross $C$ in the upstream
direction, as claimed.

Now suppose $C$ does not separate $\beta_W$ and $\beta_V$.  Because
$\sigma_2$ is running downstream, if it crosses $C$ it does so running
downstream, and Lemma \ref{lemma:utility} (Utility) implies it cannot
cross back, contradicting the fact that $C$ does not separate.  So in
this case, $\sigma_2$ does not cross $C$.  

Hence, either $\sigma_2$ terminates in the arc $\beta_V$, without
crossing any non-prime arcs, or $\sigma_2$ must cross into a non-prime
half-disk through a non-prime arc with endpoints on $C$, without exiting the half-disk.  In
the first case, the region of $V$ is as on the right of Figure
\ref{fig:VW-forms}, with $\sigma_2$ matching the labels in that
figure, since $C_V \neq C_W$.  Then notice $\sigma_1$ crosses $C$.
Similarly, in the case that $\sigma_2$ enters a non-prime half-disk without exiting,
$\beta_V$ is inside the half-disk bounded by $C$ and the non-prime arc,
and so $\sigma_1$ must cross into this half-disk as well, and because the
non-prime tentacle belongs to the shaded face of $\sigma_2$,
$\sigma_1$ must cross through a tentacle running through $C$.  In
either case, $\sigma_1$ crosses $C$.  Since $C$ does not separate
$\beta_V$ and $\beta_W$, in fact $\sigma_1$ must cross $C$ twice,
first running upstream, then running downstream, by Lemma
\ref{lemma:utility} (Utility). This finishes the proof of the claim.

\smallskip

To continue with the proof of the lemma, we change the direction of
$\sigma_1$, so it is running across $C$ in the downstream direction,
when oriented from $\beta_V$ to $\beta_W$.  We may then apply Lemma
\ref{lemma:downstream} (Downstream) to $\sigma_1$, directed toward
$\beta_W$, for note it will run downstream across $C$, and eventually
downstream across $C_W$, exiting out of every non-prime half-disk
along the way.  Hence Lemma \ref{lemma:staircase} (Staircase extension)
implies that $\sigma_1$ defines a right--down staircase between $C$
and $C_W$, with $\sigma_1$ running adjacent to each connecting
staircase in the segment.

Arguing by A--adequacy of the diagram, similar to the proof of Lemma
\ref{lemma:parallel-stairs}, the staircase of $\sigma_1$ consists of a
single segment $s_1$, connecting $C$ and $C_W$.  Recall that
$\sigma_2$ runs adjacent to a segment $s_2$, also connecting $C$ and
$C_W$.  We will argue that either $s_1\neq s_2$ and option (1) holds,
or $s_1=s_2$ and we are in option (2).

\smallskip

{{\underline {Case 1:}}} Suppose that $\beta_W$ is as on the right of
Figure \ref{fig:VW-forms}.  Direct $\sigma_2$ away from $\beta_W$.  If
$\sigma_2$ runs upstream across any other state circle before running
downstream across $C_W$, then Lemma \ref{lemma:adjacent-loop} will
imply that there is a 2--edge loop of type $\mathcal{B}$.  Hence we
assume $\sigma_2$ does not run upstream across another state circle
from the point where it leaves $\beta_W$ to the point where it crosses
$C_W$.

Similarly, consider $\sigma_1$ running (downstream) towards $\beta_W$
from $s_1$ to cross $C_W$: If it runs upstream between $s_1$ and
$C_W$, then Lemma \ref{lemma:adjacent-loop} implies that there is a
2--edge loop of type $\mathcal{D}$.

So suppose $\sigma_2$ does not run upstream between $\beta_W$ and
crossing $C_W$, and suppose that $\sigma_1$ does not run upstream
between leaving $s_1$ and crossing $C_W$.  Then $\sigma_1$ and
$\sigma_2$ are both adjacent to $C_W$ at $\beta_W$.  We claim that
$s_1$ cannot equal $s_2$.  In this case, $s_1$ and $s_2$ and the
portions of the boundary of the EPD adjacent to them form a 2--edge
loop of type $\mathcal{C}$ and we are in option (1) of the statement
of the lemma.  Suppose, on the contrary, that $s_1 = s_2$.  Then $s_1,
s_2$ are also adjacent to $C_W$ where this segment attaches to $C_W$.
Lemma \ref{lemma:adjacent-smooth} implies that both $\sigma_1$ and
$\sigma_2$ must run along $C_W$ between these two points.  But note
that $\sigma_1$ and $\sigma_2$ must run in opposite directions.  Hence
the loop following $\sigma_2$ on one side of $C_W$, following
$\sigma_1$ on the other side, connecting where these are adjacent into
a loop, becomes a loop in the diagram meeting the diagram twice,
bounding crossings on either side.  This contradicts the fact that the
diagram is prime.  Thus $s_1\neq s_2$ as desired.

\smallskip

{{\underline {Case 2:}}} Suppose that $\beta_W$ is as on the left of
Figure \ref{fig:VW-forms}.  If $\sigma_2$, directed from $W$ to $V$,
runs upstream before crossing $C_W$, then Lemma
\ref{lemma:adjacent-loop} implies there is a 2--edge loop of type
$\mathcal{F}$.  If $\sigma_1$, directed from $V$ to $W$, runs upstream
between leaving $s_1$ and crossing $C_W$, then there is a 2--edge loop
of type $\mathcal{D}$.  If $s_1 \neq s_2$, then there is a 2--edge
loop of type $\mathcal{A}$.  If none of these three things happen,
then $s_1=s_2$, and as before, Lemma \ref{lemma:adjacent-smooth}
implies that $\sigma_1$ and $\sigma_2$ both run adjacent to $C_W$
between the point where $s_1=s_2$ and the segment on the opposite side
of $C_W$ where the two arcs run adjacent on opposite sides.  In this
case, the segment $s_1=s_2$ is part of the same vertex as $\beta_W$,
as claimed in the statement of the lemma.
\end{proof}

The previous lemma is the first step in creating a maximal right--down
staircase for the vertex corresponding to $\beta_W$.  The next lemma
gives the full staircase of a zig-zag.

\begin{lemma}[Full staircase\index{Full staircase lemma}]
Suppose $C_W \neq C_V$, and $\sigma_1$ and $\sigma_2$ are directed
from $\beta_W$ to $\beta_V$, and that $\sigma_2$ crosses $C_W$ running
downstream.  Then, either
\begin{enumerate}
\item the conclusion of Theorem \ref{thm:2edgeloop} holds; or
\item $\beta_W$ is as on the left of Figure \ref{fig:VW-forms}, and
  the vertex of $\beta_W$ forms a right--down staircase, with
  $\sigma_1$ and $\sigma_2$ adjacent on either side.
\end{enumerate}
In case (2) the staircase is maximal, in the sense that at the
bottom of the right--down staircase, $\sigma_2$ either crosses $C_V$,
or crosses over a non-prime arc $\alpha$ with endpoints on some state
circle $C$, and does not exit the corresponding half-disk.  In the latter
case, the arc $\sigma_1$ also crosses into this half-disk, first running
upstream across $C$ then running downstream into the half-disk bounded by
$\alpha$ and $C$.  Additionally, $\sigma_1$ does not cross any other
state circles between its two crossings of $C$.
\label{lemma:next-steps}
\end{lemma}

The form of the graph $H_A$ in case (2) of Lemma
\ref{lemma:next-steps} (Full staircase) is illustrated in Figure
\ref{fig:stair-enter-nonprime}.

\begin{figure}
  \input{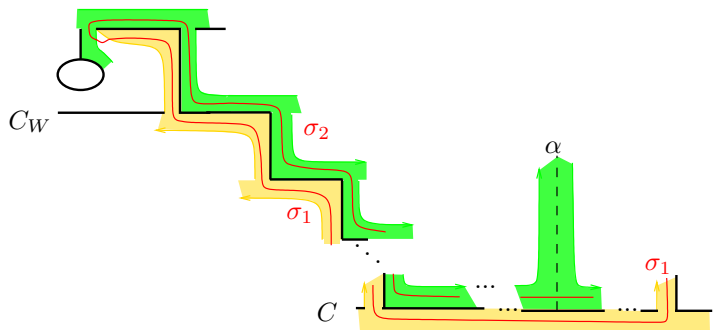}
  \caption{Lemma \ref{lemma:next-steps}: either we have desired
    2--edge loop, or the graph $H_A$ is as shown.}
  \label{fig:stair-enter-nonprime}
\end{figure}

\begin{proof}
By Lemma \ref{lemma:first-stair}, we may assume that $\beta_W$ is as
on the left of Figure \ref{fig:VW-forms}, and $\sigma_1$ and
$\sigma_2$ run parallel to the same segments on either side of
$\beta_W$, which are both part of the vertex at $\beta_W$.

\smallskip

\textbf{Claim:}  For $i=1,2$,
 $\sigma_i$ defines a right--down staircase,
with $\sigma_i$ running adjacent to each segment of the staircase.

{\em{Proof of Claim:}} We may apply Lemma \ref{lemma:downstream}
(Downstream) to the arc $\sigma_2$, directed away from $\beta_W$.
This lemma implies that $\sigma_2$ defines a right--down staircase,
with $\sigma_2$ running adjacent to each segment of the staircase in a
downstream direction, either until $\sigma_2$ crosses a non-prime arc
without exiting the half-disk it bounds with a state circle $C$, or
crosses $C_V$ and runs to $\beta_V$.

If $\sigma_2$ crosses $C_V$, then $C_V$ separates $V$ and $W$, so
$\sigma_1$ must also cross $C_V$, running upstream when directed to
$V$.  If $\sigma_2$ crosses a non-prime arc without exiting the half-disk
it bounds with $C$, then note $V$ must lie inside this half-disk, so
$\sigma_1$ must cross into this half-disk.  Moreover, $C$ cannot separate
$V$ and $W$, so $\sigma_1$ actually must cross $C$ twice, by Lemma
\ref{lemma:utility} (Utility), first running upstream, then downstream
when directed toward $V$.

In either case, $\sigma_1$ crosses the last state circle $C$ of the
staircase of $\sigma_2$ running upstream, directed toward $V$.  Change
the direction on $\sigma_1$.  It runs downstream across $C$, and
downstream across $C_W$, and must cross out of any non-prime half-disks 
between these.  So Lemma \ref{lemma:downstream}
(Downstream) implies that $\sigma_1$ defines a right--down staircase,
with $\sigma_1$ running adjacent to each segment of the
staircase. This finishes the proof of the claim.

\smallskip

To continue with the proof of the lemma we note that adequacy implies
that the segments of the staircases defined by $\sigma_1$ and by
$\sigma_2$ must actually run between the same sequence of state
circles.

Recall that we know that the first segment of the staircase, on the
other side of $C_W$ from $\beta_W$ , is shared by both $\sigma_1$ and
$\sigma_2$.  Suppose we have shown that $\sigma_1$ and $\sigma_2$ run
adjacent to the first $k$ stairs of a right--down staircase forming
the vertex of $\beta_W$, and that $\sigma_2$ runs to a $(k+1)$-st
step.  We will show the theorem holds at this step.

The arc $\sigma_2$ runs from the top of the $(k+1)$-st step,
somewhere, then downstream adjacent to the segment of the step.  If
$\sigma_2$ runs upstream first, before running downstream, then Lemma
\ref{lemma:adjacent-loop} implies that a 2--edge loop of type
$\mathcal{F}$ occurs.  Similarly, when directed downstream, the arc
$\sigma_1$ runs adjacent to the segment of the $(k+1)$-st step,
somewhere, and then downstream adjacent to the $k$-th step.  If it
runs upstream between the two segments, then Lemma
\ref{lemma:adjacent-loop} will imply there is a 2--edge loop, this
time of type $\mathcal{G}$.  In both cases option (1) holds.

Now assume neither $\sigma_1$ nor $\sigma_2$ runs upstream between the
segments of these steps.  

If the segments of $\sigma_1$ and $\sigma_2$ at this $(k+1)$-st step
are distinct, then we have a 2--edge loop of type $\mathcal{A}$; again
option (1) holds.

If the segments are not distinct, then Lemma
\ref{lemma:adjacent-smooth} implies that $\sigma_1$ and $\sigma_2$
both run adjacent to the state circle of this step between the two
segments of the step.  Thus this step is a continuation of the same
vertex corresponding to $\beta_W$.

By induction, either (1) holds or $\sigma_1$, $\sigma_2$ share every
segment of the staircase.

Finally, suppose the staircase ends with $\sigma_2$ entering into the
non-prime half-disk bounded by a non-prime arc and the state circle $C$, without
exiting.  We have already seen that $\sigma_1$ must also cross $C$ in
this case, first upstream and then downstream.  Suppose $\sigma_1$
crosses additional state circles between these two crossings of $C$.
Then Lemma \ref{lemma:adjacent-loop} implies that there is a 2--edge
loop of type $\mathcal{B}$ on the underside of the state circle $C$.
\end{proof}

\begin{lemma}
  Suppose $\sigma_2$ runs across $C_W$ in the downstream direction,
  out of every non-prime half-disk that it enters, and
  terminates with $\sigma_2$ crossing $C_V$.  Then the conclusion of
  Theorem \ref{thm:2edgeloop} holds.
\label{lemma:staircase-to-vertex}
\end{lemma}

\begin{proof}
By Proposition \ref{prop:step2} we reduce to the case
that $\beta_W$ and $\beta_V$ are as in Figure \ref{fig:VW-forms}.  The
colors on these figures are fixed, given our choice of direction in
which to pull $\beta_W$ and $\beta_V$ off their corresponding vertices
(Lemma \ref{lemma:EPDtoSquare}, EPD to oriented square).

This means that both vertices cannot be of the same form in that
figure, or a green\footnote{In grayscale versions of this
  monograph, green will appear darker gray, orange lighter gray.}
shaded face would lie adjacent to both sides of the 
same state circle, which is impossible by Lemma
\ref{lemma:escher} (Escher stairs).

Thus one vertex, $\beta_W$ say, is as on the left of Figure
\ref{fig:VW-forms}, and the other vertex, $\beta_V$, is as on the
right.  Relabel so $\sigma_1$ runs through the orange face and
$\sigma_2$ runs through the green.

By Lemma \ref{lemma:next-steps}, we reduce to the case that either
$C_W = C_V$, or both $\sigma_1$ and $\sigma_2$ run adjacent to either
side of a maximal right--down staircase from $C_W$ to $C_V$,
corresponding to the vertex of $\beta_W$.
In both cases, when $\sigma_1$ and $\sigma_2$
are directed toward $C_V$, they run adjacent to the same segment which
meets $C_V$ on the opposite side of that containing $\beta_V$: if
$C_W=C_V$, then $\sigma_1$ and $\sigma_2$ are adjacent to the segment shown
on the left in Figure \ref{fig:VW-forms}; otherwise they are adjacent to
the last segment of the staircase from $C_W$ to $C_V$.

After leaving this segment, $\sigma_1$ and $\sigma_2$
split up and run to $\beta_V$.  Thus the two are adjacent
to each other and to $C_V$ in two distinct points: at a segment on one
side of $C_V$, and at $\beta_V$ on the other side.

If $\sigma_1$ crosses $C_V$ (running upstream) then runs upstream
again before meeting $\beta_V$, Lemma \ref{lemma:adjacent-loop}
implies there will be a 2--edge loop of type $\mathcal{B}$.
Similarly, if $\sigma_2$ runs upstream before crossing $C_V$, then
there will be a 2--edge loop, and this must be of type $\mathcal{F}$,
as $\sigma_2$ is running downstream to $C_V$, so must pass through a
non-prime switch to run upstream.

If neither $\sigma_1$ nor $\sigma_2$ run upstream between the point
where they leave the segment of the zig-zag of $\beta_W$ and the point
where they meet again at $\beta_V$, then a simple closed curve crosses
the knot diagram at the base of the zig-zag and at $\beta_V$, and
nowhere else, following $\sigma_1$ on one side and $\sigma_2$ on the
other, and encircles crossings on both sides.  This contradicts the
hypothesis that the diagram is prime.
\end{proof}

\section{Step 4: Inside non-prime arcs}
At this point in the proof of Theorem \ref{thm:2edgeloop}, we either
have the conclusion of the theorem, or we have specialized to cases
where the form of the graph $H_A$ is very restricted.  In particular,
\begin{itemize}
\item  $\beta_V$ and $\beta_W$ must be in distinct polyhedral
regions (Step 1);

\item the graph near $\beta_W$ and $\beta_V$ must be of one of the two
  forms shown in Figure \ref{fig:VW-forms} (Step 2);

\item $\sigma_2$ runs down a (possibly empty) maximal right--down
  staircase and across a non-prime arc, as in Figure
  \ref{fig:stair-enter-nonprime} (Step 3).
\end{itemize}

To finish the proof, we need to analyze what happens to the EPD when
$\beta_V$ and $\beta_W$ are separated by a non-prime arc $\alpha$.


\begin{lemma}
Suppose $\beta_V$ and $\beta_W$ are separated by a non-prime arc
$\alpha$, with the arc $\sigma_2$, say, crossing $\alpha$.  Suppose
$\alpha$ is outermost\index{outermost non-prime arc} among all such arcs, with respect to $\beta_W$.
That is, $\alpha$ is the first such non-prime arc crossed by
$\sigma_2$ when directed toward $\beta_V$.  Then we have the
conclusion of Theorem 
\ref{thm:2edgeloop}.
\label{lemma:nonprime-inside}
\end{lemma}

\begin{proof}
We break the proof into two cases:  first, that $\sigma_2$ does not
run upstream after crossing $\alpha$, and second, that it does run
upstream.

\smallskip

{\underline{Case 1:}} Suppose $\sigma_2$ does not run upstream after
crossing $\alpha$.  Now suppose, by way of contradiction, that the
conclusion of Theorem \ref{thm:2edgeloop} is not true.  We will find a
contradiction to primeness of the diagram.

Since $\sigma_2$ does not run upstream after crossing $\alpha$, it
will not run downstream either, for to run downstream would be to
cross the state circle $C$ out of the non-prime half-disk bounded by  
$\alpha$, contradicting the hypotheses.  
Therefore, after crossing $\alpha$, $\sigma_2$ must run directly to
$\beta_V$ without crossing any additional state circles.  We know the
graph $H_A$ must have one of the forms of Figure \ref{fig:VW-forms},
and that $\sigma_2$ cannot cross an additional state circle after
entering the region of $\beta_V$, hence $\beta_V$ must be as on the
right of that figure, so $\alpha$ is an arc in an orange
face.

Next, Lemma \ref{lemma:next-steps} (full staircase) implies that on
the opposite side of $\alpha$, in the region containing $\beta_W$,
$\sigma_1$ and $\sigma_2$ run adjacent to the same (possibly empty)
right--down staircase corresponding to the vertex of $\beta_W$.
However, notice that if the staircase is non-empty, then $\beta_W$
must have the form of the left of Figure \ref{fig:VW-forms}, and the
colors must be as in Figure \ref{fig:stair-enter-nonprime}.  That is,
$\alpha$ is an arc in a green face.  But in the previous paragraph, we
argued that $\alpha$ is in an orange face.  This is a contradiction.
So the zig-zag of $\beta_W$ must be empty, and $\beta_W$ must have the
form of the right of Figure \ref{fig:VW-forms}.  Note that this
implies that $\sigma_2$ meets no state circles on either side of the
non-prime arc $\alpha$.  See Figure \ref{fig:CW=CV-nosep1}.

\begin{figure}
  \includegraphics{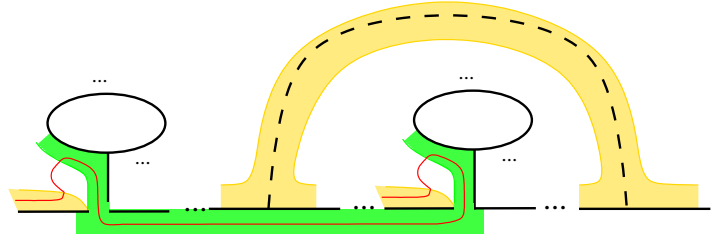}
  \caption{The contradiction in Lemma \ref{lemma:nonprime-inside},
    Case 1:  $\beta_W$ and $\beta_V$ both have the form of the right
    of Figure \ref{fig:VW-forms}, and $\sigma_2$ meets no state circles.}
  \label{fig:CW=CV-nosep1}
\end{figure}

By Lemma \ref{lemma:next-steps}, $\sigma_1$ crosses $C$ twice, but
meets no state circles between these crossings.  Then the boundary of
the EPD gives a simple closed curve in the diagram which meets the
diagram exactly twice, once each time $\sigma_1$ crosses $C$.  This
contradicts the fact that the diagram is prime.


\smallskip

{\underline{Case 2:}}  The arc $\sigma_2$ does run upstream after
crossing $\alpha$, say across some state circle $C_1$.

If $\sigma_2$ runs back to $C$ from $C_1$, then Lemma
\ref{lemma:adjacent-loop} (Adjacent loop) implies there is a 2--edge
loop of type $\mathcal{F}$.

If not, then we claim that the theorem holds or $\sigma_1$ must also
run adjacent to a segment connecting $C$ and $C_1$.  This can be seen
as follows.  First, if $C_1$ separates $\beta_V$ and $\beta_W$, then
$\sigma_1$ must also cross $C_1$.  Since $\sigma_1$ is running
downstream (by Lemma \ref{lemma:opp-sides} (Opposite sides)),
Lemma \ref{lemma:downstream} (Downstream) implies that it must run
adjacent to a segment from $C$ to $C_1$, as desired.  If $C_1$ does
not separate, then $\sigma_2$ must cross it twice, the second time
running downstream to some $C''$.  If $C''=C$, then we must have a
2--edge loop of type $\mathcal{F}$.  If $C'' \neq C$, consider
$\sigma_1$.  It runs downstream across $C$, along a segment connecting
$C$ to some $C'$.  If $C'\neq C_1$, then $C'=C''$, else we could not
connect ends of $\sigma_1$ and $\sigma_2$ at $\beta_V$. But now,
$\sigma_1$ and $\sigma_2$ cannot cross $C'=C''$, or we would build two
staircases from $C'$ contradicting Lemma \ref{lemma:parallel-stairs}
(Parallel stairs).  On the other hand, $\beta_V$ cannot lie on
$C'=C''$, since it would lie at the tails of two tentacles, which do
not meet at a vertex.  The only possibility is $C_1=C'$, as desired.

Thus $\sigma_1$ and $\sigma_2$ both run adjacent to segments from $C$
to $C_1$ inside $\alpha$.  If these segments are distinct, we have a
2--edge loop of type $\mathcal{E}$.

If not, we will show we have a contradiction to the fact that the
diagram is prime.  By assumption, $\sigma_1$ and $\sigma_2$ are
adjacent to the vertex corresponding to $\beta_W$ just outside
$\alpha$ on $C$, and they meet no additional state circles outside
$\alpha$.  If they are also adjacent to the same segment inside
$\alpha$, then we may form a loop in the diagram meeting $C$ twice,
meeting no other state circles, by following $\sigma_1$ on one side
and $\sigma_2$ on the other.  This will descend to a loop in the
diagram enclosing curves on both sides, meeting the diagram just
twice, contradicting the fact that the diagram is prime.
\end{proof}

\begin{proof}[Completion of the proof of Theorem \ref{thm:2edgeloop}]
As discussed in the begin\-ning of this section, Steps 1 through 3
imply the theorem in all cases where $\beta_W$, $\beta_V$ are not
separated by a non-prime arc.  Recall that by Step 3, either
$\sigma_2$ runs directly from $\beta_W$ across a non-prime arc
separating $V$ and $W$, or we have $\sigma_1$ and $\sigma_2$ adjacent
to a maximal right--down staircase of the vertex of $\beta_W$, as in
Figure \ref{fig:stair-enter-nonprime}. By Lemma
\ref{lemma:next-steps}, in the latter case $\sigma_2$ must also cross
a non-prime arc separating the $\beta_W$ and $\beta_V$.  Hence, in all
cases, there is a non-prime arc that separates $\beta_W$, $\beta_V$.
Now we pass to an outermost such non-prime arc, and apply Lemma
\ref{lemma:nonprime-inside} to obtain the conclusion.
\end{proof}

\chapter{Diagrams without non-prime arcs}\label{sec:nononprime}
In this chapter, which is independent from the remaining chapters, we
will restrict ourselves to $A$--adequate diagrams $D(K)$ for which the
polyhedral decomposition includes no non-prime arcs or switches.  In
this case, one can simplify the statement of Theorem
\ref{thm:guts-general} and give an easier combinatorial estimate for
the guts of $M_A$.  This is done in Theorem \ref{thm:guts-nononprime},
whose proof takes up the bulk of the chapter.  

\begin{define}\label{def:ma}
In the $A$--adequate diagram $D(K)$, let $b_A$\index{$b_A$: number of bigons in $A$--regions} denote the
number of bigons in the $A$--regions of the diagram. (Recall Definition \ref{def:long-short} on page \pageref{def:long-short} for the notion of an $A$--region.) In other words,
in Figure \ref{fig:twist-resolutions}, $b_A$ is the number of bigons
in twist regions where the $A$--resolution is short.
Define \index{$m_A$}\index{$A$--region}\index{short resolution}
$$m_A = e_A - (e'_A + b_A).$$

Since each bigon of $b_A$ corresponds to a redundant edge of the graph
$G_A$, the quantity $m_A$ is always non-negative.
\end{define}

Note that the quantity $m_A$ counts the number of distinct segments of
$H_A$ that connect the same state circles, excepting those segments
that come from twist regions and bound simple rectangles in $H_A$. In other words, $m_A = 0$ precisely when every $2$--edge loop in $\GA$ has edges belonging to the same twist region, as in Corollary \ref{cor:onlybigons} on page \pageref{cor:onlybigons}.

The main result of this chapter extends the simple diagrammatic statement of Corollary \ref{cor:onlybigons} to a context where the corollary does not directly apply.

\begin{theorem} \label{thm:guts-nononprime}\index{guts!in terms of $\chi(\GRA)$}\index{$\GRA$, $\GRB$: reduced state graph!relation to guts}\index{$m_A$!role in estimating guts}
Let $D(K)$ be a prime, $A$--adequate diagram, and let $S_A$ be the
essential spanning surface determined by this diagram. Suppose
that the polyhedral decomposition of $M_A = S^3\cut S_A$ includes no
non-prime arcs; that is, no further cutting was required in
Section \ref{subsec:primeness}.  Then
$$ \negeul (\GRA) - 8 m_A \: \leq \: \negeul ( \guts(M_A)) \: \leq \: \negeul(\GRA),$$
where the lower bound is an equality if and only if $m_A = 0$.
\end{theorem}

To derive Theorem \ref{thm:guts-nononprime} from Theorem
\ref{thm:guts-general} on page \pageref{thm:guts-general}, it suffices
to bound the number $|| E_c ||$ of complex disks required to span the 
$I$--bundle of the upper polyhedron. (See Definition \ref{def:ec} on page \pageref{def:ec}.)
%
%
Note that by Theorem \ref{thm:2edgeloop} on page
\pageref{thm:2edgeloop}, each disk $D \in E_c$ must run along a
$2$--edge loop of $\GA$. If this loop corresponds to a single twist
region, as in Corollary \ref{cor:onlybigons} on page
\pageref{cor:onlybigons}, then a disk corresponding to this loop
cannot be complex. In the following argument, we will bound $|| E_c
||$ in terms of $m_A$, where $m_A$ accounts for the loops that
\emph{do not} correspond to twist regions.

Before diving into the proof of Theorem \ref{thm:guts-nononprime}, we give a sample application. 

\begin{example}
Lickorish and Thistlethwaite introduced the notion of a \emph{strongly alternating tangle} \index{strongly alternating tangle} \cite{lick-thistle}. This is an alternating tangle $T$, such
that both its numerator and denominator closures are alternating,
prime, reduced diagrams. 
(See  Definition
\ref{def:num-denom} on page \pageref{def:num-denom} for the notions of numerator, denominator, and tangle sum.)\index{numerator closure}\index{denominator closure}
 A \emph{semi-alternating diagram}\index{semi-alternating diagram} $D$ is the numerator closure of the tangle sum $T_1 + T_2$, where each $T_i$ is strongly alternating but their sum $T_1 + T_2$ is non-alternating.
Lickorish and Thistlethwaite observed that these diagrams are both $A$-- and $B$--adequate.

If $D$ is a twist-reduced, strongly alternating diagram, there is exactly one state circle $C$ of $s_A(D)$ that runs through both tangles $T_1$ and $T_2$. In the all--$A$ resolution of $T_1$ (resp. $T_2$), this state circle appears as a pair of arcs along the north and south (resp. east and west) of the tangle. Then,
a $2$--edge loop in $\GA(D)$ can take one of two forms. The two edges of this loop either belong to a single twist region (in which case they do not contribute to $m_A$), or else they form a \emph{bridge}\index{bridge} of two edges that spans the tangle north to south, or east to west. (See Figure \ref{fig:neg-tangle} on page
\pageref{fig:neg-tangle} for an example.) The quantity $m_A$ is then exactly equal to the number of bridges in the tangles. Thus, applied to a semi-alternating diagram, Theorem \ref{thm:guts-nononprime} has the simpler formulation
$$\negeul (\GRA) - 8 \, (\mbox{number of bridges in $\GA$}) \: \leq \: \negeul ( \guts(M_A)) \: \leq \: \negeul(\GRA).$$
%
%
\end{example}

\section{Mapping EPDs to $2$--edge loops}

Recall that Theorem \ref{thm:2edgeloop} shows that every EPD in the
upper polyhedron determines a normal square of one of seven types
($\mathcal{A}$ through $\mathcal{G}$), as shown in Figure
\ref{fig:casesABCDEFG}.  Under the hypothesis that the polyhedral
decomposition includes no non-prime arcs or switches, we will simplify
these seven cases to three (see Figure \ref{fig:nononprime-types}).

\begin{define}
A \emph{brick}\index{brick} is a pair of segments $s, s'$ of the graph
$H_A$ that connect the same state circles $C_0$ and $C_1$.

Note that a closed curve in the projection plane consisting of
segments $s$ and $s'$, as well as parallel arcs of $C_0$ and $C_1$
from $s$ to $s'$, is topologically a rectangle. This is the origin of
the term \emph{brick}. We will always depict bricks with state circles
horizontal and segments of $H_A$ vertical, as in Figure
\ref{fig:nononprime-types}.

Note as well that the segments $s$ and $s'$ split the annular region
between $C_0$ and $C_1$ into two rectangular components. We will say
that tentacles adjacent to $s$ and $s'$ are on the \emph{same
  side}\index{same side (of a brick)} of
the brick if these tentacles lie in the same rectangular component,
and that these tentacles are on \emph{opposite sides}
\index{opposite sides (of a brick)} of the brick if they belong to
different components.
\end{define}

\begin{define}
Any EPD meets exactly two distinct shaded faces.  We assign each
shaded face a unique color.
A \emph{color pair}\index{color pair (of an EPD)} of an EPD is a
choice of two distinct shaded faces met by a single EPD.
\end{define}

By Lemma \ref{lemma:spanning-upper} on page
\pageref{lemma:spanning-upper}, any tentacle caries at most two EPDs
with the same color pair.

\begin{prop}\label{prop:2edgeloop-nononprime}\index{EPD!determines two-edge loop}
Let $D(K)$ be a prime, $A$--adequate diagram of a link in $S^3$ with
prime polyhedral decomposition of $S^3\cut S_A$ such that the
polyhedral decomposition contains no non-prime arcs.  Let $E$ be an
EPD embedded in the upper polyhedron, with associated normal square of
Lemma \ref{lemma:EPDtoSquare} (EPD to oriented square), and denote the
color pair of $E$ by orange--green\footnote{Note: For grayscale
  versions of this monograph, orange faces in the figures will appear
  light gray, and green ones will appear darker gray.}, with the
orientation convention of Lemma \ref{lemma:EPDtoSquare} (EPD to
oriented square).  Then we may associate $E$ with a brick in $H_A$ of
one of the following forms.
\begin{enumerate}
\item\label{type:nononprime-1} The normal square of $E$ runs through
  distinctly colored tentacles adjacent to the segments of the brick
  (hence one orange, one green tentacle), with these tentacles lying
  on the same side of the brick.
\item\label{type:nononprime-2} The normal square of $E$ runs through
  two orange colored tentacles adjacent to the segments of the brick,
  necessarily on opposite sides of the brick, and in addition, a green
  tentacle is adjacent to one of the two segments.
\item\label{type:nononprime-3} The normal square of $E$ runs through
  two orange colored tentacles adjacent to the segments of the brick,
  necessarily on opposite sides of the brick.  Moreover, the arcs of
  $\bdy E$ in these orange tentacles run to the tail of the orange
  tentacles on the state circle $C_0$ of the brick, there meet a
  vertex, and then run into green tentacles and across the state
  circle $C_0$.
\end{enumerate}
\end{prop}

The three possibilities are illustrated in Figure
\ref{fig:nononprime-types}.

\begin{figure}
  \begin{tabular}{ccc}
  \includegraphics{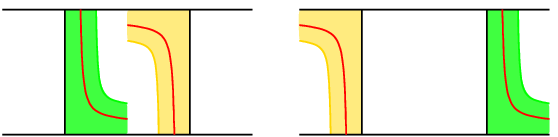}
  & \hspace{.2in} & \includegraphics{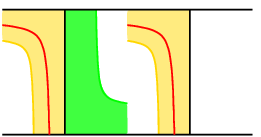} \\
  Type \eqref{type:nononprime-1}, inside or outside & &
  Type \eqref{type:nononprime-2}\\

  \vspace{.1in} & & \\
  
  \input{figures/nononprime-type3.pstex_t} & & \\
  Type \eqref{type:nononprime-3} & &
  \end{tabular}
  \caption{When there are no nonprime arcs, each EPD can be associated
  with a 2--edge loop of
  type \eqref{type:nononprime-1}, \eqref{type:nononprime-2},
  or \eqref{type:nononprime-3}, illustrated.}
\label{fig:nononprime-types}
\end{figure}

\begin{proof}
This follows from an analysis of the normal squares of types
$\mathcal{A}$ through $\mathcal{G}$ in the conclusion of Theorem
\ref{thm:2edgeloop}.  Notice that the normal squares in types
$\mathcal{E}$, $\mathcal{F}$, and $\mathcal{G}$ include non-prime arcs
as essential portions of the diagram, so none of these can occur in
the setting at hand.

Consider first the normal squares of types $\mathcal{A}$ and
$\mathcal{C}$, illustrated in Figure \ref{fig:casesABCDEFG} on page
\pageref{fig:casesABCDEFG}.  Note that the boundary of the EPD in
these cases runs in tentacles of distinct colors adjacent to the
2--edge loop.  Moreover, note that when we close off the 2--edge loop
to form a brick, these two distinguished tentacles are on the same
side of the brick.  Hence we have type \eqref{type:nononprime-1} in
these cases.

\smallskip

Next consider type $\mathcal{B}$.  The zig-zag at the top left of the
figure showing type $\mathcal{B}$ in Figure \ref{fig:casesABCDEFG} is
schematic, to represent the fact that there may be 0 or more segments
in that zig-zag.  If there are 0 segments in the zig-zag, the arc of
the EPD in the green face may run either upstream or downstream from
the top left.  To prove this proposition, when we have a 2--edge loop
of type $\mathcal{B}$, we need to condition on whether the arc of the
EPD in the green, top left, runs upstream or downstream from this
point.

Suppose first that it runs downstream.  Then by Lemma
\ref{lemma:downstream} (Downstream), it must run downstream until it
terminates.  Notice that by our orientation convention (Lemma
\ref{lemma:EPDtoSquare}, EPD to oriented square), the arc cannot
terminate in the tail of a green tentacle.  Hence the arc $\sigma_1$
in the green must cross the top state circle $C_0$ in the brick of the
2--edge loop before it terminates.  However, notice $\sigma_1$ cannot
cross $C_0$ to the left of the left--most segment of our 2--edge loop,
else it would force the orange segment to terminate, cutting off the
arc in the orange.  Thus the green terminates to the right of that
left--most segment.  This means that the tentacle adjacent to the
right of that left--most segment must be green, and our brick is of
type \eqref{type:nononprime-2} in the statement of the proposition.

Next suppose that we have a 2--edge loop of type $\mathcal{B}$, but
our arc in the top left in the green runs upstream rather than down,
adjacent to a segment $s_1$.  Then we have zero segments in the
zig-zag vertex at the top left of type $\mathcal{B}$, and the state
circle at the top of the brick, call it $C_0$, is connected by $s_1$
to some other state circle $C'$.  Now, consider the arc $\sigma_2$ of
the EPD in the orange tentacle on the right, oriented so that it is
running downstream toward $C_0$.  Either $\sigma_2$ must run
downstream across $C_0$, adjacent to a segment $s_2$ connecting $C_0$
and $C'$, or the arc $\sigma_1$ in the green tentacle must eventually
run downstream to meet $C_0$.  In the first case, we obtain a brick
between $C_0$ and $C'$, with the EPD running over distinctly colored
tentacles on the same side of the brick, and we have a brick of type
\eqref{type:nononprime-1} in the statement of the proposition.  In the
second case, the arc $\sigma_2$ terminates at a vertex on $C_0$, and
we have a brick of type \eqref{type:nononprime-3}.  This finishes the
proof in the case that our 2--edge loop coming from Theorem
\ref{thm:2edgeloop} is of type $\mathcal{B}$.

\smallskip

It remains to show that the proposition holds when our 2--edge loop
is of type $\mathcal{D}$, shown in Figure \ref{fig:casesABCDEFG} on
page \pageref{fig:casesABCDEFG}.  The argument in this case requires
three steps, illustrated in Figure \ref{fig:new-d}.

\begin{figure}
\input{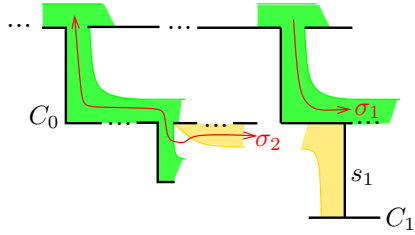}
\caption{Tentacles of an EPD of type $\mathcal{D}$, in the absence of
  non-prime arcs.}
\label{fig:new-d}
\end{figure}

\smallskip

\noindent \underline{Step 1:} Consider the arc $\sigma_2$ that lies in
an orange tentacle at the bottom of Figure \ref{fig:new-d}. We claim
that $\sigma_2$ runs upstream.  To prove this claim, we need to show
that $\sigma_2$ cannot run downstream, or terminate.

Suppose $\sigma_2$ runs downstream, across the state circle $C_0$ in
the bottom of Figure \ref{fig:new-d}. Then, observe that the arc
$\sigma_1$ in the green tentacle on the right of the figure is running
downstream, on the opposite sides of $C_0$.  Since $\sigma_1$ can only
continue downstream until it terminates (by Lemma
\ref{lemma:downstream} (Downstream)), it must terminate immediately
and connect to $\sigma_2$ at an ideal vertex.  But such a vertex would
be oriented green--orange--white (counter--clockwise), contradicting
the orientation convention of Lemma \ref{lemma:EPDtoSquare} (EPD to
oriented square).

Next, suppose that $\sigma_2$ terminates immediately, rather than
running upstream.  In that case, the arc $\sigma_1$ must run
downstream across $C_0$ and immediately connect to meet the tail of
$\sigma_2$ at a vertex.  But this is impossible: the orange tentacle
has only one tail, and this tail already forms a portion of the other
vertex of the EPD, as illustrated in the figure.  This proves the
claim: $\sigma_2$ must run upstream, adjacent to some segment $s_1$
connecting $C_0$ to a state circle $C_1$.

\smallskip

\noindent \underline{Step 2:} Now, consider the arc $\sigma_1$ lying
in the green tentacle on the right of Figure \ref{fig:new-d}.  This
arc must run downstream across $C_0$ by orientation reasons, but it
might either terminate immediately on the opposite side of $C_0$, or
continue adjacent to a segment running to $C_1$.

If $\sigma_1$ terminates, then it meets an orange arc just on the
opposite side of $C_1$.  Lemma \ref{lemma:adjacent-loop} (Adjacent
loop) implies that $\sigma_2$ runs adjacent to some segment $s_2$
connecting $C_0$ to $C_1$, and $s_1$ and $s_2$ form a 2--edge loop of
type \eqref{type:nononprime-3} of the proposition.

Next, suppose that the arc $\sigma_1$ runs downstream.  Then Lemma
\ref{lemma:utility} (Utility) implies that it runs adjacent to a
segment $s_2$ connecting $C_0$ to $C_1$.  If $s_1$ and $s_2$ are
distinct segments, then they form a 2--edge loop of type
\eqref{type:nononprime-1} as in the statement of the proposition, with
arcs of the normal square running on the same side of the brick, in
tentacles of distinct color. This proves the proposition in the case
where $s_1$ and $s_2$ are distinct segments.

\smallskip

\noindent \underline{Step 3:} Suppose that $s_1$ and $s_2$ are the
same segment. Then we will repeat the argument as above and eventually
end up with a brick of type \eqref{type:nononprime-1} or
\eqref{type:nononprime-2}, using induction and finiteness of the graph
$H_A$, as well as primeness.

First, we claim the arc in the orange must again run upstream, for if
it runs downstream we pick up a vertex of the wrong orientation, and
if it terminates, then we get a contradiction to primeness: the arc in
the orange connects across the bottom state circle to the arc in the
green, and they form a loop meeting that state circle just once more,
which gives a loop meeting the diagram twice with crossings on each
side.  Hence the orange arc runs upstream, say adjacent to a segment
$s$.

The green arc either runs downstream, or terminates.  If it
terminates, it meets the tail of another orange tentacle, and 
Lemma \ref{lemma:adjacent-loop} (Adjacent loop) implies that the
orange runs down this tentacle, adjacent to some other segment $s'$.
The segments $s$ and $s'$ form a 2--edge loop.  Notice in this case
that the green must terminate to the right of the segment $s$, else
would cut off the orange tentacle.  Thus the segment $s$ must have a
green tentacle adjacent to it on its right, and we have type
\eqref{type:nononprime-2}.  

Suppose the green arc runs downstream rather than terminating.  Then
it does so by running adjacent to a segment $s'$.  Again $s$ and $s'$
form the desired brick of the proposition, of type
\eqref{type:nononprime-1} if they are distinct.  If not, repeat
verbatim the argument above, starting from the beginning of Step 3.
By induction, we eventually get the brick given by the proposition.
\end{proof}

\section{A four--to--one mapping}

For each essential product disk $E$ in the upper polyhedron,
Proposition \ref{prop:2edgeloop-nononprime} gives a mapping from $E$
to some brick of $H_A$, with the $E$ running through tentacles as in
type \eqref{type:nononprime-1}, \eqref{type:nononprime-2}, or
\eqref{type:nononprime-3}.  Thus, when the EPDs are selected from the
spanning set $E_c$ of Lemma \ref{lemma:spanning-upper} on page
\pageref{lemma:spanning-upper}, Proposition
\ref{prop:2edgeloop-nononprime} gives a function
\begin{equation*}
f: E_c \to \{ \mbox{bricks of type \eqref{type:nononprime-1},
\eqref{type:nononprime-2}, or \eqref{type:nononprime-3}} \}.
\end{equation*}
The goal of this section is to show that the function $f$ is at most
four--to--one.

\begin{define}\label{def:brick-support}
We say that a brick (a pair of segments between the same state circles
of $H_A$) \emph{supports}\index{support (of brick)} an essential
product disk $E \in E_c$ if the function $f$ above maps $E$ to the
given brick.
 \end{define}

\begin{lemma}\label{lemma:nononprime-3not1}
A single brick of $H_A$ cannot support both an EPD of type
\eqref{type:nononprime-1} in one color pair and an EPD of type
\eqref{type:nononprime-3} in a different color pair.
\end{lemma}

\begin{proof}
Let $s$, $s'$ be the segments of the brick. Let $E_3$ be the EPD of
type \eqref{type:nononprime-3}.  For ease of exposition, we will
assume that the shaded faces of $E_3$ are colored green and orange,
with the orientation given by Lemma \ref{lemma:EPDtoSquare} (EPD to
oriented square). Hence, the tentacles of $E_3$ look identical to
those in Figure \ref{fig:nononprime-types}, type \eqref{type:nononprime-3}.

Let $E_1$ be the EPD of type \eqref{type:nononprime-1}, also supported
by the brick of $s$ and $s'$. Note that since each of $s$ and $s'$ is
adjacent to an orange tentacle, one of the shaded faces through which
$E_1$ runs must be orange. Say that the color pair of $E_1$ is
orange--blue, with the blue tentacle adjacent to segment $s'$.

Now, the orange tentacles adjacent to $s, s'$ terminate with their
tails on the heads of green tentacles on the state circle $C_0$.
Since $E_1$ does not meet a green face, the arc $\sigma_1$ of $E_1$
running through the orange tentacle adjacent to $s$ must run
downstream across $C_0$ at a segment attached to $C_0$, to the right
of the point where the orange tentacle terminates in a tail.  Since
the EPD $E_3$ runs through the tail of this tentacle, the arc of $E_1$
running through the same tentacle as the arc of $E_3$ must cross $C_0$
to the right of the point where that arc of $E_3$ crosses it.

On the other hand, the blue tentacle that $E_1$ runs through, adjacent
to a segment $s'$, must be on the right of $s'$.  Hence the arc
$\tau_1$ of $E_1$ running through the blue tentacle crosses $C_0$ to
the right of the arc of $E_3$ running adjacent to that same segment
$s'$. See Figure \ref{fig:3not1}.

\begin{figure}
\input{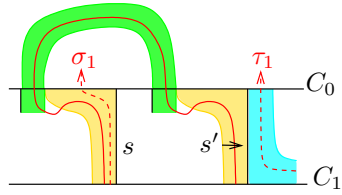}
\caption{When a single brick supports both a disk $E_1$ of type
  \eqref{type:nononprime-1} and a disk $E_3$ of type
  \eqref{type:nononprime-3}, their intersections with a state circle
  $C_0$ must interleave. This will imply that $E_1$ must cut across
  the green (darker shaded) face containing $E_3$, which is a
  contradiction.}
\label{fig:3not1}
\end{figure}

We conclude that $E_1$ and $E_3$ intersect state circle $C_0$ at
interleaving points. We will see that this interleaving implies that
the orange--blue disk $E_1$ must intersect the green shaded face,
which will give a contradiction.

Let $\rho$ be the arc of $E_3$ in the green face. We know, from Figure
\ref{fig:nononprime-types} \eqref{type:nononprime-3}, that $\rho$
crosses $C_0$ at two places, adjacent to segments $s$ and $s'$. If we
orient $\rho$ from $s$ to $s'$, then it crosses $C_0$ first going
upstream, then going downstream. By the Utility Lemma
\ref{lemma:utility}, these are the only intersections of $\rho$ with
$C_0$. Thus the green face separates the two tentacles of $E_1$
adjacent to segments $s$ and $s'$.
 
Now, consider the intersections between state circle $C_0$ and arcs
$\sigma_1, \tau_1$ of $E_1$.  As we have seen, the arc $\sigma_1$ of
$E_1$ crosses $C_0$ running downstream.  By Lemma
\ref{lemma:downstream} (Downstream), $\sigma_1$ keeps running
downstream until it meets a white face $W$. Now, orient $\tau_1$
toward face $W$, and consider the last time that it crosses $C_0$
before reaching face $W$. By Lemma \ref{lemma:parallel-stairs} 
(Parallel stairs), $\tau_1$ must cross $C_0$ running upstream.  But
we already know a place where $\tau_1$ crosses $C_0$, namely in the
tentacle to the right of segment $s'$. By the Utility Lemma
\ref{lemma:utility}, $\tau_1$ cannot cross $C_0$ twice running
upstream. Thus $\tau_1$ must cross $C_0$ to the right of $s'$ and
continue to white face $W$. This is a contradiction, since $W$ lies on
the other side of the green face from $s'$.
\end{proof}

\begin{lemma}\label{lemma:nononprime-3not2}
Suppose that a brick formed by segments $s$ and $s'$ supports an EPD
$E_2$ of type \eqref{type:nononprime-2}, as well as an EPD $E_3$ of
type \eqref{type:nononprime-3} in a different color pair than that of
$E_2$.  Then $E_2$ and $E_3$ must run through all four tentacles
adjacent to the brick.  In particular, this will happen only if the
two segments are adjacent to tentacles of the same two colors.
\end{lemma}

\begin{proof}
As in the proof of Lemma \ref{lemma:nononprime-3not1}, we will assume
that the color pair of $E_3$ is green--orange, and that the brick of
$E_3$ is positioned exactly like the brick of Figure
\ref{fig:nononprime-types} \eqref{type:nononprime-3}, with identical
colors.

Suppose for a contradiction that $E_2$ also runs through orange
tentacles in that brick, but the color pair of $E_2$ is blue--orange.
Consider the state circle $C_0$.  Each orange tentacle of the brick
terminates with its tail at a green tentacle on $C_0$.  Because $E_2$
does not meet the green face, each of the arcs of $\bdy E_2$ in the
orange tentacles must run downstream across $C_0$, to the right of the
point where the orange tentacle terminates.  Thus, as in the proof of
Lemma \ref{lemma:nononprime-3not1}, we conclude that the intersection
points of $\bdy E_2 \cap C_0$ must interleave with the points of of
$\bdy E_3 \cap C_0$.

Let $\sigma$ be the arc of $E_3$ in the orange face. The interleaving
intersections with $C_0$ mean that two points of $E_1 \cap C_0$ lie on
opposite sides of $\sigma$. Thus the arc of $E_1$ in the orange face
must intersect $\sigma \subset E_3$. By Lemma \ref{lemma:marcs} on
page \pageref{lemma:marcs}, it follows that (the normal square of)
$E_1$ must also intersect (the normal square of) $E_3$ in another
shaded face. This contradicts the hypothesis that $E_1$ is
orange--blue while $E_3$ is orange--green.
\end{proof}

\begin{lemma}\label{lemma:atmost4}
Let $D(K)$ be a prime, $A$--adequate diagram with polyhedral
decomposition with no non-prime arcs.  Then any 2--edge loop in $H_A$
supports at most four EPDs in the spanning set $E_c$, from at most two
color pairs.
\end{lemma}

\begin{proof}
Any EPD involves a color pair.  Recall that Lemma
\ref{lemma:ec-tentacle} on page \pageref{lemma:ec-tentacle} implies
that for a fixed color pair, at most two EPDs in $E_c$ between those
colors run over a given segment.  We will show that any brick of $H_A$
can support EPDs in at most two color pairs. Then, the result will
follow from Lemma \ref{lemma:ec-tentacle}.

Denote the segments of the 2--edge loop by $s_1, s_2$.  Note that if
$s_1, s_2$ support EPDs of types \eqref{type:nononprime-1} or
\eqref{type:nononprime-2}, then the color pairs of these EPDs are
determined by the colors of the tentacles adjacent to $s_1$ and $s_2$.
For type \eqref{type:nononprime-3}, the colors of adjacent tentacles
determine one of the two colors in the pair, and the other is
determined by the color of the tentacles meeting the tails of the
tentacles of the first color.

Consider the tentacles adjacent to $s_1, s_2$.  There are four such
tentacles --- one on each side of each segment.  There may be two,
three, or four distinct colors for these tentacles.

Suppose first there are four distinct colors.  Then the 2--edge
loop may only support EPDs of type \eqref{type:nononprime-1} (not
\eqref{type:nononprime-2}, not \eqref{type:nononprime-3}).  Since a
normal square of type \eqref{type:nononprime-1} lies on one side of a
brick of $s_1$ and $s_2$, and since any such brick separates the
tentacles into the same pairs inside/outside, there are at most two
color pairs in this case.

Now suppose there are three distinctly colored tentacles adjacent to
$s_1$ and $s_2$.  We may have an EPD of type \eqref{type:nononprime-2}
or \eqref{type:nononprime-3}, but not both in distinct color pairs, by
Lemma \ref{lemma:nononprime-3not2}.  If there is an EPD of type
\eqref{type:nononprime-3}, then Lemma \ref{lemma:nononprime-3not1}
implies there are none of type \eqref{type:nononprime-1} of distinct
color pairs, hence the only EPDs possible are of the same color pair
of the EPD of type \eqref{type:nononprime-3}.

If there are three distinctly colored tentacles adjacent to $s_1$ and
$s_2$, and we have an EPD of type \eqref{type:nononprime-2}, then all
color pairs must come from the colors adjacent to the two segments.
Label these colors orange, green, and blue, with the colors of the
pair of tentacles of the same color labeled orange.  Potentially, we
might have three color pairs: green--orange, blue--orange, and
green--blue.  However, note that green and blue tentacles must be on
opposite sides of the brick of $s_1$ and $s_2$.  Since they are
distinct colors, only and EPD of type \eqref{type:nononprime-1} could
run through them, but since they are not on the same side of the
brick, that is impossible.  Thus there are only two color pairs in
this case.

Finally, suppose there are only two distinct colors of tentacles
adjacent to $s_1$ and $s_2$, say green and orange.

If there is an EPD of type \eqref{type:nononprime-3}, then it
determines a color pair and there can be no EPD of type
\eqref{type:nononprime-1} with a distinct color pair by Lemma
\ref{lemma:nononprime-3not1}, nor of type \eqref{type:nononprime-2}
through the same tentacles adjacent to $s_1$ and $s_2$ but with a
distinct color pair, by Lemma \ref{lemma:nononprime-3not2}.  There
might be an EPD of type \eqref{type:nononprime-2} and a distinct color
pair which uses the other tentacles, but in that case, Lemma
\ref{lemma:nononprime-3not2} implies there cannot be another of type
\eqref{type:nononprime-3}, hence there are only two possible color
pairs.

Similarly, if we have two EPDs of type \eqref{type:nononprime-3} and
distinct color pairs, then they must use distinct tentacles of the
brick, and there can be no other types of EPDs with distinct color
pairs.

If there is no EPD of type \eqref{type:nononprime-3}, then all EPDs
must be of types \eqref{type:nononprime-1} and
\eqref{type:nononprime-2}, for which the colors in the color pair are
the colors adjacent to the segments $s_1$ and $s_2$.  Hence in this
case, there is just one color pair. 
\end{proof}

\section{Estimating the size of $E_c$}

Now we complete the proof of Theorem \ref{thm:guts-nononprime}.
Lemma \ref{lemma:atmost4} has the following immediate consequence.

\begin{lemma}\label{lemma:atmost4-twist}
Let $e_1, \ldots, e_n$ and $f_1, \ldots f_m$ be edges of $\GA$
(equivalently, segments of the graph $H_A$), all of which connect the
same pair of state circles $C$ and $C'$. Suppose that the $e_i$ belong
to the same twist region, and that the $f_j$ belong to the same twist
region.  Then the collection of \emph{all} $2$--edge loops of the form
$\{ e_i, f_j \}$ supports a total of at most four EPDs in the spanning
set $E_c$.
\end{lemma}

\begin{proof}
Let $\widehat{D}$ be the diagram obtained from $D$ by removing all but
one crossing from every $A$--region of $D$. Under this operation,
$e_1, \ldots, e_n$ become the same edge $e$ of $\GA(\widehat{D})$, and
$f_1, \ldots f_m$ become the same edge $f$ of
$\GA(\widehat{D})$. Furthermore, by Lemma \ref{lemma:remove-bigon} on
page \pageref{lemma:remove-bigon}, there is a one--to--one
correspondence between complex disks of $E_c(D)$ and the complex disks
of $E_c(\widehat{D})$. In particular, a complex EPD in
$E_c(\widehat{D})$ that runs through a tentacle adjacent to the single
edge $e$ corresponds to a complex EPD in $E_c(\widehat{D})$ that runs
through a tentacle adjacent to one of the $e_i$, and similarly for the
$f_j$.

Thus, applying Lemma \ref{lemma:atmost4} to the diagram $\widehat{D}$
gives the desired result for $E_c(D)$.
\end{proof}

Now Theorem \ref{thm:guts-nononprime} will follow immediately from
Theorem \ref{thm:guts-general} on page \pageref{thm:guts-general} and
the following lemma.

\begin{lemma}\index{$m_A$!role in estimating guts}
Let $E_c$ be the spanning set of Lemma \ref{lemma:spanning-upper} on
page \pageref{lemma:spanning-upper}, and let $m_A$ be the diagramatic
quantity defined in Definition \ref{def:ma} on page
\pageref{def:ma}. Then, in the absence of non-prime arcs,
$$0 \: \leq \: ||E_c|| \: \leq \: 8 \, m_A,$$
with equality if and only if $m_A = 0$.
\end{lemma}

\begin{proof}
Consider a pair of state circles $C, C'$ of $\GA$, which are connected
by at least one edge. There are $e(C,C')$ edges of $\GA$ connecting
these circles, which belong to $m(C, C')$ twist regions. Then the
number of bigons between $C$ and $C'$ is $b_A(C,C')$, where
\begin{equation}\label{eq:bigon-count}
b_A(C, C') = e(C, C') - m(C, C').
\end{equation}

Associated to the pair of circles $C$ and $C'$, we construct a planar
surface $S(C,C')$, contained in the projection sphere $S^2$. Take the
disjoint disks in $D^2$ whose boundaries are $C$ and $C'$, and connect
these disks by $m(C, C')$ rectangular bands --- with each band
containing the segments of the corresponding twist
region. Topologically, $S(C,C')$ is a sphere with $m(C,C')$ holes.

Let $D \in E_c$ be an essential product disk that runs through
tentacles between $C$ and $C'$. Then $\bdy D$ is a simple closed curve
in $S(C, C')$. Now, the conclusion of Lemma \ref{lemma:atmost4-twist}
can be rephrased to say that at most four distinct EPDs of $E_c$
running through tentacles between $C$ and $C'$ can have boundaries
that are isotopic in $S(C,C')$. This is because isotopy in $S(C,C')$
is exactly the same equivalence relation as running through tentacles
in the same pair of twist regions.

Recall that a sphere with $m(C,C')$ holes contains at most
$2m(C,C')-3$ isotopy classes of disjoint essential simple closed
curves. Since the disks in $E_c$ are disjoint, and since at most four
of these disks can can have boundaries that are isotopic in $S(C,C')$,
we conclude that there are at most
\begin{equation}\label{eq:sphere-curves}
4 (2m(C, C')-3) \: < \: 8 ( m(C,C') - 1)
\end{equation}
disks in $E_c$ that run through tentacles between $C$ and $C'$. (Note
that if $m(C,C') = 1$, i.e.\ all segments between $C$ and $C'$ belong
to the same twist region, then the left side of
\eqref{eq:sphere-curves} is negative.  But in this case, all EPDs
running through tentacles between $C$ and $C'$ must be simple or
semi-simple, and cannot belong to $E_c$. Thus the estimate is
meaningful precisely when a 2--edge loop between $C$ and $C'$
contributes to $E_c$.  Meanwhile, the right side of
\eqref{eq:sphere-curves} is always non-negative when $C$ and $C'$ are
connected in $\GA$.)

Summing over all pairs of state circles $C, C'$ that are connected by
at least one edge of $\GA$, we obtain

$$
\begin{array}{r c l l}
|| E_c ||
& \leq & 8 \sum_{C,C'} ( m(C,C') - 1) & \mbox{ by \eqref{eq:sphere-curves}} \\
& = & 8 \sum_{C,C'} ( e(C,C') - b_A(C,C') - 1) & \mbox{ by \eqref{eq:bigon-count} } \\
& = & 8 \left( \sum_{C,C'}  e(C,C') - \sum_{C,C'} b_A(C,C') - \sum_{C,C'} 1  \right) & \\
& = & 8 \left( e_A - b_A - e'_A \right) \\
& = & 8\, m_A & \mbox{ by Def \ref{def:ma}}.
\end{array}
$$

Notice that the inequality is sharp precisely when the estimate of
\eqref{eq:sphere-curves} applies at least once in a non-trivial way,
i.e.\ when $m_A > 0$.
\end{proof}

\chapter{Montesinos links}\label{sec:montesinos}
In this chapter, we study state surfaces of Montesinos links, and
calculuate their guts.  Our main result is
Theorem \ref{thm:monteguts}.  In that theorem, we show that for every
sufficiently complicated Montesinos link $K$, either $K$ or its mirror
image admits an $A$--adequate diagram $D$ such that the quantity $|| E_c||$ of
Definition \ref{def:ec} vanishes. Then, it will follow
that $\negeul( \guts(M_A))= \negeul(\GRA)$.

\section{Preliminaries}
We begin by reviewing some classically known facts. A reference for
this material is, for example, Burde--Zieschang \cite[Chapter
12]{burde-zieschang:knots}.
A \emph{rational tangle}\index{rational tangle}\index{tangle!rational} is a pair $(B,L)$
where $B$ is a $3$--ball and $L$ is a pair of arcs in $B$ that are
isotopic to $\bdy B$, with the isotopies following disjoint disks in
$B$. Note that a rational tangle is unique up to homeomorphism of
pairs. Also, a rational tangle $(B,L)$ contains a unique compression
disk that separates the two arcs of $L$.

A \emph{marked}\index{rational tangle!marked}\index{marked rational tangle} rational tangle is an embedding of $(B,L)$ into $\RR^3$, with
$B$ being embedded into the regular neighborhood of a unit square in
$\RR^2$ (called a \emph{pillowcase}) and the four
endpoints of $L$ sent to the four corners of the pillow. For
concreteness, we label these four corners NW, NE, SE, and SW.  A marked
rational tangle specifies a planar projection of $K$ to the unit
square, uniquely up to Reidemeister moves in the square.

Marked rational tangles are in 1--1 correspondence with
\emph{slopes}\index{slope of rational tangle} in $\QQ \cup \{ \infty \}$. This can be
seen in several ways. First, it is well--known that isotopy classes of
essential simple closed curves in a $4$--punctured sphere are
parametrized by $\QQ \cup \{\infty \}$. Thus a rational number
determines the slope of a compression disk in the tangle, and this
disk determines an embedding of the tangle up to isotopy. A more
concrete way to specify the correspondence is to picture the the
pillowcase boundary of $B$ as constructed from the union of two
Euclidean squares. Then, a rational slope $q$ specifies a Euclidean
geodesic that starts from a corner and travels with Euclidean slope
$q$. There are exactly two disjoint arcs with this slope; their union
is $L$.

We adopt the standard convention that rational tangles of slope $0$
and $\infty$ have crossing--free diagrams, and a rational tangle of
slope $1$ projects to a single positive crossing. See Figure
\ref{fig:tangles}.

\begin{figure}
\psfrag{1}{$1$}
\psfrag{0}{$0$}
\psfrag{i}{$\infty$}
\begin{center}
\includegraphics{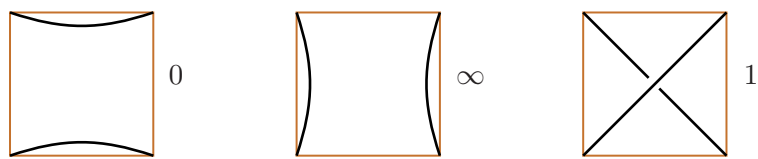}
\end{center}
\caption{Marked rational tangles of slope $0$, $\infty$, and $1$.}
\label{fig:tangles}
\end{figure}

Given marked rational tangles $T_1, T_2$ of slope $q_1, q_2$, one may
form a new tangle, called the \emph{sum}\index{sum of tangles} of
$T_1$ and $T_2$, by joining the NE corner of $T_1$ to the NW corner of
$T_2$ and the SE corner of $T_1$ to the SW corner of $T_2$. One may
check that if $q_i \in \ZZ$ for either $i=1$ or $i=2$, the result is
again a rational tangle of slope $q_1 + q_2$. This is called a
\emph{trivial}\index{sum of tangles!trivial} sum. Otherwise, if $q_i
\notin \ZZ$, the sum tangle will not be rational.

\begin{define}\label{def:num-denom}
For any tangle diagram $T$ with corners labeled NW, NE, SE, and SW,
the \emph{numerator closure}\index{numerator closure} of $T$ is
defined to be the link diagram obtained by connecting NW to NE and SW
to SE by simple arcs with no crossings. The \emph{denominator
  closure}\index{denominator closure} of $T$ is defined to be the
diagram obtained by connecting NW to SW, and NE to SE by simple arcs
with no crossings.

Given marked rational tangles $T_1, \ldots, T_r$, a \emph{Montesinos
  link}\index{Montesinos link} is constructed by taking the numerator
closure of the sum $T_1 + \ldots + T_r$. We also call this the
\emph{cyclic sum}\index{cyclic sum} of $T_1, \ldots, T_r$. See Figure
\ref{fig:Montesinos}. In particular, if $r$ is the number of rational
tangles used, then the Montesinos link $K$ is determined by an
$r$--tuple of slopes $q_1, \ldots, q_r \in \QQ \cup \{ \infty \}$. To
avoid trivial sums, we always assume $q_i \notin \ZZ$ for all $i$.
\end{define}

\begin{figure}[h]
\begin{center}
\includegraphics{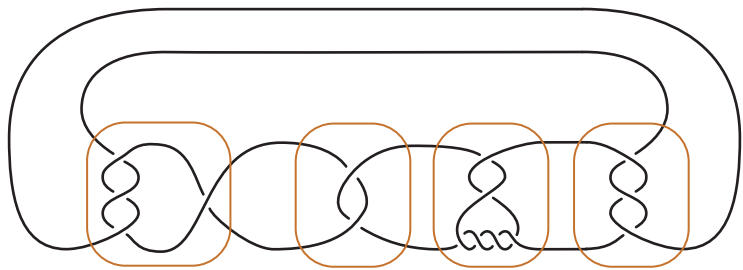}
\end{center}
\caption{A Montesinos knot constructed from rational tangles of slope
  $4/3$, $1/2$, $4/7$, $-1/3$. This diagram is not reduced, according
  to Definition \ref{def:monte-reduced}.}
\label{fig:Montesinos}
\end{figure}

A cyclic sum of two rational tangles is a two--bridge link. Since
two--bridge links are alternating, the guts of the checkerboard
surfaces in this case are known by Lackenby's work
\cite{lackenby:volume-alt}, or equivalently by Corollary \ref{cor:onlybigons}.  Thus we assume that $r \geq 3$. In
addition, since summing with a tangle of slope $\infty$ produces a
composite or split link, and we are interested in prime links, we also
prohibit tangles of slope $\infty$.

Note, in Figure \ref{fig:Montesinos}, that a cyclic permutation of the
tangles $T_1, \ldots, T_r$ produces the same diagram, up to isotopy in
$S^2$. Furthermore, because every rational tangle admits a
rotationally symmetric diagram (Figure \ref{fig:euc-algorithm}, left),
reversing the order of $T_1, \ldots, T_r$ also does not affect the
link. The following theorem of Bonahon and Siebenmann \cite[Theorem
  12.28]{burde-zieschang:knots} implies that the converse is also
true: dihedral permutations of the tangles are essentially the only
moves that will produce the same link.

\begin{theorem}[Theorem 12.28 of  \cite{burde-zieschang:knots}]
\label{thm:monte-classification}
Let $K$ be a Montesinos link obtained as a cyclic sum of $r \geq 3$
rational tangles whose slopes are $q_1, \ldots, q_r \in \QQ \setminus
\ZZ$. Then $K$ is determined up to isomorphism by the the rational
number $\sum_{i=1}^r q_i$ and the vector $((q_1 \mod 1), \ldots, (q_r
\mod 1))$, up to dihedral permutation.
\end{theorem}

Once consequence of Theorem \ref{thm:monte-classification} is that
$r$, the number of rational tangles used to construct $K$, is a link
invariant. This number is called the \emph{length}\index{length (of a Montesinos link)}\index{Montesinos link!length} of $K$. In this
framework, two--bridge links are Montesinos links of length $2$.

For a rational number $q$, the integer vector $[a_0, a_1, \ldots,
  a_n]$ is called a \emph{continued fraction expansion}
\index{continued fraction expansion} of $q$ if
$$ q \: = \: a_0 + 
\cfrac{1}{a_1 + 
\cfrac{1}{a_2 +
\cfrac{1}{\ddots + 
\cfrac{1}{a_n}}}} \:.$$
This continued fraction expansion specifies a tangle diagram, as
follows. Moving from the outside of the unit square toward the inside,
place $a_0$ positive crossings in a horizontal band, followed by $a_1$
crossings in a vertical band, etc., until the final $a_n$ crossings in
a (vertical or horizontal) band connect all four strands of the
braid. The convention is that positive integers correspond to positive
crossings and negative integers to negative crossings. The integer
$a_0$ will be $0$ if $\abs{q}<1$, but all subsequent $a_i$ are
required to be nonzero.
A continued fraction expansion where all nonzero $a_j$ have the same
sign as $q$ determines an alternating diagram of the tangle. See
Figure \ref{fig:euc-algorithm}, where both the alternating diagram and
the continued fraction are constructed via a Euclidean algorithm.

\begin{figure}
\psfrag{35}{$ \cfrac{3}{5} $}
\psfrag{123}{$= \cfrac{1}{1 + \cfrac{2}{3}} $}
\psfrag{112}{$= \cfrac{1}{1 + \cfrac{1}{1 + \cfrac{1}{2}}} \, . $}
\begin{center}
\includegraphics{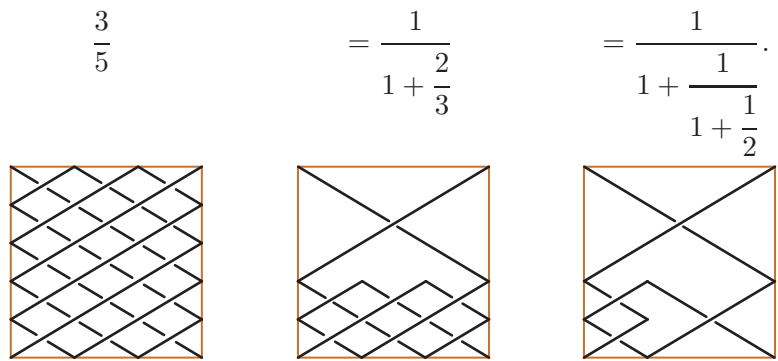}
\end{center}
\caption{Performing the Euclidean algorithm on $p/q$ produces a
  continued fraction expansion and an alternating tangle
  diagram. Here, $3/5 = [0,1,1,2]$.}
\label{fig:euc-algorithm}
\end{figure}

\begin{define}\label{def:monte-reduced}
Let $K$ be a Montesinos link of length $r \geq 3$, obtained as the
cyclic sum of tangles $T_1, \ldots, T_r$ of slope $q_1, \ldots,
q_r$. A diagram $D(K)$ is called a \emph{reduced Montesinos diagram}\index{reduced Montesinos diagram}\index{Montesinos link!reduced diagram}
if it is a cyclic sum of diagrams of the $T_i$, and both of the
following hold:
\begin{enumerate}
\item Either all $q_i$ have the same sign, or $0 < \abs{q_i} < 1$ for
  all $i$.
\item The diagram of $T_i$ comes from a constant--sign continued
  fraction expansion of $q_i$.  If the sign is positive, we say $T_i$
  is a \emph{positive tangle}.\index{positive tangle}\index{tangle!positive, negative}
  Otherwise, $T_i$ is
  a \emph{negative tangle}.\index{negative tangle}
\end{enumerate}
Note that $D(K)$ will be an alternating diagram iff all $q_i$ have the
same sign.
\end{define}

It is an easy consequence of Theorem \ref{thm:monte-classification}
that every (prime, non-split, non-2--bridge) Montesinos link has a
reduced diagram. For example, if $q_i <0$ while $q_j > 1$, one may add
$1$ to $q_i$ while subtracting $1$ from $q_j$. By Theorem
\ref{thm:monte-classification}, this does not change the link
type. Continuing in this fashion will eventually satisfy condition
$(1)$ of the definition.

The significance of reduced diagrams lies in the result, due to
Lickorish and Thistlethwaite, that for prime, non-split Montesinos
links of length $r \geq 3$ the crossing number of $K$ is realized by a
reduced diagram \cite{lick-thistle}.  The proof of this result makes
extensive use of adequacy. In particular, they make the following
observation.

\begin{lemma}[Lickorish--Thistlethwaite \cite{lick-thistle}]\label{lemma:monte-adequate}
Let $D(K)$ be a reduced Montesinos diagram with $r>0$ positive tangles
and $s>0$ negative tangles. Then $D(K)$ is $A$--adequate iff $r \geq
2$ and $B$--adequate iff $s \geq 2$. Since $r+s \geq 3$ in a reduced
diagram, $D$ must be either $A$--adequate or $B$--adequate.
\end{lemma}

Note that if $r=0$ or $s=0$, then $D(K)$ is an alternating diagram,
which is both $A$-- and $B$--adequate. Thus it follows from Lemma
\ref{lemma:monte-adequate} that every Montesinos link is $A$-- or
$B$--adequate. This turns out to be enough to determine the crossing
number of $K$.

In constructing an alternating tangle diagram from a continued
fraction, we had a number of choices, as follows. The integer $a_1$
corresponds to $a_1$ positive crossings in a vertical band --- which
can be at the top or bottom of the band. Similarly, the second integer
$a_2$ corresponds to crossings in a horizontal band --- which can be
at the left or right of the band. For example, in Figure
\ref{fig:euc-algorithm}, the first crossing was chosen to go on the
top of the tangle, and the second crossing was chosen on the left side
of the horizontal band. Reversing these choices still produces a
reduced diagram. However, for our analysis of $I$--bundles and guts,
we will prefer a particular choice.

\begin{define}\label{def:admissible-tangle}
Let $T$ be a rational tangle of slope $q$, where $0 < \abs{q} < 1$. If
$q>0$, we say that an alternating diagram $D(T)$ is
\emph{admissible}\index{admissible} if all the crossings in a vertical
band are at the top of the band, and all the crossings in a horizontal
band are on the right of the band. If $q<0$, we say that an
alternating diagram $D(T)$ is \emph{admissible}\index{admissible} if
all the crossings in a vertical band are at the top of the band, and
all the crossings in a horizontal band are on the left of the
band. For example, the diagram in Figure \ref{fig:euc-algorithm} is
admissible.

A reduced Montesinos diagram $D(K)$ is called
\emph{admissible}\index{admissible}\index{Montesinos link!admissible diagram}
if the sub-diagram $D(T_i)$ is admissible for every tangle of slope $0
< \abs{q_i} < 1$.
\end{define}

For the rest of this chapter, we will assume that $D(K)$ is a reduced,
admissible, Montesinos diagram. This assumption does not restrict the
class of links under consideration, because every reduced diagram can
be made admissible by a sequence of flypes.  We also remark that the
placement of crossings in vertical and horizontal bands implies that
every reduced, admissible diagram is also twist--reduced (see
Definition \ref{primetwist} on page \pageref{primetwist}).

Our goal is to understand the guts of $M_A = S^3
\cut S_A$ corresponding to a reduced, admissible diagram.

\begin{theorem}\label{thm:monteguts}\index{Montesinos link!guts of state surface}\index{guts!for Montesinos links}\index{$\GRA$, $\GRB$: reduced state graph!relation to guts}
Suppose $K$ is a Montesinos link with a reduced admissible diagram
$D(K)$ that contains at least three tangles of positive slope.  Then
$D$ is $A$--adequate, and
$$\negeul( \guts(M_A))=\negeul(\GRA).$$
\end{theorem}

Note that the $A$--adequacy of $D$ follows immediately from Lemma
\ref{lemma:monte-adequate}.  Similarly, if $D(K)$ contains at least
three tangles of negative slope, then it is $B$--adequate and
$$\negeul( \guts(M_B))=\negeul(\GRB).$$

Recall that in Theorem \ref{thm:guts-general} on page
\pageref{thm:guts-general}, we have expressed $\negeul( \guts(M_A))$
in terms of the negative Euler characteristic $\negeul(\GRA)$ and the
number $||E_c||$ of complex disks required to span the $I$--bundle 
of the upper polyhedron.
 Thus, to prove
Theorem \ref{thm:monteguts}, it will suffice to show that $||E_c|| = 0$.  
If $D(K)$ is alternating and twist--reduced, all
$2$--edge loops in $\GA$ belong to twist regions, hence the result
follows by Corollary \ref{cor:onlybigons}. Thus, for the rest of the
chapter, we will assume that $D(K)$ is non-alternating; that is,
$D(K)$ contains at least three tangles of positive slope and at least
one tangle of negative slope.

We will prove the desired statement at the end of the chapter, in
Proposition \ref{prop:monte-main}.  In turn, Proposition
\ref{prop:monte-main} relies on knowing a lot of detailed information
about the structure of shaded faces in the upper polyhedron, along
with their tentacles.  The next section is devoted to compiling this
information.

\section{Polyhedral decomposition}
In this section we will describe the polyhedral decomposition of
$M_A$, in the case of Montesinos diagrams.  Then, we prove several
tentacle--chasing lemmas about the shape of shaded faces in the upper
polyhedron.

\begin{lemma}\label{lemma:monte-poly-regions}
Suppose that a non-alternating diagram $D(K)$ is the cyclic sum of
positive slope tangles $P_1, \ldots, P_r$ and negative slope tangles
$N_1, \ldots, N_s$. Here, the order of the indices indicates that
$P_1$ is clockwise (i.e. west) of $P_2$, but does not give any
information about the position of $P_i$ relative to any $N_j$. Then
the polyhedral regions of the projection plane are as follows:
\begin{enumerate}
\item There is one polyhedral region containing all of the positive
  tangles. The lower polyhedron of this region corresponds to the
  alternating diagram obtained from a cyclic sum of $P_1, \ldots,
  P_r$.
\item Every negative tangle $N_j$ corresponds to its own polyhedral
  region. The lower polyhedron of this region has a diagram coinciding
  with the alternating diagram of the denominator closure of $N_j$.
\end{enumerate}
\end{lemma}

Recall that the numerator and denominator closures of a tangle, as
well as the cyclic sum of several tangles, are defined in Definition
\ref{def:num-denom}.

\begin{proof}
Consider the way in which the state circles of $s_A$ intersect the
individual tangles. For each tangle $T$, its intersection with $s_A$
will contain some number of closed circles, along with exactly two
arcs that connect to the four corners of $T$. When $T$ is a rational
tangle, one can easily check that the non-closed arcs of $s_A \cap T$
will run along the north and south sides of $T$ if its slope is
positive, and along the east and west sides of $T$ if its slope is
negative. See Figure \ref{fig:tangle-resolutions}.

\begin{figure}
\begin{center}
\includegraphics{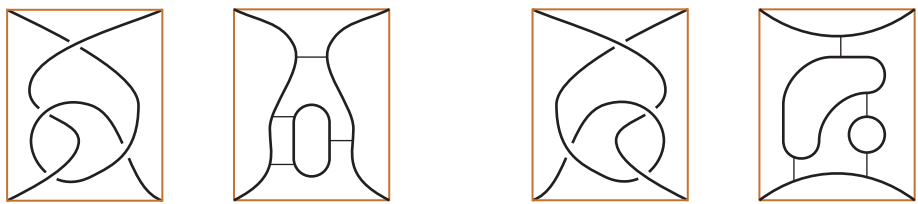}
\end{center}
\caption{Left: A positive tangle and its $A$--resolution. Right: A
  negative tangle and its $A$--resolution.}
\label{fig:tangle-resolutions}
\end{figure}

Now, recall that the polyhedral decomposition described in Chapter
\ref{sec:decomp} proceeds in two steps. In the first step, we cut
$M_A$ along white faces; the resulting polyhedra correspond to the
(non-trivial) regions in the complement of the state circles $s_A$.
Here, all of the positive tangles are grouped into the same
complementary region. On the other hand, every maximal consecutive
sequence of negative tangles $N_i, \ldots, N_{i+k}$ defines its own
complementary region.  See Figure \ref{fig:tangle-decomp}. We call
such a maximal string of negative tangles a \emph{negative
  block}\index{negative block}.

\begin{figure}
\begin{center}
\psfrag{p}{$\oplus$}
\psfrag{m}{$\ominus$}
\includegraphics[width=5in]{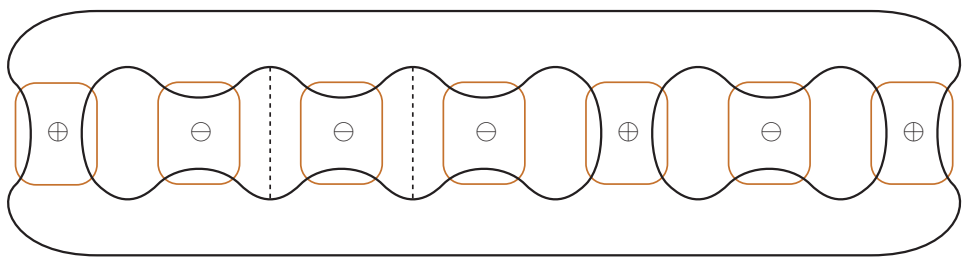}
\end{center}
\caption{Polyhedral regions of $D$ correspond to individual negative
  tangles, as well as the sum of all positive tangles. Dashed lines
  are non-prime arcs. A maximal string of consecutive negative
  tangles, called a \emph{negative block}\index{negative block!example}, is
  shown on the left.}
\label{fig:tangle-decomp}
\end{figure}

The second step of the polyhedral decomposition is the cutting along
non-prime arcs. The region containing all the positive tangles is
prime, and does not need to be cut further. On the other hand, if
negative tangles $N_i$ and $N_{i+1}$ are adjacent in $D$, they will be
separated by a non-prime arc $\alpha$. Thus, at the end of this second
step, every negative tangle corresponds to its own lower polyhedron.

The claimed correspondence between these lower polyhedra and
alternating link diagrams is a direct consequence of Lemma
\ref{lemma:nonprime-3balls} on page \pageref{lemma:nonprime-3balls}.
\end{proof}

\begin{remark}
Examining Figure \ref{fig:tangle-decomp}, along with the shape of
individual positive tangles in Figure \ref{fig:tangle-resolutions},
gives a quick proof of Lemma \ref{lemma:monte-adequate}.
\end{remark}

\begin{lemma}\label{lemma:monte-tentacle}
Let $D(K)$ be a reduced, admissible, non-alternating, $A$--adequate
Montesinos diagram.  Let $s$ be a segment of the graph $H_A$
constructed from $D$. Consider the two tentacles that run along
$s$. Then
\begin{enumerate}
\item\label{item:terminate-now} At least one of these two tentacles
  will terminate immediately downstream from $s$.
\item\label{item:terminate-corner} Both tentacles terminate
  immediately downstream from $s$, unless $s$ is adjacent to the NE or
  SW corner of a positive tangle, or the NW or SE corner of a negative
  tangle.
\end{enumerate}
\end{lemma}

\begin{proof}
Let $C$ and $C'$ be state circles connected by $s$. Note that $s$ is
contained in some rational tangle $T$. Suppose, as a warm-up, that
that $C$ is an innermost circle entirely contained in $T$. Then the
tentacle that runs downstream toward $C$ will terminate immediately
after reaching $C$, and conclusion \eqref{item:terminate-now} is
satisfied. We store this observation for later.

To prove the lemma, we consider three cases, conditioned on whether
$C$ and/or $C'$ are entirely contained in $T$.

\smallskip

\noindent \underline{Case 1:} $C$ and $C'$ are entirely contained in
$T$.  Then, by the above observation, the tentacles that run
downstream to both $C$ and $C'$ terminate immediately. Thus
\eqref{item:terminate-now} and \eqref{item:terminate-corner} both
hold, proving the lemma in this case.

\smallskip

\noindent \underline{Case 2:} $C$ is entirely contained in $T$ but
$C'$ is not. By the above observation, the tentacle that runs
downstream toward $C$ will terminate immediately.  Also, Figure
\ref{fig:tangle-resolutions} shows that a state circle $C'$ that is
not entirely contained in $T$ must constitute the north or south side
of a negative tangle, or else the east or west side of a positive
tangle.

Suppose, for example, that $T$ is a positive tangle and that $C'$
forms the east side of it. Then the tentacle running from $s$ toward
$C'$ will turn north along $C'$. Hence, if there is another segment of
$H_A$ that meets the east side of $T$, further north than $s$, then
the tentacle will need to terminate immediately downstream from
$s$. The only way that this tentacle can continue running downstream
along $C'$ is if $s$ is adjacent to the NE corner of the tangle.

By the same argument, if $T$ is a positive tangle and $C'$ is the west
side of $T$, then the tentacle that runs downstream toward $C'$ will
have to terminate immediately after $s$, unless $s$ is adjacent to the
SW corner of the tangle. Similarly, if $T$ is a negative tangle and
$C'$ is its north or south side, then the tentacle that runs
downstream toward $C'$ will have to terminate immediately after $s$,
unless $s$ is adjacent to the NW or SE corner of the tangle. This
proves the lemma in Case 2.

\smallskip

\noindent \underline{Case 3:} neither $C$ nor $C'$ are contained in
$T$. Then, by the hypotheses of the lemma, these sides of the tangle
must be connected by a segment $s$ of $H_A$. If $T$ is a negative
tangle, then the north and south sides of $T$ belong to the same state
circle. Thus any segment $s$ spanning $T$ north to south would violate
$A$--adequacy. So $T$ must be a positive tangle, whose west side is
$C$ and whose east side is $C'$.

If $s$ is not adjacent to the NE corner of tangle $T$, then the
argument above implies that the tentacle running downstream toward
$C'$ will terminate immediately downstream from $s$. Similarly, the
tentacle running downstream toward $C$ will terminate, unless $s$ is
adjacent to the SW corner of $T$. The only situation in which neither
of these tentacles terminate immediately after $s$ is the one where
$s$ is simultaneously adjacent to the NE and SW corners of the
tangle. But then the tangle $T$ contains a single crossing, and has
slope $+1$, violating Definition \ref{def:monte-reduced} of a reduced
diagram. This proves the lemma.
\end{proof}

\begin{lemma}\label{lemma:monte-stair-length}
Let $D(K)$ be as in Lemma \ref{lemma:monte-tentacle} and let $\gamma$
be a simple arc in a shaded face that starts in an innermost disk of
$H_A$. (See Definition \ref{def:simple-face} on page
\pageref{def:simple-face}.) Then the course of $\gamma$ must satisfy
one of the following:
\begin{enumerate}
\item\label{item:terminate-2steps} $\gamma$ terminates after running
  downstream along 2 or fewer segments of $H_A$.
\item\label{item:terminate-3steps} $\gamma$ terminates after running
  downstream along 3 segments of $H_A$, where the middle segment spans
  a positive tangle east to west.
\item\label{item:run-into-head} $\gamma$ runs downstream along one
  tentacle, through a non-prime switch, and then upstream along one
  tentacle into another innermost disk. In this case, both innermost
  disks belong to consecutive negative tangles.
\end{enumerate}
\end{lemma}

See Figure \ref{fig:directed-spine} on page
\pageref{fig:directed-spine} for a review of tentacles and non-prime
switches. One way to summarize the conclusion of the lemma is that, in
the special case of admissible Montesinos diagrams, right--down
staircases in $H_A$ have at most 3 stairs.

\begin{proof}
By hypothesis, $\gamma$ starts in an innermost disk of $H_A$. Let
$s_1$ be the segment of $H_A$ along which $\gamma$ runs out of this
innermost disk. If $s_1$ is not adjacent to the NE or SW corner of a
positive tangle, or the NW or SE corner of a negative tangle, then
Lemma \ref{lemma:monte-tentacle} implies $\gamma$ must terminate
immediately, and conclusion \eqref{item:terminate-2steps} holds.

Suppose $s_1$ belongs to a positive tangle $T_1$, and is adjacent to
its NE corner. By Lemma \ref{lemma:monte-poly-regions}, the polyhedral
region containing $T_1$ does not have any non-prime arcs, hence
$\gamma$ cannot enter a non-prime switch. Thus the only way $\gamma$
can continue running downstream is by entering a negative tangle $T_2$
along a segment $s_2$ on the north side of $T_2$. As we have already
seen, a negative tangle cannot be spanned north to south by a single
segment (otherwise, it would have integer slope and violate Definition
\ref{def:admissible-tangle}). Thus $s_2$ connects to an innermost
circle, and $\gamma$ must terminate after two segments.

Similarly, if $s_1$ is adjacent to the SW corner of a positive tangle
$T_1$, then $\gamma$ can only continue running downstream by entering
a negative tangle $T_2$ along a segment $s_2$ in on the south side of
$T_2$. After this, $\gamma$ must terminate after two segments, and
conclusion \eqref{item:terminate-2steps} holds.

Suppose $s_1$ belongs to a negative tangle $T_1$, and is adjacent to
its SE corner. After running along $s_1$, $\gamma$ can enter a
non-prime switch, or continue running downstream through tentacles.
Suppose, first, that $\gamma$ enters a non-prime switch.  All the
non-prime arcs of $D(K)$ separate consecutive negative tangles, as in
Figure \ref{fig:tangle-decomp}.  After entering a non-prime switch
from the SE corner of $T_1$, $\gamma$ can run downstream to the NE
corner of $T_1$ and terminate, run downstream to the SW corner of the
next negative tangle $T_2$ and terminate, or run upstream to the NW
corner of the negative tangle $T_2$.  After this, $\gamma$ is forced
to run upstream into an innermost disk, as in conclusion
\eqref{item:run-into-head}.

Next, suppose that after running along $s_1$, $\gamma$ continues
downstream through tentacles.  Thus $\gamma$ must enter a positive
tangle $T_2$ through a segment $s_2$ on the west side of $T_2$.  If
$s_2$ leads to an innermost disk, then $\gamma$ must terminate after
two steps.  Alternately, if $s_2$ spans the positive tangle $T_2$ east
to west but is not adjacent to the NE corner of $T_2$, then $\gamma$
must also terminate after two steps.  Finally, if $s_2$ spans the
positive tangle $T_2$ east to west and is adjacent to the NE corner of
$T_2$, then we can repeat the analysis of positive tangles at the
beginning of the proof (with $s_2$ playing the role of $s_1$).  In
this case, $\gamma$ must terminate after at most three steps.  If
there are indeed 3 steps, then the middle segment $s_2$ spans a
positive tangle from west to east, and conclusion
\eqref{item:terminate-3steps} holds.

Finally, suppose $s_1$ belongs to a negative tangle $T_1$, and is
adjacent to its NW corner.  Then we argue as above, with all the
compass directions reversed, and reach the same conclusions.
\end{proof}

\begin{lemma}[Head locator\index{Head locator lemma}]\label{lemma:head-locator}
Let $D(K)$ be a reduced, admissible, $A$--adequate Montesinos diagram
that is non-alternating.  If a shaded face in the upper polyhedron of
the polyhedral decomposition of $M_A$ meets an innermost disk (has a
head) in a positive tangle, then it meets no other innermost disks
elsewhere.  If a shaded face meets an innermost disk in a negative
tangle, then it may meet another innermost disk in the same negative
block, but it will not meet any innermost disk in any other negative
block.
\end{lemma}

See Figure \ref{fig:tangle-decomp} for the notion of a negative
block. Recall as well that a \emph{head}\index{head!of shaded face}\index{shaded face!head}\index{shaded face!innermost disk} of a shaded face
is an innermost disk, as in Figure \ref{fig:shaded-pieces} on page
\pageref{fig:shaded-pieces}.

\begin{proof}
A shaded face meets more than one innermost disk only when it runs
through a non-prime switch.  Since all non-prime arcs in our diagrams
lie inside of negative blocks, only innermost disks inside the same
negative block can belong to the same shaded face.
\end{proof}

\begin{lemma}\label{lemma:mont-topandbottom}
Let $D(K)$ be a reduced, admissible, non-alternating Montesinos
diagram with at least three positive tangles.  Then tentacles of the
same shaded face in the upper polyhedron of the polyhedral
decomposition of $M_A$ cannot run across the north and the south of
the outside of a single negative block.
\end{lemma}

\begin{proof}
Note that by Lemma \ref{lemma:monte-adequate} $D(K)$ is $A$--adequate.
The tentacle across the north of the outside of a negative block runs
from a segment in the positive tangle directly to the west of that
negative block. See Figure \ref{fig:tangle-decomp}. Hence, by Lemma
\ref{lemma:head-locator} (Head locator), its head is either inside
that positive tangle, or, if it came from a segment running east to
west in that positive tangle, its head will be inside the negative
block directly to the west.

Similarly, the tentacle across the south of the outside of a negative
block runs from a segment in the negative tangle directly to the east
of that negative block, hence has its head inside the positive tangle
directly to the east, or the negative block directly to the east.

Note that the positive tangles directly to the east and west cannot
agree, by the assumption that our diagram has at least three positive
tangles. Similarly, the negative blocks to the east or west cannot
agree.  Hence the conclusion follows.
\end{proof}

\begin{lemma}\label{lemma:mont-outsidenegative}
Suppose that $D(K)$ is a reduced, admissible Montesinos diagram with
at least three positive tangles that is not alternating.  Any tentacle
in the polyhedral decomposition of $M_A$ running over the outside of a
negative block cannot have a head inside that negative block.
\end{lemma}

\begin{proof}
The head of a tentacle running over the outside of a negative block
lies in the positive tangle or the negative block directly to the
west, in case the tentacle runs across the north, or in the positive
tangle or negative block directly to the east, in case the tentacle
runs across the south.  Then the result follows from Lemma
\ref{lemma:head-locator} (Head locator).
\end{proof}

\section{Two-edge loops and essential product disks}
Next, we study EPDs in the upper polyhedron of the polyhedral
decomposition of $M_A$.  Recall that, by Corollary \ref{cor:epd-loop},
every EPD in the upper polyhedron $P$ must run through tentacles
adjacent to a 2--edge loop in $\GA$.  In Lemma
\ref{lemma:monte-2edge-loop} below, we show that these 2--edge loops
can be classified into three different types.  Most of our attention
will be devoted to particular type of 2--edge loop, depicted in Figure
\ref{fig:neg-tangle}.

\begin{lemma}\label{lemma:monte-2edge-loop}
Let $D(K)$ be a reduced, admissible, $A$--adequate, non-alternating  
Montesinos diagram. Let $C, C'$ be a pair of state circles of
$s_A$. These circles are connected by multiple segments of $H_A$
(corresponding to a two-edge loop of $\GA$) if and only if one of the
following happens:
\begin{enumerate}
\item\label{item:bigon-loop} $C$ and $C'$ co-bound one or more bigons
  in the short resolution of a twist region, which is entirely
  contained in a tangle. See Figure \ref{fig:twist-resolutions}.
\item\label{item:neg-loop} $C$ is contained inside a negative tangle
  $N_i$ of slope $-1 < q \leq -1/2$, and is connected by segments of
  $H_A$ to the state circle $C'$ that runs along the north and south
  of $N_i$. See Figure \ref{fig:neg-tangle}.
\item\label{item:pos-loop} There are exactly two positive tangles
  $P_1$ and $P_2$, and $C,C'$ are the state circles that run along the
  east and west sides of these tangles.
\end{enumerate}
\end{lemma}

\begin{proof}
If the circles $C,C'$ satisfy one of the conditions of the lemma, it
is easy to see that they will be connected by two or more segments of
$H_A$. To prove the converse, suppose that $C$ and $C'$ are connected
by two segments of $H_A$. Each of these segments corresponds to a
crossing of the diagram $D$, and belongs to a particular rational
tangle. We consider several cases, conditioned on how $C$ and $C'$
intersect this common tangle.

First, suppose that $C$ and $C'$ are both closed loops inside a tangle
$T$.  Since the tangle diagram is alternating, each of these state
circles is innermost. Thus there is a loop in the projection plane
that runs through the regions bounded by $C$ and $C'$, and intersects
the projection of $K$ exactly at the two crossings where $C$ meets
$C'$. Then these crossings are twist--equivalent. Since a reduced,
admissible Montesinos diagram must be twist--reduced, the two
crossings are connected by one or more bigons. Thus conclusion
\eqref{item:bigon-loop} holds.

Next, suppose that $C$ is entirely contained in a tangle $T$, and that
$C' \cap T$ consists of one or two arcs.  If the two segments of $H_A$
connect $C$ to the same arc on the side of a tangle, then again the
corresponding two crossings are twist--equivalent, and must be
connected by one or more bigons.  Thus conclusion
\eqref{item:bigon-loop} holds.  If the two segments of $H_A$ connect
$C$ to opposite sides of the tangle, then $C'$ must contain both of
those arcs (east and west in the case of a positive tangle, north and
south in the case of a negative tangle).  We investigate this
possibility further.

If $T = P_i$ is a positive tangle, then Figure \ref{fig:tangle-decomp}
shows that the state circle on the east side of $P_i$ also runs along
the west side $P_{i+1}$, but is disjoint from every other positive
tangle. Thus the east and west sides of $P_i$ belong to the same state
circle only if there is exactly one positive tangle --- but this
violates $A$--adequacy, by Lemma \ref{lemma:monte-adequate}.

If $T$ is a negative tangle, then Figure \ref{fig:tangle-decomp} shows
that the north and south sides of $T$ will indeed belong to the same
state circle $C'$. Now, let $q<0$ be the slope of $T$, and let
$[a_0,a_1, \ldots, a_n]$ be the continued fraction expansion of
$q$. Since the diagram $D$ is reduced, we have $q \in (-1,0)$, hence
$a_0 = 0$. Moving from the boundary of the tangle inward, the first
crossings will be $\abs{a_1}$ negative crossings in a vertical band.
If $\abs{a_1} \geq 3$, or if $\abs{a_1} = 2$ and $\abs{a_2} > 0$, then
the north and south sides of $T$ will not connect to the same state
circle in $T$.  Note that these conditions on $a_1$ and $a_2$ are
describing exactly the rational numbers $\abs{q} < 1/2$. On the other
hand, if $-1 < q \leq -1/2$, there is exactly one circle $C$ that
connects to the north and south sides of the tangle. This state circle
is the boundary of an innermost disk $I$ depicted in Figure
\ref{fig:neg-tangle}. Thus conclusion \eqref{item:neg-loop} holds.

\begin{figure}
 \input{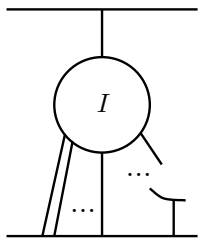}
 \caption{The form of a negative tangle with a 2--edge loop of type
   \eqref{item:neg-loop} of Lemma \ref{lemma:monte-2edge-loop}.  There
   is exactly one segment at the north, and one or more parallel
   segments on the south--west.  The south--east of the tangle may or
   may not have additional segments and state circles.
  The segment on the north of innermost disk $I$, plus a segment on the south $I$, forms a \emph{bridge}\index{bridge} from the north to the south of the tangle.}
 \label{fig:neg-tangle}
\end{figure}

Finally, suppose that neither $C$ nor $C'$ is entirely contained in a
single tangle. Then each segment of $H_A$ that connects $C$ to $C'$
must span all the way from the north to the south sides of a tangle,
or from the east to the west. Since the diagram $D(K)$ is reduced and
non-alternating, every tangle has slope $\abs{q_i} < 1$, hence no
single edge of $H_A$ can span a tangle from north to south. The
remaining possibility is that each segment of $H_A$ that connects $C$
to $C'$ spans a positive tangle east to west. If these two segments
lie in the same tangle $T$, the corresponding crossings are
twist--equivalent, hence conclusion \eqref{item:bigon-loop} holds. If
these two segments lie in two different positive tangles $P_1$ and
$P_2$, and the east and west sides of these tangles belong to the same
state circles $C, C'$, then $P_1$ and $P_2$ must be the only positive
tangles in the diagram. Thus conclusion \eqref{item:pos-loop} holds.
\end{proof}

Lemma \ref{lemma:monte-2edge-loop}, combined with Corollary
\ref{cor:epd-loop}, will allow us to find and classify EPDs in the
polyhedral decomposition for Montesinos links. Looking over the
conclusions of Lemma \ref{lemma:monte-2edge-loop}, we find that
two-edge loops of type \eqref{item:bigon-loop} are very standard, and
easy to deal with using Lemma \ref{lemma:remove-bigon}. Loops of type
\eqref{item:pos-loop} will be ruled out once we assume that $D(K)$ has
at least three positive tangles. Thus most of our effort is devoted to
studying EPDs that run over a two-edge loop of type
\eqref{item:neg-loop}.

It is worth taking a closer look at negative tangles that support a
two-edge loop of \eqref{item:neg-loop}.  The $A$--resolution of such a
negative tangle is illustrated in Figure \ref{fig:neg-tangle}.  Note
in particular that there is exactly one segment connecting the
innermost disk $I$ to the outside of the negative block at the north
of the tangle.  On the south, one or more parallel segments connects
the innermost disk to the outside of the negative block on the south,
and these segments are all at the south--west of the diagram.  The
portion of the diagram on the south--east can have additional state
circles and edges, or not.  The 2--edge loop runs over the single
segment in the north and one of the parallel strands in the
south--west.

\begin{lemma}\label{lemma:mont-innermost-color}
Let $N_i$ be a negative tangle in a reduced, $A$--adequate, Montesinos
diagram $D(K)$, with at least three positive tangles. Let $E$ be an
essential product disk in the upper polyhedron, which runs over a
2--edge loop of that spans $N_i$ north to south, as in Figure
\ref{fig:neg-tangle}. Then $\bdy E$ must run through the innermost
disk $I$ shown in Figure \ref{fig:neg-tangle}. Furthermore, $\bdy E$
must run adjacent to the 2--edge loop through at least one tentacle
whose head is the innermost disk $I$.
\end{lemma}

\begin{proof}
By Lemma \ref{lemma:monte-tentacle}, all tentacles that run downstream
toward $I$ must terminate upon reaching $I$. Thus, each time $\bdy E$
runs along a segment that connects $I$ to the north or south sides of
the tangle, it must either pass through the tentacle that runs
downstream out of $I$ (hence, through the innermost disk $I$), or
through a tentacle that runs downstream toward $I$ (hence, into
$I$). In either case, the disk $E$ must run through $I$.

Next, suppose for a contradiction that on both the north and south
sides of the tangle, $\bdy E$ runs in tentacles that terminate at
$I$. Since $\bdy E$ intersects only two shaded faces, one of which has
a head at $I$, the two tentacles running toward $I$ from the north and
south must belong to the same shaded face. But then the heads of these
two tentacles are both attached to the outside of the negative block,
one on the north and one on the south, and so the negative block must
have tentacles of this same color both on the north and on the south.
This contradicts Lemma \ref{lemma:mont-topandbottom}.  Therefore,
$\bdy E$ must run downstream out of $I$, either on the north or on the
south (or both).
\end{proof}

\section{Excluding complex disks}

We are now ready to prove the following proposition. As remarked
earlier, the statement that $||E_c || = 0$ and the upper polyhedron
contains no complex disks, combined with Theorem
\ref{thm:guts-general}, immediately implies the non-alternating case
of Theorem \ref{thm:monteguts}.

\begin{prop}\label{prop:monte-main}
Suppose that $D(K)$ is a reduced, admissible, non-alter\-nating
Montesinos link diagram with at least three positive tangles.  Then
every EPD in the upper polyhedron is either parallel to a white bigon
face (simple), or parabolically compresses to bigon faces
(semi-simple). In particular, in the terminology of Lemma
\ref{lemma:spanning-upper}, $E_c = \emptyset$.
\end{prop}

\begin{proof}
Let $E$ be an essential product disk in the upper polyhedron. By
Corollary \ref{cor:epd-loop}, $\bdy E$ must run over tentacles
adjacent to a 2--edge loop in $\GA$. With the hypothesis that $D(K)$
has at least three positive tangles, Lemma
\ref{lemma:monte-2edge-loop} implies that every 2--edge loop in $\GA$
is of type \eqref{item:bigon-loop} or \eqref{item:neg-loop}.

Type \eqref{item:bigon-loop} loops correspond to crossings in a single
twist region, in which the all--$A$ resolution is the short resolution
(see Figure \ref{fig:twist-resolutions} on page
\pageref{fig:twist-resolutions}). Note that by Lemma
\ref{lemma:remove-bigon}, removing all the bigons in the $A$--twist
regions does not affect the number $|| E_c ||$ of complex disks in the
spanning set of the upper polyhedron. But if all bigons in
$A$--regions are removed, the only remaining two--edge loops will be
of type \eqref{item:neg-loop}, spanning a negative tangle north to
south. Thus, if we can show that every EPD running over a type
\eqref{item:neg-loop} loop is simple or semi-simple, it will follow
that the same conclusion holds for type \eqref{item:bigon-loop} loops
as well.

For the remainder of the proof, we assume that $E$ is an essential
product disk, such that $\bdy E$ runs over tentacles adjacent to a
2--edge loop that spans a negative tangle north to south, as in Figure
\ref{fig:neg-tangle}. Then, by Lemma \ref{lemma:mont-innermost-color},
$\bdy E$ must run through an innermost disk $I$, as in Figure
\ref{fig:neg-tangle}. We will show that $E$ is either parallel to a
bigon face, or parabolically compressible to a collection of bigon
faces.

Following the setup of Chapter \ref{sec:epds}, color the shaded faces
met by $E$ orange and green,\footnote{Note: For versions of the
  monograph in grayscale, orange faces will appear in the figures as
  light gray, green as darker gray.} so that the shaded face containing the
innermost disk $I$ is green. As in Lemma \ref{lemma:EPDtoSquare} (EPD
to oriented square), we may pull $\bdy E$ off the shaded faces,
forming a normal square. Note that under the orientation convention of
Lemma \ref{lemma:EPDtoSquare} (EPD to oriented square), any arc of the
normal square on a white face cuts off a vertex at the head of a green
tentacle and at the tail of an orange tentacle.

We consider the possible locations of the ideal vertices of
$E$. Equivalently, we consider the possible white faces into which
$\bdy E$ has been pulled.  Using Lemma \ref{lemma:monte-stair-length},
we may enumerate the possible locations for an ideal vertex of $E$.
These are shown in Figures \ref{fig:green-enumerate123},
\ref{fig:green-enumerate4}, and \ref{fig:green-enumerate5}.

\begin{enumerate}
\item\label{green-monte:case1} An ideal vertex of $\bdy E$ may appear
  on the innermost disk of Figure \ref{fig:neg-tangle} itself. This
  means the vertex is at the head of a tentacle running north or south
  out of the innermost disk $I$.
  
  \smallskip
  
\item\label{green-monte:case2} An ideal vertex of $\bdy E$ may appear
  at the head of a tentacle running out of the negative block
  containing $I$, if $\bdy E$ runs downstream from $I$ to the next
  adjacent positive tangle.
  
  \smallskip

\begin{figure}
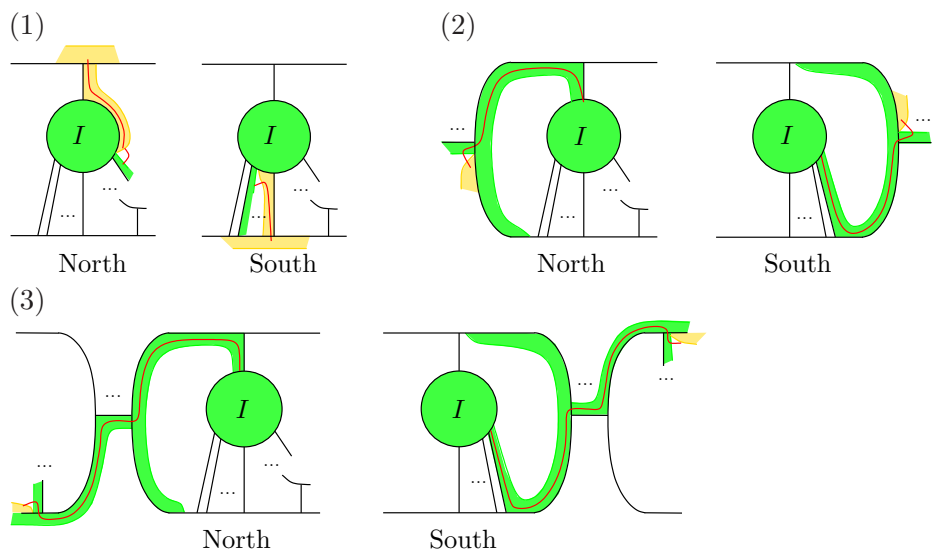

 \begin{center}
 \begin{tabular}{lcl}
   \eqref{green-monte:case1} & \hspace{.2in} & \eqref{green-monte:case2} \\
 \input{figures/neg-tangle-green1.pstex_t} & &
 \input{figures/neg-tangle-green2.pstex_t} \\
 \multicolumn{3}{l}{\eqref{green-monte:case3}} \\
 \multicolumn{3}{l}{\input{figures/neg-tangle-green3.pstex_t}}
 \end{tabular}
 \end{center}

 \caption{Ideal vertices of an EPD, of types
   \eqref{green-monte:case1}, \eqref{green-monte:case2}, and
   \eqref{green-monte:case3}.  In each type, the boundary of the EPD,
   shown in red, and can run north or south from the innermost disk.
   Both cases are shown.}
 \label{fig:green-enumerate123}
\end{figure}

\item\label{green-monte:case3} If $\bdy E$ runs downstream from $I$ to
  the next adjacent positive tangle, then across a segment spanning
  the positive tangle east to west, and then downstream across the
  outer state circle of the next negative block, the vertex may appear
  on the next adjacent negative block.

\smallskip

\begin{figure}
 \begin{tabular}{l}
 \eqref{green-monte:case4} \\
 \input{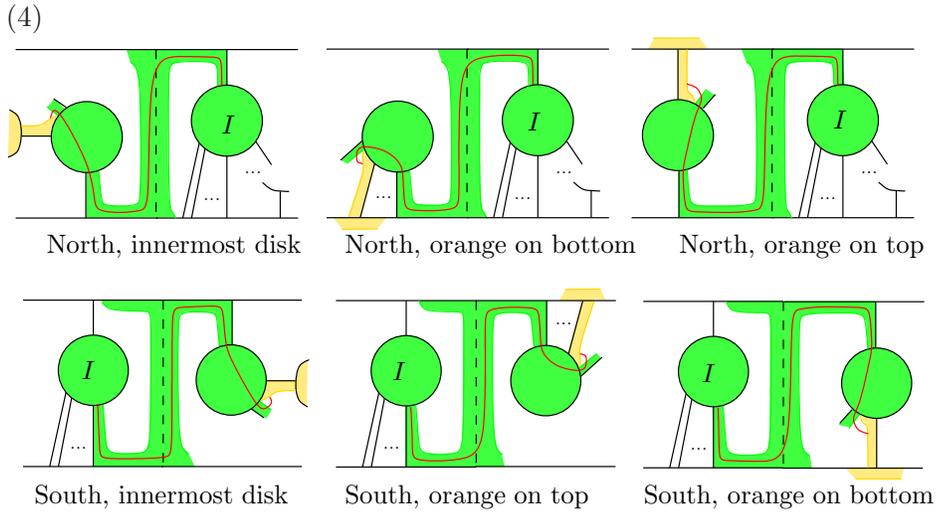}
 \end{tabular}
 \caption{Ideal vertices of an EPD, of type \eqref{green-monte:case4}.
   Note the boundary of the EPD, shown in red, can run north or south.
   The orange tentacle it meets can either come from an innermost
   disk, from the south, or from the north of the negative block.  All
   possibilities are shown.}
 \label{fig:green-enumerate4}
\end{figure}

\item\label{green-monte:case4} If $\bdy E$ runs downstream from $I$,
  across a non-prime arc, and then upstream, then it will run through
  an innermost disk in the next adjacent negative tangle.  The vertex
  may appear on this innermost disk.  Note the corresponding orange
  tentacle will either come from an innermost disk inside the negative
  tangle, or from a tentacle across the north or south of the negative
  block.  All three of these possibilities are shown in Figure
  \ref{fig:green-enumerate4}.

\begin{figure}
 \begin{tabular}{l}
   \eqref{green-monte:case5} \\
   \input{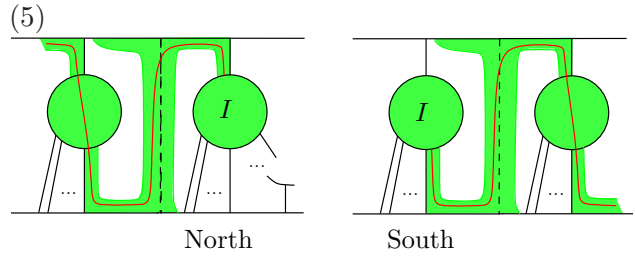}
 \end{tabular}
 \caption{Ideal vertices of an EPD, of type \eqref{green-monte:case5}.
   The boundary of the EPD, shown in red, can run north or south.
   Both cases are shown.}
 \label{fig:green-enumerate5}
\end{figure}

\smallskip

\item\label{green-monte:case5} If $\bdy E$ runs across a non-prime
  arc, then upstream into an innermost disk $J$, and the vertex does
  not appear on this innermost disk, then $\bdy E$ must run downstream
  again from $J$. In fact, to obtain the correct orientation near an
  ideal vertex of $E$, its boundary must run downstream for at least
  two stairs in a staircase starting at $J$. Thus, by Lemma
  \ref{lemma:monte-stair-length}, $\bdy E$ runs along a second 2--edge
  loop, with an innermost disk at $J$. (See Figure
  \ref{fig:green-enumerate5}.) Exiting this 2--edge loop, one of the
  vertices of types (\ref{green-monte:case2}),
  (\ref{green-monte:case3}), (\ref{green-monte:case4}), or
  (\ref{green-monte:case5}) must occur.  Note that between the pairs
  of 2--edge loops is a collection of bigons.  We will handle this
  last type by induction on the number of negative tangles in a
  negative block.
\end{enumerate}

Now, each EPD has two vertices.  From the green innermost disk $I$,
$\bdy E$ runs through the green shaded face in two directions (north
and south) toward these two vertices.  Each must be one of the above
enumerated types.  We consider all the combinations of these types of
vertices.

\smallskip

{\underline{ Type (\ref{green-monte:case1}) and type
    (\ref{green-monte:case1})}:} This combination cannot happen. For
if both vertices lie on the given innermost disk, then $\bdy E$ does
not run through any green tentacles exiting the innermost disk, which
contradicts Lemma \ref{lemma:mont-innermost-color}.

\smallskip

{\underline{ Type (\ref{green-monte:case1}) and type
    (\ref{green-monte:case2})}:} In this case, type
(\ref{green-monte:case2}) implies the negative tangle containing the
innermost disk $I$ is to the far west or far east of the negative
block, and the green tentacle leaves this negative tangle and wraps
around to the side of the adjacent positive tangle.  Note that
although $\bdy E$ meets a vertex here, the tentacle itself must
continue until it terminates at the same negative tangle, forming a
white bigon face. See Figure \ref{fig:case1-2}.

Next, note that since one of the vertices is of type
(\ref{green-monte:case1}), a tentacle across the outside of the given
negative block must be orange (either top or bottom, depending on
whether the EPD runs through an orange tentacle of the 2--edge loop on
the north or south).  Since the other vertex, which lies on the
outside of the negative block, must also meet an orange tentacle,
Lemma \ref{lemma:head-locator} (Head Locator) implies the only
possibility is that the other vertex meets the same orange tentacle,
and $\bdy E$ runs through this orange tentacle connecting the
vertices.  In this case, $\bdy E$ bounds only bigon(s), and $E$ is
parabolically compressible to bigons.

\begin{figure}
 \input{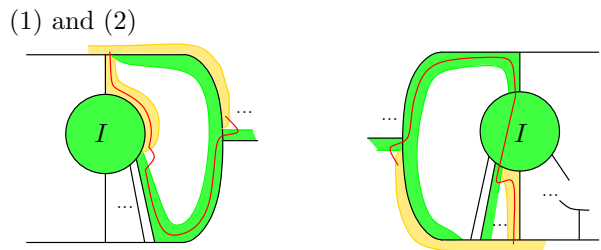}
 \caption{Combining vertices of type \eqref{green-monte:case1} and
   \eqref{green-monte:case2}. Left: type \eqref{green-monte:case1}
   north and \eqref{green-monte:case2} south; right: type
   \eqref{green-monte:case1} south and \eqref{green-monte:case2}
   north.}
 \label{fig:case1-2}
\end{figure}

\smallskip

{\underline{ Type (\ref{green-monte:case1}) and type
    (\ref{green-monte:case3})}:} This case cannot occur by the Head
Locator Lemma \ref{lemma:head-locator}; for, the head of the orange in
type (\ref{green-monte:case1}) would lie in the positive tangle or
negative block adjacent to one side, and the head of the orange in
type (\ref{green-monte:case3}) would lie in the negative block to the
opposite side.

\smallskip

{\underline{ Type (\ref{green-monte:case2}) and type
    (\ref{green-monte:case2})}:} In this case, tentacles from the
north and south of the negative tangle run along the sides of the
positive tangles to the east and west of the negative block.  Lemma
\ref{lemma:head-locator} (Head Locator) puts serious restrictions on
the diagram from here. (See Figure \ref{fig:case2-2}.)  In particular,
by Lemma \ref{lemma:head-locator}, any head(s) of the orange face must
lie in the same positive tangle, or in the same negative block.  Note
that when the vertex on the west of the negative block is of this
type, then the orange tentacle it meets must either

\begin{enumerate}[(a)]
\item run across the south of the given negative block, if the vertex
  is at the very south-east of the positive tangle, and have its head
  in the positive tangle or negative block to the east, or
\item come from an innermost disk in that positive tangle to the west
  of our negative block, or
\item come from an innermost disk inside the negative block to the
  west of our original negative block.
\end{enumerate}

Three similar options hold for the east.  Lemma
\ref{lemma:head-locator} (Head locator) implies that only one of two
possibilities may occur: on the west, the vertex lies at the
south--east tip of the positive tangle, and the orange head(s) are in
the positive tangle or negative block to the east, or on the east the
vertex lies at the north--west tip of the positive tangle, and all
orange head(s) are in the positive tangle or negative block to the
west.  We argue that $E$ parabolically compresses to bigons.  Both
cases are similar; we go through the case that the head of the orange
lies to the east (Figure \ref{fig:case2-2}, left).

First, note that since each innermost disk of a positive tangle has a
distinct color, the two segments attached to orange tentacles that
meet vertices of $\bdy E$ must have a head in the same orange
innermost disk in the positive tangle.  Thus the segment at the
south--west of the positive tangle to the east, which must have orange
on one side, must be attached to the same state circles as the segment
near the vertex on the east of the negative block.  Then these
segments are twist--equivalent, hence bound a chain of bigons Hence,
$\bdy E$ must run from the vertex on the west of the negative block,
through the tentacle across the south, up the segment at the
south--west of the positive tangle to the east, then encircle bigons,
and connect to the vertex on the east of the negative block.  The
diagram must be as shown in Figure \ref{fig:case2-2}, and the EPD is
parabolically compressible to bigons.

\begin{figure}
 \input{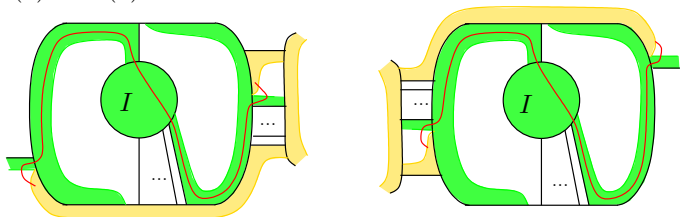}
 \caption{Possibilities for type \eqref{green-monte:case2} and
   \eqref{green-monte:case2}.}
 \label{fig:case2-2}
\end{figure}

\smallskip

{\underline{ Type (\ref{green-monte:case2}) and type
    (\ref{green-monte:case3})}:} The appearance of a type
(\ref{green-monte:case3}) vertex forces the head(s) of the orange
shaded face to lie in the negative block adjacent to one side, which,
just as above in the case (\ref{green-monte:case2}) and
(\ref{green-monte:case2}), puts restrictions on the portion of the
diagram with the vertex of type (\ref{green-monte:case2}).  The
argument is similar if the orange innermost lies to the east or to the
west.  For ease of explanation, we go through the case that the orange
lies to the east.  The result is illustrated in Figure
\ref{fig:case2-3}.

\begin{figure}
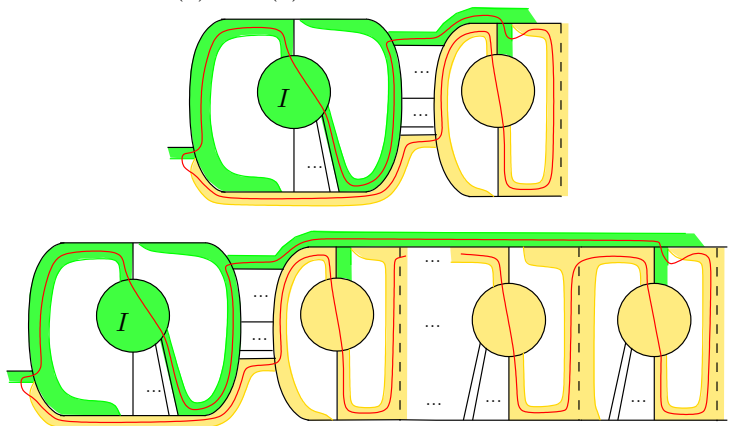

\begin{center}
\input{figures/green-monte-2and3.pstex_t}\\
\vskip 0.1in
\input{figures/green-monte-2and3long.pstex_t}
\end{center}
\caption{Possibilities for types \eqref{green-monte:case2} north and
 \eqref{green-monte:case3} south.  Note there are similar
 possibilities for \eqref{green-monte:case2} south and
 \eqref{green-monte:case3} north.}
\label{fig:case2-3}
\end{figure}

In particular, in this case the orange tentacle of vertex
(\ref{green-monte:case2}) must run across the south of the outside of
the negative block, and across a segment spanning the next (eastward)
positive tangle from east to west.  Further, there must be an orange
tentacle just inside the adjacent negative block to the east, adjacent
to this positive tangle.  This tentacle comes from an innermost disk
$J$ in the first negative tangle in this negative block.  The other
vertex, of type (\ref{green-monte:case3}) meets a tail of another
orange tentacle inside this negative block, at the top. Thus, the
innermost disk at the head of this other tentacle either agrees with
$J$, or is connected to $J$ by a sequence of non-prime switches.  If
$J$ is the head of both tentacles, then they must bound bigons between
them, and $E$ parabolically compresses to bigons, as in Figure
\ref{fig:case2-3}.  If $J$ is not the head of both tentacles, then
there is a sequence of 2--edge loops of the type shown in Figure
\ref{fig:green-enumerate5}, this time with an orange innermost disk,
and $E$ running across each loop in a string separated by non-prime
arcs.  Again this bounds bigons, and parabolically compresses to
bigons.

\smallskip

{\underline{ Type (\ref{green-monte:case3}) and type
    (\ref{green-monte:case3})}:} This case cannot occur.
Two vertices of this type would require orange heads in both negative
blocks to the east and west of the given negative block, contradicting
Lemma \ref{lemma:head-locator} (Head locator).

\medskip

Next, consider vertices of type (\ref{green-monte:case4}): the vertex
lies on an innermost disk in the adjacent negative tangle.  Again we
analyze the possible locations for orange heads from the other vertex,
using Lemma \ref{lemma:head-locator} (Head Locator).

\smallskip

{\underline{ Type (\ref{green-monte:case1}) and type
    (\ref{green-monte:case4})}:} In this case, an orange tentacle
meeting a vertex of type (\ref{green-monte:case1}) must run over the
outside of the negative block, so Lemmas \ref{lemma:mont-topandbottom}
and \ref{lemma:mont-outsidenegative} imply that the innermost disk in
the adjacent negative tangle must meet an orange tentacle connected to
the same side of the outside of the given negative block.  This
implies $\bdy E$ will bound a sequence of bigons, and hence $E$
parabolically compresses to bigons.  See Figure
\ref{fig:case4-innermostv}, left.

\begin{figure}
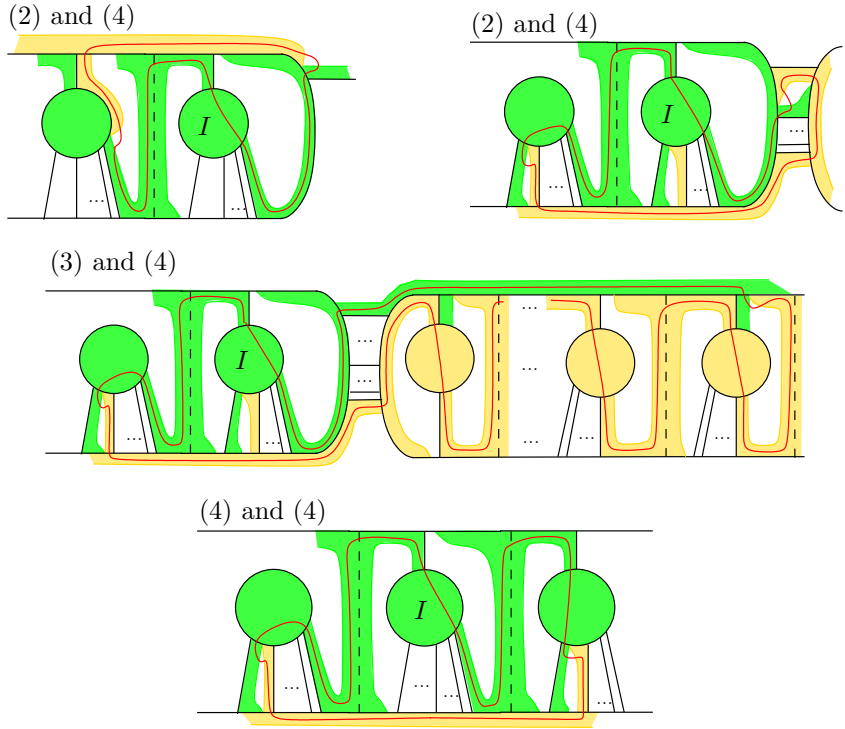

 \begin{center}
   \begin{tabular}{ccc}
     \input{figures/green-monte-4and2a.pstex_t} & \hspace{.3in} &
     \input{figures/green-monte-4and2b.pstex_t} \\
   \end{tabular}
 \end{center}
 \input{figures/green-monte-4and3.pstex_t} \\
 \vskip 0.15in
 \input{figures/green-monte-4and4.pstex_t}
 \caption{All possibilities for type (\ref{green-monte:case4}) north.}
 \label{fig:case4-innermostv}
\end{figure}

\smallskip

{\underline{ Type (\ref{green-monte:case2}) and type
    (\ref{green-monte:case4})}:} Here, the vertex of type
(\ref{green-monte:case2}) meets an orange tentacle on the outside of
the negative block. Thus Lemmas \ref{lemma:mont-topandbottom} and
\ref{lemma:mont-outsidenegative} imply that the orange tentacle
meeting the vertex of type (\ref{green-monte:case4}) must have its
head attached to the outside of the negative block.  Then, $\bdy E$
must run through this outside tentacle toward the vertex of type
(\ref{green-monte:case2}).  If the tentacle terminates at the vertex
of type (\ref{green-monte:case2}), then $\bdy E$ encloses only
bigons. Otherwise, the tentacle has its head attached to another state
circle.  In this case, the argument is the same as the one in case
(\ref{green-monte:case2}) and (\ref{green-monte:case2}) above.  The
orange tentacle meeting the other vertex is also attached to this
state circle, and we have a parabolic compression to bigons.

\smallskip

{\underline{ Type (\ref{green-monte:case3}) and type
    (\ref{green-monte:case4})}:} In this case, the head of an orange
tentacle meeting the vertex of type (\ref{green-monte:case3}) must be
in the next adjacent negative block. Thus, by the Head Locator Lemma
\ref{lemma:head-locator}, the orange tentacle meeting the vertex of
type (\ref{green-monte:case4}) must have its head inside the same
adjacent negative block.  Then, an argument similar to that in case
(\ref{green-monte:case2}) and (\ref{green-monte:case3}) implies that
$\bdy E$ encircles only bigons. Compare Figure \ref{fig:case2-3} to
Figure \ref{fig:case4-innermostv}.

\smallskip

{\underline{ Type (\ref{green-monte:case4}) and type
    (\ref{green-monte:case4})}:} Here, Lemmas
\ref{lemma:mont-topandbottom}, and \ref{lemma:mont-outsidenegative}
imply that the orange tentacles meeting the two vertices cannot come
from the north on one side and the south on the other, or the north or
south on one side and an innermost disk on the other.  Neither can
both orange tentacles come from innermost disks in distinct negative
tangles, for those tangles will be separated by green non-prime
switches, not orange connecting switches.  The only remaining
possibility is that the orange tentacles come from outside the
negative block on the same side.  In this case, $\bdy E$ bounds
bigons, as desired.  See Figure \ref{fig:case4-innermostv}.

\medskip

{\underline{ Type (\ref{green-monte:case5}) cases:}} Recall that in
type (\ref{green-monte:case5}), an arc of $\bdy E$ runs over a
non-prime arc and upstream, but the next vertex does not lie on an
innermost disk in this negative tangle.  Then $\bdy E$ must continue
downstream out of this innermost disk. Thus we have another 2--edge
loop as in Figure \ref{fig:neg-tangle}, and the options from Lemma
\ref{lemma:monte-stair-length} imply that from here, $\bdy E$ cannot
meet a vertex immediately, so its path toward a vertex is one of type
(\ref{green-monte:case2}), (\ref{green-monte:case3}),
(\ref{green-monte:case4}), or (\ref{green-monte:case5}).  By induction
on the number of negative tangles in a negative block, there will be
some finite number of consecutive 2--edge loops corresponding to
instances of case (\ref{green-monte:case5}), but eventually $\bdy E$
will run to a vertex of types (\ref{green-monte:case2}),
(\ref{green-monte:case3}), or (\ref{green-monte:case4}).  Note that
between these 2--edge loops from type (\ref{green-monte:case5}), we
have only bigon faces.  Combining the above arguments with these
additional bigon faces, we find that in all cases $\bdy E$ encloses
only bigons.

This phenomenon is illustrated in the bottom panel in Figure
\ref{fig:case4-innermostv}.  Thinking of the middle green innermost
disk $I$ as the innermost disk of the representative 2--edge loop, we
argued that this figure arose by combining vertices of type
(\ref{green-monte:case4}) and (\ref{green-monte:case4}).  However, if
we think of the right--most innermost green disk as the innermost disk
of our representative 2--edge loop, then this figure illustrates a
vertex of type (\ref{green-monte:case1}) (bottom right), from which
$\bdy E$ runs over a second 2--edge loop to the west, which is type
(\ref{green-monte:case5}), followed by a vertex of type
(\ref{green-monte:case4}).  More generally, we could have $n$ negative
tangles as in the middle of the bottom panel of Figure
\ref{fig:case4-innermostv}, strung end to end.  Between such tangles,
$\bdy E$ bounds only bigons.

\smallskip

This completes the enumeration of cases. For every combination of
ideal vertices, $E$ parabolically compresses to bigons. Thus there are
no complex EPDs in the upper polyhedron. This completes the proof of
Proposition \ref{prop:monte-main}, hence also the proof of Theorem
\ref{thm:monteguts}.
\end{proof}

\chapter{Applications}\label{sec:applications}
In this chapter, we will use the calculations of $ \guts(S^3\cut S_A)$
obtained in earlier chapters to relate the geometry of $A$--adequate
links to diagrammatic quantities and to Jones polynomials. In Section
\ref{sec:volume-apps}, we combine Theorem \ref{thm:guts-general} with
results of Agol, Storm, and Thurston \cite{ast} to obtain bounds on
the volumes of hyperbolic $A$--adequate links. A sample result is
Theorem \ref{thm:positive-volume}, which gives tight diagrammatic
estimates on the volumes of positive braids with at least $3$
crossings per twist region. The gap between the upper and lower bounds
on volume is a factor of about $4.15$.
 
In Section \ref{sec:monte-volume}, we apply these ideas to Montesinos
links, and obtain diagrammatic estimates for the volume of those
links. Again, the bounds are fairly tight, with a factor of $8$
between the upper and lower bounds.
 
In Section \ref{sec:jones-apps}, we relate the quantity
$\negeul(\guts(S^3\cut S_A))$ to coefficients of the Jones and colored
Jones polynomials of the link $K = \partial S_A$. One sample
application here is Corollary \ref{cor:beta-fiber}: for $A$--adequate
links, the next-to-last coefficient $\beta'_K$ detects whether a
state surface is a fiber in $S^3 \setminus K$.  Finally, in Section
\ref{sec:jones-volume}, we synthesize these ideas to obtain relations
between the Jones polynomial and volume. As a result, the volumes of
both positive braids and Montesinos links can be bounded above and
below in terms of these coefficients.

\section{Volume bounds for hyperbolic links}\label{sec:volume-apps}

Using Perelman's estimates for volume change under Ricci flow with
surgery, Agol, Storm, and Thurston \cite{ast} have obtained a
relationship between the guts of an essential surface $S \subset M$
and the hyperbolic volume of the ambient $3$--manifold $M$.  The
following result is an immediate consequence of \cite[Theorem
  9.1]{ast}, combined with work of Miyamoto \cite[Proposition 1.1 and
  Lemma 4.1]{miyamoto}.\index{hyperbolic volume}\index{volume|see{hyperbolic volume}}

\begin{theorem}\label{thm:ast-estimate}\index{guts!give volume estimates}\index{hyperbolic volume!bounded below by guts}
Let $M$ be finite--volume hyperbolic $3$--manifold, and let $S \subset
M$ be a properly embedded essential surface.  Then
$$\vol(M) \:\geq \: v_8\, \negeul(\guts(M\cut S)),$$
where $v_8 = 3.6638...$ is the volume of a regular ideal octahedron.
\end{theorem}

\begin{remark}
By \cite{ast} and work of Calegari, Freedman and Walker \cite{cfw},
the inequality of Theorem \ref{thm:ast-estimate} is an equality
precisely when $S$ is \emph{totally geodesic} and $M \cut S$ is a
union of regular ideal octahedra. We will not need this stronger
statement.
\end{remark}

In general, it is hard to effectively compute the quantity
$\negeul(\guts(M\cut S))$ for infinitely many pairs $(M, S)$.  To 
date, there have only been a handful of results  computing the
guts of essential surfaces in an infinite family of manifolds: see e.g.\ 
\cite{agol:guts, kuessner:guts, lackenby:volume-alt}.
In particular, Lackenby's computation of the guts of checkerboard surfaces
of alternating links  \cite[Theorem 5]{lackenby:volume-alt}
enabled him to estimate the volumes of these link
complements directly from a diagram. See \cite[Theorem
  1]{lackenby:volume-alt} and \cite[Theorem 2.2]{ast}.

In the $A$--adequate setting, we have the following volume estimate.

\begin{theorem}\label{thm:volume}\index{hyperbolic volume!in terms of $\chi(\GRA)$}
\index{$\GRA$, $\GRB$: reduced state graph!relation to volume}
 Let $D=D(K)$ be a prime
$A$--adequate diagram of a hyperbolic link $K$.  Then
$$\vol(S^3 \setminus K) \: \geq \: v_8\, (\negeul(\GRA)-|| E_c||),$$
where $\negeul(\GRA)$ and 
$|| E_c||$  are as in the statement of Theorem \ref{thm:guts-general}
and  $v_8 = 3.6638...$ is the volume of a regular ideal octahedron.
\end{theorem}

\begin{proof}
We will apply Theorem \ref{thm:ast-estimate} to the essential surface
$S_A$ and the 3--manifold $S^3\setminus K$. Since $S^3\cut S_A$ is
homeomorphic to $(S^3\setminus K)\cut S_A$, we have
$$\vol(S^3 \setminus K) \: \geq \: v_8\, \negeul(\guts(S^3\cut S_A)) \: =\: \negeul(\GRA)-|| E_c||,$$
where the equality comes from Theorem \ref{thm:guts-general}.
The result now follows.
\end{proof}

Theorem \ref{thm:volume} becomes particularly effective in the case
where $|| E_c||=0$. For example, this will happen when every $2$--edge
loop in the state graph $\GA$ comes from a single twist region of the
diagram $D$.

\begin{corollary}\label{cor:no2edge-volume}
Let $D(K)$ be a prime, $A$--adequate diagram of a hyperbolic link $K$,
such that for each 2--edge loop in $\GA$, the edges belong to the same
twist region of $D(K)$.  Then
$$\vol(S^3 \setminus K) \: \geq \: v_8 \, (\negeul(\GRA)).$$
\end{corollary}

\begin{proof}
This follows immediately from Theorem \ref{thm:volume}
and Corollary \ref{cor:onlybigons}.
\end{proof}

\begin{remark}
If $D=D(K)$ is a prime reduced alternating link diagram, then the
hypotheses of Corollary \ref{cor:no2edge-volume} are satisfied by both
the state graphs $\GA$ and $\GB$. Thus Corollary
\ref{cor:no2edge-volume} gives lower bounds on volume in terms of both
$\negeul(\GRA)$ and $\negeul(\GRB)$. By averaging these two lower
bounds, one recovers Lackenby's lower bound on the volume of
hyperbolic alternating links, in terms of the twist number $t(D)$
\cite[Theorem 2.2]{ast}.
\end{remark}

\begin{figure}
\psfrag{s}{$\sigma_1$}
\psfrag{t}{$\sigma_2$}
\begin{center}
\includegraphics{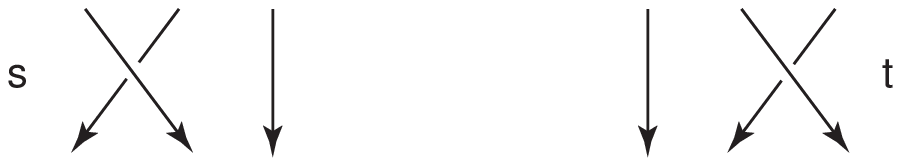}
\end{center}
\caption{The generators $\sigma_1$ and $\sigma_2$ of the 3--string braid group.\index{braid!generators ($\sigma_1, \ldots, \sigma_{n-1}$)}}
\label{fig:3braid-generators}
\end{figure}

Corollary \ref{cor:no2edge-volume} also applies to certain closed
braids\index{braid}.

\begin{define}\label{def:positive-braid}
Let $B_n$ denote the braid group on $n$ strings. The elementary braid 
generators are denoted $\sigma_1,
\ldots, \sigma_{n-1}$ (see
Figure \ref{fig:3braid-generators} for the case $n=3$).  A braid
$b=\sigma_{i_1}^{r_1}\sigma_{i_2}^{r_2} \cdots \sigma_{i_k}^{r_k}$ is
called \emph{positive}\index{braid!positive, negative}\index{positive braid} if
all the exponents $r_j $ are positive, and
\emph{negative} if all the exponents $r_j$ are negative.
\end{define}

Suppose that $D_b$ is the closure of a positive braid $b\in B_n$. Then
it follows immediately that the diagram $D_b$ is $B$--adequate. In
fact, the reduced graph $\GRB$ is a line segment with $n$
vertices. Thus, by Theorem \ref{thm:fiber-tree}, the state surface
$S_B$ is a fiber for $S^3 \setminus K$. (This recovers a classical
result of Stallings \cite{stallings:fibered} and Gabai
\cite{gabai:fibered}.) In particular, $S^3 \cut S_B$ is an
$I$--bundle, and does not contain any guts. On the other hand, under
stronger hypotheses about the exponents $r_j$, one can get non-trivial
volume estimates from the guts of the other state surface $S_A$.

\begin{theorem}\label{thm:positive-volume}\index{hyperbolic volume!of positive braids}\index{positive braid!hyperbolic volume}\index{braid!hyperbolic volume}\index{hyperbolic volume!in terms of twist number}\index{twist number!and hyperbolic volume}
Let $D = D_b$ be a diagram of a hyperbolic link $K$, obtained as the
closure of a positive braid $b=\sigma_{i_1}^{r_1}\sigma_{i_2}^{r_2}
\cdots \sigma_{i_k}^{r_k}$. Suppose that $r_j \geq 3$ for all $1\leq
j\leq k$; in other words, each of the $k$ twist regions in $D$
contains at least $3$ crossings.  Then
$$\frac{2 v_8}{3} \, t(D) \: \leq \:\vol(S^3 \setminus K) \: < \: 10v_3(t(D)-1),$$ 
where $v_3 = 1.0149...$ is the volume of a regular ideal tetrahedron  
and $v_8 = 3.6638...$ is the volume of a regular ideal octahedron. 
\end{theorem}

Recall that $t(D)$\index{$t(D)$|see{twist number}} denotes the \emph{twist
  number}\index{twist number}: the number of twist regions\index{twist region} in the diagram $D$.  Observe that the multiplicative
constants in the upper and lower bounds differ by a rather small
factor of $4.155...$ .

The proof of Theorem \ref{thm:positive-volume} will require two lemmas.

\begin{lemma}\label{lemma:prime-hyp-braid}
$D = D_b$ be a diagram of a hyperbolic link $K$, obtained as the
  closure of the positive braid $b =
  \sigma_{i_1}^{r_1}\sigma_{i_2}^{r_2} \cdots \sigma_{i_k}^{r_k}$,
  where $k \geq 2$. Then
\begin{enumerate} 
\item\label{item:hyp-prime} If $K$ is hyperbolic and $r_j \geq 2$ for all $j$, then $D$ is a prime, $A$--adequate diagram.
\item\label{item:prime-hyp} If $D$ is a prime diagram and $r_j \geq 6$, for all $j$, then $K$ is hyperbolic.
\end{enumerate}
\end{lemma}

\begin{proof}
First, suppose that $r_j \geq 2$ for all $j$. Since $b$ is a positive
braid, the $A$--resolution of every twist region is the long
resolution (see Figure \ref{fig:twist-resolutions} on page
\pageref{fig:twist-resolutions}). Thus every edge of $\GA$ connects to
a bigon on at least one end, and no edge of $\GA$ is a loop. Thus $D$
is $A$--adequate.

If $K$ is hyperbolic, it must be prime and non-split. Thus, by
Corollary \ref{cor:primelinkadequate} on page
\pageref{cor:primelinkadequate}, either $D$ is prime or contains
nugatory crossings. But a nugatory crossing in a braid diagram can
only be created by stabilization, which would imply there is a term
$\sigma_i^1$, contradicting the hypothesis that $r_j \geq 2$ for all
$j$. This proves \eqref{item:hyp-prime}.

Statement \eqref{item:prime-hyp} follows immediately from
\cite[Theorem 1.4]{fp:twisted}, once one knows that $D$ is
twist--reduced. As we will not need conclusion \eqref{item:prime-hyp}
in the sequel (it was mainly included as a pleasant quasi-converse to
\eqref{item:hyp-prime}), we leave it to the reader to show that $D$ is
twist--reduced.
\end{proof}

\begin{lemma}\label{lemma:positive-twist}\index{twist number!in terms of $\chi(\GRA)$}\index{$\GRA$, $\GRB$: reduced state graph!relation to twist number}
Let $D = D_b$ be a diagram of a hyperbolic link $K$, obtained as the
closure positive braid $b=\sigma_{i_1}^{r_1}\sigma_{i_2}^{r_2} \cdots
\sigma_{i_k}^{r_k}$. Suppose that $r_j \geq 3$ for all $j$. Then $D$
is $A$--adequate, and
$$\chi(\GRA) \: = \: \chi(\GA) \: \leq \: -2k/3 \: = \: -2t(D)/3 \:<  \: 0.$$ 
\end{lemma}

\begin{proof}
The diagram $D$ is $A$--adequate by Lemma
\ref{lemma:prime-hyp-braid}. Since the $A$--resolution is the long
resolution, every loop in $\GA$ has length at least $3$. Thus $\GA =
\GRA$. It remains to count the vertices and edges of $\GA$.

Recall that the edges of $\GA$ are in one--to--one correspondence with
the crossings in $D_b$; thus there are a total of $\sum r_j$ edges of
$\GA$. The vertices of $\GA$ are in one--to--one correspondence with
the state circles in the $A$--resolution of $D_b$. In a twist region
with $r_j$ crossings, there are $(r_j - 1)$ bigon circles in the long
resolution; thus there are a total of $(\sum r_j) - k$ bigon state
circles.  It remains to count the non-bigon state circles of the
$A$--resolution. We call these the \emph{wandering}\index{state circle!wandering} state circles, as
they wander through multiple twist regions.

Consider the $S^1$--valued height function on $D(K)$ that arises from
the braid position of the diagram. Relative to this height function,
all segments of $H_A$ are vertical, and connect two critical points of
state circles. Thus the number of critical points on a state circle
$C$ equals the number of segments of $H_A$ (equivalently, edges of
$G_A$) met by $C$. To complete the proof of the lemma we need the
following.

\smallskip

\textbf{Claim:} Every wandering state circle $C$ has at least $6$
critical points.

\smallskip

\emph{Proof of claim:} Since $C$ has the same number of minima as
maxima, the total number of critical points on $C$ must be even. Also,
note that between critical points, $C$ runs directly along one of the
$n$ strands of the braid. At a critical point, it crosses from the
$j$-th to the $(j \pm 1)$-st strand.

Consider the number of distinct strands that $C$ runs along. If $C$
only runs along one strand of the braid, with no critical points, then
that strand is a link component with no crossings: absurd. If $C$ only
runs along the $i$-th and $(i+1)$-st strands of the braid, then it
must have exactly $2$ critical points, and is a bigon. This
contradicts the hypothesis that $C$ is wandering.

If $C$ runs along four or more strands of the braid, then it needs at
least $6$ critical points (to get from the $i$-th to the $(i+3)$-rd
strand, and back), hence we are done.  The remaining possibility is
that $C$ runs along exactly three strands of the braid. This means
that $C$ must have at least $4$ critical points. If it has more than
$4$, then we are done.

\begin{figure}
\psfrag{C}{$C$}
\includegraphics{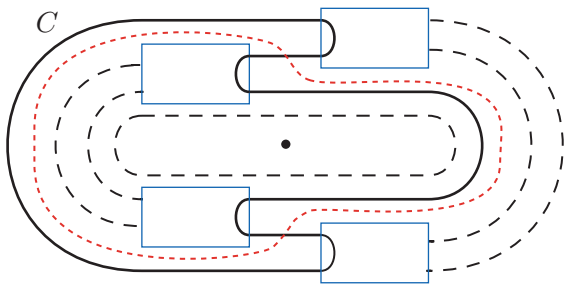}
\caption{A wandering state circle $C$ with exactly $4$ critical
  points. The rectangular boxes are twist regions. The dashed arcs are
  strands of the braid, which may run through other twist regions. The
  red dotted loop provides a contradiction to primeness.}
\label{fig:braid-contradiction}
\end{figure}

Suppose, for a contradiction, that $C$ runs along exactly three
strands and has exactly $4$ critical points. Then, with some choice of
orientation along $C$, it must run from the $i$-th to the $(i+1)$-st
strand at a maximum, then to the $(i+2)$-nd at a minimum, then to the
$(i+1)$-st strand at a maximum, then finally back to the $i$-th at a
minimum. In other words, $C$ must look exactly like the state circle
of Figure \ref{fig:braid-contradiction}. But the figure reveals an
essential loop (dotted, red) meeting $D(K)$ twice, which contradicts
primeness. Since $D$ is prime by Lemma \ref{lemma:prime-hyp-braid},
this finishes the proof of the claim.

\smallskip

To continue with the proof of the lemma, observe that every twist
region contains two critical points of wandering state circles; these
are the ends of the long resolution in Figure
\ref{fig:twist-resolutions} on page \pageref{fig:twist-resolutions}.
On the other hand, by the claim, each wandering circle has at least
$6$ critical points. Thus there must be at least three times as many
twist regions as wandering circles. We may now compute:
\begin{eqnarray*}
\chi(\GRA) = \chi(\GA) &=&
(\mbox{bigon circles}) + (\mbox{wandering circles}) - (\mbox{crossings}) \\
& \leq & \left(\sum r_j \: - k \right) + (k/3) - \left( \sum r_j \right) \\
& = & - \frac{2k}{3} \: = \: - \frac{2t(D)}{3}.  
\end{eqnarray*}

\vspace{-3ex}
\end{proof}

We can now complete the proof of Theorem \ref{thm:positive-volume}.

\begin{proof}[Proof of Theorem \ref{thm:positive-volume}]
The upper bound on volume is due to Agol and D.\ Thurston
\cite[Appendix]{lackenby:volume-alt}, and holds for all diagrams.

To prove the lower bound on volume, recall that $D$ must be prime by
Lemma \ref{lemma:prime-hyp-braid}. By Lemma
\ref{lemma:positive-twist}, we know that $\GA = \GRA$, hence $\GA$ has
no $2$--edge loops. Thus Corollary \ref{cor:no2edge-volume} applies.
Plugging the estimate
$$\negeul(\GRA) \: = \: - \chi(\GRA)  \: \geq \: 2 t(D) / 3$$ 
into Corollary \ref{cor:no2edge-volume} completes the proof. 
\end{proof}

Similar relations between volume and twist number are known for
alternating links, links that admit diagrams with at least seven
crossings in each twist region, and closed 3--braids \cite{ast,
  fkp:filling, fkp:farey}.  To this list, we may add a result about
the volumes of Montesinos links.

\section{Volumes of Montesinos links}\label{sec:monte-volume}

In this section, we will prove Theorem \ref{thm:monte-volume}, which
estimates the volume of Montesinos links.  We begin with a pair of
lemmas.  For the statement of the lemmas, recall Definition
\ref{def:monte-reduced} on page \pageref{def:monte-reduced} and
Definition \ref{def:admissible-tangle} on page
\pageref{def:admissible-tangle}.

\begin{lemma}\label{lemma:monte-twist-tangle}
Let $D(K)$ be a reduced, admissible Montesinos diagram with at least
three positive tangles and at least three negative tangles.  Let
$\GRA$ and $\GRB$ be the reduced all--$A$ and all--$B$ graphs
associated to $D$. Then
$$ - \chi(\GRA) - \chi(\GRB) \: = \: t(D) - Q_{1/2}(D),$$
where $Q_{1/2}(D)$ is the number of rational tangles in $D$ whose
slope has absolute value $\abs{q} \in [1/2, 1).$
\end{lemma}

\begin{proof}
The link diagram $D$ can be used to construct a \emph{Turaev
  surface}\index{Turaev surface} $T$: this is a closed, unknotted
surface in $S^3$, onto which $K$ has an alternating projection. The
graphs $\GA$ and $\GB$ naturally embed in $T$ as checkerboard graphs
of the alternating projection, and are dual to one another.
Furthermore, because $D$ is constructed as a cyclic sum of alternating
tangles, the Turaev surface is a torus. See \cite[Section
  4]{dasbach-futer...} for more details.

Recall that $\chi(\GRA) = v_A - e'_A$, where $v_A$ is the number of
vertices and $e_A$ is the number of edges, and similarly for
$\chi(\GRB)$. We can use the topology of $T$ to get a handle on these
quantities.  Because $\GA$ and $\GB$ are dual, the number of vertices
of $\GB$ equals the number of regions in the complement of $\GA$.
Thus, since $T$ is a torus, we have
$$v_A - e_A + v_B = \chi(T) = 0.$$

Now, consider the number of edges of $\GA$ that are discarded when we
pass to $\GRA$. Because $D$ has at least three positive tangles, Lemma
\ref{lemma:monte-2edge-loop} on page \pageref{lemma:monte-2edge-loop}
implies that edges can be lost in one of two ways:
\begin{enumerate}
\item If $r$ is an $A$--region with $c(r)>1$ crossings, hence $c(r) >
  1$ parallel edges in $\GA$, then $c(r) -1$ of these edges will be
  discarded as we pass to $\GRA$. See Definition \ref{def:long-short}
  and Figure \ref{fig:twist-resolutions} on page
  \pageref{fig:twist-resolutions}. \smallskip

\item If $N_i$ is a negative tangle of slope $q_i \in (-1, -1/2]$,
then one edge of $\GA$ will be lost from the two-edge loop that spans
$N_i$ north to south. See Figure \ref{fig:neg-tangle} on page \pageref{fig:neg-tangle}.
\end{enumerate}
The same principle holds for the $B$--graph $\GB$, with $B$--regions
replacing $A$--regions and positive tangles replacing negative ones.

Combining these facts gives
\begin{eqnarray*}
(e_A - e'_A) + (e_B - e'_B) &=& \sum_{\mbox{twist regions}} (c(r) - 1)
  \: \:+ \#\{ i: \abs{q_i} \in [1/2, 1) \} \\
    &=& c(D) - t(D) + Q_{1/2}(D).
\end{eqnarray*}
Finally, since the edges of $\GB$ are in one-to-one correspondence
with the crossings of $D$,
\begin{eqnarray*}
 -  \chi(\GRA) - \chi(\GRB) 
 &=& e'_A + e'_B - v_A - v_B \\
 &=& (e'_A + e'_B - e_A - e_B) \qquad + e_B \quad + (e_A - v_A - v_B) \\
 &=& - c(D) + t(D) - Q_{1/2} \quad + c(D) \quad + 0 \\
 &=&  t(D) - Q_{1/2}(D).
\end{eqnarray*}

\vspace{-3ex}
\end{proof}

\begin{lemma}\label{lemma:monte-twist}\index{twist number!in terms of $\chi(\GRA)$}\index{$\GRA$, $\GRB$: reduced state graph!relation to twist number}
Let $D(K)$ be a reduced, admissible Montesinos diagram with at least
three positive tangles and at least three negative tangles.  Then
$$ - \chi(\GRA) - \chi(\GRB) \: \geq \: \frac{t(D) -\# K}{2}. $$
where $\# K$ is the number of link components of $K$.
\end{lemma}

\begin{proof}
By Lemma \ref{lemma:monte-twist-tangle}, it will suffice to estimate
the quanity $Q_{1/2}(D)$.  Consider a rational tangle $R_i$ whose
slope satisfies $\abs{q_i} \in [1/2, 1)$. Each such tangle contributes
one unit to the count $Q_{1/2}(D)$. If $\abs{q_i} > 1/2$, then the
continued fraction expansion of $q_i$ has at least two terms, hence
$R_i$ has at least two twist regions. Only one of those twist
regions will be lost to the count $Q_{1/2}(D)$.

Alternately, suppose $q_i = \pm 1/2$. In this case, one strand of $K$
in this tangle runs from the NW to the SW corner of the tangle, and
another strand runs from the NE to the SE corner. In other words, the
number of link components of $K$ will remain unchanged if we replace
$R_i$ by a tangle of slope $\infty$. See Figure \ref{fig:tangles} on
page \pageref{fig:tangles}.

Let $n$ be the number of tangles of slope $\pm 1/2$ in the diagram
$D$. If we replace each such tangle by one of slope $\infty$, the
number $\# K$ of link components is unchanged. But after this
replacement, there are $n$ ``breaks'' in the diagram, hence $K$ is a
link of at least $n$ components. This proves that $n \leq \# K$. In
other words, there is a one--to--one mapping from tangles of slope $\pm
1/2$ to link components.
We conclude that 
\begin{eqnarray*}
Q_{1/2}(D) &=&
\sum_{i \colon \abs{q_i} > 1/2} 1 \: \: + \sum_{i \colon \abs{q_i} = 1/2} 1 \\
& \leq & \sum_{i \colon \abs{q_i} > 1/2} \frac{t(R_i)}{2} \: \: + \sum_{i \colon \abs{q_i} = 1/2} \frac{t(R_i) + 1}{2} \\
& \leq & \frac{t(D) + \# K}{2},
\end{eqnarray*}
and the result follows by Lemma \ref{lemma:monte-twist-tangle}.
\end{proof}

\begin{theorem}\label{thm:monte-volume}\index{hyperbolic volume!of Montesinos links}\index{Montesinos link!hyperbolic volume}\index{hyperbolic volume!in terms of twist number}\index{twist number!and hyperbolic volume}
Let $K \subset S^3$ be a Montesinos link with a reduced Montesinos
diagram $D(K)$.  Suppose that $D(K)$ contains at least three positive
tangles and at least three negative tangles.  Then $K$ is a hyperbolic
link, satisfying
$$ \frac{v_8}{4} \, \left( t(D) - \#K \right) 
 \: \leq \: 
\vol(S^3  \setminus K) \: < \: 2 v_8 \, t(D),$$
where $v_8 = 3.6638...$ is the volume of a regular ideal octahedron
and $\# K$ is the number of link components of $K$.
The upper bound on volume is sharp.
\end{theorem}

We note that the upper bound on volume applies to all Montesinos links, without any restriction on the number of positive and negative tangles.

The lower bound on volume is proved using Lemma
\ref{lemma:monte-twist}.  In fact, using Lemma
\ref{lemma:monte-twist-tangle} instead of Lemma
\ref{lemma:monte-twist}, one can obtain the sharper estimate
$$\vol(S^3 \setminus K) \geq \frac{v_8}{2} \, \left( t(D) - Q_{1/2}(D) \right),$$
where $Q_{1/2}(D)$ is the number of rational tangles of slope
$\abs{q_i} \in [1/2, 1)$.

\begin{proof}[Proof of Theorem \ref{thm:monte-volume}]
Let $D(K)$ be a reduced Montesinos diagram that contains at least
three positive tangles and at least three negative tangles. As we have
observed following Definition \ref{def:admissible-tangle} on page
\pageref{def:admissible-tangle}, any reduced diagram can be made
admissible by a sequence of flypes. Since flyping does not change the
twist number of $D$, we may also assume that $D$ is admissible. Thus
Theorem \ref{thm:monteguts} and Lemmas \ref{lemma:monte-twist-tangle}
and \ref{lemma:monte-twist} all apply to $D(K)$.

The hyperbolicity of $K$ follows from Bonahon and Siebenmann's
enumeration of non-hyperbolic arborescent links
\cite{bonsieb:monograph}. See also Futer and Gu\'eritaud \cite[Theorem
  1.5]{fg:arborescent}.
	
The lower bound on volume follows quickly by applying Theorem
\ref{thm:ast-estimate} to both the all--$A$ and all--$B$ state
surfaces:
\begin{eqnarray*}
 \vol(S^3 \setminus K)
& \geq & v_8 /2  \, \left( \negeul \guts(S^3 \cut S_A)  + \negeul \guts(S^3 \cut S_B) \right) \\
& = & v_8 /2 \, \left(\negeul(\GRA) + \negeul(\GRB) \right) \\
& \geq & v_8 /4 \, \left( t(D) - \#K \right)  ,
\end{eqnarray*}
where the last two lines used Theorem \ref{thm:monteguts} and Lemma
\ref{lemma:monte-twist}.

\smallskip

\begin{figure}[h]
\input{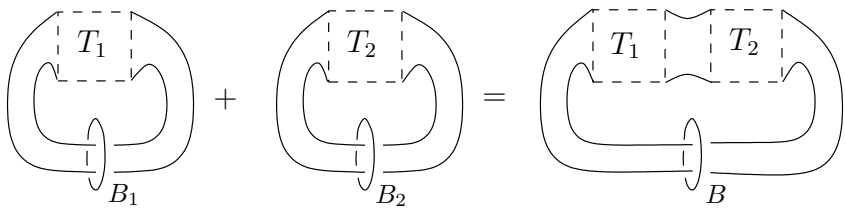}
\caption{A belted sum\index{belted sum}\index{sum of tangles!belted} of tangles $T_1$ and
  $T_2$. The twice--punctured disks bounded by $B_1$ and $B_2$ are
  glued to form the twice--punctured disk bounded by the belt $B$.}
\label{fig:beltsum}
\end{figure}

The upper bound on volume will follow from a standard Dehn filling
argument. Add a link component $B$ to $K$, which encircles the two
eastern ends of some rational tangle $T_i$. Note that $B$ can be moved
by isotopy to lie between any pair of consecutive tangles. Thus the
longitude of $B$ forms (part of) the boundary of $n$ different
twice-punctured disks in $S^3 \setminus (K \cup B)$, with one disk
between every pair of consecutive tangles.

The link $K \cup B$ is arborescent, hence also hyperbolic by
\cite{fg:arborescent}. Each twice--punctured disk bounded by $B$ will
be totally geodesic in this hyperbolic metric, by a theorem of Adams
\cite{adams:3-punct}.

Let $L_i$ be the link obtained by taking the numerator closure\index{numerator closure} 
of
tangle $T_i$, and adding an extra circle $B_i$ about the eastern ends
of the tangle. Then $(K \cup B)$ is a \emph{belted sum}\index{belted sum}\index{sum of tangles!belted} of the tangles $T_1, \ldots, T_n$: it is obtained by cutting
each $S^3 \setminus L_i$ along the twice--punctured disk bounded by
$B_i$, and gluing these manifolds cyclically along the
twice--punctured disks. (See Figure \ref{fig:beltsum}, and see Adams
\cite{adams:3-punct} for more information about belted sums.)

Since the numerator closure of each rational tangle $T_i$ is a
$2$--bridge link, the link $L_i$ is an augmented $2$--bridge
link. Thus each $L_i$ is hyperbolic. Furthermore, if tangle $T_i$
contains $t(T_i)$ twist regions, then $L_i$ is the augmentation of a
$2$--bridge link with at most $t(T_i) + 1$ twist regions. Therefore,
\cite[Theorem B.3]{gf:two-bridge} implies that
$$\vol(S^3 \setminus L_i) \: < \: 2 v_8 \, t(T_i).$$

When we perform the belted sum to obtain $K \cup B$, we cut and reglue
along totally geodesic twice--punctured disks. Volume is additive
under this operation \cite{adams:3-punct}. This gives the estimate
$$\vol(S^3 \setminus (K \cup B)) \: < \: 2 v_8 \, t(D).$$
Since volume goes down when we Dehn fill the meridian of $B$, the
upper bound on $\vol(S^3 \setminus K)$ follows. 

\smallskip

To prove the sharpness of the upper bound, consider the following
sequence of examples. Let $n$ be an even number, and let $K_n$ be a
Montesinos link with $n$ rational tangles, where the slope of the
$j$-th tangle is $(-1)^j / n$. 
This is a $(n, -n, \ldots, n, -n)$ pretzel link.
Then every rational tangle has slope $\pm
1/n$, with alternating signs. The diagram $D_n$ has exactly $n$ twist
regions, with exactly $n$ crossings in each region.

Let $J_n$ be a link obtained by adding a crossing circle about every
twist region, as well as the belt component $B$ of Figure
\ref{fig:beltsum}.  Then, by the above discussion, $J_n$ is a belted
sum of $n$ copies of the Borromean rings, hence
$$\vol(S^3 \setminus J_n) \: = \: 2v_8 \, n \: = \: 2v_8 \, t(D_n).$$
Furthermore, $K_n$ can be recovered from $J_n$ by $(\pm 1, n/2)$ Dehn
filling on the crossing circles and meridional Dehn filling on the
belt component $B$.
 
By \cite[Theorem 3.8]{fp:twisted}, there is an embedded horospherical
neighborhood of the cusps of $J_n$, such that in each of the many
$3$--punctured spheres in $S^3 \setminus J_n$, the cusp neighborhoods
of the $3$ punctures are pairwise tangent. Then, by \cite[Corollary
  3.9 and Theorem 3.10]{fp:twisted}, the Dehn filling curves have the
following length on the horospherical tori:
\begin{itemize}
\item The meridian of the belt $B$ has length $\ell(\mu) \geq n$,
\item The $(\pm 1, n/2)$ curves on the crossing circles have length  $\ell \geq \sqrt{n^2 + 1}$.
\end{itemize}
In particular, each filling curve has length at least $n$.

Now, we may use \cite[Theorem 1.1]{fkp:filling} to bound the change in
volume under Dehn filling. As a corollary of that theorem, it follows
that when several cusps of a manifold $M$ are filled along slopes of
length at least $\ell_{\min} > 2\pi$, the additive change in volume
satisfies
$$\Delta V \: \leq \: \frac{6 \pi \cdot \vol(M)}{\ell_{\min}^2 } .$$
(Deriving this corollary requires expanding the Taylor series for
$(1-x)^{3/2}$; see \cite[Section 2.3]{fkp:filling}.) 
In our setting, $\vol(S^3 \setminus J_n) = 2v_8 \, n$, and all the
filling curves also have length at least $n$. Thus 
$$2v_8 \, t(D_n) - \vol(S^3 \setminus K_n) \: = \: \Delta V \: \leq \:  (6 \pi \cdot 2 v_8 \cdot n)  / n^2,$$
which becomes arbitrarily small as $n \to \infty$. Thus the upper
bound on volume is sharp.
\end{proof}

\begin{remark}	
The bounds on the change of volume under Dehn filling, obtained in
\cite{fkp:filling}, can be fruitfully combined with the results of
this chapter. This combination results in relations between simple
diagrammatic quantities (such as the twist number of an $A$--adequate
knot) and the hyperbolic volume of 3--manifolds obtained by Dehn
surgery on a that knot.  For example, one may combine Theorem
\ref{thm:monte-volume} with \cite[Theorem 3.4]{fkp:filling} and obtain
the following: Let $K \subset S^3$ be a Montesinos knot as in Theorem
\ref{thm:monte-volume}, and let $N$ be a hyperbolic manifold obtained
by $p/q$--Dehn surgery along $K$, where $\abs{q} \geq 12$. Then
$$\frac{v_8}{4} \left(1-\frac{127}{q^2}\right)^{3/2} \: (t(D)-1) \: <\: \vol(N) \: < \: 2 v_8 \, t(D).$$
\end{remark}
\smallskip

\section{Essential surfaces and colored Jones polynomials}\label{sec:jones-apps}
For a knot $K$ let
$$J^n_K(t)= \alpha_n t^{m_n}+ \beta_n t^{m_n-1}+ \ldots + \beta'_n
t^{r_n+1}+ \alpha'_n t^{r_n},$$ denote its $n$-th \emph{colored Jones
  polynomial}\index{colored Jones polynomial}.  One recently observed
relation between the colored Jones polynomials and classical topology is the \emph{slope conjecture}\index{slope conjecture} of
Garoufalidis \cite{garoufalidis:jones-slopes}, which postulates that
the sequence of degrees of the colored Jones polynomials detects
certain boundary slopes of a knot $K$. This conjecture has been proved
for several classes of knots \cite{dunfield-garoufalidis:2fusion,
  fkp:PAMS, garoufalidis:jones-slopes}, including a proof by the
authors for the family of adequate knots \cite{fkp:PAMS}. See Theorem
\ref{thm:slopes} in the Introduction for a precise statement.

In the same spirit, we can now show that certain coefficients of
$J^n_K(t)$ measure how far the surface $S_A$ is from being a fiber.
We need the following lemma; a similar statement holds for
$B$--adequate diagrams.

\begin{lemma}\label{lemma:stabilized}\index{$\GRA$, $\GRB$: reduced state graph!relation to Jones coefficients}\index{$\beta_K$, $\beta_K'$!in terms of $\chi(\GRA)$ or $\chi(\GRB)$}
Let $D=D(K)$ be an $A$--adequate diagram with reduced all--$A$ state
graph $\GRA$.  Then for every $n>1$
\begin{enumerate}
\item $\abs{\alpha'_n}=1 $; and
\item $\abs{\beta'_n}= 1 -\chi(\GRA)$,
\end{enumerate}
where as above $ \alpha'_n$, $\beta'_n$ are the last and next-to-last coefficients of $J^n_K(t)$.
\end{lemma}

\begin{proof}
Part (1) is proved in \cite{lickorish:book}; part (2) is proved in
\cite{dasbach-lin:head-tail}. See \cite{dfkls:determinant} for an alternate proof of both results.
\end{proof}

\begin{define}\label{def:stablevalue}\index{stable value}
With the setting and notation of Lemma \ref{lemma:stabilized}, we
define the \emph{stable penultimate coefficient}\index{colored Jones polynomial!stable coefficients}\index{stable coefficients} of $J^n_K(t)$ to be
$\beta'_K:=\abs{\beta'_n}$\index{$\beta_K$, $\beta_K'$}, for $n>1$. For
completeness we define the \emph{stable last
  coefficient} to be
$\alpha'_K:=\abs{ \alpha'_n}=1 $.

We also define $\epsilon'_K=1$\index{$\epsilon_K$, $\epsilon'_K$} if $\beta'_K=0$
and 0 otherwise.

Similarly, for a $B$--adequate knot $K$, we define the \emph{stable
  second coefficient}\index{colored Jones polynomial!stable coefficients}\index{stable coefficients} of $J^n_K(t)$ to be
$\beta_K:=\abs{\beta_n}$, for $n>1$, and the
\emph{stable first coefficient} to be
$\alpha_K:=\abs{ \alpha_n}=1 $.  We also define
$\epsilon_K=1$ if $\beta_K=0$ and 0 otherwise.
\end{define}

The next result, which is a corollary of Theorem \ref{thm:fiber-tree},
shows that the stable coefficients $\beta'_K$, $\beta_K$ are exactly
the obstructions to $S_A$ or $S_B$ being fibers. We only state the
result for $A$--adequate links.

\begin{corollary}\label{cor:beta-fiber}\index{fiber!detected by Jones coefficient}\index{$\beta_K$, $\beta_K'$!detects fiber surfaces}
For an $A$--adequate link $K$, the following are equivalent:
\begin{enumerate}
\item\label{item:beta=0} $\beta'_K=0$. 
\item\label{item:fiber-every} For \emph{every} $A$--adequate diagram of $D(K)$,  $S^3 \setminus K$ fibers over $S^1$ with fiber the corresponding state surface $S_A = S_A(D)$. 
\item\label{item:ibundle-some} For \emph{some} $A$--adequate diagram $D(K)$,  $M_A = S^3 \cut S_A$ is an $I$--bundle over $S_A(D)$.
\end{enumerate}
\end{corollary}

\begin{proof}
By Lemma \ref{lemma:stabilized}, $\beta'_K=0$ precisely when $\GRA$ is
a tree, for every $A$--adequate diagram of $K$. Thus $\eqref{item:beta=0} \Rightarrow \eqref{item:fiber-every}$ follows immediately from Theorem
\ref{thm:fiber-tree} on page \pageref{thm:fiber-tree}. The implication $\eqref{item:fiber-every} \Rightarrow \eqref{item:ibundle-some}$ is trivial, by specializing to a particular $A$--adequate diagram. Finally, $\eqref{item:ibundle-some} \Rightarrow \eqref{item:beta=0}$ is again immediate from Theorem
\ref{thm:fiber-tree}  and Lemma \ref{lemma:stabilized}.
\end{proof}

\begin{remark}\label{rem:genus-detect}
Given a knot $K$, the \emph{Seifert genus}\index{Seifert genus} $g(K)$
is defined to be the smallest genus over all orientable surfaces
spanned by $K$.  Since a fiber realizes the genus of a knot
\cite{burde-zieschang:knots}, Corollary \ref{cor:beta-fiber} implies
that $g(K)$ can be read off from any $A$--adequate diagram of $K$ when
$\beta'_K = 0$.  Since $\GA$ is a spine for $S_A$ (Lemma
\ref{lemma:ga-spine}), in this case we have
$$g(K)=\frac{1 - \chi(\GA)}{2}.$$
\end{remark}

Note that $\abs{\beta_K'}-1+\epsilon'_K=0$ precisely when
$\abs{\beta_K'} \in \{0, 1\}$.  By Corollary \ref{cor:beta-fiber}, having $\beta'_K=0$ corresponds to $S_A$ being a fiber. Our next result is that $\abs{\beta_K'} = 1$ precisely when $M_A$ is a book of $I$--bundles (hence, $S_A$ is a fibroid) of a particular type.

\begin{theorem}\label{thm:fibroid-detect}\index{book of $I$--bundles!detected by Jones coefficient}\index{fibroid!detected by Jones coefficient}\index{$\beta_K$, $\beta_K'$!detects book of $I$--bundles}
Let $K$ be an $A$--adequate link, and let $\beta'_K$ be as in
Definition \ref{def:stablevalue}.  Then the following are equivalent:
\begin{enumerate}
\item\label{item:beta=1} $\abs{\beta'_K}=1$.
\item\label{item:fibroid-every} For \emph{every} $A$--adequate diagram of $K$, the 
corresponding 3--manifold $M_A$ is a book of
  $I$--bundles, with $\chi(M_A)= \chi(\GA) - \chi(\GRA)$, and is not a
  trivial $I$--bundle over the state surface $S_A$.
\item\label{item:fibroid-some} For  \emph{some} $A$--adequate diagram of $K$, the 
corresponding 3--manifold $M_A$ is a book of
  $I$--bundles, with $\chi(M_A)= \chi(\GA) - \chi(\GRA)$.
  \end{enumerate}
\end{theorem}

\begin{proof}
For $\eqref{item:beta=1} \Rightarrow \eqref{item:fibroid-every}$, suppose that $\abs{\beta_K'}=1$, and let $D$ be an $A$--adequate diagram. Then, by Theorem \ref{thm:guts-general} and Lemma \ref{lemma:stabilized}, 
$$\negeul(\guts(M_A)) \: =\: \negeul(\GRA) - ||E_c|| \: = \:  1 - \abs{\beta_K'}  - ||E_c|| \: \leq \: 0.$$ 
Since $\negeul(\cdot) \geq 0$ by definition, it follows that $\chi(\guts(M_A)) = 0$, hence $\guts(M_A) = \emptyset$. In other words, if there are no guts, all of $M_A$ is a book of $I$--bundles. But, by Corollary \ref{cor:beta-fiber}, $M_A$ cannot be an $I$--bundle over $S_A$, because $\abs{\beta'_K} \neq 0$.

\smallskip

$\eqref{item:fibroid-every} \Rightarrow \eqref{item:fibroid-some}$ is trivial.

\smallskip

For $\eqref{item:fibroid-some} \Rightarrow \eqref{item:beta=1}$, suppose that for some $A$--adequate diagram, $M_A$ is a book of $I$--bundles, satisfying $\chi(M_A)= \chi(\GA) - \chi(\GRA)$. 
By equation \eqref{eq:alexander} on page \pageref{eq:alexander}, this means $\chi(\GRA) = 0$. Thus, by Lemma \ref{lemma:stabilized}, $\abs{\beta_K'} = 1$.
\end{proof}

One of the main results in this manuscript is the following theorem, which
shows that $\beta'_K$ monitors the topology of $M_A$ quite effectively.

\begin{theorem} \label{thm:jonesguts}\index{guts!in terms of Jones coefficients}\index{$\beta_K$, $\beta_K'$!relation to guts}
Let $D=D(K)$ be a prime $A$--adequate diagram of a link $K$ with prime
polyhedral decomposition of $M_A = S^3\cut S_A$ and let $\beta'_K$ and
and $\epsilon'_K$ be as in Definition \ref{def:stablevalue}. Then we
have
$$|| E_c||\: + \: \negeul( \guts(M_A))= \:
\abs{\beta'_K}-1+\epsilon'_K,$$ where $||E_c||$ is as in Definition \ref{def:ec}, on page \pageref{def:ec}.
\end{theorem}

There is a similar statement for the stable second coefficient of
$J^n_K(t)$ of $B$--adequate links. If $D=D(K)$ be a prime
$B$--adequate diagram of a link $K$, then
$$|| E_c||\: + \: \negeul( \guts(M_B))= \:
\abs{\beta_K}-1+\epsilon_K,$$ where again $||E_c||$ is the smallest number of complex disks required to span the $I$--bundle of the upper polyhedron, as in Definition \ref{def:ec} on page \pageref{def:ec}.

\smallskip

\begin{proof}
By Theorem \ref{thm:guts-general}, page \pageref{thm:guts-general}, we
have
$$\negeul(\guts(M_A))\: + \: || E_c||=\: \negeul(\GRA).$$
By Definition \ref{def:neg-euler} on page
\pageref{def:neg-euler},
$$ \negeul(\GRA)=\: -\chi(\GRA)\: + \:  \chi_{+}(\GRA),$$
where $ \chi_{+}(\GRA)=1$ if $\GRA$ is tree and 0 otherwise.
By Lemma \ref{lemma:stabilized}(2)
$$\abs{\beta'_K}-1=-\chi(\GRA),$$ which implies that
$\abs{\beta'_K}=0$ if and only if $\GRA$ is a tree. This in turn
implies that $\abs{\beta'_K}=0$ if and only if $ \chi_{+}(\GRA)=1$.
Combining all these we see that the quantity
$$\epsilon'_K:=\negeul(\guts(M_A))\: + \: || E_c||-\abs{\beta'_K}+
1,$$ is equal to 1 if $\abs{\beta'_K}=0$ and 0 otherwise. This proves
the equation in the statement of the theorem.
\end{proof}

A simpler version of Theorem \ref{thm:jonesguts} is
Theorem \ref{thm:no2loops}, which was stated in the
introduction. 

\begin{theorem}\label{thm:no2loops}
Suppose $K$ is an $A$--adequate link whose stable penultimate coefficient is $\beta_K' \neq 0$.
Then, for every $A$--adequate diagram $D(K)$,
$$\negeul( \guts(M_A)) + || E_c ||\: = \: \abs{\beta'_K} - 1, $$
where $|| E_c|| \geq 0$ is as in Definition \ref{def:ec}.
Furthermore, if $D$ is prime and every 2--edge loop in $\GA$ has edges belonging to
the same twist region, then $||E_c|| = 0$ and 
$$\negeul( \guts(M_A)) \: = \: \abs{\beta'_K}- 1. $$
\end{theorem}

\begin{proof}
The first equation of the theorem follows immediately from Theorem \ref{thm:jonesguts}, since $\epsilon_K' = 0$ when $\beta_K' \neq 0$. The second equation of the theorem follows by combining Corollary \ref{cor:onlybigons} with Lemma \ref{lemma:stabilized}, since $|| E_c || =0$ when the edges of each 
 2--edge loop in $\GA$ the
edges belong in the same twist region of the diagram.
\end{proof}

When $|| E_c||=0$, Theorem \ref{thm:jonesguts} provides particularly
striking evidence that coefficients of the Jones polynomials measure
something quite geometric: when $\abs{\beta'_K}$ is large, the link
complement $S^3 \setminus K$ contains essential surfaces that are
correspondingly far from being fibroids.  As a result, if $K$ is
hyperbolic, $S^3 \setminus K$ is forced to have large volume.  As
noted earlier, classes of links with $|| E_c||=0$ include alternating
knots and Montesinos links.  In this case we have the following.

\begin{corollary}\label{cor:monte-exact}\index{guts!for Montesinos links}\index{Montesinos link!guts of state surface}
Suppose $K$ is a Montesinos link with a reduced admissible diagram
$D(K)$ that contains at least three tangles of positive slope.  Then
$$\negeul( \guts(M_A))= \abs{\beta'_K}-1.$$
Similarly, if $D(K)$ contains at least three tangles of negative
slope, then 
 $$\negeul( \guts(M_B))=\abs{\beta_K}-1.$$
\end{corollary}

\begin{proof}
Suppose that $D(K)$ has $r \geq 3$ tangles of positive slope.  Then
Theorem \ref{thm:monteguts} on page \pageref{thm:monteguts} implies
that $\negeul( \guts(M_A)) = \negeul (\GRA)$. Furthermore, observe in
Figure \ref{fig:tangle-decomp} on page \pageref{fig:tangle-decomp}
that the graph $\GA$ contains at least one loop of length $r \geq 3$;
this is the loop that spans every positive tangle west to east.  All
the edges of this loop are distinct in $\GRA$. Thus $\GRA$ contains at
least one non-trivial loop, and is not a tree. Therefore, by Lemma
\ref{lemma:stabilized} on page \pageref{lemma:stabilized},
$$\negeul( \guts(M_A)) \: = \: \negeul(\GRA) \: = \:  - \chi(\GA) \: = \: \abs{\beta'_K}-1.$$

The argument for three negative tangles is identical.
\end{proof}

\section{Hyperbolic volume and colored Jones polynomials}\label{sec:jones-volume}

If the volume conjecture is true, then for large $n$, it would imply a
relation between the volume of a knot complement $S^3 \setminus K$ and coefficients of
$J^n_K(t)$.  For example, for $ n \gg 0$ one would have
$\vol(S^3\setminus K) < C || J^n_K||$, where $|| J^n_K||$ denotes the
$L^1$--norm of the coefficients of $J_K^n(t)$ and $C$ is an
appropriate constant.  A series of articles
written in recent years \cite{dasbach-lin:head-tail, fkp:filling, fkp:conway, fkp:farey}
 has established such relations for several
classes of knots. In fact, in all the known cases the upper bounds on
volume are paired with similar lower bounds. In several cases, our results here
provide an intrinsic and satisfactory explanation for the existence of
the lower bounds.

To illustrate this, let us look at the example of hyperbolic links $K$
that have diagrams $D=D(K)$ that are positive closed braids, such that
each twist region has at least seven crossings.  As before, let
$\beta_K, \beta'_K$ denote the stable second and penultimate
coefficients of of $J^n_K(t)$ (Definition \ref{def:stablevalue}).
Corollary 1.6 of \cite{fkp:filling} states that the quantity
$\abs{\beta_K} + \abs{\beta'_K}$ provides two--sided bounds for the
volume $\vol(S^3 \setminus K)$.  As we saw in the the discussion
before Theorem \ref{thm:positive-volume}, the graph $\GRB$ is a tree,
hence $\beta_K = 0$. Thus the two--sided bound on volume is in terms
of $\abs{\beta'_K}$ alone. However, since the argument of
\cite{fkp:filling} is somewhat indirect (requiring twist number as an intermediate quantity), 
the upper and lower bounds
differ by a factor of about $86$.

Our results in this monograph (Corollary \ref{cor:onlybigons}) reveal
that the quantity $\abs{\beta'_K}-1$ realizes the guts of the
incompressible surface $S_A$; hence, in the light of Theorem
\ref{thm:ast-estimate}, we expect it to show up as a lower bound on
the volume of $S^3\setminus K$. In fact we can now show that
$\abs{\beta'_K}-1$ gives two--sided bounds on the volume of positive
braids that have only three crossings per twist region, rather than
seven. Furthermore, because the argument using guts is more direct and
intrinsic, the factor between the upper and lower bounds is now about
$4.15$.

\begin{corollary}\label{cor:positive-vol-jones}\index{hyperbolic volume!of positive braids}\index{positive braid!hyperbolic volume}\index{braid!hyperbolic volume}\index{hyperbolic volume!in terms of Jones coefficients}\index{$\beta_K$, $\beta_K'$!relation to volume}
Suppose that a hyperbolic link $K$ is the closure of a positive braid
$b=\sigma_{i_1}^{r_1}\sigma_{i_2}^{r_2} \cdots \sigma_{i_k}^{r_k}$,
where $r_j \geq 3$ for all $1\leq j\leq k$. Then
$$v_8 \, ( \abs{\beta'_K}-1 )\: \leq \:\vol(S^3 \setminus K) \: < \: 15 v_3 \, (\abs{\beta'_K}-1) - 10 v_3,$$ 
where $v_3 = 1.0149...$ is the volume of a regular ideal tetrahedron  
and $v_8 = 3.6638...$ is the volume of a regular ideal octahedron.
\end{corollary}

\begin{proof}
By Lemma \ref{lemma:positive-twist}, the graph $\GA$ has no $2$--edge
loops, and $\chi(\GA) = \chi(\GRA) < 0$. Thus, by Corollary
\ref{cor:no2edge-volume},
$$\vol(S^3 \setminus K) \: \geq \: -v_8 \,  \chi(\GRA) \: = \: v_8 (\abs{\beta'_K} -1).$$

For the upper bound on volume, we also use Lemma
\ref{lemma:positive-twist}.  The estimate of that lemma implies that
$$t(D) \: \leq \: - \tfrac{3}{2} \, \chi(\GRA) \: = \: \tfrac{3}{2}
\abs{\beta'_K} - \tfrac{3}{2}.$$
Combined with Agol and D.\ Thurston's bound $\vol(S^3 \setminus K) < 10v_3
(t(D) -1)$, this completes the proof.
\end{proof}

A second result in this vein concerns Montesinos knots and links.

\begin{corollary}\label{cor:jones-volumemonte}\index{hyperbolic volume!of Montesinos links}\index{Montesinos link!hyperbolic volume}\index{hyperbolic volume!in terms of Jones coefficients}\index{$\beta_K$, $\beta_K'$!relation to volume}
Let $K \subset S^3$ be a Montesinos link with a reduced Montesinos
diagram $D(K)$.  Suppose that $D(K)$ contains at least three positive
tangles and at least three negative tangles.  Then $K$ is a hyperbolic
link, satisfying
$$ v_8 \left( \max \{ \abs{\beta_K}, \abs{\beta'_K} \} -1 \right) \:
\leq \: \vol(S^3 \setminus K) \: < \: 4 v_8 \left( \abs{\beta_K} +
\abs{\beta'_K} -2 \right) + 2 v_8 \, (\# K),$$
where $\#K$ is the number of link components of $K$.
\end{corollary}

We remark that the number of link components $\# K$ is recoverable
from the Jones polynomial evaluated at $1$: $J_K(1) = (-2)^{\#K-1}$.
See \cite{jones:announce}.

\begin{proof}
The lower bound on volume is Theorem \ref{thm:ast-estimate} combined
with Corollary \ref{cor:monte-exact}. For the upper bound on volume,
we combine the upper bound of Theorem \ref{thm:monte-volume} with the
estimate of Lemma \ref{lemma:monte-twist}:
\begin{eqnarray*}
\vol(S^3 \setminus K) & < & 2 v_8 \, t(D) \\
& \leq & 2 v_8 \left(  -2 \chi(\GRA) -2 \chi(\GRA) + \# K \right) \\
& = & 4 v_8  \left(  \abs{\beta_K} + \abs{\beta'_K} -2 \right) + 2 v_8 \, (\# K).
\end{eqnarray*}

\vspace{-3ex}
\end{proof}

\chapter{Discussion and questions}\label{sec:questions}
In this final chapter, we state some questions that arose from this
work and speculate about future directions related to this project. In
Section \ref{sec:efficient}, we discuss modifications of the diagram $D$ 
that preserve $A$--adequacy. 
 In Section \ref{sec:control}, we
speculate about using normal surface theory in our polyhedral
decomposition of $M_A$ to attack various open problems, for
example the Cabling Conjecture and the determination of hyperbolic
$A$--adequate knots.  In Section \ref{sec:other-states}, we discuss
extending the results of this monograph to states other than the all--$A$
(or all--$B$) state.  Finally, in Section \ref{sec:coarse}, we discuss
a coarse form of the hyperbolic volume conjecture.

\section{Efficient diagrams}\label{sec:efficient}
To motivate our discussion of diagrammatic moves,  recall  the well-known \emph{Tait conjectures}\index{Tait conjectures for alternating links}
for alternating links:
\begin{enumerate}
\item\label{i:alt-crossing} Any two reduced alternating projections of the same link have the same number of crossings.
\item\label{i:alt-minimal} A reduced alternating diagram of  a link has the least number of crossings among all the projections of the link.
\item\label{i:alt-flyping} Given two reduced, prime alternating  diagrams $D$ and $D'$ of the same link, it is possible to transform $D$ to $D'$ by a finite sequence of \emph{flypes}\index{flyping conjecture}.
\end{enumerate}

Statements \eqref{i:alt-crossing} and \eqref{i:alt-minimal} where proved by Kauffman \cite{KaufJones} and Murasugi \cite{murasugitait} using properties of the Jones polynomial.
A shorter proof along similar lines was given by Turaev \cite{turaevsurface}.
 Statement \eqref{i:alt-flyping}, which is known as the ``flyping conjecture"\index{flyping conjecture} was proven by Menasco and Thistlethwaite \cite{taitflype}. Note that the Jones polynomial is also used in that proof.

One can ask to what extend the statements above can be generalized to semi-adequate links.
It is easy to see that statements \eqref{i:alt-crossing} and \eqref{i:alt-minimal} are not true in this case: For instance, the two diagrams 
in Example \ref{example:bad} on page \pageref{example:bad} are both $A$--adequate, but have different numbers of crossings. Nonetheless, some information is known about crossing numbers of semi-adequate diagrams:
Stoimenow showed that the number of crossings of any  semi-adequate projection of a link is bounded above by a link invariant
that is expressed in terms of the 2--variable Kauffman polynomial and the maximal Euler characteristic of the link.
As a result, he concluded that each semi-adequate link has only finitely many semi-adequate reduced diagrams \cite[Theorem 1.1]{stoimenow}.
In view of his work, it seems natural to ask for analogue of the flyping conjecture in the setting of semi-adequate links.

\begin{problem}\label{prob:adequate-flypes} Find a set of diagrammatic moves that preserve $A$--adequacy and that suffice to pass between any pair of reduced, $A$--adequate diagrams of a link $K$.
\end{problem}

A solution to Problem \ref{prob:adequate-flypes} would help to clarify to what extent the various quantities introduced in this monograph actually depend on the choice of $A$--adequate diagram $D(K)$. 
Recall the prime polyhedral decomposition of $M_A = S^3\cut S_A$ introduced above, and let
$\beta'_K$ and $\epsilon_K$ be as in Definition \ref{def:stablevalue}
on page \pageref{def:stablevalue}.  Since
$\abs{\beta'_K}-1+\epsilon'_K$ is an invariant of $K$, Theorem
\ref{thm:jonesguts} implies that the quantity $||E_c|| +
\negeul(\guts(M_A))$ is also an invariant of $K$.  As noted earlier,
$|| E_c||$ and $\negeul( \guts(M_A))$ are not, in general, invariants
of $K$: they depend on the $A$--adequate diagram used. For instance,
in Example \ref{example:bad} on page \pageref{example:bad}, we show
that by modifying the diagram of a particular link, we can eliminate
the quantity $|| E_c||$.  This example, along with the family of
Montesinos links (see Theorem \ref{thm:monteguts} on page
\pageref{thm:monteguts}), prompts the following question.

\begin{question}\label{quest:allequal}\index{$E_c$!dependence on diagram}
Let $K$ be a non-split, prime $A$--adequate link. Is there an
$A$--adequate diagram $D(K)$, such that if we consider the
corresponding prime polyhedral decomposition of $M_A(D) = S^3\cut
S_A(D)$, we will have $|| E_c||=0$? This would imply that
$$\negeul( \guts(M_A)) \:= \: \negeul (\GRA) \: = \: \abs{\beta'_K}-1+\epsilon'_K \, .$$
\end{question}

Among the more accessible special cases of Question
\ref{quest:allequal} is the following.

\begin{question}\label{quest:montegeneral}
Does Theorem \ref{thm:monteguts} generalize to {\em all} Montesinos
links? That is: can we remove the hypothesis that a reduced diagram
$D(K)$ must contain at least three tangles of positive slope?
\end{question}

Note that if $D(K)$ has no positive tangles, then it is alternating,
hence the conclusion of Theorem \ref{thm:monteguts} is known by
\cite{lackenby:volume-alt}. If $D(K)$ has one positive tangle, then it
is not $A$--adequate by Lemma \ref{lemma:monte-adequate}. Thus
Question \ref{quest:montegeneral} is open only in the case where
$D(K)$ contains exactly two tangles of positive slope.

Another tractable special case of Question \ref{quest:allequal} is the
following.

\begin{question}\label{quest:tanglesgen}
Let $K$ be an $A$--adequate link that can be depicted by a diagram  $D(K)$, obtained
by Conway summation of alternating tangles. Each such link admits a Turaev surface\index{Turaev surface} of genus one \cite{dasbach-futer...}.
Does there exist a (possibly different) diagram of $K$, for which $|| E_c|| = 0$?
\end{question}

Prior to this manuscript, there have been only a few cases in
which the $\guts$ of essential surfaces have been explicitly
understood and calculated for an infinite family of $3$--manifolds
\cite{agol:guts, kuessner:guts, lackenby:volume-alt}.
Affirmative answers to Questions \ref{quest:allequal},
\ref{quest:montegeneral}, and \ref{quest:tanglesgen} would add to the
list of these results, and could have further applications.  In
particular, combined with Theorem \ref{thm:ast-estimate}, they would lead to
new relations between quantum invariants and hyperbolic volume.

\smallskip

Next, recall from the end of Section \ref{sec:mod-diagram}, that given an
$A$--adequate diagram $D:=D(K)$, we denote by $D^n$ the $n$--cabling of
$D$ with the blackboard framing.  If $D$ is $A$--adequate then $D^n$
is $A$--adequate for all $n\in {\NN}$.  Furthermore, we have
$$\chi(\GRA (D^n)) =\chi(\GRA(D)),$$ for all $n\geq 1$.  In other
words, the quantity $\chi(\GRA)$ remains invariant under cabling
\cite[Chapter 5]{lickorish:book}.  Recall, from Corollary \ref{cor:cable} on page
\pageref{cor:cable}, that the quantity $\negeul \guts(S^3 \cut
S^n_A)+ || E_c(D^n)||$ is also invariant under planar cabling. This
prompts the following question.

\begin{question}\label{quest:cableques}\index{cabling}\index{$D^n$, the $n$--cabling of a diagram}\index{guts!stability under cabling}\index{cabling!effect on guts}
Let $D:=D(K)$ be a prime, $A$--adequate diagram, of a link $K$.  For
$n \geq 1$, let $D^n$ denote the $n$--cabling of $D$ using the
blackboard framing.  Is it true that $ || E_c(D^n)||= || E_c(D)||$,
hence $\negeul ( \guts(S^3 \cut S^n_A)) = \negeul (\guts(S^3 \cut
S_A))$, for every $n$ as above?
\end{question}

We note that an affirmative answer to Question \ref{quest:cableques}
would provide an intrinsic explanation for the fact that the
coefficient $\beta'_n$ of the colored Jones polynomials stabilizes.
\smallskip

\section{Control over surfaces}\label{sec:control}

In Chapters \ref{sec:ibundle} and \ref{sec:spanning}, we controlled
pieces of the characteristic submanifold of $M_A$ by putting them in
normal form with respect to the polyhedral decomposition constructed
in Chapter \ref{sec:polyhedra}. The powerful tools of normal surface
theory have been used (sometimes in disguise) to obtain a number of
results about alternating knots and links: see, for example,
\cite{lackenby:volume-alt, lackenby:tunnel, menasco:incompress,
  menasco-thist}. It seems natural to ask what other results in
this vein can be proved for $A$--adequate knots and links.

One sample open problem that should be accessible using these methods
is the following the following problem posed by Ozawa \cite{ozawa}.

\begin{problem}\label{problem:composite}\index{prime!diagram}
Prove that an $A$--adequate knot is prime if and only if \emph{every}
$A$--adequate diagram without nugatory crossings is prime.
\end{problem}

Recall that one direction of the problem is Corollary
\ref{cor:primelinkadequate} on page \pageref{cor:primelinkadequate}:
if $K$ is prime and $D(K)$ has no nugatory crossings, then $D$ must be
prime.  To attack the converse direction of the problem, one might try
showing that if $K$ is not prime, then an $A$--adequate diagram $D(K)$
cannot be prime.

Suppose that $K$ is not prime, and $\Sigma \subset S^3 \setminus K$ is
an essential, meridional annulus in the prime decomposition. Then,
since $S_A$ is also an essential surface, $\Sigma$ can be moved by
isotopy into a position where it intersects $S_A$ in a collection of
essential arcs. Thus, after $\Sigma$ is cut along these arcs, it must
intersect $M_A = S^3 \cut S_A$ in a disjoint union of EPDs. Now, all
the machinery of Chapter \ref{sec:ibundle} can be used to analyze
these EPDs, with the aim of proving that $D$ must not be prime.

The same ideas can be used to attack other problems that depend on an
understanding of ``small'' surfaces in the link complement. For
example, if $\Sigma \subset S^3 \setminus K$ is an essential torus,
then $\Sigma \cap S_A$ must consist of simple closed curves that are
essential on both surfaces. Cutting $\Sigma$ along these curves, we
conclude that $\Sigma \cap M_A$ is a union of annuli, which are
contained in the maximal $I$--bundle of $M_A$. Thus once again, the
machinery of Chapter \ref{sec:ibundle} can be brought to bear: by
Lemma \ref{lemma:squares}, each annulus intersects the polyhedra in
normal squares, and so on. This leads to the following question.

\begin{problem}\label{problem:hyperbolic}
Give characterization of hyperbolic $A$--adequate links in terms of
their $A$--adequate diagrams.
\end{problem}

We are aware of only three families of $A$--adequate diagrams that
depict non-hyperbolic links.  First, the standard diagram of a
$(p,q)$--torus link (where $p > 0$ and $q < 0$) is a negative braid,
hence $A$--adequate by the discussion following Definition
\ref{def:positive-braid} on page \pageref{def:positive-braid}.
Second, by Corollary \ref{cor:primelinkadequate} on page
\pageref{cor:primelinkadequate}, a non-prime $A$--adequate diagram
(without nugatory crossings) must depict a composite link. Third, a
planar cable of (some of the components of) a link $K$ in an
$A$--adequate diagram $D$ also produces an $A$--adequate diagram
$D^n$, but clearly is not hyperbolic. Thus the following na\"ive
question has a chance of a positive answer:

\begin{question}\label{quest:hyperbolic}
Suppose $D(K)$ is a prime $A$--adequate diagram that is not a planar
cable and not the standard diagram of a $(p,q)$--torus link.  Is $K$
necessarily hyperbolic?
\end{question}

A related open problem is the celebrated \emph{Cabling Conjecture},
which implies that a hyperbolic knot $K$ does not have any reducible
Dehn surgeries.  While the conjecture has been proved for large
classes of knots \cite{eudave-munoz:band, luft-zhang:symmetric,
  menasco-thist, scharlemann:reducible}, including all non-hyperbolic knots, 
  it is still a major open
problem.  Note that if a Dehn filling of a knot $K$ does contain an
essential $2$--sphere, then $S^3 \setminus K$ must contain an
essential planar surface $\Sigma$, whose boundary is the slope along
which we perform the Dehn filling.  The Cabling Conjecture asserts
that $K$ must be a cable knot and $\Sigma$ is the \emph{cabling
  annulus}. Given existing work \cite{moser:torus-knot-surgery, scharlemann:reducible}, an equivalent formulation is that hyperbolic knots do not have any reducible surgeries.\index{cabling conjecture}
  
   If $K$ is an $A$--adequate knot, then our results here
provide a nice ideal polyhedral decomposition of associated
3--manifold $M_A$.  It would be interesting to attempt to analyze
essential planar surfaces in $S^3 \setminus K$ by putting them in
normal form with respect to this decomposition, to attack the
following problem.

\begin{problem}\label{problem:cabling}
If $\Sigma$ is an essential planar surface in the complement of an
$A$--adequate knot $K$, show that either $\partial \Sigma$ consists of
meridians of $K$, or $\Sigma$ is a cabling annulus.  That is, prove
the Cabling Conjecture for $A$--adequate knots.
\end{problem}

Recall that the class of $A$--adequate knots is very large; see Section
\ref{subsec:largeclass} on page \pageref{subsec:largeclass}.
Therefore, the resolution of Problem \ref{problem:cabling} would be a
major step toward a proof of the Cabling Conjecture.

\section{Other states}\label{sec:other-states}

As we mentioned in Chapter \ref{sec:decomp}, one may associate many
states to a link diagram.   Any choice of state $\sigma$ defines a
state graph ${\G}_{\sigma}$ and a state surface $S_{\sigma}$ (see also
\cite{fkp:PAMS}).  
A natural and interesting question is: to what
extent do the methods and results of this manuscript generalize to
states other than the all--$A$ and the all--$B$ state?  For example,
one can ask the following question.
\begin{question}\label{quest:find-good-state}\index{$S_\sigma$, state surface of $\sigma$}
Does every knot $K$ admit a diagram $D(K)$ and a state $\sigma$ so
that $S_{\sigma}$ is essential in $S^3 \setminus K$?
\end{question}

As we have seen
in Sections
\ref{subsec:generalization}, \ref{subsec:idealsigma}, 
\ref{sec:gensigmahomo}, and \ref{subsec:spanningsigma}, all of our structural results about the polyhedral decomposition  generalize to state surfaces of
 $\sigma$--homogeneous, $\sigma$--adequate  states. In particular,  the state surface $S_{\sigma}$ of such a state must be essential, recovering Ozawa's Theorem \ref{thm:ozawa}.
In \cite{dasbach-futer...}, Dasbach, Futer, Kalfagianni, Lin, and
Stoltzfus show that for any diagram $D(K)$, the entire Jones
polynomial $J_K(t)$ can be computed from the  Bollob\'as--Riordan polynomial \cite{bo-ri, bo-ri1}
of the \emph{ribbon graph}
associated to the all--$A$ graph $\GA$ or the all--$B$ graph
$\GB$. It is natural to ask whether these results extend to
other states.
\begin{question}
Let $D(K)$ be a link diagram that is  $\sigma$--adequate and
$\sigma$--homogeneous. Does the
Bollob\'as--Riordan polynomial of the graph $\G_\sigma$
associated to $\sigma$ carry all of the information in the Jones
polynomial of $K$? How do these polynomials relate to the topology of
the  state surface $S_\sigma$?
\end{question}

\section{A coarse volume conjecture}\label{sec:coarse}

Our results here, as well as several recent articles
\cite{dasbach-lin:head-tail, fkp:filling, fkp:conway, fkp:farey}, have
established two--sided bounds on the hyperbolic volume of a link
complement in terms of coefficients of the Jones and colored Jones
polynomials. These results motivate the following question.

\begin{define}\label{def:coarse}
Let $f,g : Z \to \RR_+$ be functions from some (infinite) set $Z$ to
the non-negative reals. We say that $f$ and $g$ are \emph{coarsely
  related}\index{coarsely related} if there exist universal constants
$C_1 \geq 1$ and $C_2 \geq 0$ such that
$$C_1^{-1} f(x) - C_2 \: \leq \: g(x) \: \leq \: C_1 f(x) + C_2 \quad
\forall x \in Z.$$ This notion is central in coarse geometry. For
example, a function $\varphi: X \to Y$ between two metric spaces is a
\emph{quasi-isometric embedding}\index{quasi-isometric embedding} if
$d_X(x,x')$ is coarsely related to $d_Y(\varphi(x), \varphi(x'))$. Here, $Z = X \times X$.
\end{define}

\begin{question}[Coarse Volume Conjecture] \label{quest:voljp}\index{coarse volume conjecture}\index{volume conjecture!coarse}\index{hyperbolic volume!coarse volume conjecture}\index{$\beta_K$, $\beta_K'$!relation to volume}
Does there exist a function $B(K)$ of the coefficients of the colored
Jones polynomials of a knot $K$, such that for hyperbolic knots,
$B(K)$ is coarsely related to hyperbolic volume?
  
Here, we are thinking of both $\vol : Z \to \RR_+$ and $B : Z \to
\RR_+$ as functions on the set $Z$ of hyperbolic knots.
\end{question}

Work of Garoufalidis and Le \cite{garoufalidisLe} implies that for a
given link $K$, the sequence $\{ J^n_K(t)| n\in \NN \}$ is determined
by finitely many values of $n$. This implies that the coefficients
satisfy linear recursive relations with constant coefficients
\cite{garquasi}.  For $A$--adequate links, the recursive relations
between coefficients of $J^n_K(t)$ manifest themselves in the
stabilization properties discussed in Lemma \ref{lemma:stabilized} on
page \pageref{lemma:stabilized}, and Definition \ref{def:stablevalue}
on page \pageref{def:stablevalue}.  Lemma \ref{lemma:stabilized} is
not true for arbitrary knots. However, numerical evidence and
calculations (by Armond, Dasbach, Garoufalidis, van der Veen, Zagier,
etc.)
prompt the question of whether the first and last two coefficients of
$J^n_K(t)$ ``eventually" become periodic.

\begin{question} \label{quest:periodic}
Given a knot $K$, do there exist a ``stable'' integer $N = N(K)>0$ and a ``period" $p=p(K)>0$, depending
on $K$, such that for all $m \geq N$ where $m-N$ is a multiple of $p$,
$$ \abs{\alpha_m} = \abs{\alpha_N}, \quad \abs{\beta_m}=\abs{\beta_N}, \quad
\abs{\beta'_m}=\abs{\beta'_N}, \quad  \abs{ \alpha'_m }=\abs{ \alpha'_N} \ ? $$
\end{question}

As discussed above,  for knots that are both $A$ and $B$--adequate, any integer $N\geq2$ is ``stable" with period $p=1$.
Examples show that in general, we cannot hope that $p=1$ for arbitrary knots. For example, \cite[Proposition 6.1]{codyoliver} states that for 
several families of torus knots we have $p=2$.
In general, 
if the answer to Question \ref{quest:periodic} is \emph{yes}, then we
if we take $N$ to be the smallest ``stable" integer then we may consider the $4p$ values
\begin{equation}\label{eq:periodic}
 \abs{\alpha_m}, \quad \abs{\beta_m}, \quad  \abs{\beta'_m}, \quad  
 \abs{\alpha'_m} , \quad {\rm for} \quad  N\leq m\leq N+p-1.
 \end{equation}
The results \cite{dasbach-lin:head-tail, fkp:filling, fkp:conway,
  fkp:farey}, as well as Corollary \ref{cor:jones-volumemonte} in Chapter \ref{sec:applications}, prompt the question of whether this family of coefficients of $J^n_K(t)$ 
determines the volume of $K$ up to a bounded constant.

%

\begin{question}\label{quest:volconj-special}\index{stable coefficients!relation to volume}
Suppose the answer  to Question \ref{quest:periodic} is \emph{yes}, and the stable values 
$ \abs{\alpha_m}$ $\abs{\beta_m}$, $\abs{\beta'_m}$, $\abs{\alpha'_m}$ of equation \eqref{eq:periodic} are well--defined.
Is there a function $B(K)$ of these stable coefficients
that is coarsely related to the hyperbolic volume $\vol(S^3\setminus K)$?
\end{question}

\begin{remark} If $K$ is an alternating knot then $ \beta_K, \beta'_K $
are equal to the second and penultimate coefficient  of the ordinary Jones polynomial
$J_K(t)$, respectively.  Since the quantity $\abs{ \beta_K}+ \abs{\beta'_K }$ provides two sided bounds on the volume of
hyperbolic alternating links one may wonder whether there is a function of the second and the penultimate coefficient of 
$J_K(t)$ that controls the volume of all hyperbolic knots $K$. In \cite[Theorem 6.8]{fkp:farey} we show that is is not the case. That is
there is no single function of the the second and the penultimate coefficient of 
the Jones polynomial that can control the volume of all hyperbolic knots.
\end{remark}
Finally, we note that the quantity on the right-hand side of the
equation in the statement of Theorem \ref{thm:jonesguts} can be
rewritten in the form $\abs{\beta'_K}-\abs{\alpha'_K}+\epsilon'_K$.
In the view of this observation, it is tempting to ask whether
analogues of Theorem \ref{thm:jonesguts} on page
\pageref{thm:jonesguts} hold for \emph{all} knots.

\begin{question}\label{quest:jonesguts-analogue}
Given a knot $K$ for which the stable coefficients of Question \ref{quest:periodic} exist, is there an essential spanning surface $S$ with
boundary $K$ such that the stable coefficients $\alpha'_K, \beta'_K$
capture the topology of $S^3\cut S$ in the sense of Theorem
\ref{thm:jonesguts}, Corollary \ref{cor:beta-fiber}, and Theorem
\ref{thm:fibroid-detect}?
\end{question}

\bibliographystyle{hamsplain}
\bibliography{biblio}
\printindex
\end{document}